\documentclass[11pt,letterpaper]{amsart}
\usepackage[
    paperwidth=8.5in,
    paperheight=11in,
    margin=1.25in,
    marginparwidth=0.75in,
    marginparsep=0.1in,
    footskip=0.25in
]{geometry}

\usepackage[shortlabels]{enumitem}
\setlist{
    itemsep=1ex,
    listparindent=\parindent,
    parsep=0em,
    topsep=2ex,
}
\allowdisplaybreaks[1]

\usepackage{xpatch}
\usepackage{amsfonts}
\usepackage{amsmath}
\usepackage{amssymb}
\usepackage{amsthm,thmtools,thm-restate}
\usepackage{bm}
\usepackage{mathrsfs}
\usepackage{mathtools}

\usepackage{booktabs}
\usepackage{comment}
\usepackage{float}
\usepackage{graphicx}
\usepackage{multirow}
\usepackage{tikz,pgfplots}
\pgfplotsset{compat=1.18}
\usepgfplotslibrary{groupplots}
\usetikzlibrary{calc}

\usepackage{silence}
\usepackage{caption}
\usepackage[
    backend=biber,
    style=numeric,
    giveninits=true,
    sorting=nyt,
    maxnames=100,
    backref=false,
    doi=false,
    url=false,
    isbn=false,
]{biblatex}
\renewbibmacro{in:}{}
\DeclareFieldFormat[
	article,
	inbook,
	incollection,
	inproceedings,
	patent,
	thesis,
	unpublished
]{title}{#1}

\newcommand\linkcolor{black}

\usepackage{hyperref}
\hypersetup{
    colorlinks,
    linktoc=all,
    citecolor=black,
    linkcolor=\linkcolor,
    urlcolor=\linkcolor
}
\usepackage[capitalize,nameinlink,noabbrev]{cleveref}
\crefname{claim}{Claim}{Claims}
\crefname{fact}{Fact}{Facts}
\crefname{inequality}{Inequality}{Inequalities}
\creflabelformat{inequality}{#2(#1)#3}
\crefname{lemma}{Lemma}{Lemmas}
\crefname{subsection}{Subsection}{Subsections}

\makeatletter
\def\cref@short@range@type{}
\def\cref@lemma@type{lemma}
\def\cref@section@type{section}
\newcommand{\cref@reinsert@short@range}[1]{%
  \cref@isstackfull{\cref@consecutive@stack}%
  \@whilesw\if@cref@stackfull\fi{%
    \edef\@tempa{\cref@stack@top{\cref@consecutive@stack}}%
    \expandafter\cref@stack@push\expandafter{\@tempa}{#1}%
    \cref@stack@pop{\cref@consecutive@stack}%
    \cref@isstackfull{\cref@consecutive@stack}}}
\xpatchcmd{\cref@processconsecutive}
  {\let#4\relax%
  #5=1\relax}
  {\let#4\relax%
  #5=1\relax%
  \cref@stack@init{\cref@consecutive@stack}}
  {}
  {\PackageWarning{preamble}{Could not initialize cleveref range stack}}
\xpatchcmd{\cref@processconsecutive}
  {\advance#5 1\relax%
      \let#4\@nextref%
      \cref@stack@pop{#2}}
  {\advance#5 1\relax%
      \let#4\@nextref%
      \expandafter\cref@stack@push\expandafter{\@nextref}{\cref@consecutive@stack}%
      \cref@stack@pop{#2}}
  {}
  {\PackageWarning{preamble}{Could not record cleveref range members}}
\xpatchcmd{\@cref}
  {\ifnum\count@consecutive=2\relax%
        \expandafter\cref@stack@push\expandafter{\@endref,}{\@refsubstack}%
        \let\@endref\relax%
        \count@consecutive=1\relax%
      \fi}
  {\ifnum\count@consecutive>1\relax%
        \ifnum\count@consecutive=2\relax%
          \cref@reinsert@short@range{\@refsubstack}%
          \let\@endref\relax%
          \count@consecutive=1\relax%
        \else%
          \ifnum\count@consecutive=3\relax%
            \expandafter\cref@gettype\expandafter{\@beginref}{\cref@short@range@type}%
            \ifx\cref@short@range@type\cref@lemma@type%
              \cref@reinsert@short@range{\@refsubstack}%
              \let\@endref\relax%
              \count@consecutive=1\relax%
            \else%
              \ifx\cref@short@range@type\cref@section@type%
                \cref@reinsert@short@range{\@refsubstack}%
                \let\@endref\relax%
                \count@consecutive=1\relax%
              \fi%
            \fi%
          \fi%
        \fi%
      \fi}
  {}
  {\PackageWarning{preamble}{Could not adjust cleveref range threshold}}
\makeatother

\theoremstyle{plain}
\newtheorem{theorem}{Theorem}[section]
\newtheorem{claim}[theorem]{Claim}
\newtheorem{corollary}{Corollary}[theorem]
\newtheorem{fact}[theorem]{Fact}
\newtheorem{lemma}[theorem]{Lemma}
\newtheorem{proposition}[theorem]{Proposition}
\newtheorem{example}[theorem]{Example}
\newtheorem{definition}[theorem]{Definition}

\theoremstyle{definition}
\newtheorem{conjecture}[theorem]{Conjecture}
\newtheorem{remark}[theorem]{Remark}

\newcommand\E{{\mathbb E}}
\newcommand\N{{\mathbb N}}
\newcommand\bbP{{\mathbb P}}
\newcommand\R{{\mathbb R}}
\newcommand\Z{{\mathbb Z}}
\newcommand\bsone{\boldsymbol{1}}
\newcommand\cB{{\mathcal B}}
\newcommand\cC{{\mathcal C}}
\newcommand\cE{{\mathcal E}}
\newcommand\cF{{\mathcal F}}
\newcommand\cG{{\mathcal G}}
\newcommand\cH{{\mathcal H}}
\newcommand\cI{{\mathcal I}}
\newcommand\cJ{{\mathcal J}}
\newcommand\cK{{\mathcal K}}
\newcommand\cM{{\mathcal M}}
\newcommand\cN{{\mathcal N}}
\newcommand\cP{{\mathcal P}}
\newcommand\cR{{\mathcal R}}
\newcommand\cS{{\mathcal S}}
\newcommand\cT{{\mathcal T}}
\newcommand\cV{{\mathcal V}}
\newcommand\cW{{\mathcal W}}
\newcommand\cX{{\mathcal X}}
\newcommand\cY{{\mathcal Y}}
\newcommand\cZ{{\mathcal Z}}
\newcommand\frp{{\mathfrak p}}
\newcommand\frP{{\mathfrak P}}
\newcommand\frt{{\mathfrak t}}
\newcommand\sB{{\mathscr B}}
\newcommand\sD{{\mathscr D}}
\newcommand\sE{{\mathscr E}}
\newcommand\sP{{\mathscr P}}
\newcommand\sR{{\mathscr R}}
\newcommand{\red}{{\normalfont\texttt r}}
\newcommand{\green}{{\normalfont\texttt g}}
\newcommand{\blue}{{\normalfont\texttt b}}

\newcommand\erdos{Erd\H{o}s}
\newcommand\renyi{R\'enyi}
\newcommand\erdosrenyi{\erdos--\renyi}

\makeatletter
\newcommand{\nHookrightarrow}{\mathrel{\mathpalette\@nHookrightarrow\relax}}
\newcommand{\@nHookrightarrow}[2]{%
  \begin{tikzpicture}[baseline={(H.base)}]
    \node[inner sep=0pt, outer sep=0pt] (H) {$#1\hookrightarrow$};
    \draw[line cap=round, line width=.07em]
      ([xshift=.425em,yshift=.15ex]H.south west) --
      ([xshift=-.525em,yshift=-.055ex]H.north east);
  \end{tikzpicture}%
}
\makeatother

\pgfmathsetmacro{\opacity}{50}
\colorlet{lightred}{red!\opacity}
\colorlet{lightgreen}{green!\opacity}
\definecolor{lightblue}{RGB}{144,202,249}

\newcommand{\grg}{\tikz[baseline=-.6ex,x=1em,y=1em]{%
  \draw[lightred, line width=3pt,line cap=round] (0,0)--(1.25,0);
  \fill[lightgreen] (0,0) circle (3pt);
  \fill[lightgreen] (1.25,0) circle (3pt);
}}

\newcommand{\gbg}{\tikz[baseline=-.6ex,x=1em,y=1em]{%
  \draw[lightblue, line width=3pt,line cap=round] (0,0)--(1.25,0);
  \fill[lightgreen] (0,0) circle (3pt);
  \fill[lightgreen] (1.25,0) circle (3pt);
}}

\newcommand{\brb}{\tikz[baseline=-.6ex,x=1em,y=1em]{%
  \draw[lightred, line width=3pt,line cap=round] (0,0)--(1.25,0);
  \fill[lightblue] (0,0) circle (3pt);
  \fill[lightblue] (1.25,0) circle (3pt);
}}

\newcommand{\rrbtriangle}{\tikz[baseline=-0.15ex,x=1em,y=1em,scale=0.75]{%
  \coordinate (A) at (0,0);
  \coordinate (B) at (1.25,0);
  \coordinate (C) at (.625,1.08);
  \draw[lightblue,line width=2.5pt,line cap=round] (A)--(B); 
  \draw[lightred,line width=2.5pt,line cap=round] (B)--(C); 
  \draw[lightred,line width=2.5pt,line cap=round] (C)--(A); 

  \fill[lightblue] (C) circle (3pt);
  \filldraw[fill=white, draw=black, very thin] (A) circle (3pt);
  \filldraw[fill=white, draw=black, very thin] (B) circle (3pt);
}}

\newcommand{\rrgtriangle}{\tikz[baseline=-0.15ex,x=1em,y=1em,scale=0.75]{%
  \coordinate (A) at (0,0);
  \coordinate (B) at (1.25,0);
  \coordinate (C) at (.625,1.08);
  \draw[lightgreen,line width=2.5pt,line cap=round] (A)--(B); 
  \draw[lightred,line width=2.5pt,line cap=round] (B)--(C); 
  \draw[lightred,line width=2.5pt,line cap=round] (C)--(A); 

  \fill[lightgreen] (C) circle (3pt);
  \filldraw[fill=white, draw=black, very thin] (A) circle (3pt);
  \filldraw[fill=white, draw=black, very thin] (B) circle (3pt);
}}

\newcommand{\exampletriangle}{\tikz[baseline=-0.15ex,x=1em,y=1em,scale=0.75]{%
  \coordinate (A) at (0,0);
  \coordinate (B) at (1.25,0);
  \coordinate (C) at (.625,1.08);
  \draw[lightred,line width=2.5pt,line cap=round] (A)--(B); 
  \draw[lightred,line width=2.5pt,line cap=round] (B)--(C); 
  \draw[lightred,line width=2.5pt,line cap=round] (C)--(A); 

  \fill[lightgreen] (A) circle (3pt);
  \fill[lightgreen] (B) circle (3pt);
  \fill[lightblue] (C) circle (3pt);
}}

\makeatletter
\newcommand{\vast}{\bBigg@{3}}
\makeatother

\newcommand\blambda{{\boldsymbol{\lambda}}}
\newcommand\bLambda{{\boldsymbol{\Lambda}}}
\newcommand\bsa{\boldsymbol{a}}
\newcommand\bsm{\boldsymbol{m}}
\newcommand{\deltatilde}{\tilde{\delta}}

\newcommand\close{{\hspace{0.2mm}\ensuremath{\mathsf{close}}}}
\newcommand\far{{\hspace{0.2mm}\ensuremath{\mathsf{far}}}}
\newcommand\high{{\hspace{0.2mm}\ensuremath{\mathsf{hi}}}}
\newcommand\mat{{\hspace{0.2mm}\ensuremath{\mathsf{mat}}}}
\newcommand\ind{{\hspace{0.2mm}\ensuremath{\mathrm{ind}}}}
\newcommand\spl{\ensuremath{\mathrm{spl}}}
\newcommand\degenerate{{\hspace{0.2mm}\ensuremath\mathsf{deg}}}
\newcommand\med{{\hspace{0.2mm}\ensuremath\mathsf{med}}}
\newcommand\res{{\hspace{0.2mm}\ensuremath\mathsf{res}}}

\newcommand\nar{{\ensuremath\mathsf{nar}}}
\renewcommand\root{{\ensuremath\mathsf{root}}}
\newcommand\leftsf{{\hspace{0.2mm}\ensuremath\mathsf{left}}}
\newcommand\insf{{\hspace{0.2mm}\ensuremath\mathsf{in}}}
\newcommand\outsf{{\hspace{0.2mm}\ensuremath\mathsf{out}}}
\newcommand\Tail{{\ensuremath\mathsf{Tail}}}
\newcommand\Upper{{\ensuremath\mathsf{U}}}
\newcommand\Lower{{\ensuremath\mathsf{L}}}
\newcommand\bin{{\ensuremath\mathsf{bin}}}
\newcommand\Ent{{\ensuremath\mathsf{Ent}}}
\newcommand\Err{{\ensuremath\mathsf{Err}}}
\newcommand\ret{{\ensuremath\mathsf{ret}}}
\newcommand\matsf{{\ensuremath\mathsf{mat}}}
\newcommand\lgsf{{\ensuremath\mathsf{lg}}}
\newcommand\smsf{{\ensuremath\mathsf{sm}}}
\newcommand\sparse{{\ensuremath\mathsf{sp}}}

\DeclareMathOperator{\argmax}{argmax}
\DeclareMathOperator{\image}{image}
\DeclareMathOperator{\Bin}{Bin}
\DeclareMathOperator{\rand}{rand}

\DeclarePairedDelimiter\ceil{\lceil}{\rceil}
\DeclarePairedDelimiter\floor{\lfloor}{\rfloor}
\DeclarePairedDelimiter\abs{\lvert}{\rvert}
\DeclarePairedDelimiter\norm{\lVert}{\rVert}
\renewcommand\geq{\geqslant}
\renewcommand\leq{\leqslant}

\newcommand{\clean}{{\mathrm{cl}}}
\newcommand{\type}{{\mathrm{ty}}}

\title[Excluding an induced star in dense random graphs]{Excluding an induced star in dense random graphs}

\author{Sam van der Poel}
\address{Georgia Institute of Technology}
\email{samvanderpoel@gatech.edu}

\begin{document}

\begin{abstract}
For fixed $k\geq3$, we study the asymptotic number and typical structure of dense graphs with no induced copy of the star $K_{1,k}$. We solve the associated graphon variational problems both at fixed constant edge density $\gamma$ and for the conditioned Erd\H{o}s--R\'enyi random graph $G(n,p)$ for constant $p$. As consequences, we obtain explicit formulas for the entropy density of induced-$K_{1,k}$-free graphs with $\Theta(n^2)$ edges and for the large deviation rate function for the event that $G(n,p)$ is induced-$K_{1,k}$-free. The entropy density exhibits a second-order phase transition at an explicit critical density $\gamma_k$, while the rate function exhibits a first-order phase transition at a critical parameter $p_k$.

We completely characterize the optimizers of both variational problems. Both models have parameter values for which there are infinitely many optimal graphons, but there is always a unique graphon that represents the typical structure in cut metric. We refine the graphon-level results by giving a detailed structural description of both models. For supercritical parameters, each random graph model is the complement of a $(k-1)$-partite graph with high probability. In the subcritical regime of the fixed-density model, the typical structure is the disjoint union of the complement of a $(k-1)$-partite graph, and a sparse remainder. In the subcritical regime of the conditioned Erd\H{o}s--R\'enyi random graph, a typical sample has $o(n^2)$ edges.
\end{abstract}

\maketitle

\section{Introduction}
\label{sec:intro}
Determining the structure of a typical $H$-free graph is a problem at the heart of combinatorics.
Because it lies at the interface of several subjects including asymptotic enumeration, extremal graph theory, and nonlinear large deviations, the problem has bridged different areas and given new perspectives on classical results.
An example is the \citeyear{erdos1976enumeration} theorem of \citeauthor{erdos1976enumeration} \cite{erdos1976enumeration} stating that almost all triangle-free graphs are bipartite. This theorem introduced the idea that typical $H$-free graphs closely resemble a subgraph of an extremal $H$-free graph. In the case of triangle-free graphs, Mantel's theorem \cite{mantel1907vraagstuk} states that the extremal structure is complete balanced bipartite.
Powerful results have since shown that the asymptotic enumeration and typical structure of graph properties often reduce to a related optimization problem. This paper is about the solution to a graphon variational problem and a finer counting argument that yield a precise structural description of typical graphs of given density not containing an induced star.

A graph is \emph{induced-$H$-free} if none of its induced subgraphs is isomorphic to $H$. \citeauthor{promel1991excluding} pioneered the study of typical induced-$H$-free graphs in papers on induced-$C_4$-free graphs \cite{promel1991excluding} and induced-$C_5$-free graphs \cite{promel1992almost}. These papers are notable for determining the first-order asymptotics of the number of graphs in question. \citeauthor{promel1992excluding} \cite{promel1992excluding} also determined the first-order asymptotics of the logarithm of the number of induced-$H$-free graphs for \emph{all} $H$ in terms of the coloring number of $H$, a quantity we define below. \citeauthor{alekseev1993entropy} \cite{alekseev1993entropy} and \citeauthor{bollobas1997hereditary} \cite{bollobas1997hereditary} generalized this to hereditary properties, and \citeauthor{alon2011structure} \cite{alon2011structure} strengthened it to a rough structure result for hereditary properties. \citeauthor{balogh2011excluding} \cite{balogh2011excluding} proved a precise typical structure result for all $H$ that satisfy a certain notion of criticality, and this has been extended by \citeauthor{norin2025typicalI} \cite{norin2025typicalI,norin2025typicalII}. Detailed typical structure results have also been proven for induced-$C_{2\ell}$-free graphs \cite{kim2018forbidding} and induced-$(K_{a+b}\setminus K_b)$-free graphs \cite{keevash2017structure}.

The typical structure of $H$-free graphs with a given number of edges $m=m(n)$ was also first studied by \citeauthor{promel1996asymptotic}. Refining \erdos{}--Kleitman--Rothschild, they proved that if $m=\Omega(n^{7/4}\log n)$ then almost all triangle-free graphs on $n$ vertices and $m$ edges are bipartite \cite{promel1996asymptotic}. Today, the typical structure of triangle-free graphs \cite{osthus2003densities,steger2005evolution,stark2018probability,jenssen2025evolution} and more generally $K_{r+1}$-free graphs \cite{luczak2000triangle,balogh2009typical,balogh2015independent,balogh2016typical} with $m$ edges are well understood. By contrast, induced-$H$-free graphs with $m$ edges have received much less attention. \citeauthor{bottcher2012perfect} \cite{bottcher2012perfect} proved that induced-$C_5$-free graphs and perfect graphs of fixed constant density have the same exponential growth rate. \citeauthor{morris2024asymmetric} \cite{morris2024asymmetric} proved that if $n^{4/3}(\log n)^4\leq m \leq o(n^2)$ then almost all induced-$C_4$-free graphs with $m$ edges are within $o(m)$ edit distance of a split graph, while if $m=o(n^{4/3}(\log n)^{1/3})$ then this is not the case. Recently, Perkins and the author \cite{perkins2025typical} proved that claw-free graphs of fixed constant edge density exhibit a phase transition, characterized the optimizers of the associated graphon variational problem, and proved detailed typical structure results.

From another perspective, conditioning the \erdosrenyi{} random graph $G(n,p)$ to be $H$-free places the problem in lower-tail large deviations for subgraph counts. Writing $X_H$ for the number of copies of $H$ in $G(n,p)$, the lower-tail event is $X_H\leq(1-\eta)\E X_H$ for a parameter $0<\eta\leq1$.
A seminal result of \citeauthor{chatterjee2011large} \cite{chatterjee2011large} implies as a special case that for constant $p$, the lower-tail \emph{rate function} (i.e.\ the constant in the leading term of the exponent of the lower-tail probability) is given by a variational problem over graphons.
\citeauthor{zhao2017lower} \cite{zhao2017lower} proved that the lower-tail variational problem exhibits a phase transition in the sense of a non-analyticity of the rate function in $p$ and $\eta$: for small $\eta$, the optimizer is constant, while for $\eta$ close to 1, this is not the case. A full characterization of the optimizers for all $p$ and $\eta$ remains a difficult open problem, even for triangles. \citeauthor{kozma2023lower} \cite{kozma2023lower} proved that the lower-tail rate function is still given by a variational problem down to $p=\Omega(n^{-1/m_2(H)})$, where $m_2(H)$ is the 2-density
of $H$, which is the full range of densities where the reduction could hold. A similar reduction for the rate function of the lower tail for \emph{induced} copies is only known in the dense setting due to \citeauthor{chatterjee2011large}.

In this paper, we completely solve the variational problem over graphons with zero induced $K_{1,k}$ density and fixed constant edge density. The solution yields the entropy density of induced-$K_{1,k}$-free graphs of given edge density, the large deviation rate function for induced-$K_{1,k}$-freeness in $G(n,p)$ for constant $p$, and a rough structural description of both random graph models. Most notably, the variational problem exhibits a phase transition that is second-order in the fixed-density model and first-order in the conditioned \erdosrenyi{} random graph. Both models have infinitely many optimal graphons for some parameter values, but there is always a unique graphon that represents the typical structure in cut metric. A more detailed analysis of the counts and probabilities yields a finer structural characterization of a typical sample from each model. Finally, we show that the induced-$H$-free fixed-density variational problem does not necessarily exhibit a phase transition in general: we prove that almost all induced-$C_4$-free graphs of fixed constant edge density are split graphs, resolving the dense case of a conjecture of \citeauthor{morris2024asymmetric} \cite[Conjecture~6.1]{morris2024asymmetric} and refining a theorem of \citeauthor{promel1991excluding} \cite{promel1991excluding}.

\subsection{Main results}
For a graph $H$, let $\cF^\ast_{n,m}(H)$ be the set of induced-$H$-free graphs on $n$ vertices and $m$ edges, and let $N^\ast_{n,m}(H):=\abs{\cF^\ast_{n,m}(H)}$. The \emph{entropy density} of a property $\cP(n)$ of graphs is defined as $\lim_{n\to\infty}\binom{n}{2}^{-1}\log|\cP(n)|$, and its \emph{large deviation rate function} is defined as $-\lim_{n\to\infty}\binom{n}{2}^{-1}\log\bbP\{G(n,p)\in\cP(n)\}$, provided the limits exist. For an integer $r\geq2$, a graph is \emph{co-$r$-partite} if it is the complement of an $r$-partite graph.

Our first main result establishes the entropy density of $\cF^\ast_{n,m}(K_{1,k})$ for $m=\Theta(n^2)$. For every integer $k\geq3$, let $p_k\in(0,1)$ be the unique solution to the equation $p_k=(1-p_k)^{k-1}$. Define the constant
\begin{equation}\label{eqn:gammak-dfn}
\gamma_k:=\frac{1+(k-2)p_k}{k-1}
\end{equation}
and define the function $\sE_k:[0,1]\to\R$ by
\begin{equation}\label{eqn:entropy-dfn}
\sE_k(x):=\begin{cases}\frac{(k-2)x}{1+(k-2)p_k}\,H(p_k)&0\leq x\leq\gamma_k\\[5pt]\big(1-\frac{1}{k-1}\big)\,H\!\left(\frac{(k-1)x-1}{k-2}\right)&\gamma_k\leq x\leq1\end{cases}\,,
\end{equation}
where $H(x)=-x\log x-(1-x)\log(1-x)$ is the binary entropy. The notation $f(n)\sim g(n)$ means $\lim_{n\to\infty}f(n)/g(n)=1$.

\begin{theorem}\label{thm:main-entropy}
Fix constants $k\geq3$ and $\gamma\in(0,1)$. If $m\sim\gamma\binom{n}{2}$ then
\begin{equation*}
\lim_{n\to\infty}\frac{1}{\binom{n}{2}}\log N^\ast_{n,m}(K_{1,k})=\sE_k(\gamma) \,.
\end{equation*}
\end{theorem}

The graph of $\sE_k$ is shown on the left axis of \Cref{fig:rate-funcs} for several values of $k$. Since $\sE_k$ has a continuous first derivative and a discontinuous second derivative at $\gamma_k$, the uniformly random graph $G\in\cF^\ast_{n,m}(K_{1,k})$ is said to exhibit a \emph{second-order phase transition} at edge density $\gamma_k$. See \cite{jaeger1998ehrenfest} for background on the Ehrenfest classification of phase transitions. To understand the entropy density $\sE_k(\gamma)$, note that the convex piece on the interval $[\gamma_k,1]$ is the entropy density of co-$(k-1)$-partite graphs of edge density $\gamma+o(1)$, while the linear piece on the interval $[0,\gamma_k]$ is the entropy density of graphs of edge density $\gamma+o(1)$ that are the disjoint union of a co-$(k-1)$-partite graph and an empty remainder.

In the next theorem, we present the rate function for induced-$K_{1,k}$-freeness. For all $k\geq3$ define the function $\sR_k:[0,1]\to\R$ by
$$\sR_k(x):=\begin{cases}\log(\frac{1}{1-x})&0\leq x\leq p_k\\[4pt]\frac{1}{k-1}\log(\frac{1}{x})&p_k\leq x\leq1\end{cases}\,.$$
Let us write $H\leq G$ to declare that $H$ is an induced subgraph of $G$.

\begin{theorem}\label{thm:main-rate}
Fix constants $k\geq3$ and $p\in(0,1)$. Then
\begin{equation*}
\lim_{n\to\infty}\frac{1}{\binom{n}{2}}\log\bbP\{G(n,p)\not\geq K_{1,k}\}=-\,\sR_k(p) \,.
\end{equation*}
\end{theorem}

To understand the rate function $\sR_k(p)$, note that the supercritical piece on $[p_k,1]$ is the rate function for the event that $G(n,p)$ is co-$(k-1)$-partite, while the subcritical piece on $[0,p_k]$ is the rate function for the event that $G(n,p)$ is empty. In fact, \Cref{thm:main-rate} can be read off from results of \citeauthor{thomason2011graphs} \cite{thomason2011graphs} (specifically, by combining Theorem 4.4 and Section 7.5), who studied extremal problems that complement those of this paper. Connections to this prior literature are discussed further below.

\begin{figure}
	\begin{center}
	\includegraphics{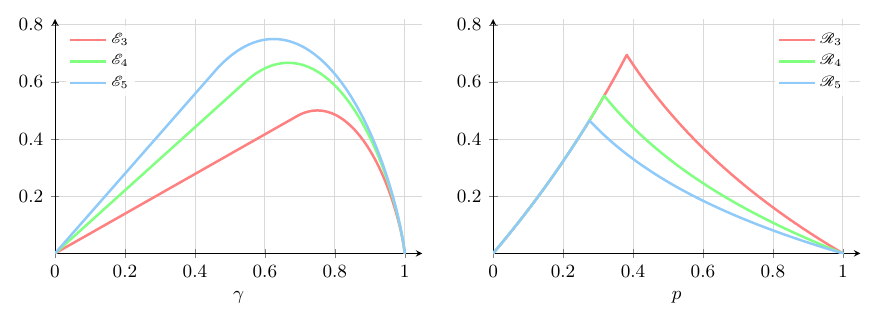}
	\end{center}
	\caption{The entropy density $\sE_k$ (left) and rate function $\sR_k$ (right) for $k=3,4,5$. The curves $\sR_3$, $\sR_4$, and $\sR_5$ coincide on the interval $[0,p_5]$.}
	\label{fig:rate-funcs}
\end{figure}

The most important feature of both \Cref{thm:main-entropy,thm:main-rate} is that the entropy density and rate function are given by piecewise formulas, reflecting that there are competing structural mechanisms that determine the exponential rates. These competing structures were heuristically described above by noting that the entropy densities or rate functions match the formulas given in $\sE_k$ and $\sR_k$. Our next set of results show that, in fact, these heuristics precisely characterize the typical structure of the corresponding random graph models.

Here and throughout, \emph{almost all} and \emph{almost every} mean all but a $o(1)$ fraction.

\begin{theorem}\label{thm:main-almostall}
Fix constants $k\geq3$ and $\gamma\in(0,1)$. If $m\sim\gamma\binom{n}{2}$ then the following hold:
\begin{enumerate}[
  label=\textit{(\roman*)},
  ref=(\textit{\roman*}),
  topsep=5pt
]
	\item\label{item:thm:main-almostall-super} If $\gamma\in(\gamma_k,1)$ then almost every $G\in\cF^\ast_{n,m}(K_{1,k})$ is co-$(k-1)$-partite.
	\item\label{item:thm:main-critical} If $m=\floor{\gamma_k\binom{n}{2}}$ then almost every $G\in\cF^\ast_{n,m}(K_{1,k})$ is the disjoint union of a co-$(k-1)$-partite graph and a graph on $O_k(\log n)$ vertices.
	\item\label{item:thm:main-almostall-sub} If $\gamma\in(0,\gamma_k)$ then almost every $G\in\cF^\ast_{n,m}(K_{1,k})$ is the disjoint union of two graphs $G=G_1\sqcup G_2$ such that $G_1$ is a co-$(k-1)$-partite graph on
	$$(1+o(1))\left(\frac{\gamma(k-1)}{1+(k-2)p_k}\right)^{1/2}n$$
	vertices and $G_2$ is a graph with $\Omega(n)\leq e(G_2)\leq o(n^2)$ edges.
\end{enumerate}
\end{theorem}

For $\gamma\in(\gamma_k,1)$, \Cref{thm:main-almostall} \ref{item:thm:main-almostall-super} significantly strengthens \Cref{thm:main-entropy} since it gives the first-order asymptotics of $N^\ast_{n,m}(K_{1,k})$: letting $r=k-1$ and assuming $r\mid n$, the number $C^r_{n,m}$ of co-$r$-partite graphs on $n$ vertices and $m\sim\gamma\binom{n}{2}$ edges is
$$C^r_{n,m}=(1+o(1))\vast(\sum_{\substack{z_1,\dots,z_r\in\mathbb Z\\ z_1+\cdots+z_r=0}}
\left(\frac{r\gamma-1}{r-1}\right)^{\frac12(z_1^2+\cdots+z_r^2)}\vast)\,\frac{1}{r!}\,
\binom{n}{\frac{n}{r},\dots,\frac{n}{r}}
\binom{\binom{r}{2}\frac{n^2}{r^2}}{\,m-r\binom{n/r}{2}}\,.$$
Meanwhile, in the subcritical regime in \Cref{thm:main-almostall} \ref{item:thm:main-almostall-sub}, there is a nonempty \emph{sparse} remainder on $\Theta(n)$ vertices outside the co-$(k-1)$-partite core.

The next main result describes the typical structure of the conditioned \erdosrenyi{} random graph $G(n,p)$ for constant $p$, where we again observe a structural dichotomy.

\begin{theorem}\label{thm:gnp-typ-struc}
Fix constants $k\geq3$ and $p\in(0,1)$. Let $G$ be the \erdosrenyi{} random graph $G(n,p)$ conditional on the event $G\in\cF^\ast_n(K_{1,k})$. Then the following hold:
\begin{enumerate}[
  label=\textit{(\roman*)},
  ref=(\textit{\roman*}),
  topsep=5pt
]
	\item If $p\in(p_k,1)$ then with high probability $G$ is co-$(k-1)$-partite and
	$$e(G)=(1+o(1))\left(p+\frac{1-p}{k-1}\right)\binom{n}{2}\,.$$
	\item If $p\in(0,p_k]$ then with high probability $G$ has $o(n^2)$ edges.
\end{enumerate}
\end{theorem}

We give intuition further below for the subcritical behavior in \Cref{thm:gnp-typ-struc}. First, we describe the graphon variation problems and their solutions in more detail. Formally, a \emph{graphon} is a symmetric Lebesgue-measurable function $W:[0,1]^2\to[0,1]$, where symmetric means $W(x,y)=W(y,x)$ for all $x,y\in[0,1]$. For background on graphons, we refer the reader to \cite{lovasz2006limits,borgs2008convergent,lovasz2012large}. The set of all graphons is denoted $\cW$. Two graphons $W_1$ and $W_2$ are \emph{equivalent} if there exist measure-preserving maps $\sigma_1,\sigma_2:[0,1]\to[0,1]$ such that $W_1(\sigma_1(x),\sigma_1(y))=W_2(\sigma_2(x),\sigma_2(y))$ almost everywhere. Statements below about the uniqueness of graphons are meant up to this notion of equivalence.

For all $k\geq3$ and $\gamma\in(0,1)$, define the variational problem over graphons $W\in\cW$
\begin{equation}\label{eqn:var-prob-intro-gamma}
\arraycolsep=3pt
\begin{array}[t]{lll}
\Phi_k(\gamma) \hspace{0.5mm}:= & \sup & \displaystyle\int_{[0,1]^2}H(W(x,y)) \, dx \, dy \\[16pt]
&\text{s.t.} & \displaystyle\int_{[0,1]^{k+1}}\Bigg(\prod_{i=2}^kW(x_1,x_i)\Bigg)\Bigg(\prod_{2\leq i<j\leq k+1}(1-W(x_i,x_j))\Bigg) \prod_{i=1}^{k+1}dx_i = 0 \,, \\[22pt]
&& \displaystyle\int_{[0,1]^2}W(x,y) \, dx \, dy=\gamma \,.
\end{array}
\end{equation}
The first constraint states that $W$ has zero $K_{1,k}$ density, or equivalently, for the familiar reader, that the $W$-random graph $G(n,W)$ is induced-$K_{1,k}$-free with probability 1. The second constraint states that $W$ has edge density $\gamma$.

\begin{theorem}\label{thm:ent-graphons}
Fix $k\geq3$. For all $\gamma\in[\gamma_k,1)$, $\Phi_k(\gamma)$ has a unique optimizer up to equivalence, while for all $\gamma\in(0,\gamma_k)$, there are infinitely many non-equivalent optimizers.
\end{theorem}

Next, we state the solution to the variational problem for the conditioned \erdosrenyi{} random graph. For $p\in(0,1)$, define the relative entropy
\begin{equation}\label{eqn:rel-ent-dfn}
I_p(x):=x\log\frac{x}{p}+(1-x)\log\frac{1-x}{1-p}\,.
\end{equation}
For all $k\geq3$ and $p\in(0,1)$, define the variational problem over graphons $W\in\cW$
\begin{equation}\label{eqn:var-prob-intro-p}
\arraycolsep=3pt
\begin{array}[t]{lll}
\Psi_k(p) \hspace{0.5mm}:= & \inf & \displaystyle\int_{[0,1]^2}I_p(W(x,y)) \, dx \, dy \\[16pt]
&\text{s.t.} & \displaystyle\int_{[0,1]^{k+1}}\Bigg(\prod_{i=2}^kW(x_1,x_i)\Bigg)\Bigg(\prod_{2\leq i<j\leq k+1}(1-W(x_i,x_j))\Bigg) \prod_{i=1}^{k+1}dx_i = 0 \,.
\end{array}
\end{equation}

\begin{theorem}\label{thm:gnp-graphons}
Fix $k\geq3$. For all $p\in(0,p_k)$, the unique optimizer of $\Psi_k(p)$, up to equivalence, is the all-zero graphon; for $p=p_k$, there are infinitely many non-equivalent optimizers; and for $p\in(p_k,1)$, there is again a unique optimizer up to equivalence.
\end{theorem}

To understand heuristically why the all-zero graphon is the subcritical optimizer in \Cref{thm:gnp-graphons}, consider that the rate function for induced-$K_{1,k}$-freeness in $G(n,p)$ can be written as an infimum over fixed edge densities $\gamma\in(0,1)$. For all $\gamma$, \Cref{thm:main-almostall} implies that for the uniformly random graph $G\in\cF^\ast_{n,m}(K_{1,k})$ with $m\sim\gamma\binom{n}{2}$, with high probability there is a partition of $V(G)$ into parts of size $\Theta(n)$ whose internal and between-part densities are either $o(1)$, $1-o(1)$, or some $p_\gamma+o(1)\geq p_k$. Hence, if $p<p_k$ then there is no optimizer of $\Phi_k(\gamma)$ that takes the value $p$ anywhere. But since the relative entropy $I_p(x)$ is uniquely minimized at $x=p$, this means no fixed-density optimizer of $\Phi_k(\gamma)$ for $\gamma\in(0,1)$ can also be an optimizer of $\Psi_k(p)$ whenever $p<p_k$, leaving only the all-zero graphon.

\begin{figure}
	\begin{center}
	\includegraphics{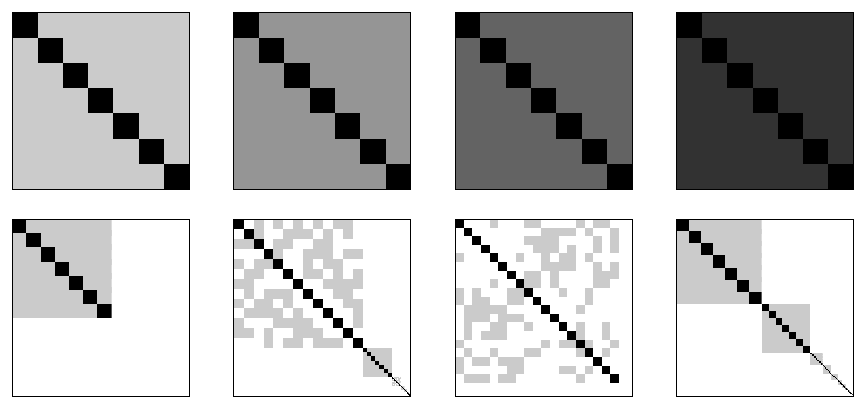}
	\end{center}
	\caption{Optimal graphons for induced-$K_{1,8}$-free graphs of fixed edge density. The top row shows the unique optimal graphons of edge densities $\gamma_8\approx0.317$, $\frac{1}{2}$, $\frac{2}{3}$, and $\frac{5}{6}$. The bottom row shows examples of optimal graphons of edge density $\frac{1}{10}$. The white areas represent density 0, gray areas density $p_\gamma\in(0,1)$, and black areas density 1.}
	\label{fig:graphons}
\end{figure}

The next theorem shows that the induced-$H$-free fixed-density variational problem does \emph{not} necessarily exhibit a phase transition for general $H$. A graph $G$ is a \emph{split graph} if there is a partition $V(G)=A\cup B$ such that $G[A]$ is an independent set and $G[B]$ is a clique.

\begin{theorem}\label{thm:c4-main}
Fix a constant $\gamma\in(0,1)$ and let $m\sim\gamma\binom{n}{2}$. Almost all induced-$C_4$-free graphs on $n$ vertices and $m$ edges are split graphs. In particular,
\begin{equation}\label{eqn:c4-entropy-main}
\lim_{n\to\infty}\frac{1}{\binom{n}{2}}\log N^\ast_{n,m}(C_4)=\max_{\lambda\in[0,1]}\left\{2\lambda(1-\lambda)\,H\!\left(\frac{\gamma-\lambda^2}{2\lambda(1-\lambda)}\right)\right\}\,,
\end{equation}
where the maximum is over $\lambda$ for which the right-hand side is well-defined.
\end{theorem}

The right-hand side of \eqref{eqn:c4-entropy-main} is an analytic function of $\gamma\in(0,1)$, reflecting the absence of a phase transition. \Cref{thm:c4-main} resolves the dense case of a conjecture of \citeauthor{morris2024asymmetric} \cite[Conjecture~6.1]{morris2024asymmetric} and refines the theorem of \citeauthor{promel1991excluding} \cite{promel1991excluding} stating that almost all induced-$C_4$-free graphs are split graphs.

\subsection{Overview of proofs}\label{subsec:overview}
Our starting point is the graphon variational problem \eqref{eqn:var-prob-intro-gamma}. We solve \eqref{eqn:var-prob-intro-gamma} via a reduction to an extremal problem over 3-edge-colored graphs. This extremal problem is closely related to prior literature discussed in \Cref{subsec:edit-distance}. We summarize the reduction here.

First, \Cref{lemma:UFHatGraphonSeq} shows that a graphon that is feasible for \eqref{eqn:var-prob-intro-gamma} is the limit of a sequence of \emph{types} $T$ of induced-$K_{1,k}$-free graphs. A type is an $\epsilon$-regular partition equipped with a 2-vertex-colored, 3-edge-colored graph $J$ that encodes whether clusters and pairs of clusters have density close to 0 (green), 1 (blue), or an intermediate value (red). If $T$ is a type for an induced-$K_{1,k}$-free graph, then the induced embedding lemma implies there is no induced homomorphism from $K_{1,k}$ into $J$, written $K_{1,k}\nHookrightarrow J$. If $0\leq d_{ij}\leq1$ represents the density between clusters $i$ and $j$ for all $1\leq i<j\leq\ell$, then the number of induced-$K_{1,k}$-free graphs of edge density $\gamma+o(1)$ that are consistent with $T$ and $\{d_{ij}\}_{1\leq i<j\leq\ell}$ is $e^{o(n^2)}\prod_{i<j}\binom{(n/\ell)^2}{d_{ij}(n/\ell)^2}$. For a fixed type $T$, if the logarithm of this quantity is near-maximum then concavity of the entropy implies the densities $d_{ij}$ corresponding with red pairs $ij$ are nearly equal to one another. The common value of these $d_{ij}$ is $(\gamma\binom{\ell}{2}-e_\blue(J))/e_\red(J)+o(1)$, where $e_\red(J)$ and $e_\blue(J)$ are the numbers of red and blue edges of $J$, respectively.

This motivates the following finite objective. For a 3-edge-colored graph $J$ on $\ell$ vertices with $e_\red(J)>0$, define
\begin{equation}\label{eqn:overview-fgamma}
f_\gamma(J):=e_\red(J)\cdot H\!\left(\frac{\gamma\binom{\ell}{2}-e_\blue(J)}{e_\red(J)}\right)\,.
\end{equation}
Up to the normalizing factor $2/\ell^2$ and a $o(1)$ error, the entropy contribution of a type with colored graph $J$ is governed by $f_\gamma(J)$. Consequently, solving the graphon variational problem reduces to proving extremal and stability results for $f_\gamma(J)$ over colored graphs $J$ satisfying $K_{1,k}\nHookrightarrow J$.

For excluded induced stars, we first prove stability for the linear extremal inequality
\begin{equation}\label{eqn:overview-ineq}
e_\red(J)\leq (k-2)e_\blue(J)+O_k(\ell)\,.
\end{equation}
That is, we prove that every $J$ with $K_{1,k}\nHookrightarrow J$ satisfies \eqref{eqn:overview-ineq} and, moreover, every $J$ for which \eqref{eqn:overview-ineq} is met with near equality is close in Hamming distance to an explicit extremal structure. Passing back to graphons, this gives the constraint $|R_W|\leq (k-2)|O_W|$, where $R_W=\{0<W<1\}$ and $O_W=\{W=1\}$. Jensen's inequality then reduces \eqref{eqn:var-prob-intro-gamma} to a two-variable calculus problem in $|R_W|$ and $|O_W|$, whose solution is $\sE_k(\gamma)$, and whose optimizers satisfy $|R_W|=(k-2)|O_W|$. The proofs of \eqref{eqn:overview-ineq} and the stability of the extremal structure use symmetrization, a local \erdos{}--Simonovits stability step, and an iterative extraction of reduced $(k-2)$-regular components. The extremal structure for \eqref{eqn:overview-ineq} is reflected in the graphons in \Cref{fig:graphons}. For induced-$C_4$-free graphs, we directly prove stability of the objective $f_\gamma(J)$ for every fixed $\gamma\in(0,1)$; in this case, the stable extremal structures are split colorings.

The entropy density and rough structure of fixed-density induced-$H$-free graphs follow from the solution to the graphon variational problem. The finer typical structure assertions in \Cref{thm:main-almostall,thm:gnp-typ-struc,thm:c4-main} require a more detailed analysis. These proofs have the following common architecture, inspired by prior work including \cite{balogh2020efficient}. For every graph $G$ that is close in cut metric to an optimal graphon, we fix a partition $\Pi(G)$ of its vertex set minimizing the number of \emph{defect edges}, that is, edges not consistent with the pure structure given by this optimal graphon. We then classify different defect structures and prove that in each case, the penalty to the count that is incurred by the defect outweighs the number of ways of choosing it. The term `penalty' reflects the fact that the number of graphs with a given optimal partition $\Pi$ and defect graph $T$ is constrained quantitatively in terms of $T$. We establish these penalties using random multipartite graphs that are specifically attuned to a given optimal partition and class of defect graph.

\Cref{sec:subcritical}, covering the fixed-density subcritical regime for excluded induced stars, uses a bookkeeping object we call a \emph{profile} to carefully track the defect edges occurring at multiple vertices along with the penalties incurred by each defect. There are three distinct structural mechanisms that lead to penalties for defect edges:
\begin{enumerate}[
	label=\textit{(\roman*)},
	ref=(\textit{\roman*}),
	topsep=5pt
]
	\item Defect edges give rise to a family of potential copies of $K_{1,k}$ that the aforementioned multipartite random graph must avoid for it to be induced-$K_{1,k}$-free. Choosing these copies positively correlated, we apply Janson's inequality (\Cref{thm:JansonsInequality}), giving a bound for the probability of induced-$K_{1,k}$-freeness. In some cases, this penalty compensates for all possibilities for the defect graph.
	\item Positive defect edges (as opposed to missing defect edges) could instead be placed in the bipartite interfaces that give the leading order of the entropy density. The comparison between the number of choices in both cases leads to a penalized count.
	\item Finally, the closeness of an induced-$K_{1,k}$-free graph $G$ to an optimal graphon $W$ deterministically excludes certain adjacency patterns. These excluded patterns imply that if a vertex has a certain defect profile then it must exhibit upper- or lower-tail adjacency into neighboring parts in the corresponding multipartite random graph, yielding a per-vertex large deviation penalty of order $e^{-\Theta(n)}$.
\end{enumerate}
By collecting distinct penalties and matching them with the choices for the defect graph, we show almost all induced-$K_{1,k}$-free graphs close in cut metric to an optimal graphon $W$ have the pure structure given by $W$.

To prove that the one-part graphon on the lower-left of \Cref{fig:graphons} is the unique optimal graphon that represents the typical structure in the subcritical regime, we analyze lower-order terms of the count, analogously to \cite{perkins2025typical}. Specifically, for the distinguished one-part graphon $W^\ast$, the set $\{W^\ast=0\}$ on which $W^\ast$ is zero has strictly larger measure than that of every other optimal graphon. For graphs close in cut metric to $W^\ast$, this translates to an additional $\Theta(n)$ vertices on which a graph with $o(n^2)$ edges can exist. In particular, since there are $e^{\Theta(n\log n)}$ matchings on $\Theta(n)$ vertices, the gain from using the additional $\Theta(n)$ vertices among graphs close to $W^\ast$ outweighs the $e^{O(n)}$ choices for the vertex partition.

\subsection{Related work}
\subsubsection{Coloring number and criticality.}\label{subsec:col-num}
Generalizing the \erdos--Kleitman--Rothschild theorem, \citeauthor{promel1992asymptotic} \cite{promel1992asymptotic} proved that for a non-bipartite graph $H$ of chromatic number $r+1$, almost all $H$-free graphs are $r$-partite if and only if $\chi(H-e)<\chi(H)$ for some edge $e$. The natural induced analogue of this theorem, proved by \citeauthor{balogh2011excluding} \cite{balogh2011excluding}, is formulated as follows and gives important context for our main conjecture below.

For nonnegative integers $s$ and $t$, a graph $H$ is \emph{$(s,t)$-colorable} if there exists a partition $V(H)=V_1\cup\cdots\cup V_{s+t}$ such that $H[V_i]$ is a clique for all $1\leq i\leq s$ and $H[V_j]$ is an independent set for all $s+1\leq j\leq s+t$. The \emph{coloring number} of $H$, denoted $\tau(H)$, is the minimum integer $k$ such that for every pair $(s,t)$ of nonnegative integers with $s+t=k$, $H$ is $(s,t)$-colorable. If $s+t=\tau(H)-1$ and $H$ is not $(s,t)$-colorable then $(s,t)$ is called a \emph{witnessing pair} for $H$.
Thus if $(s,t)$ is a witnessing pair for $H$ then every $(s,t)$-colorable graph is induced-$H$-free.
Write $\cF(H,s,t)$ for the set of graphs $F$, minimal under induced subgraph containment, such that $H$ can be covered by $s$ cliques, $t$ independent sets, and an induced subgraph isomorphic to $F$. Also define
$$\cP_n(H,s,t):=\bigcap_{F\in\cF(H,s,t)}\cF^\ast_n(F)\,.$$
The main theorem of \citeauthor{balogh2011excluding} \cite{balogh2011excluding} states that for a graph $H$ of coloring number at least 3, almost all induced-$H$-free graphs are $(s,t)$-colorable for some witnessing pair $(s,t)$ if and only if $\cP_n(H,s,t)\subseteq\{K_n,K_n^c\}$ for all $s+t=\tau(H)-2$ and large enough $n$. We say $H$ is \emph{critical} if it satisfies the latter property. In other words, $\mathcal{F}(H,s,t)$ contains the minimal residual induced subgraphs of $H$ after using $s$ cliques and $t$ independent sets, and criticality states that when $s+t=\tau(H)-2$, the corresponding residual class is eventually just $\{K_n,K_n^c\}$, so no positive-entropy internal structure can exist in the $\tau(H)-1$ parts.

Below, we state a conjecture that criticality leads to a pure $(s,t)$-colorable structure in almost all $G\in\cF^\ast_{n,m}(H)$ when $m\sim\gamma\binom{n}{2}$, where the witnessing pair $(s,t)$ depends on $\gamma$.

\subsubsection{Edit distance and hereditary properties.}\label{subsec:edit-distance}
Our work is related to
literature on extremal problems for multigraphs, most notably by \citeauthor{marchant2010extremal} \cite{marchant2010extremal,marchant2011structure} and \citeauthor{thomason2011graphs} \cite{thomason2011graphs}. This literature focuses on calculating an extremal parameter $\kappa_p$ that is analogous to the Tur\'an density. Theorem 4.4 in \cite{marchant2010extremal} shows how to convert between $\kappa_p$ and the large deviation rate function for a hereditary property in $G(n,p)$. Consequently, one can read off from \cite{marchant2010extremal} the large deviation rate function for induced-$H$-freeness for various specific graphs $H$. An example in \cite{thomason2011graphs} states that for excluded induced stars and the specific parameter value $p=p_k$, there are infinitely many colored extremal structures, which matches the extremal characterization given in this paper.
Our work differs from this previous literature in that we focus mainly on the fixed-density setting, characterize the graphon optimizers for all parameter values, and explore fine-grained structural results.

\subsubsection{Other density-constrained models}
Numerous prior works have studied typical structure and phase transitions in random graph models defined by subgraph density constraints. A well-studied case is the uniformly random graph of given edge and triangle density; several works have characterized optimal graphons in subsets of the feasible region of densities and proven the existence of phase transitions \cite{radin2013phase,radin2014asymptotics,radin2015singularities}.

\subsection{Future directions}
Here we present a conjecture that generalizes \Cref{thm:main-entropy}. Let $H$ be a graph of coloring number $\tau(H)=k\geq3$. For $\gamma\in(0,1)$ and nonnegative integers $s$ and $t$ with $s+t=k-1$, define the following optimization problem:
\[\arraycolsep=1.5mm
\begin{array}{llll}
\Phi_{s,t}(\gamma) ~:= &\max & \displaystyle H(\beta)\sum_{1\leq i<j\leq s+t}\sigma_i\sigma_j \\[18pt]
&\text{s.t.} & 0\leq \sigma_1,\dots,\sigma_{s+t}\leq1 \,, \\[6pt]
&& 0\leq \beta\leq1 \,, \\[6pt]
&& \sigma_1+\cdots+\sigma_{s+t} \leq 1 \,, \\[6pt]
&& \displaystyle \sum_{i=1}^s\sigma_i^2 + 2\beta\sum_{1\leq i<j\leq s+t}\sigma_i\sigma_j = \gamma \,.
\end{array}\]
If $H$ has no isolated vertices and $(s,t)$ is a witnessing pair for $H$ then since every $(s,t)$-colorable graph along with some $\ell$ isolated vertices is induced-$H$-free, one can see that $\Phi_{s,t}(\gamma)$ gives a lower bound for the entropy density of induced-$H$-free graphs of edge density $\gamma+o(1)$. We expect that for many graphs $H$, the maximum of $\Phi_{s,t}(\gamma)$ over witnessing pairs $(s,t)$ characterizes the entropy density.

\begin{conjecture}
Let $H$ be a critical graph (as defined in \Cref{subsec:col-num}) of coloring number $\tau(H)\geq3$ with no isolated vertices. If $\gamma\in(0,1)$ is fixed and $m\sim\gamma\binom{n}{2}$ then
$$\lim_{n\to\infty}\frac{1}{\binom{n}{2}}\log N^\ast_{n,m}(H)=\max_{s,t}\,\Phi_{s,t}(\gamma)\,,$$
where the maximum is over witnessing pairs $(s,t)$ of $H$.
\end{conjecture}

From the results of this paper, the conjecture is true for stars and $C_4$. The assumption of criticality rules out nontrivial residual hereditary classes after covering $H$ with $\tau(H)-2$ witness classes, and is thus a natural hypothesis for the leading entropy to come only from $(s,t)$-colorable constructions. If $H$ has an isolated vertex then we do not expect witnessing pairs to fully characterize the entropy density. For example, for $H=K_{1,3}\sqcup K_1$, we expect there to be a phase transition in which the supercritical structure matches that of $K_{1,4}$ (i.e.\ a co-3-partite graph), but where the subcritical structure is the union of a co-bipartite graph and a sparse graph; the subcritical structure observed for $K_{1,4}$ is \emph{not} induced-$H$-free and thus cannot carry over to $H$. In this case, the subcritical entropy density would be given by a maximum over a larger class of covering pairs than just witnessing pairs.

In general, $\max\,\Phi_{s,t}(\cdot)$ need not be unimodal, convex, or continuously differentiable. For $C_7$, which has witnessing pairs $(3,0)$ and $(2,1)$, the graph of $\max\,\Phi_{s,t}(\cdot)$ is bimodal and has a first-order non-analyticity at $\gamma^\ast\approx0.622$. Another example is the complete multipartite graph $M_k$ with $k\geq3$ color classes of sizes $1,\dots,k$. One can show that $\tau(M_k)=k$ and every pair $(s,t)$ with $s+t=k-1$ is a witnessing pair. In this case, $\max\,\Phi_{s,t}(\cdot)$ has $k-1$ first-order non-analyticities.

\subsection{Organization}
In \Cref{sec:prelim}, we include background on colored graphs, regularity, graphons, and cut metric that will be used throughout the paper. In \Cref{sec:graphon-prob}, we solve the graphon variational problems, fully characterizing all optimizers. In \Cref{sec:entropy} we prove \Cref{thm:main-entropy,thm:main-rate,thm:ent-graphons,thm:gnp-graphons}. In \Cref{sec:supercritical,sec:critical,sec:subcritical}, we prove \Cref{thm:main-almostall} \ref{item:thm:main-almostall-super}, \ref{item:thm:main-critical}, and \ref{item:thm:main-almostall-sub}, respectively. In \Cref{sec:erdosrenyi}, we prove \Cref{thm:gnp-typ-struc}. In \Cref{sec:edge-coloring}, we establish the extremal and stability results for edge colorings that underlie the graphon analysis. In \Cref{sec:deferred}, we give deferred proofs from the earlier sections. Finally, in \Cref{sec:c4-free} we prove \Cref{thm:c4-main}.

\section{Preliminaries}
\label{sec:prelim}
\subsection{Notation}
We will write both $V(K_n)$ and $V$ for a vertex set of size $n$. The complement of a set or graph $X$ is denoted $X^c$. The base-2 logarithm is denoted $\log$ and the natural logarithm is denoted $\ln$. Induced subgraph containment is denoted $H\leq G$. For subsets $A,B\subseteq V$, we write $G[A,B]$ for the graph on $A\cup B$ containing all edges of $G$ with one endpoint in $A$ and another in $B$. To emphasize the underlying graph $G$, we write $e_G(\cdot)$, $d_G(\cdot)$, $N_G(\cdot)$, etc. We use the superscript $c$ to denote neighborhood and degree in the complementary graph $G^c$; for example, $N^c(v)=V\setminus(N(v)\cup\{v\})$, $N^c(v,A):=N^c(v)\cap A$, $d^c(v):=|N^c(v)|$, and $d^c(v,A):=|N^c(v,A)|$. For vertices $i,j$ we use the shorthand $ij:=\{i,j\}$. If $\cH$ is a collection of graphs then we write $\cH(n)$ for the set of $H\in\cH$ on $n$ vertices; conversely, if $\cH(n)$ is defined for various $n$ then we write $\cH:=\bigcup_{n\geq1}\cH(n)$.

\subsection{Janson's inequality}
The following version of Janson's inequality due to \citeauthor{riordan2015janson} \cite{riordan2015janson} applies to families of events that are positively correlated and plays an important role in this paper. On several occasions, we will define a family of copies of $K_{1,k}$ and use \Cref{thm:JansonsInequality} to bound the probability that none of these copies appears in a random graph.

\begin{theorem}[\normalfont{\cite{riordan2015janson}}]\label{thm:JansonsInequality}
Let $(\Omega,\cF,\bbP)$ be a probability space. Let $\cI\subseteq\cF$ be a family of events such that for all $A,B\in\cI$ we have $\bbP\{A\cap B\}\geq\bbP\{A\}\bbP\{B\}$, $A\cap B\in\cI$, and $A\cup B\in\cI$. Let $B_1,\dots,B_n\in\cI$ and define
$$\mu:=\sum_{i=1}^n\bbP\{B_i\}\hspace{7mm}\text{and}\hspace{7mm}\Delta:=\sum_{i\sim j}\bbP\{B_i\cap B_j\}\,,$$
the second sum being over unordered pairs $i\neq j$ such that $B_i$, $B_j$ are not independent. Then
$$\bbP\!\left\{\bigcap_{i=1}^nB_i^c\right\}\leq\exp\!\left(-\min\left\{\frac{\mu}{2}\,,\,\frac{\mu^2}{4\Delta}\right\}\right)\,.$$
\end{theorem}

\subsection{Colored homomorphisms}\label{subsec:colored-graphs}
Graph regularity and induced embeddings are important tools in our proofs. It is convenient and standard \cite{marchant2010extremal,marchant2011structure,bottcher2012perfect,martin2013edit} to work with vertex- and edge-colored graphs that encode which induced subgraphs can be embedded into a regular partition. We first introduce colored graphs and homomorphisms in this subsection, and then their applications to our setting in \Cref{subsec:regularity}.

A \emph{colored graph} $J=(F,\sigma)$ is a graph $F$ with a 3-coloring $\sigma:V(F)\cup E(F)\to\{\red,\green,\blue\}$ of its vertices and edges such that $\sigma(V(F))\subseteq\{\green,\blue\}$. We write $V(J):=V(F)$ and $E(J):=E(F)$. Let $G$ be a graph and $J$ a colored graph. A \emph{colored homomorphism} from $G$ to $J$ is a mapping $\phi:V(G)\to V(J)$ such that the following conditions hold:
\begin{enumerate}[
  label=\textit{(\roman*)},
  ref=(\textit{\roman*}),
  topsep=5pt
]
    \item If $ij\in E(G)$ and $\phi(i)\neq\phi(j)$ then $\phi(i)\phi(j)\in E(J)$.
    \item If $uv\in E(G)$ then either $\phi(u)\neq\phi(v)$ and $\sigma(\phi(u)\phi(v))\in\{\red,\blue\}$, or $\phi(u)=\phi(v)$ and $\sigma(\phi(u))=\blue$.
    \item If $uv\not\in E(G)$ then either $\phi(u)\neq\phi(v)$ and $\sigma(\phi(u)\phi(v))\in\{\green,\red\}$, or $\phi(u)=\phi(v)$ and $\sigma(\phi(u))=\green$.
\end{enumerate}
In other words, edges (and only edges) are embedded into \blue, non-edges (and only non-edges) are embedded into \green, and both can be embedded into \red. We write $G\hookrightarrow J$ if there exists a colored homomorphism from $G$ to $J$. For example, $C_4\hookrightarrow\grg$ and $K_{1,3}\hookrightarrow\exampletriangle$.

Let $\cC(n)$ be the set of colored graphs on $n$ vertices, and let
$$\cC(n,G):=\{J\in\cC(n):G\nHookrightarrow J\}\,.$$
An \emph{edge-colored graph} $J=(F,\tau)$ is a graph $F$ with a 3-coloring $\tau:E(F)\to\{\red,\green,\blue\}$ of its edges. Again, we write $V(J):=V(F)$ and $E(J):=E(F)$.

\subsection{Graph regularity}\label{subsec:regularity}
Here we introduce \emph{types}, which are regularity graphs used for subgraph embedding. We mostly follow the framework of \citeauthor{bottcher2012perfect} \cite{bottcher2012perfect}. To simplify the exposition, we instead directly define a type as a structure that satisfies the conclusions of the induced embedding lemma \cite[Lemma~2.9]{bottcher2012perfect}. In analogy with the regularity lemma, the results of \cite{bottcher2012perfect} imply such a structure always exists, as we show in \Cref{lemma:type-lemma}.

A pair $A,B\subseteq V(G)$ is \emph{$\epsilon$-regular} in a graph $G$ if for all $A'\subseteq A$ and $B'\subseteq B$ such that $\abs{A'}\geq\epsilon\abs{A}$ and $\abs{B'}\geq\epsilon\abs{B}$, $\abs{d(A',B')-d(A,B)}\leq\epsilon$, where $d(X,Y):=\frac{e(X,Y)}{|X|\cdot|Y|}$. An \emph{$\epsilon$-regular partition} of $G$ is a partition $V(G)=V_0\cup V_1\cup\cdots\cup V_k$ such that $\abs{V_1}=\cdots=\abs{V_k}$, $\abs{V_0}\leq\epsilon\abs{V(G)}$, and all but at most $\epsilon\binom{k}{2}$ of the pairs $1\leq i<j\leq k$ are $\epsilon$-regular.

\begin{definition}[Type]\label{dfn:type}
Let $G$ be a graph on $n$ vertices, and let $\epsilon,\delta\in(0,\frac{1}{2})$ and $\ell\in\N$. An \emph{$(\epsilon,\delta,\ell)$-type} for $G$ is a pair $(P,R)$ where $P=\{V_0,V_1,\dots,V_k\}$ is an $\epsilon$-regular partition of $G$ and $R=([k],E_R,\sigma)$ is a colored graph such that the following conditions hold:
\begin{enumerate}[
  label=\textit{(\roman*)},
  ref=(\textit{\roman*}),
  topsep=5pt
]
    \item For all $1\leq i<j\leq k$, $ij\in E_R$ if and only if $\{V_i,V_j\}$ is $\epsilon$-regular in $G$.
    \item For all $ij\in E_R$, we have $\sigma(ij)=\green$ if $d(V_i,V_j)<\delta$; $\sigma(ij)=\red$ if $d(V_i,V_j)\in[\delta,1-\delta]$; and $\sigma(ij)=\blue$ if $d(V_i,V_j)>1-\delta$.
    \item\label{item:type-colored-homo} If $H$ is a graph on at most $\ell$ vertices and $H\hookrightarrow R$ then $H\leq G$.
\end{enumerate}
If all the above conditions hold then we say $G$ \emph{has} type $(P,R)$.
\end{definition}

\subsection{Graphons and entropy}\label{subsec:graphons}
The \emph{entropy} of a graphon $W\in\cW$ is defined as
$$H(W):=\int_{[0,1]^2}H(W(x,y))\,dx\,dy\,,$$
and for all $p\in[0,1]$, the \emph{relative entropy} of $W$ is defined as
$$I_p(W):=\int_{[0,1]^2}I_p(W(x,y))\,dx\,dy\,,$$
The following lemmas, which are \cite[Propositions~2.10~\&~2.11]{perkins2025typical}, state that the entropy density and rate function for induced-$H$-freeness are given by variational problems over graphons. We write $t_\ind$ and $t$ for the induced and non-induced homomorphism densities; these are written explicitly in \eqref{eqn:var-prob-intro-gamma} in the case of $K_{1,k}$ and $K_2$, and we refer the reader to \cite{lovasz2012large} for the general definition. We also define
$$\rand(W):=\abs{\{(x,y)\in[0,1]^2:0<W(x,y)<1\}}\,.$$

\begin{lemma}\normalfont{(\cite{perkins2025typical})}\label{lemma:LDPforGNPkl}
Let $p\in(0,1)$ be a constant and $G\sim G(n,p)$. If $H$ is a fixed graph and
$$\cF:=\{W\in\cW:t_{\ind}(H,W)=0,\,\rand(W)>0\}$$
is nonempty then
$$\lim_{n\to\infty}\frac{1}{\binom{n}{2}}\log\bbP\{G\in\cF^\ast_n(H)\}=-\inf_{W\in\cF}I_p(W)\,.$$
\end{lemma}

\begin{lemma}\normalfont{(\cite{perkins2025typical})}\label{lemma:FnmGraphonOptAsymps}
Let $\gamma\in(0,1)$, $n,m\in\N$, and $m\sim\gamma\binom{n}{2}$. If $H$ is a fixed graph and
$$\cF_\gamma:=\{W\in\cW:t(K_2,W)=\gamma,\,t_{\ind}(H,W)=0,\,\rand(W)>0\}$$
is nonempty then
$$\lim_{n\to\infty}\frac{1}{\binom{n}{2}}\log|\cF^\ast_{n,m}(H)|=\sup_{W\in\cF_\gamma}H(W)\,.$$
\end{lemma}

We now define graphons naturally associated to matrices, graphs, and types. For integers $i,j\in[n]$, define
\begin{equation}\label{dfn:snij}
S^n_{ij}:=\left(\frac{i-1}{n},\frac{i}{n}\right)\times\left(\frac{j-1}{n},\frac{j}{n}\right)\cup\left(\frac{j-1}{n},\frac{j}{n}\right)\times\left(\frac{i-1}{n},\frac{i}{n}\right)\subseteq[0,1]^2\,.
\end{equation}
For a symmetric matrix $M\in\R^{n\times n}$, the graphon $W_M$ associated to $M$ is defined as
$$W_M:=\sum_{1\leq i\leq j\leq n}M_{ij}\,\bsone_{S_{ij}^n}\,.$$
If $G$ is either an unweighted or weighted graph on $n$ vertices with adjacency or weight matrix $A$ then we define the graphon $W_G:=W_A$.
If a graph $G$ has the $(\epsilon,\delta,\ell)$-type $T=(P,R)$, where $P=\{V_0,V_1,\dots,V_k\}$, then the graphon $W_T$ associated with $T$ and $G$ is defined by
$$W_T:=\sum_{1\leq i\leq j\leq k}d_G(V_i,V_j)\,\bsone_{S_{ij}^k}\,.$$
Note that the exceptional set $V_0$ is not represented in the graphon $W_T$.

The cut metric and related notions of distance between graphons are important tools in this paper. We record the necessary definitions here and refer the reader to \cite{lovasz2012large} for more details. The \emph{cut norm} of a graphon $W\in\cW$ is defined as
$$\norm{W}_\square:=\sup_{S,T\subseteq[0,1]}\bigg|\int_{S\times T}W(x,y)\,dx\,dy\bigg|\,.$$
For a graphon $W\in\cW$ and a map $\varphi:[0,1]\to[0,1]$, define the graphon $W^\varphi$ by $W^\varphi(x,y)=W(\varphi(x),\varphi(y))$. The \emph{cut metric} is a distance between graphons $W,W'\in\cW$ defined as
$$\delta_\square(W,W'):=\inf_\varphi\,\norm{W^\varphi-W'}_\square\,,$$
where the infimum is over measure-preserving bijections $\varphi$. The cut metric $\delta_\square$ is defined on pairs of graphons, (weighted or unweighted) graphs, and matrices, and any combination thereof by using the corresponding graphon associated to those objects (i.e.\ $W_G$ or $W_M$). The \emph{cut norm} of a real $n\times n$ matrix $A$ is defined as
$$\norm{A}_\square:=\frac{1}{n^2}\,\max_{S,T\subseteq[n]}\Bigg|\sum_{i\in S,j\in T}A_{ij}\Bigg|\,.$$
For weighted graphs $G$ and $H$ on $n$ vertices with adjacency matrices $A_G$ and $A_H$, we define the two notions of distance
\begin{align*}
d_\square(G,H) &:= \norm{A_G-A_H}_\square \,, \\[3pt]
\widehat{\delta}_\square(G,H) &:= \min_P\,\norm{A_G-P^TA_HP}_\square \,,
\end{align*}
where the minimum is over $n\times n$ permutation matrices. Finally, the \emph{edit distance} between two graphs $G$ and $H$ is defined as
$$d_1(G,H):=\frac{1}{n^2}\,|E(G)\,\triangle\,E(H)|\,,$$
where $\triangle$ is the symmetric difference.

The following lemma plays an important role in the solution to the graphon variational problem in \Cref{sec:graphon-prob}: it allows us to transfer from a graphon $W$ with zero induced $H$ density to a sequence of types that converge to $W$. The lemma is proven in \Cref{sec:deferred}.

\begin{restatable}{lemma}{GraphonSeq}\label{lemma:UFHatGraphonSeq}
Fix a graph $F$. For all $W\in\cW$ such that $t_{\ind}(F,W)=0$ and all $\delta>0$, there exists a sequence $\{W_m\}_{m\in\N}$ of graphons such that the following conditions hold:
\begin{enumerate}[
  label=\textit{(\roman*)},
  ref=(\textit{\roman*}),
  topsep=5pt
]
	\item $\delta_\square(W_m,W)\to0$ as $m\to\infty$.
    \item There exists a graphon $W'$ such that $\norm{W_m-W'}_1\to0$ as $m\to\infty$.
	\item For all $m$ there exists $\eta_m>0$ such that $W_m$ is the graphon associated with an $(\eta_m,\delta,v(F))$-type $(P_m,R_m)$ of an induced-$F$-free graph.
	\item $\eta_m\to0$ and $v(R_m)\to\infty$ as $m\to\infty$.
\end{enumerate}
\end{restatable}

\section{Solving the Graphon Variational Problem}
\label{sec:graphon-prob}
In this section we solve the variational problems \eqref{eqn:var-prob-intro-gamma} and \eqref{eqn:var-prob-intro-p}. Let $k\geq3$ be an integer fixed throughout this section, and let $\Delta:=k-2$. For all $\gamma\in(0,1)$ define the set of graphons
$$\cX_\gamma:=\{W\in\cW:t_\ind(K_{1,k},W)=0\,,\,t(K_2,W)=\gamma\}\,.$$
Then the variational problems \eqref{eqn:var-prob-intro-gamma} and \eqref{eqn:var-prob-intro-p} can be written
\begin{align}
\Phi_k(\gamma) &= \sup\{H(W):W\in\cX_\gamma\} \,,\label{eqn:var-prob-fixed-gamma} \\[4pt]
\Psi_k(p) &= \inf\{I_p(W):t_\ind(K_{1,k},W)=0\} \,.
\end{align}
Let $\cX_\gamma^\ast$ be the set of optimizers in \eqref{eqn:var-prob-fixed-gamma}, which is a nonempty set since $\cX_\gamma$ is compact and $H$ is upper semicontinuous in the cut metric topology (see \cite[Theorem~3.7]{borgs2008convergent} and \cite[Lemma~2.1]{chatterjee2011large}).

We now define the set of graphons we will prove are the optimizers in \eqref{eqn:var-prob-fixed-gamma}.
For a graph $G$ on $[n]$, define the function $\xi_G:\R^2\to[0,1]$ by
\begin{equation}\label{eqn:xiG-dfn}
\xi_G(x,y)=\sum_{i=1}^n\bsone_{S^n_{ii}}+\sum_{ij\in E(G)}p_k\,\bsone_{S_{ij}^n}\,,
\end{equation}
where $S_{ij}^n$ is as defined in \eqref{dfn:snij}.

Let $\bLambda$ denote the set of all (finite or infinite) sequences $((\lambda_1,G_1),(\lambda_2,G_2),\dots)$ of pairs where the numbers $\lambda_i\in[0,1]$ are strictly increasing in $i\geq1$, the numbers $\lambda_i-\lambda_{i-1}$ are nonincreasing in $i\geq1$ (where $\lambda_0:=0$), and $G_i$ is a connected $\Delta$-regular graph. For all $\blambda=((\lambda_i,G_i))_{i\geq1}\in\bLambda$ define the graphon $W_\blambda$ by
\begin{equation}\label{eqn:Wblambda}
W_\blambda(x,y)=\sum_{i=1}^{\abs{\blambda}}\xi_{G_i}\!\left(\frac{x-\lambda_{i-1}}{\lambda_{i}-\lambda_{i-1}}\,,\,\frac{y-\lambda_{i-1}}{\lambda_{i}-\lambda_{i-1}}\right)\,.
\end{equation}
For all $\gamma\in(0,\gamma_k)$, let $\cV_{\gamma}$ be the set of graphons defined by
\begin{align}
\cV_{\gamma} &:= \left\{W_\blambda:\int_{[0,1]^2}W_\blambda(x,y)\,dx\,dy=\gamma\,,\,\blambda\in\bLambda\right\} \nonumber\\
&= \Bigg\{W_\blambda:\sum_{i=1}^{\abs{\blambda}}\frac{(\lambda_{i}-\lambda_{i-1})^2}{v(G_i)}=\frac{\gamma}{1+\Delta\,p_k}\,,\,\blambda=((\lambda_i,G_i))_{i\geq1}\in\bLambda\Bigg\}\,,\label{eqn:vgamma-condition}
\end{align}
while for $\gamma\geq\gamma_k$, let $\cV_{\gamma}$ be the set containing the single graphon $W^\ast_{\gamma}$ defined by
\begin{equation}\label{eqn:opt-graphon-super}
W^\ast_{\gamma}(x,y):=\begin{cases}1&(x,y)\in\bigcup_{i=1}^{k-1}S^{k-1}_{ii}\\[4pt]\frac{(k-1)\gamma-1}{k-2}&\text{otherwise}\end{cases}\,.
\end{equation}
Define the union $\cV=\bigcup_{\gamma\in(0,1)}\cV_\gamma$. Note that both $\cV_\gamma$ and $\cV$ are closed sets in the cut metric topology. We also write $\cV_0:=\{W_0\}$, where $W_0\equiv0$ is the all-zero graphon.

\begin{example}
Several graphons in $\cV_\gamma$ for $k=8$ are shown in \Cref{fig:graphons}. The unit squares are rotated 90 degrees in the figure to match the ``adjacency matrix'' interpretation. We now concretely describe some of these graphons to clarify the definitions given above. The leftmost supercritical graphon in \Cref{fig:graphons} is precisely $W^\ast_{\gamma_k}$ from \eqref{eqn:opt-graphon-super}. At this edge density, we can see from \eqref{eqn:gammak-dfn} and \eqref{eqn:opt-graphon-super} that the off-diagonal density is $p_k$. The second supercritical graphon is $W^\ast_{1/2}$, the unique element of $\cV_{1/2}$; in this case, the off-diagonal density is $\frac{5}{12}$.

All four subcritical graphons are of edge density $\gamma=\frac{1}{10}$ and are given by $W_\blambda$ for some $\blambda\in\bLambda$. The first of these graphons has $\blambda=((\lambda_1,K_7))$, where $\lambda_1:=\big(\frac{7\,\cdot\,1/10}{1+6p_8}\big)^{1/2}\approx0.561$, which can be seen from the condition in \eqref{eqn:vgamma-condition}; this is the unique graphon in $\cV_\gamma$ satisfying both $|\blambda|=1$ and $G_1=K_7$. The third subcritical graphon has $\blambda=((\lambda_1,G_1))$, where $\lambda_1=\big(\frac{19\,\cdot\,1/10}{1+6p_8}\big)^{1/2}\approx0.925$ and $G_1$ is a 6-regular connected graph on 19 vertices. The fourth subcritical graphon has $\blambda=((\lambda_1,K_7),\dots,(\lambda_{19},K_7))$.
\end{example}

The following proposition is the main result of this section. Before proving it, we will state and prove two lemmas.

\begin{proposition}\label{prop:graphon-char-fixed-gamma}
For all $\gamma\in(0,1)$ we have $\cX^\ast_\gamma=\cV_\gamma$ up to equivalence of graphons.
\end{proposition}

\begin{lemma}\label{lemma:Vgamma-opt}
For all $\gamma\in(0,1)$ and $W\in\cV_\gamma$, we have $t_\ind(K_{1,k},W)=0$ and $H(W)=\sE_k(\gamma)$.
\end{lemma}
\begin{proof}
If $\gamma\in(0,\gamma_k)$ then write $W=W_\blambda$ with $\blambda=((\lambda_i,G_i))_{i\geq1}\in\bLambda$ and $\alpha_i:=\lambda_i-\lambda_{i-1}$. Since different blocks are joined by zero density and each $G_i$ is $\Delta$-regular, the integrand defining $t_\ind(K_{1,k},W)$ can be positive only on tuples lying in a single block; such a tuple would require some vertex of $G_i$ to have degree at least $k>\Delta+1$, which is impossible. Also,
$$H(W)=H(p_k)\sum_{i=1}^{|\blambda|}\frac{\Delta\alpha_i^2}{v(G_i)}=\frac{\Delta\gamma}{1+\Delta p_k}H(p_k)=\sE_k(\gamma)\,,$$
where we used \eqref{eqn:vgamma-condition}. If $\gamma\in[\gamma_k,1)$ then $W=W^\ast_\gamma$, and since $W^\ast_\gamma$ has only $\Delta+1=k-1$ diagonal blocks, we have $t_\ind(K_{1,k},W)=0$ as above. Moreover, we have
$$H(W)=\frac{\Delta}{\Delta+1}H\!\left(\frac{(\Delta+1)\gamma-1}{\Delta}\right)=\sE_k(\gamma)\,,$$
completing the proof.
\end{proof}

\begin{lemma}\label{lemma:k1k-calc-prob}
Let $\gamma\in(0,1)$, $k\geq3$, and $\Delta:=k-2$. Let $p_k$ and $\gamma_k$ be as defined in \eqref{eqn:gammak-dfn}. The optimization problem
\begin{equation}\label{eqn:calculus-problem}
\begin{array}{ll}
\max & y\,H\big(\frac{\gamma-x}{y}\big) \\[6pt]
\mathrm{s.t.} & 0\leq x\leq 1 \,, \\[3pt]
& 0\leq y\leq\min\{1-x,\Delta\,x\} \,, \\[3pt]
& 0\leq \frac{\gamma-x}{y} \leq 1 \,,
\end{array}
\end{equation}
has the unique optimizer
\begin{equation}\label{eqn:xsys-opt-coords}
(x^\ast,y^\ast)=\begin{cases}
\left(\dfrac{\gamma}{1+\Delta p_k}\,,\,\dfrac{\Delta\gamma}{1+\Delta p_k}\right)\,, & 0<\gamma\leq \gamma_k\,,\\[10pt]
\left(\dfrac{1}{\Delta+1}\,,\,\dfrac{\Delta}{\Delta+1}\right)\,, & \gamma_k\leq \gamma<1\,,
\end{cases}
\end{equation}
and has optimal value $\sE_k(\gamma)$, where $\sE_k$ is the function defined in \eqref{eqn:entropy-dfn}.
\end{lemma}
\begin{proof}
For every fixed $x$, the objective is increasing in $y$. Thus, every optimizer lies on the boundary $y=\min\{1-x,\Delta\,x\}$. The boundaries intersect at $x_0=1/(\Delta+1)$, $y_0=\Delta/(\Delta+1)$. Along $y=1-x$ for $x\geq1/(\Delta+1)$, the objective is decreasing in $x$. Thus, its maximum on this segment is at the intersection point, where $y=\Delta\,x$. Otherwise if $y\neq1-x$ then we must have $y=\Delta\,x$ per our previous observation.

It remains to maximize the univariate objective $\Delta x\,H\big(\frac{\gamma-x}{\Delta x}\big)$ subject to $x\leq x_0$. With the change of variable $p=\frac{\gamma-x}{\Delta x}$, this becomes $f(p):=\frac{\Delta\gamma}{1+\Delta p}\,H(p)$ subject to $\max\{0,p_0\}\leq p\leq1$ where $p_0:=\frac{\gamma(\Delta+1)-1}{\Delta}$. An easy calculation shows $f'(p)=0$ if and only if $p=p_k$. Thus if $p_k\geq p_0$ then the global optimizer is $p^\ast=p_k$; otherwise if $p_k<p_0$ then since $f$ is strictly decreasing for $p\geq p_k$, the optimizer is $p^\ast=p_0$.

The condition $p_k\geq p_0$ is equivalent to $\gamma\leq\gamma_k$. Hence using $y^\ast=\Delta x^\ast$, we establish \eqref{eqn:xsys-opt-coords} by obtaining $x^\ast$ and $y^\ast$ in the equation $p_k=\frac{\gamma-x^\ast}{\Delta x^\ast}$ if $\gamma\leq\gamma_k$, and $p_0=\frac{\gamma-x^\ast}{\Delta x^\ast}$ if $\gamma\geq\gamma_k$. The objective value $\sE_{\Delta+2}(\gamma)$ follows immediately by substituting $(x,y)=(x^\ast,y^\ast)$.
\end{proof}

For a graphon $W\in\cW$ define
\begin{align*}
R_W&:=\{(x,y)\in[0,1]^2:0<W(x,y)<1\}\,, \\
O_W&:=\{(x,y)\in[0,1]^2:W(x,y)=1\}\,,
\end{align*}
and write $|\cdot|$ for Lebesgue measure.

\begin{proof}[Proof of \Cref{prop:graphon-char-fixed-gamma}]
The proof will go as follows. First, we will show every graphon $W\in\cX^\ast_\gamma$ satisfies $|R_W|\leq\Delta|O_W|$ and $H(W)\leq\sE_k(\gamma)$ by invoking the optimization problem \eqref{eqn:calculus-problem} and the extremal inequality in \Cref{lemma:kthOrderMantel}. Since every graphon in $\cV_\gamma$ has entropy $\sE_k(\gamma)$, this means $\cV_\gamma\subseteq\cX^\ast_\gamma$. To show that every optimal graphon is equivalent to a graphon in $\cV_\gamma$, we will need the stability of the extremal structure from \Cref{thm:kth-order-stability}.

\smallskip
\noindent\textbf{Preparation.}
This part applies to the whole remainder of the proof.
Let $W\in\cX^\ast_\gamma$ and
$$p:=\frac{1}{|R_W|}\int_{R_W}W(x,y)\,dx\,dy\,,$$
which is well-defined since $|R_W|>0$, as otherwise we would have $H(W)=0$, contradicting optimality of $W$. Strict concavity of $H(\cdot)$ implies $W$ takes the three values $\{0,p,1\}$ almost everywhere: otherwise the graphon defined by $W':=\bsone_{O_W}+p\,\bsone_{R_W}$, which also satisfies $t(K_2,W')=\gamma$ and $t_\ind(K_{1,k},W')=0$, would have $H(W')>H(W)$ by Jensen's inequality.

Let $\delta\in\big(0,\frac12\min\{p,1-p\}\big)$. Apply \Cref{lemma:UFHatGraphonSeq} with $W_{\ref{lemma:UFHatGraphonSeq}}=W$, $F_{\ref{lemma:UFHatGraphonSeq}}=K_{1,k}$ and $\delta_{\ref{lemma:UFHatGraphonSeq}}=\delta$. This yields graphons $W_m$, a graphon $W'$, real numbers $\eta_m\to0$, and $(\eta_m,\delta,k+1)$-types $(P_m,R_m)$ such that $\delta_\square(W_m,W)\to0$ and $\norm{W_m-W'}_1\to0$. Since $\delta_\square(W_m,W')\leq \norm{W_m-W'}_1$, we have $\delta_\square(W,W')=0$. Thus $W$ and $W'$ are equivalent, so $|R_W|=|R_{W'}|$ and $|O_W|=|O_{W'}|$. Replacing $W$ by $W'$ if necessary, we may assume that $\norm{W_m-W}_1\to0$.

For each $m$, write $R_m=([q_m],E_m,\sigma_m)$ and define the coloring $\varphi_m$ of $E(K_{q_m})$ by
\begin{equation}\label{eqn:varphim-ij}
\varphi_m(ij)=\begin{cases}\red&\text{$ij\in E_m$ and $\sigma_m(ij)=\red$}\\\green&\text{$ij\in E_m$ and $\sigma_m(ij)=\green$}\\\blue&\text{otherwise}\end{cases}\,.
\end{equation}
Note that $\varphi_m\in\cC_k(q_m)$, where $\cC_k$ is the set of 3-edge-colorings defined in \Cref{dfn:Fk-free-cols}. Indeed, if $\varphi_m$ contained a forbidden pattern from $\cF_k$ then all pairs used by that pattern would lie in $E_m$ (because the pattern has no blue edges), and the induced embedding lemma would give an induced copy of $K_{1,k}$ in the underlying graph of the type, contradicting the choice of $R_m$.

Define $\tau\colon[0,1]\to\{0,p,1\}$ by
$$\tau(t):=\begin{cases}
0\,,&t<\delta\,,\\
p\,,&\delta\leq t\leq1-\delta\,,\\
1\,,&t>1-\delta\,.
\end{cases}$$
Since $\delta<\frac12\min\{p,1-p\}$, there exists a constant $C=C(p,\delta)$ such that $|\tau(t)-s|\leq C|t-s|$ for all $t\in[0,1]$ and $s\in\{0,p,1\}$.
Hence if we define the graphon $G'_m:=\tau\circ W_m$ then
$$\norm{G'_m-W}_1\leq C\norm{W_m-W}_1=o(1)\,.$$
Let $G_m\in\cW$ be the graphon obtained from $\varphi_m$ defined by
$$G_m:=\sum_{ij\in E_\red(\varphi_m)}p\,\bsone_{S^{q_m}_{ij}}+\sum_{ij\in E_\blue(\varphi_m)}\bsone_{S^{q_m}_{ij}}+\sum_{i=1}^{q_m}\bsone_{S^{q_m}_{ii}}\,,$$
where $E_\red$ and $E_\blue$ are the red and blue edge sets of $\varphi_m$, respectively. Since $G_m$ and $G'_m$ differ only on diagonal squares and on irregular pairs of the type,
$$\norm{G_m-G'_m}_1=O\!\left(\eta_m+\frac{1}{q_m}\right)=o(1)\,,$$
and hence $\norm{G_m-W}_1=o(1)$. Because both $G_m$ and $W$ take values in $\{0,p,1\}$, we have $|R_{G_m}|=|R_W|+o(1)$ and $|O_{G_m}|=|O_W|+o(1)$. Since
\begin{equation}\label{eqn:RGm-OGm-ana}
|R_{G_m}|=\frac{2e_\red(\varphi_m)}{q_m^2}\qquad\text{and}\qquad|O_{G_m}|=\frac{2e_\blue(\varphi_m)}{q_m^2}+O\!\left(\frac{1}{q_m}\right)
\end{equation}
while \Cref{lemma:kthOrderMantel} gives $e_\red(\varphi_m)-\Delta e_\blue(\varphi_m)\leq\Delta q_m/2$, it follows that $|R_W|\leq \Delta|O_W|$.

\smallskip
\noindent\textbf{Proving $\cV_\gamma\subseteq\cX^\ast_\gamma$.}
Let $x:=|O_W|$ and $y:=|R_W|$. Above we showed $0\leq y\leq\Delta x$, and we also have $0\leq y\leq 1-x$ and $0\leq\frac{\gamma-x}{y}\leq1$. Hence $(x,y)$ is feasible for \eqref{eqn:calculus-problem} and \Cref{lemma:k1k-calc-prob} implies $H(W)\leq \sE_k(\gamma)$. It is easy to calculate that every $W'\in\cV_\gamma$ satisfies $H(W')=\sE_k(\gamma)$, so $\cV_\gamma\subseteq\cX^\ast_\gamma$. In the remainder we prove $W\in\cX^\ast_\gamma$ is equivalent to a graphon in $\cV_\gamma$.

\smallskip
\noindent\textbf{Proving $\cX^\ast_\gamma\subseteq\cV_\gamma$ up to equivalence.} Since $H(W)=\sE_k(\gamma)$, \Cref{lemma:k1k-calc-prob} proves that the pair $(|O_W|,|R_W|)$ is given precisely by the right-hand side of \eqref{eqn:xsys-opt-coords}, which in particular gives $|R_W|=\Delta|O_W|$.
Moreover, since $W$ takes the values $\{0,p,1\}$ almost everywhere and $t(K_2,W)=\gamma$, this means that for all $\gamma\in(0,\gamma_k)$ we have $p=p_k$, while for all $\gamma\in[\gamma_k,1)$ we have $p=\frac{(k-1)\gamma-1}{k-2}$.

Combining the identities $|R_{G_m}|=|R_W|+o(1)$ and $|O_{G_m}|=|O_W|+o(1)$ proven above with \eqref{eqn:RGm-OGm-ana} and $|R_W|=\Delta|O_W|$ implies
\begin{equation}\label{eqn:graphon-prob-near-extr}
\Phi(\varphi_m)=e_\red(\varphi_m)-\Delta e_\blue(\varphi_m)=o(q_m^2)\,.
\end{equation}
Let $\epsilon>0$, and let $\delta_{\epsilon}>0$ be given by \Cref{thm:kth-order-stability}. Since $q_m=v(R_m)\to\infty$ by \Cref{lemma:UFHatGraphonSeq} and \eqref{eqn:graphon-prob-near-extr} gives $\Phi(\varphi_m)\geq-\delta_{\epsilon}q_m^2$ for all sufficiently large $m$, \Cref{thm:kth-order-stability} gives $d(\varphi_m,\cE_k(q_m))\leq\epsilon q_m^2$ for all sufficiently large $m$. Since $\epsilon>0$ was arbitrary, there exist colorings $\psi_m\in\cE_k(q_m)$, where $\cE_k$ is the extremal family defined in \Cref{dfn:k1k-extremal-fam}, satisfying
$$d(\varphi_m,\psi_m)=o(q_m^2)\,.$$
Similar to $G_m$ above, let $H_m\in\cW$ be the graphon obtained from $\psi_m$ by assigning the values $p$, $0$, and $1$ to red, green, and blue off-diagonal squares, respectively, and $1$ to diagonal squares. We then have
$$\norm{H_m-G_m}_1\leq \frac{2\,d(\varphi_m,\psi_m)}{q_m^2}=o(1)\,,$$
which implies $\norm{H_m-W}_1=o(1)$.

By the definition of $\cE_k(q_m)$, after relabeling the vertices of $K_{q_m}$, there exist disjoint sets $C_1^m,\dots,C_{s_m}^m\subseteq[q_m]$ such that all edges not contained in one of the sets $C_a^m$ are green, and, for every $a\in[s_m]$, the coloring $\psi_m|_{C_a^m}$ is obtained from a connected $\Delta$-regular graph $G_a^m$ on $[\ell_a^m]$ by replacing each vertex $i\in[\ell_a^m]$ by a set $C_{a,i}^m$, where
$$C_a^m=C_{a,1}^m\cup\cdots\cup C_{a,\ell_a^m}^m\qquad\text{and}\qquad\big||C_{a,i}^m|-|C_{a,j}^m|\big|\leq1$$
for all $i,j\in[\ell_a^m]$. Write $\kappa_a^m:=|C_a^m|$ and $\alpha_a^m:=\kappa_a^m/q_m$. Let $I_a^m\subseteq[0,1]$ be the interval of length $\alpha_a^m$ corresponding to $C_a^m$.

We now define a graphon $U_m$. Outside $\bigcup_{a\in[s_m]}I_a^m\times I_a^m$, set $U_m=0$. For each $a\in[s_m]$, partition $I_a^m$ into intervals $I_{a,1}^m,\dots,I_{a,\ell_a^m}^m$ of equal length $\alpha_a^m/\ell_a^m$. On $I_a^m\times I_a^m$, set $U_m=1$ on $I_{a,i}^m\times I_{a,i}^m$ for all $i\in[\ell_a^m]$, set $U_m=p$ on $I_{a,i}^m\times I_{a,j}^m$ whenever $ij\in E(G_a^m)$, and set $U_m=0$ otherwise.

We claim $\norm{U_m-H_m}_1=O(1/q_m)$. Indeed, fix $a\in[s_m]$ and write $n_{a,i}:=|C_{a,i}^m|$. On the square corresponding to $C_a^m$, passing from the intervals of lengths $n_{a,i}/q_m$ to equal intervals of length $\alpha_a^m/\ell_a^m$ changes the area of each square or rectangle on which the value may be positive by at most
$$O\!\left(\left|\frac{n_{a,i}^2}{q_m^2}-\frac{(\alpha_a^m)^2}{(\ell_a^m)^2}\right|+\left|\frac{n_{a,i}n_{a,j}}{q_m^2}-\frac{(\alpha_a^m)^2}{(\ell_a^m)^2}\right|\right)
=O\!\left(\frac{\kappa_a^m}{\ell_a^mq_m^2}\right)\,,$$
since $n_{a,i}=\kappa_a^m/\ell_a^m+O(1)$. There are at most $\ell_a^m+\Delta\ell_a^m/2$ such squares or rectangles, so the total contribution to $\norm{U_m-H_m}_1$ from $C_a^m$ is $O(\kappa_a^m/q_m^2)$. Summing over all $a\in[s_m]$ gives $\norm{U_m-H_m}_1=O(1/q_m)$, which implies $\norm{U_m-W}_1=o(1)$.

If $\gamma\in(0,\gamma_k)$ then $p=p_k$, so $U_m\in\cV$ for every $m$. Hence $\delta_\square(W,\cV)\leq\norm{U_m-W}_1=o(1)$. Thus $W$ is equivalent to a graphon in the $L^1$-closure of $\cV$. Since the family $\cV$ is closed up to equivalence, and since $t(K_2,W)=\gamma$, it follows that $W$ is equivalent to a graphon in $\cV_\gamma$.

If instead $\gamma\in[\gamma_k,1)$ then \eqref{eqn:xsys-opt-coords} gives $|O_W|+|R_W|=1$, so since $W$ takes only the values $\{p,1\}$, the zero area of $U_m$ is $o(1)$. Writing the nonzero component intervals of $U_m$ in nonincreasing order of length as $\alpha_i^m$, and writing $G_i^m$ for the connected $\Delta$-regular graph defining the corresponding component interval, we have
$$1-o(1)=|O_{U_m}|+|R_{U_m}|=\sum_{i\geq1}\frac{(\Delta+1)(\alpha_i^m)^2}{v(G_i^m)}\leq \sum_i(\alpha_i^m)^2\leq1\,,$$
because every connected $\Delta$-regular graph has at least $\Delta+1$ vertices. Hence $\sum_i(\alpha_i^m)^2\to1$, so $\alpha_1^m\to1$ and $\sum_{i\geq2}\alpha_i^m\to0$ as $m\to\infty$. It follows that
$$\frac{(\Delta+1)(\alpha_1^m)^2}{v(G_1^m)}\to1\,,$$
and therefore $v(G_1^m)=\Delta+1$ for all large $m$. It follows that $G_1^m=K_{\Delta+1}$ for all large $m$, and after a rearrangement of the vertex set, $U_m$ differs by $o(1)$ in $L^1$ from $W^\ast_\gamma$. Since also $\norm{U_m-W}_1=o(1)$, we obtain $\delta_\square(W,W^\ast_\gamma)=0$, so $W$ is equivalent to $W^\ast_\gamma$.
\end{proof}

\section{Entropy Density, Rate Function, and Rough Structure}
\label{sec:entropy}
In this section we prove \Cref{thm:main-entropy,thm:main-rate,thm:ent-graphons,thm:gnp-graphons} as well as rough structure of the conditional distribution in \Cref{lemma:rough-struc-fixed-g,lemma:rough-struc-gnp}.

\begin{proof}[Proof of \Cref{thm:main-entropy}]
Fix $\gamma\in(0,1)$. \Cref{prop:graphon-char-fixed-gamma} proves $\cX^\ast_\gamma=\cV_\gamma$ up to equivalence, so every graphon $W\in\cV_\gamma$ achieves $H(W)=\Phi_k(\gamma)$. From the definition of $\cV_\gamma$, it is easy to calculate that $H(W)=\sE_k(\gamma)$, where $\sE_k$ is defined in \eqref{eqn:entropy-dfn}. Combining with \Cref{lemma:FnmGraphonOptAsymps} now immediately implies the theorem.
\end{proof}

\begin{proof}[Proof of \Cref{thm:main-rate}]
Let $\cF:=\{W\in\cW:t_\ind(K_{1,k},W)=0\}$.
Exactly as in \cite[Eq.~17]{perkins2025typical}, we can write, for all fixed $p\in(0,1)$,
\begin{equation}\label{eqn:sup-gamma-calc}
\inf_{W\in\cF}I_p(W)=\inf_{0\leq\gamma\leq1}\left\{-\sE_k(\gamma)+\gamma\log\!\left(\frac{1-p}{p}\right)\right\}-\log(1-p)\,,
\end{equation}
where $\sE_k$ is defined in \eqref{eqn:entropy-dfn}. It remains to show the right-hand side equals $\sR_k(p)$.

We solve the optimization of $N_p(\gamma):=-\sE_k(\gamma)+\gamma\log\big(\frac{1-p}{p}\big)$ over $0\leq\gamma\leq1$ as follows. First consider $0\leq\gamma\leq\gamma_k$. The identity $p_k=(1-p_k)^{k-1}$ directly implies
$$H(p_k)=-(1+(k-2)p_k)\log(1-p_k)\,,$$
and hence
$$\frac{k-2}{1+(k-2)p_k}H(p_k)=-(k-2)\log(1-p_k)=\log\!\left(\frac{1-p_k}{p_k}\right)\,.$$
It follows that
$$N_p(\gamma)=\gamma\log\!\left(\frac{(1-p)p_k}{p(1-p_k)}\right)\,,$$
so $N_p$ is increasing on $[0,\gamma_k]$ if $p<p_k$, constant if $p=p_k$, and decreasing if $p>p_k$.

Now consider $\gamma_k\leq\gamma\leq1$, and write $x:=\frac{(k-1)\gamma-1}{k-2}\in[p_k,1]$.
We then have
$$N_p(\gamma)=-\frac{k-2}{k-1}H(x)+\frac{1+(k-2)x}{k-1}\log\!\left(\frac{1-p}{p}\right)\,,$$
which is strictly convex in $x$ and has derivative
$$\frac{k-2}{k-1}\left(-H'(x)+\log\!\left(\frac{1-p}{p}\right)\right)=\frac{k-2}{k-1}\log\!\left(\frac{x(1-p)}{(1-x)p}\right)\,.$$
Hence there is a unique critical point at $x=p$. It follows that if $p<p_k$ then $N_p$ is increasing on $[\gamma_k,1]$, while if $p\geq p_k$ then the minimum on $[\gamma_k,1]$ is attained uniquely at $x=p$, which corresponds to $\gamma=\frac{1+(k-2)p}{k-1}$.

Summarizing the above, if $p<p_k$ then $N_p$ is increasing on both pieces, so the unique minimizer is $\gamma=0$ and $\inf_\gamma N_p(\gamma)=0$. If $p=p_k$ then $N_p\equiv0$ on $[0,\gamma_k]$, and convexity on $[\gamma_k,1]$ shows that every $\gamma\in[0,\gamma_k]$ is optimal. If $p>p_k$ then the unique minimizer is $\gamma^*=\frac{1+(k-2)p}{k-1}$, and at this point, $N_p(\gamma^\ast)=\log(1-p)-\frac{1}{k-1}\log p$. Subtracting $\log(1-p)$ proves the right-hand side of \eqref{eqn:sup-gamma-calc} is $\sR_k(p)$. Finally, combining this with \Cref{lemma:LDPforGNPkl} implies the theorem.
\end{proof}

\begin{proof}[Proof of \Cref{thm:ent-graphons}]
This is an immediate consequence of \Cref{prop:graphon-char-fixed-gamma}.
\end{proof}

\begin{proof}[Proof of \Cref{thm:gnp-graphons}]
The proof of \Cref{thm:main-rate} gave a characterization of the optimizers of \eqref{eqn:sup-gamma-calc}. Combining this characterization with \Cref{prop:graphon-char-fixed-gamma} directly implies that the set of optimizers $\cX^p_\ast$ of $\Psi_k(p)$ is given by
\begin{equation}\label{eqn:CpStarFormalDef}
\cX^p_\ast=\begin{cases}\{W_0\}&p\in(0,p_k)\\[5pt]\bigcup_{\gamma\in[0,\gamma_k]}\cV_\gamma&p=p_k\\[5pt]\cV_{\gamma_p}&p\in(p_k,1)\end{cases}\,,
\end{equation}
where $W_0\equiv0$ is the all-zero graphon and $\gamma_p:=\frac{1+(k-2)p}{k-1}$. The theorem follows immediately from \eqref{eqn:CpStarFormalDef}.
\end{proof}

\subsection{Rough structure from graphon optimizers}
The following lemmas follow from the characterization of the optimizers of $\Phi_k(\gamma)$ and $\Psi_k(p)$ given in \Cref{prop:graphon-char-fixed-gamma} and \eqref{eqn:CpStarFormalDef}, respectively. See \cite[Propositions~3.8~\&~3.9]{perkins2025typical} for proofs, which are based on \citeauthor{chatterjee2011large} \cite[Theorem~3.1]{chatterjee2011large} and use a result of \citeauthor{hatami2018graph} \cite{hatami2018graph}.

\begin{lemma}\label{lemma:rough-struc-fixed-g}
Fix $k\geq3$ and $\gamma\in(0,1)$. Let $n,m\in\N$ and assume $m\sim\gamma\binom{n}{2}$. Let $G$ be the uniformly random element of $\cF^\ast_{n,m}(K_{1,k})$. For all $\epsilon>0$ and large enough $n$,
$$\bbP\{\delta_\square(G,\cX^\ast_\gamma)\geq\epsilon\}\leq e^{-Cn^2}\,,$$
where $C>0$ is a constant depending only on $k$, $\gamma$, and $\epsilon$.
\end{lemma}

\begin{lemma}\label{lemma:rough-struc-gnp}
Let $p\in(0,1)$ be a constant and let $G\sim G(n,p)$. For all $\epsilon>0$ and large enough $n$,
$$\bbP\{\delta_\square(G,\cX_\ast^p)\geq\epsilon\,|\,G\in\cF^\ast_n(K_{1,k})\}\leq e^{-Cn^2}\,,$$
where $C>0$ is a constant depending only on $\epsilon$ and $p$.
\end{lemma}

\section{The Supercritical Regime}
\label{sec:supercritical}
In this section we prove \Cref{thm:main-almostall} \ref{item:thm:main-almostall-super}. We first introduce recurring parameters and definitions. Throughout this section, fix $k\geq3$, let $\Delta:=k-2$, fix $\gamma\in[\gamma_k,1)$, let $m\sim\gamma\binom{n}{2}$, and define
$$\rho:=\frac{(k-1)\gamma-1}{k-2}\,.$$
The lemmas up to and including \Cref{lemma:super-FPiT} apply to the full half-open interval $\gamma\in[\gamma_k,1)$, as the equality case $\gamma=\gamma_k$ is needed in \Cref{sec:critical}. The remainder of the section applies only to the open interval $\gamma\in(\gamma_k,1)$. Throughout the section, we write $C_k>0$ for a constant depending only on $k$ that may vary from line to line.

The parameters below are chosen in the following order. First choose $0<\alpha<1/(100k)$ sufficiently small. The proof of \Cref{lemma:super-FPiT} below gives a constant $c_\mat=c_\mat(k,\gamma)>0$. Next choose $\delta>0$ sufficiently small so that
\begin{equation}\label{eqn:super-parameter-alpha-delta-K1k}
\delta<\frac{\alpha}{100}\,,\qquad C_k\bigl(H(3\alpha)+\delta\bigr)<\frac{c_\mat}{100}\,,\qquad \rho\pm3\delta\subseteq(0,1)\,.
\end{equation}
After $\alpha$ and $\delta$ are fixed, the proof of \Cref{lemma:super-FPiprime} gives a constant $c_\med=c_\med(k,\gamma,\alpha)>0$. Choose $\epsilon>0$ sufficiently small so that
\begin{equation}\label{eqn:super-parameter-epsilon-K1k}
C_k\bigl(H(\epsilon)+\delta\bigr)<\frac{c_\med}{100}\,,\qquad \epsilon=o(\alpha^k)\,.
\end{equation}
Finally define
$$\tau:=\frac{1}{2}\left(\frac{\epsilon}{2^8k}\right)^3\,.$$

For pairwise disjoint nonempty sets $P_1,\dots,P_{k-1}\subseteq V$, call
$$\Pi=\{P_1,\dots,P_{k-1}\}$$
a \emph{division} of $V$. For a division $\Pi$, write
$$V(\Pi):=\bigcup_{P\in\Pi}P\,,\qquad \Pi_\sparse:=V\setminus V(\Pi)\,.$$
Let $\sD$ denote the set of all divisions of $V$. For all $s\geq0$ define $\sD_s:=\{\Pi\in\sD:|\Pi_\sparse|=s\}$.
For all $G\in\cF^\ast_{n,m}(K_{1,k})$ and divisions $\Pi=\{P_1,\dots,P_{k-1}\}\in\sD$, define the graph $T(G,\Pi)$ on the vertex set $V$ by
$$E(T(G,\Pi)):=\bigcup_{i=1}^{k-1}E(G^c[P_i])\cup E(G[V(\Pi),\Pi_\sparse])\,.$$
Define the quantity
$$b(G,\Pi):=e(T(G,\Pi))+e(G[\Pi_\sparse])\,.$$
Let $\Pi(G)$ denote a canonically chosen division minimizing $b(G,\cdot)$, and let us write $\Pi_\sparse(G)$ instead of $\Pi(G)_\sparse$. Define
$$T(G):=T(G,\Pi(G))\,,\qquad b(G):=b(G,\Pi(G))\,.$$

\subsection{Proof anatomy}
Let $W^\ast:=W^\ast_\gamma$ be the unique optimal graphon in the supercritical regime defined in \eqref{eqn:opt-graphon-super}. Define the set of far graphs
$$\cF^\far:=\{G\in\cF^\ast_{n,m}(K_{1,k}):\delta_\square(G,W^\ast)\geq\tau\}\,.$$
For all $\Pi\in\sD$ define the sets of graphs
$$\arraycolsep=1.4pt
\begin{array}{ll}
\cF_\Pi &:= \{G\in\cF^\ast_{n,m}(K_{1,k})\setminus\cF^\far:\Pi(G)=\Pi\,,\,e(T(G))\geq1\} \,, \\[5pt]
\cF^\ast_\Pi &:= \{G\in\cF^\ast_{n,m}(K_{1,k})\setminus\cF^\far:\Pi(G)=\Pi\,,\,e(T(G))=0\} \,.
\end{array}$$
as well as
$$\arraycolsep=1.4pt
\begin{array}{ll}
\cT_\Pi &:= \{T(G)\cup G[\Pi_\sparse]:G\in\cF_\Pi\} \,, \\[5pt]
\cT^\ast_\Pi &:= \{G[\Pi_\sparse]:G\in\cF^\ast_\Pi\} \,.
\end{array}$$
Using the same degree trichotomy as in \cite{perkins2025typical}, for all $T\in\cT_\Pi$ and $P\in\Pi$ we say that $v$ has
\begin{enumerate}[(\emph{\roman*})]
	\item \emph{low degree in $P$}  if $d_T(v,P)<\alpha|P|$,
    \item \emph{medium degree in $P$} if $\alpha|P|\leq d_T(v,P)\leq(1-\alpha)|P|$, or
	\item \emph{high degree in $P$} if $d_T(v,P)>(1-\alpha)|P|$.
\end{enumerate}
If a vertex $v$ has medium degree in some $P\in\Pi$ with respect to $T\in\cT_\Pi$, we say that $v$ has $(\Pi,T)$-medium degree. For all $\Pi\in\sD$ and $T\in\cT_\Pi$ let
$$\arraycolsep=1.4pt
\begin{array}{ll}
\cF'_\Pi &:= \{G\in\cF_\Pi:\text{there exists a vertex of $(\Pi,T(G))$-medium degree}\} \,, \\[5pt]
\cF_{\Pi,T} &:= \{G\in\cF_\Pi\setminus\cF'_\Pi:T(G)\cup G[\Pi_\sparse]=T\} \,.
\end{array}$$

\begin{lemma}\label{lemma:SuperCloseStructureK1k}
For large enough $n$, every graph $G\in\cF^\ast_{n,m}(K_{1,k})\setminus\cF^\far$ satisfies the following conditions:
\begin{enumerate}[
  label=\textit{(\roman*)},
  ref=(\textit{\roman*}),
  topsep=5pt
]
    \item\label{item:lemSuperK1ki} By editing at most $\epsilon n^2$ edges of $G$, one can obtain a co-$(k-1)$-partite graph.
    \item\label{item:lemSuperK1kii} For all $P\in\Pi(G)$ it holds that $|P|\in\big(\frac{1}{k-1}\pm\delta\big)n$.
    \item\label{item:lemSuperK1kiii} We have $|\Pi_\sparse(G)|\leq\delta n/2$.
    \item\label{item:lemSuperK1kiv} For all distinct $P,Q\in\Pi(G)$ it holds that $|\rho-d_G(P,Q)|\leq\delta$.
    \item\label{item:lemSuperK1kv} For every $v\in V$ and every $P\in\Pi(G)$, the vertex $v$ does not have high degree in $P$ with respect to $T(G)\cup G[\Pi_\sparse(G)]$.
\end{enumerate}
\end{lemma}

For an integer $r\geq2$ let $\cC^r_{n,m}$ denote the set of co-$r$-partite graphs on $n$ vertices and $m$ edges. As a direct consequence of \Cref{lemma:SuperCloseStructureK1k}, we have
\begin{equation}\label{eqn:union-bound-super}
\cF^\ast_{n,m}(K_{1,k})\subseteq\cC^{k-1}_{n,m}\cup\bigcup_{s=1}^{\delta n/2}\bigcup_{\Pi\in\sD_s}\cF^\ast_\Pi\cup\bigcup_{\Pi\in\sD}\Bigg(\bigcup_{T\in\cT_\Pi}\cF_{\Pi,T}\cup\cF'_\Pi\Bigg)\cup\cF^\far\,.
\end{equation}
The proof of \Cref{thm:main-almostall} \ref{item:thm:main-almostall-super} in this section consists in establishing bounds for the cardinalities of the sets on the right-hand side of \eqref{eqn:union-bound-super}. To establish \Cref{lemma:SuperCloseStructureK1k}, we need the following definition, similar to that in \cite{perkins2025typical}.

\begin{definition}[Discretization of a graphon]\label{def:DiscreteGraphon}
For a graphon $W\in\cW$ and an integer $n$, define a weighted graph $H_n$ on vertex set $[n]$ by assigning weight $W(i/n,j/n)$ to the pair $ij$, and assigning weight $1$ to each vertex. The $n$th discretization of $W$ is the graphon $W_n:=W_{H_n}$, where the graphon $W_H$ associated with a weighted graph $H$ is defined in \Cref{subsec:graphons}.
\end{definition}

\begin{proof}[Proof of \Cref{lemma:SuperCloseStructureK1k}]
For all $n\in\N$, let $H_n$ be the weighted graph on $V=[n]$ whose $ij$ edge weight is $W^\ast(i/n,j/n)$ and whose node-weights all equal $1$, and let $W_n:=W_{H_n}$ be the $n$th discretization of $W^\ast$. In the remainder, all statements hold for large enough $n$.

Since $W_n\to W^\ast$ pointwise almost everywhere, we have $\delta_\square(W^\ast,W_n)\to0$. Hence if $\delta_\square(G,W^\ast)<\tau$ then $\delta_\square(G,W_n)<2\tau$ for large enough $n$. By \cite[Lemma~8.9]{lovasz2012large} we have $\delta_\square(G,H_n)=\delta_\square(W_G,W_n)$, and by \cite[Theorem~2.3]{borgs2008convergent} we have
$$\widehat{\delta}_\square(G_1,G_2)\leq32\cdot(\delta_\square(G_1,G_2))^{1/3}$$
for edge-weighted graphs with weights in $[-1,1]$ (the exponent of $1/3$ instead of $1/67$ follows from the remark on p.~1831 of \cite{borgs2008convergent}). Since $32(2\tau)^{1/3}=\epsilon/(8k)$, we deduce that $\widehat{\delta}_\square(G,H_n)<\epsilon/(8k)$.

\smallskip
\noindent\emph{Proof of \ref{item:lemSuperK1ki}.}
The vertex set of $H_n$ can be partitioned into $k-1$ parts $V_1\cup\cdots\cup V_{k-1}$ along with at most $k$ exceptional vertices coming from the discretization, where $||V_i|-|V_j||\leq1$ and every edge inside each $V_i$ has weight $1$. Since $\widehat{\delta}_\square(G,H_n)<\epsilon/(8k)$, we may assume the vertices of $G$ are labeled such that $d_\square(G,H_n)=\widehat{\delta}_\square(G,H_n)\leq\epsilon/(8k)$. For every $i\in[k-1]$, the discrepancy between $G$ and $H_n$ on $V_i\times V_i$ is $e(G^c[V_i])$, and therefore $e(G^c[V_i])\leq\epsilon n^2/(8k)$ for all $i\in[k-1]$, which further implies
$$\sum_{i=1}^{k-1}e(G^c[V_i])\leq \frac{\epsilon n^2}{8}\,.$$
By adding all missing edges inside each $V_i$ and adding at most $2kn$ edges to include the exceptional vertices in $V_1$, we obtain a co-$(k-1)$-partite graph after at most $\epsilon n^2$ edits, proving \ref{item:lemSuperK1ki}.

\smallskip
\noindent\emph{Proof of \ref{item:lemSuperK1kii} and \ref{item:lemSuperK1kiii}.}
Write $\Pi:=\Pi(G)=\{P_1,\dots,P_{k-1}\}$ for the remainder of the proof. By \ref{item:lemSuperK1ki} there exists a partition $V=V_1\cup\cdots\cup V_{k-1}$ such that $\sum_i e(G^c[V_i])\leq \epsilon n^2$ and hence $b(G)\leq\epsilon n^2$.

Define the graph $H$ on $V$ with edge set
$$E(H):=\Bigg(\bigcup_{i=1}^{k-1}E(K_{P_i})\Bigg)\cup\Bigg(\bigcup_{1\leq i<j\leq k-1}E(G[P_i,P_j])\Bigg)\,,$$
that is, $H$ is obtained from $G$ by adding all missing edges inside each $P_i$ and deleting all edges incident to $\Pi_\sparse$. By construction, we have $\widehat{\delta}_\square(G,H)\leq b(G)/n^2\leq\epsilon$. Using $\widehat{\delta}_\square(G,H_n)<\epsilon/8$ we conclude that
\begin{equation}\label{eqn:triangle-superK1k}
\widehat{\delta}_\square(H,H_n)\leq \widehat{\delta}_\square(H,G)+\widehat{\delta}_\square(G,H_n)\leq \epsilon+\frac{\epsilon}{8}\,.
\end{equation}
We first show $|\Pi_\sparse|\leq \delta n/2$. Indeed, edges between $\Pi_\sparse$ and $V(\Pi)$ have weight $0$ in $H$, while edges in $H_n$ outside the exceptional rows and columns have weight at least $\rho$, hence
$$\widehat{\delta}_\square(H,H_n)\geq\rho\cdot\frac{|\Pi_\sparse|(n-|\Pi_\sparse|)}{n^2}-\frac{2k}{n}\,.$$
If $|\Pi_\sparse|>\delta n/2$ then the right-hand side is at least $\rho\frac{\delta}{4}(1-\frac{\delta}{2})$, which is larger than $\epsilon+\epsilon/8$ by the choice of $\epsilon$, contradicting \eqref{eqn:triangle-superK1k} and proving \ref{item:lemSuperK1kiii}.

Next we show that $\Pi$ is balanced. Order the sets in $\Pi$ so that $|P_1|\leq\cdots\leq|P_{k-1}|$. If $|P_1|\leq(\frac{1}{k-1}-\delta)n$ then since $|\Pi_\sparse|\leq\delta n/2$ we have
$$|P_{k-1}|\geq n-|\Pi_\sparse|-(k-2)|P_1|\geq \left(\frac{1}{k-1}+\frac{\delta}{2}\right)n$$
for all sufficiently small $\delta$. In $H$, the induced subgraph on $P_{k-1}$ is complete, while in $H_n$ no cluster of weight $1$ can contain more than $\lceil n/(k-1)\rceil$ vertices. Consequently, there are at least $\frac{\delta n}{4}\cdot\frac{n}{k-1}$ pairs $\{x,y\}\subseteq P_{k-1}$ whose edge has weight $1$ in $H$ but weight $\rho$ in $H_n$. It follows that
$$\widehat{\delta}_\square(H,H_n)\geq (1-\rho)\cdot \frac{1}{n^2}\cdot \frac{\delta n}{4}\cdot\frac{n}{k-1}\,,$$
which again exceeds $\epsilon+\epsilon/8$, contradicting \eqref{eqn:triangle-superK1k} and proving \ref{item:lemSuperK1kii}.

\smallskip
\noindent\emph{Proof of \ref{item:lemSuperK1kiv}.}
The preceding paragraph also shows that, after relabeling the parts of $\Pi$, each $P_i$ differs from the corresponding weight-$1$ block of $H_n$ in at most $C_k\delta n$ vertices. If $|d_G(P_i,P_j)-\rho|>\delta$ for some $i\neq j$ then the discrepancy between $H$ and $H_n$ on $P_i\times P_j$ is at least $c_k\delta n^2-O_k(\delta n^2)-O_k(n)$. By taking $\delta$ small first and then $\epsilon$ sufficiently small, this contradicts \eqref{eqn:triangle-superK1k} and proves \ref{item:lemSuperK1kiv}.

\smallskip
\noindent\emph{Proof of \ref{item:lemSuperK1kv}.}
First suppose there exists $v\in P_i$ such that
$$d_{T(G)}(v,P_i)=d^{c}_G(v,P_i)>(1-\alpha)|P_i|\,.$$
Let $j\neq i$ and consider the division $\Pi'$ obtained from $\Pi(G)$ by moving $v$ from $P_i$ to $P_j$. The quantity $b(G,\Pi)$ changes only through edges involving $v$, and we have
$$b(G,\Pi')-b(G,\Pi(G))=d^{c}_G(v,P_j)-d^{c}_G(v,P_i)\,.$$
By optimality of $\Pi(G)$ this is nonnegative, hence $d^{c}_G(v,P_j)>(1-\alpha)|P_i|$. Using \ref{item:lemSuperK1kii}, it follows that
$$d_G(v,P_j)=|P_j|-d^{c}_G(v,P_j)\leq |P_j|-(1-\alpha)|P_i|\leq 3\alpha n$$
for all sufficiently small $\delta$. Similarly we have $d_G(v,P_i)\leq\alpha|P_i|\leq\alpha n$. Summing over $j\in[k-1]$ and using \ref{item:lemSuperK1kiii}, we obtain $d_G(v)\leq4\alpha n$.

Now define the division $\Pi''$ obtained from $\Pi(G)$ by moving $v$ from $P_i$ into $\Pi_\sparse$. Then we have
$$b(G,\Pi'')-b(G,\Pi(G))=d_G(v)-d^{c}_G(v,P_i)\,.$$
Since $d^{c}_G(v,P_i)>(1-\alpha)|P_i|\geq (\frac{1}{k-1}-\delta)(1-\alpha)n$ while $d_G(v)\leq4\alpha n$, the right-hand side above is negative, contradicting optimality of $\Pi(G)$.

It remains to consider $v\in\Pi_\sparse$. If $v$ is moved from $\Pi_\sparse$ into $P_i$ then
$$b(G,\Pi')-b(G,\Pi(G))=d_G^c(v,P_i)-d_G(v,V(\Pi))\,,$$
which again is nonnegative by optimality. Therefore
$$|P_i|-d_G(v,P_i)=d_G^c(v,P_i)\geq d_G(v,V(\Pi))\geq d_G(v,P_i)\,,$$
and hence $d_G(v,P_i)\leq |P_i|/2<(1-\alpha)|P_i|$. Thus no vertex in $\Pi_\sparse$ has high degree into any $P_i$, completing the proof.
\end{proof}

\subsection{The counting argument}
For all $\Pi=\{P_1,\dots,P_{k-1}\}\in\sD$ and integers $t$, define
$$\cM_{\Pi,t}:=
\left\{\bsm\in\N^{\binom{k-1}{2}}:
\begin{array}{l}
\norm{\bsm}_1+e(\Pi^c)+t=m \text{ and } \\[8pt]
\text{for all $1\leq i<j\leq k-1$,}~~\frac{\bsm_{ij}}{|P_i|\cdot|P_j|}\in \rho\pm\delta
\end{array}\right\}\,,$$
where $e(\Pi^c)=\sum_{i=1}^{k-1}\binom{|P_i|}{2}$. Let us also write
$$\cM_\Pi:=\bigcup_{-\epsilon n^2\leq t\leq\epsilon n^2}\cM_{\Pi,t}\,.$$
Thus an $\bsm\in\cM_{\Pi,t}$ assigns $\bsm_{ij}$ edges to the bipartite graph between $P_i$ and $P_j$, and $t$ records the net contribution of $E(T(G))$ and $E(G[\Pi_\sparse])$ to the edge count. For all $\Pi=\{P_1,\dots,P_{k-1}\}\in\sD$ and $\bsm\in\cM_\Pi$, let us define
$$\binom{\Pi}{\bsm}:=\prod_{1\leq i<j\leq k-1}\binom{|P_i|\cdot|P_j|}{\bsm_{ij}}\,.$$
For $T\in\cT_\Pi$, define
$$\sigma_\Pi(T):=e_T(V(\Pi),\Pi_\sparse)+e(T[\Pi_\sparse])-e(T[V(\Pi)])\,,$$
where $e_T(V(\Pi),\Pi_\sparse)$ denotes the number of edges of $T$ with one endpoint in $V(\Pi)$ and the other in $\Pi_\sparse$.

\begin{lemma}\label{lemma:super-FPiprime}
There exists $c_\med=c_\med(k,\gamma,\alpha)>0$ such that the following holds. For all $\Pi\in\sD$ and $\bsm\in\cM_\Pi$, we have
$$\abs{\cF'_\Pi} \leq \binom{\Pi}{\bsm}\,e^{-c_\med n^2}\,.$$
\end{lemma}
\begin{proof}
Let us write $\Pi=\{P_1,\dots,P_{k-1}\}$. For all $T\in\cT_\Pi$, $v\in V$, graphs $F$ on $P_1\cup\cdots\cup P_{k-1}\cup\{v\}$ whose edges are all incident with $v$, and $\bsm'\in\cM_\Pi$, let $\cF'_{\Pi,T,v,F,\bsm'}$ be the set of graphs $G\in\cF'_\Pi$ such that $T(G)\cup G[\Pi_\sparse]=T$, the vertex $v$ has $(\Pi,T)$-medium degree, the adjacencies between $v$ and $V(\Pi)$ are given by $F$, and $e_G(P_i,P_j)=\bsm'_{ij}$ for all $1\leq i<j\leq k-1$.
Then
\begin{equation}\label{eqn:super-medium-cover-K1k}
\abs{\cF'_\Pi}\leq \sum_{T\in\cT_\Pi}\sum_{v\in V}\sum_F\sum_{\bsm'\in\cM_\Pi}\abs{\cF'_{\Pi,T,v,F,\bsm'}}\,.
\end{equation}

Fix $T,v,F,\bsm'$ such that $\cF'_{\Pi,T,v,F,\bsm'}$ is nonempty. Since defect edges lie either inside one of the sets $P\in\Pi$ or are incident with $\Pi_\sparse$, any vertex of $(\Pi,T)$-medium degree outside $\Pi_\sparse$ must lie in the unique set $P\in\Pi$ in which it has medium degree. After relabeling the sets in $\Pi$, we may therefore assume that $v$ has medium degree in $P_1$ and that $v\in P_1\cup\Pi_\sparse$.

For all $G\in\cF'_{\Pi,T,v,F,\bsm'}$, \Cref{lemma:SuperCloseStructureK1k} gives $|P_i|\in(\frac{1}{k-1}\pm\delta)n$ for all $i\in[k-1]$. Write $a_i:=|P_i|$ and define
$$N:=\begin{cases}\{x\in P_1\setminus\{v\}:vx\not\in E(F)\}&\text{if }v\in P_1\\[8pt]\{x\in P_1:vx\in E(F)\}&\text{if }v\in \Pi_\sparse
\end{cases}\,.$$
We then have $\alpha a_1\leq |N|\leq (1-\alpha)a_1$. If $j\geq2$ then optimality of $\Pi(G)$ gives a set $N_j\subseteq P_j$ such that $|N_j|\geq\alpha a_j/2$ and $vx\not\in E(F)$ for all $x\in N_j$. Indeed, if $v\in P_1$ then $d_G^c(v,P_j)\geq d_G^c(v,P_1)=|N|$, while if $v\in\Pi_\sparse$ then $d_G^c(v,P_j)\geq d_G(v,V(\Pi))\geq d_G(v,P_1)=|N|$.

Let $Q_1:=P_1\setminus\{v\}$ if $v\in P_1$, and $Q_1:=P_1$ if $v\in\Pi_\sparse$. Also set $Q_i:=P_i$ for all $i\geq2$. Define integers $\widetilde\bsm_{ij}$ by setting $\widetilde\bsm_{ij}:=\bsm'_{ij}$ for all $2\leq i<j\leq k-1$, and also for all $1j$ if $v\in\Pi_\sparse$, while for $v\in P_1$ and $j\geq2$ we set
$$\widetilde\bsm_{1j}:=\bsm'_{1j}-\abs{\{x\in P_j:vx\in E(F)\}}\,.$$
Let $G=G_{\Pi,T,v,F,\bsm'}$ be the random graph on $V$ obtained by choosing, independently for $1\leq i<j\leq k-1$, a uniformly random $\widetilde\bsm_{ij}$-element subset $A_{ij}\subseteq E(K_{Q_i,Q_j})$, and setting
\begin{equation}\label{eqn:super-fixed-model-medium-K1k}
E(G):=\left(\bigcup_{i=1}^{k-1}\binom{P_i}{2}\cup E(F[V(\Pi)])\cup\bigcup_{1\leq i<j\leq k-1}A_{ij}\right)\triangle\,E(T)\,.
\end{equation}
Every graph in $\cF'_{\Pi,T,v,F,\bsm'}$ is an outcome of $G$, and $\prod_{1\leq i<j\leq k-1}\binom{|Q_i||Q_j|}{\widetilde\bsm_{ij}}\leq\binom{\Pi}{\bsm'}$. Combining this with \eqref{eqn:super-medium-cover-K1k} gives
\begin{equation}\label{eqn:super-medium-counting-setup-K1k}
\abs{\cF'_\Pi}\leq \sum_{T\in\cT_\Pi}\sum_{v\in V}\sum_F\sum_{\bsm'\in\cM_\Pi}\binom{\Pi}{\bsm'}\,\bbP\{G_{\Pi,T,v,F,\bsm'}\in\cF^\ast_{n,m}(K_{1,k})\}\,,
\end{equation}
where the summand is interpreted as $0$ if $\cF'_{\Pi,T,v,F,\bsm'}=\emptyset$.

For $1\leq i<j\leq k-1$, set $q_{ij}:=\widetilde\bsm_{ij}/(|Q_i||Q_j|)$. The random graph $G^\bin=G^\bin_{\Pi,T,v,F,\bsm'}$ is defined in the same way as $G$, except that now each $A_{ij}$ is obtained by including every edge of $E(K_{Q_i,Q_j})$ independently with probability $q_{ij}$.
Since $\cF'_{\Pi,T,v,F,\bsm'}$ is nonempty, all probabilities $q_{ij}$ lie in $\rho\pm2\delta$ for large enough $n$. Conditioning on the exact values $|A_{ij}|=\widetilde\bsm_{ij}$ for all $i<j$ recovers the random graph $G$. Since there are only $\binom{k-1}{2}$ pairs and the probabilities $q_{ij}$ are bounded away from $0$ and $1$, we have that for every event $\cE$, $\bbP\{G\in\cE\}\leq n^{O_k(1)}\bbP\{G^\bin\in\cE\}$.

We now define a family $\cK$ of copies of $K_{1,k}$. If $v\in P_1$ then for $x\in N$, $z\in Q_1\setminus(N\cup\{x\})$, and $y_j\in N_j$ for $2\leq j\leq k-1$, let $K(x,z,y_2,\dots,y_{k-1})$ be the copy whose center is $z$ and whose leaves are $v,x,y_2,\dots,y_{k-1}$. If $v\in\Pi_\sparse$ then for $x\in N$, $z\in P_1\setminus(N\cup\{x\})$, and $y_j\in N_j$ for $2\leq j\leq k-1$, let $K(x,z,y_2,\dots,y_{k-1})$ be the copy whose center is $x$ and whose leaves are $v,z,y_2,\dots,y_{k-1}$. In both cases, discard every choice for which some deterministic adjacency required inside $P_1$ is an edge of $T$. Since $e(T)\leq\epsilon n^2$ and $\epsilon=o(\alpha^k)$, the number of discarded choices is $o(\alpha^k n^k)$, so $|\cK|\geq c\alpha^kn^k$ for a constant $c=c(k)>0$.

For $K\in\cK$, let $A_K$ be the event that $K$ appears in $G^\bin$ as an induced copy of $K_{1,k}$. All random edges appearing in $A_K$ lie between distinct parts $Q_i$ and $Q_j$, and all probabilities $q_{ij}$ lie in $\rho\pm2\delta$. Therefore $\bbP\{A_K\}\geq c_1$ for a constant $c_1=c_1(k,\gamma)>0$, and hence
$$\mu:=\sum_{K\in\cK}\bbP\{A_K\}\geq c_2\alpha^kn^k\,.$$
If $A_K$ and $A_{K'}$ are dependent then the two copies share a random pair between two distinct parts of the partition $\{Q_1,\dots,Q_{k-1}\}$. For each fixed $K$, there are at most $C_kn^{k-2}$ possible choices of $K'$, and consequently
$$\Delta:=\sum_{K\sim K'}\bbP\{A_K\cap A_{K'}\}\leq C_kn^{2k-2}\,.$$
It follows that $\mu/2\geq c_3\alpha^kn^k$ and $\mu^2/(4\Delta)\geq c_3\alpha^{2k}n^2$. By \Cref{thm:JansonsInequality} we obtain
$$\bbP\!\left\{\bigcap_{K\in\cK}\{G^\bin\not\geq K\}\right\}\leq e^{-c_4\alpha^{2k}n^2}\,,$$
where $\not\geq$ denotes exclusion as an induced subgraph. It follows that
$$\bbP\{G\not\geq K_{1,k}\}\leq n^{O_k(1)}\bbP\{G^\bin\not\geq K_{1,k}\}\leq n^{O_k(1)}\bbP\!\left\{\bigcap_{K\in\cK}\{G^\bin\not\geq K\}\right\}\leq e^{-c_5\alpha^{2k}n^2}\,.$$
Inserting this bound into \eqref{eqn:super-medium-counting-setup-K1k}, and using that every $T\in\cT_\Pi$ has at most $\epsilon n^2$ edges by \Cref{lemma:SuperCloseStructureK1k}, that there are at most $2^n$ possibilities for $F$, and that $|\cM_\Pi|=n^{O(1)}$, we obtain
\begin{align*}
\abs{\cF'_\Pi}&\leq \sum_{T\in\cT_\Pi}\sum_{v\in V}\sum_F\sum_{\bsm'\in\cM_\Pi}\binom{\Pi}{\bsm'}e^{-c_5\alpha^{2k}n^2} \\
&\leq \max_{\bsm'\in\cM_\Pi}\binom{\Pi}{\bsm'}\exp\!\left(-c_5\alpha^{2k}n^2+H(\epsilon)(\log2)n^2+O(n)\right).
\end{align*}
Finally, if $\bsm,\bsm'\in\cM_\Pi$ then the corresponding densities all lie in $\rho\pm\delta$, and Stirling's formula implies $\log\binom{\Pi}{\bsm'}-\log\binom{\Pi}{\bsm}=O(\delta n^2)$. The choice of $\epsilon$ and $\delta$ completes the proof, after setting $c_\med>0$ sufficiently small.
\end{proof}

For $T\in\cT_\Pi$, let $M(T)$ be a canonically chosen maximum matching in the graph $T[V(\Pi),V]$ and define $h(T):=|M(T)|$.

\begin{lemma}\label{lemma:super-FPiT}
There exists $c_\mat=c_\mat(k,\gamma)>0$ such that the following holds. Let $\Pi\in\sD$, $T\in\cT_\Pi$, and $\bsm\in\cM_{\Pi,\sigma_\Pi(T)}$. If we define
$$\cF_{\Pi,T,\bsm}:=\left\{G\in\cF_{\Pi,T}:e_G(P_i,P_j)=\bsm_{ij}\text{ for all }1\leq i<j\leq k-1\right\}$$
then we have
$$\abs{\cF_{\Pi,T,\bsm}} \leq \binom{\Pi}{\bsm}\,e^{-c_\mat h(T)n}\,.$$
\end{lemma}
\begin{proof}
If $\cF_{\Pi,T,\bsm}=\emptyset$ there is nothing to prove. Write $\Pi=\{P_1,\dots,P_{k-1}\}$ and $h:=h(T)$. By \Cref{lemma:SuperCloseStructureK1k} and the definition of $\cF_{\Pi,T}$, every vertex has low degree in every $P_i$ with respect to $T$. In particular, in the graph $T[V(\Pi),V]$, every neighborhood in one of the parts $P_i$ has size at most $\alpha|P_i|$.

The matching $M(T)$ is partitioned into the $2(k-1)$ submatchings $M(T)[P_i]$ and $M(T)[P_i,\Pi_\sparse]$ for $i\in[k-1]$. Let $M_0$ be one such submatching with $|M_0|\geq h/(2(k-1))$, and let $M'\subseteq M_0$ be a submatching of size $\ceil*{|M_0|/2}$.

Let $G=G_{\Pi,T,\bsm}$ be the random graph obtained as follows: inside each $P_i$ we put all edges except those in $T[P_i]$, between $V(\Pi)$ and $\Pi_\sparse$ and inside $\Pi_\sparse$ we put exactly the edges of $T$, and between each pair $P_i,P_j$ with $i<j$ we choose a uniformly random $\bsm_{ij}$-element subset of $E(K_{P_i,P_j})$, independently over all such pairs. Let $G^\bin$ be defined in the same way as $G$, except that each edge between $P_i$ and $P_j$ is included independently with probability $q_{ij}:=\bsm_{ij}/(|P_i||P_j|)\in\rho\pm\delta$. Conditioning $G^\bin$ on the events $|E(G^\bin[P_i,P_j])|=\bsm_{ij}$ for all $1\leq i<j\leq k-1$ recovers $G$. Since the probabilities $q_{ij}$ are bounded away from $0$ and $1$, Stirling's formula gives that the intersection of these events has probability at least $n^{-O_k(1)}$, and therefore $\bbP\{G\in\cE\}\leq n^{O_k(1)}\bbP\{G^\bin\in\cE\}$ for every event $\cE$.

First suppose $M'\subseteq T[P_i]$. For every $xy\in M'$, let $\cK_{xy}$ be the family of copies of $K_{1,k}$ whose center is a vertex $z\in P_i\setminus V(M')$, whose leaves include $x$ and $y$, and whose remaining leaves consist of one vertex in each part $P_j$ with $j\neq i$. We discard choices for which $xz\in E(T)$ or $yz\in E(T)$; this discards only at most $O_k(\alpha n^{k-1})$ copies, since both $x$ and $y$ have low degree in $P_i$, and since $P_i\setminus V(M')$ and all other parts have size $\Omega_k(n)$. It follows that $|\cK_{xy}|\geq c_1n^{k-1}$ for every $xy\in M'$.

Now suppose $M'\subseteq T[P_i,\Pi_\sparse]$, and let $X:=V(M')\cap P_i$. For every $xy\in M'$, with $x\in P_i$ and $y\in\Pi_\sparse$, let $\cK_{xy}$ be the family of copies of $K_{1,k}$ centered at $x$, whose leaves include $y$, one vertex $z\in P_i\setminus X$, and one vertex in each part $P_j$ with $j\neq i$. We discard choices in which $z$ is adjacent to $y$, some chosen vertex in $P_j$ is adjacent to $y$, or $xz\in E(T)$. Since $y$ has low degree into every $P_j$, $x$ has low degree in $P_i$, and $P_i\setminus X$ has size $\Omega_k(n)$, we again have $|\cK_{xy}|\geq c_1n^{k-1}$ for every $xy\in M'$.

In either case, let $\cK:=\bigcup_{xy\in M'}\cK_{xy}$. For $K\in\cK$, define the event $A_K:=\{G^\bin[V(K)]\cong K\}$. Since $q_{ij}\in\rho\pm\delta$ for all $i,j$, we calculate that
$$\mu:=\sum_{K\in\cK}\bbP\{A_K\}\geq c_2|M'|n^{k-1}\,.$$
If $A_K$ and $A_{K'}$ are dependent then the two copies share a random pair between two distinct parts of $\Pi$. For a fixed $K$, after choosing such a shared pair, there are at most $C_kn^{k-2}+C_k|M'|n^{k-3}\leq C_kn^{k-2}$ choices for $K'$, hence
$$\Delta:=\sum_{K\sim K'}\bbP\{A_K\cap A_{K'}\}\leq C_k|M'|n^{2k-3}\,.$$
It follows that $\mu/2\geq c_3|M'|n$ and $\mu^2/(4\Delta)\geq c_3|M'|n$, after decreasing $c_3>0$. Since the events $A_K$ are positively correlated by the choice of the families $\cK_{xy}$, \Cref{thm:JansonsInequality} gives
$$\bbP\!\left\{\bigcap_{K\in\cK}\{G^\bin\not\geq K\}\right\}\leq e^{-c_3|M'|n}\leq e^{-c_4hn}\,.$$
Every graph in $\cF_{\Pi,T,\bsm}$ is an outcome of $G$, hence
$$\abs{\cF_{\Pi,T,\bsm}}\leq \binom{\Pi}{\bsm}\bbP\{G\not\geq K_{1,k}\}\leq \binom{\Pi}{\bsm}n^{O(1)}e^{-c_4hn}\leq \binom{\Pi}{\bsm}e^{-c_\mat hn}$$
for large enough $n$, after decreasing $c_\mat>0$.
\end{proof}

For all $\Pi=\{P_1,\dots,P_{k-1}\}\in\sD_s$, after ordering the parts so that $|P_1|\leq\cdots\leq|P_{k-1}|$, define the division
$$\Pi_\ast:=\{P_1\cup\Pi_\sparse,P_2,\dots,P_{k-1}\}$$
by moving the vertices from the sparse part $\Pi_\sparse$ to $P_1$.

\begin{lemma}\label{lemma:super-Fstar}
Assume $\gamma\in(\gamma_k,1)$. For all $s\geq1$, $\Pi=\{P_1,\dots,P_{k-1}\}\in\sD_s$, integers $0\leq t\leq\binom{s}{2}$, and $\bsm\in\cM_{\Pi,t}$, define the set of graphs
$$\cF^\ast_{\Pi,\bsm}:=\left\{G\in\cF^\ast_\Pi:e_G(P_i,P_j)=\bsm_{ij}\text{ for all }1\leq i<j\leq k-1\right\}\,.$$
We then have
\begin{equation}\label{eqn:fstarpim-enusn}
\abs{\cF^\ast_{\Pi,\bsm}}\leq\binom{\Pi_\ast}{\bsm}\,e^{-\nu sn}
\end{equation}
for a constant $\nu=\nu(k,\gamma)>0$.
\end{lemma}
\begin{proof}
If $\cF^\ast_{\Pi,\bsm}=\emptyset$ there is nothing to prove. For all $G\in\cF^\ast_{\Pi,\bsm}$, since $\Pi(G)=\Pi$, \Cref{lemma:SuperCloseStructureK1k} gives $|P_i|\in(\frac{1}{k-1}\pm\delta)n$ for all $i\in[k-1]$ and $s\leq\delta n/2$.
Write $a_i:=|P_i|$ for $i\in[k-1]$. Since every $G\in\cF^\ast_{\Pi,\bsm}$ satisfies $e(T(G))=0$, it has no missing edges inside the parts $P\in\Pi$ and no edges between $V(\Pi)$ and $\Pi_\sparse$. Thus such a graph is determined by the induced bipartite graphs $G[P_i,P_j]$ together with $G[\Pi_\sparse]$. It follows that $\abs{\cF^\ast_{\Pi,\bsm}}\leq\binom{\Pi}{\bsm}2^{\binom{s}{2}}$. All factors in $\binom{\Pi}{\bsm}$ and $\binom{\Pi_\ast}{\bsm}$ agree except those involving $P_1$, hence
\begin{equation}\label{eqn:super-Fstar-ratio}
\binom{\Pi}{\bsm}\binom{\Pi_\ast}{\bsm}^{-1}
=\prod_{j=2}^{k-1}\binom{a_1a_j}{\bsm_{1j}}\binom{(a_1+s)a_j}{\bsm_{1j}}^{-1}\,.
\end{equation}
Fix $j\in\{2,\dots,k-1\}$. Since $\bsm\in\cM_{\Pi,t}$, we have $\bsm_{1j}\geq(\rho-\delta)a_1a_j$, further implying
\begin{align*}
\binom{a_1a_j}{\bsm_{1j}}\binom{(a_1+s)a_j}{\bsm_{1j}}^{-1}&=\prod_{r=0}^{\bsm_{1j}-1}\frac{a_1a_j-r}{(a_1+s)a_j-r}
\leq\left(\frac{a_1}{a_1+s}\right)^{\bsm_{1j}} \\
&\leq\exp\!\left(-\frac{\bsm_{1j}s}{a_1+s}\right)
\leq\exp\!\left(-\frac{(\rho-\delta)a_1a_j}{a_1+s}\,s\right).
\end{align*}
Now $a_i\geq(\frac{1}{k-1}-\delta)n$ for all $i\in[k-1]$, while $a_1+s\leq n$. Hence, for $\delta$ sufficiently small,
$$\frac{(\rho-\delta)a_1a_j}{a_1+s}\geq \frac{\rho}{2(k-1)^2}\,n\,.$$
Combining this estimate over all $j=2,\dots,k-1$ in \eqref{eqn:super-Fstar-ratio}, we obtain
$$\binom{\Pi}{\bsm}\leq\binom{\Pi_\ast}{\bsm}\exp\!\left(-\frac{(k-2)\rho}{2(k-1)^2}\,sn\right)\,.$$
Since $s\leq\delta n/2$, the factor $2^{\binom{s}{2}}$ is absorbed into the exponential penalty for sufficiently small $\delta$, proving \eqref{eqn:fstarpim-enusn}.
\end{proof}

\begin{proof}[Proof of \Cref{thm:main-almostall} \ref{item:thm:main-almostall-super}]
From \eqref{eqn:union-bound-super} we have
\begin{equation}\label{eqn:super-union-bd}
N^\ast_{n,m}(K_{1,k})\leq C^{k-1}_{n,m}+\sum_{s=1}^{\delta n/2}\sum_{\Pi\in\sD_s}\abs{\cF^\ast_\Pi}+\sum_{\Pi\in\sD}\Bigg(\sum_{T\in\cT_\Pi}\abs{\cF_{\Pi,T}}+\abs{\cF'_\Pi}\Bigg)+\abs{\cF^\far}\,,
\end{equation}
where $C^{k-1}_{n,m}:=|\cC^{k-1}_{n,m}|$. We bound the second, third, and fourth terms on the right-hand side of \eqref{eqn:super-union-bd} one at a time.

Define the set of balanced divisions
$$\sB:=\left\{\Pi\in\sD_0:|P|\in\left(\frac{1}{k-1}\pm\delta\right)n~\text{ for all }P\in\Pi\right\}\,.$$
If $\Pi\in\sD_s$ and $\cF^\ast_\Pi\neq\emptyset$ then \Cref{lemma:SuperCloseStructureK1k}, together with the choice of $P_1$, implies that $\Pi_\ast\in\sB$. Moreover, for all $\Pi'\in\sB$ and $s\geq1$, there are at most $\binom{n}{s}k^s$ divisions $\Pi\in\sD_s$ such that $\Pi_\ast=\Pi'$. Hence, using \Cref{lemma:super-Fstar},
\begin{align*}
\sum_{s=1}^{\delta n/2}\sum_{\Pi\in\sD_s}|\cF^\ast_\Pi|&= \sum_{s=1}^{\delta n/2}\sum_{\Pi\in\sD_s}\sum_{t=0}^{\binom{s}{2}}\sum_{\bsm\in\cM_{\Pi,t}}|\cF^\ast_{\Pi,\bsm}| \leq \sum_{s=1}^{\delta n/2}\sum_{\Pi\in\sD_s}\sum_{t=0}^{\binom{s}{2}}\sum_{\bsm\in\cM_{\Pi,t}}\binom{\Pi_\ast}{\bsm}e^{-\nu sn}\,.
\end{align*}
For each $\Pi\in\sD_s$ and $t\leq\binom{s}{2}$, the quantities $e(\Pi_\ast^c)$ and $e(\Pi^c)+t$ differ by $O(sn)$. Also, if $\bsm\in\cM_{\Pi,t}$ then the densities defining $\binom{\Pi_\ast}{\bsm}$ all lie in $\rho\pm C_k\delta$. Using Stirling's formula,
$$\sum_{t=0}^{\binom{s}{2}}\sum_{\bsm\in\cM_{\Pi,t}}\binom{\Pi_\ast}{\bsm}\leq n^{O(1)}e^{O(\delta sn)}\max_{\bsm'\in\cM_{\Pi_\ast,0}}\binom{\Pi_\ast}{\bsm'}\,.$$
All but a $o(1)$ fraction of co-$(k-1)$-partite graphs have a unique covering by $k-1$ cliques, and consequently
$$\sum_{\Pi\in\sB}\sum_{\bsm'\in\cM_{\Pi,0}}\binom{\Pi}{\bsm'}=O(C^{k-1}_{n,m})\,.$$
Therefore, after taking $\delta$ sufficiently small, the factor $e^{O(\delta sn)}$ is absorbed into $e^{-\nu sn}$ and
\begin{align*}
\sum_{s=1}^{\delta n/2}\sum_{\Pi\in\sD_s}|\cF^\ast_\Pi|&\leq O(C^{k-1}_{n,m})\,n^{O(1)}\sum_{s=1}^{\delta n/2}\binom{n}{s}k^se^{-\nu sn/2} \leq e^{-\nu n/5}C^{k-1}_{n,m}\,.
\end{align*}

We next bound the contribution from $\cF_{\Pi,T}$. Fix $\Pi\in\sD_s$ and write $T_0:=T[V(\Pi),V]$ and $J:=T[\Pi_\sparse]$. If $h(T)=h$ then the endpoints of $M(T)$ form a vertex cover of $T_0$. Since all vertices' neighborhoods in the parts $P_i$ are low and $s\leq\delta n/2$, we have
\begin{equation}\label{eqn:super-T0-edge-bound-K1k}
e(T_0)\leq C_kh(\alpha+\delta)n\,.
\end{equation}
Hence for fixed $h$, the number of possibilities for $T_0$ is at most $e^{C_kh(H(3\alpha)+\delta)n+C_kh\log n}$. Indeed, choose the at most $2h$ cover vertices, and then choose for each cover vertex its low neighborhoods inside the parts $P_i$ and its arbitrary neighborhood inside $\Pi_\sparse$.

For fixed $T_0$, write $\sigma_0:=e_{T_0}(V(\Pi),\Pi_\sparse)-e(T_0[V(\Pi)])$. By \eqref{eqn:super-T0-edge-bound-K1k}, $|\sigma_0|\leq C_kh(\alpha+\delta)n$. Combining \Cref{lemma:super-FPiT} with the same comparison used above for clean sparse graphs gives
\begin{align}\label{eqn:super-FPiT-summed-K1k}
\sum_{J\subseteq\binom{\Pi_\sparse}{2}}\sum_{\bsm\in\cM_{\Pi,\sigma_0+e(J)}}|\cF_{\Pi,T_0\cup J,\bsm}|&\leq e^{-c_\mat hn}\sum_{J\subseteq\binom{\Pi_\sparse}{2}}\sum_{\bsm\in\cM_{\Pi,\sigma_0+e(J)}}\binom{\Pi}{\bsm} \nonumber\\
&\leq n^{O(1)}e^{-c_\mat hn+O((\alpha+\delta)hn)}e^{-\nu sn/2}
\max_{\bsm'\in\cM_{\Pi_\ast,0}}\binom{\Pi_\ast}{\bsm'}.
\end{align}
Here the factor $e^{-\nu sn/2}$ is interpreted as $1$ if $s=0$. The proof of the second inequality is exactly the proof of \Cref{lemma:super-Fstar}, with the additional shift $\sigma_0=O(h(\alpha+\delta)n)$; the comparison with $\cM_{\Pi_\ast,0}$ contributes the factor $e^{O((\alpha+\delta)hn)}$, and the choices of $J$ contribute only $e^{O(\delta sn)}$.
Using \eqref{eqn:super-parameter-alpha-delta-K1k}, summing \eqref{eqn:super-FPiT-summed-K1k} over all $T_0$ with $h(T)=h$, and then summing over $h\geq1$, gives
$$\sum_{T\in\cT_\Pi}|\cF_{\Pi,T}|\leq n^{O(1)}e^{-c_\mat n/3}e^{-\nu sn/2}\max_{\bsm'\in\cM_{\Pi_\ast,0}}\binom{\Pi_\ast}{\bsm'}\,.$$
Summing over $\Pi\in\sD_s$ exactly as in the clean case yields
$$\sum_{\Pi\in\sD}\sum_{T\in\cT_\Pi}|\cF_{\Pi,T}|\leq e^{-c_\mat n/4}C^{k-1}_{n,m}\,.$$

For the medium-degree term, \Cref{lemma:super-FPiprime} gives $|\cF'_\Pi|\leq\binom{\Pi}{\bsm}e^{-c_\med n^2}$ for every $\Pi\in\sD$ and $\bsm\in\cM_\Pi$. Summing over $\Pi$ and $\bsm$ and using the same comparison as above gives
$$\sum_{\Pi\in\sD}|\cF'_\Pi|\leq e^{-c_\med n^2/2}C^{k-1}_{n,m}$$
for large enough $n$, by the choice of $\epsilon$ and $\delta$. Finally, \Cref{lemma:rough-struc-fixed-g} yields
$$|\cF^\far|\leq e^{-Cn^2}\,N^\ast_{n,m}(K_{1,k})$$
for some constant $C=C(k,\gamma,\tau)>0$. Substituting these bounds into \eqref{eqn:super-union-bd}, we find
$$N^\ast_{n,m}(K_{1,k})\leq(1+o(1))C^{k-1}_{n,m}+e^{-Cn^2}N^\ast_{n,m}(K_{1,k})\,.$$
Since every co-$(k-1)$-partite graph is induced-$K_{1,k}$-free, this proves $N^\ast_{n,m}(K_{1,k})=(1+o(1))C^{k-1}_{n,m}$. The estimates above also show that all but a $o(1)$ fraction of graphs in $\cF^\ast_{n,m}(K_{1,k})$ are co-$(k-1)$-partite, completing the proof.
\end{proof}

\section{The Critical Density}
\label{sec:critical}
In this section we prove \Cref{thm:main-almostall} \ref{item:thm:main-critical}. Throughout the section, let $p:=p_k$ and $m=\floor{\gamma_k\binom{n}{2}}$. We use the definitions from \Cref{sec:supercritical} with $\gamma=\gamma_k$ and $\rho=p$, and we write $W^\ast=W^\ast_\gamma$ for the unique optimal graphon at edge density $\gamma$. The proofs of \Cref{lemma:SuperCloseStructureK1k,lemma:super-FPiprime,lemma:super-FPiT} remain valid at $\gamma=\gamma_k$ since $p\in(0,1)$. Thus the far graphs, the medium-degree graphs, and the graphs with $e(T(G))\geq1$ contribute only $o(C^{k-1}_{n,m})$ to $N^\ast_{n,m}(K_{1,k})$. The calculations below replace the supercritical estimate \eqref{eqn:fstarpim-enusn}, and the additional shift coming from $T[V(\Pi),V]$ is $O(h(T)(\alpha+\delta)n)$, which is absorbed by the matching penalty in \Cref{lemma:super-FPiT}. It remains to bound the number of graphs in $\cF^\ast_\Pi$ when $|\Pi_\sparse|\geq L\log n$ for a sufficiently large constant $L=L(k)$. Our estimates, including \eqref{eqn:critical-binomial-expansion-K1k}, are generalizations of \cite{perkins2025typical} for general $k\geq3$.

\begin{lemma}\label{lemma:critical-clean-sparse-K1k}
There exists $L=L(k)>0$ such that
$$\sum_{s\geq L\log n}\sum_{\Pi\in\sD_s}|\cF^\ast_\Pi|=o(C^{k-1}_{n,m})\,.$$
\end{lemma}
\begin{proof}
Let $\Pi=\{P_1,\dots,P_{k-1}\}\in\sD_s$ and write $\pi_i:=|P_i|$. If $\cF^\ast_\Pi\neq\emptyset$ then \Cref{lemma:SuperCloseStructureK1k} gives $\pi_i\in(\frac{1}{k-1}\pm\delta)n$ and $s\leq\delta n/2$. Choose a partition $\Pi_\sparse=S_1\cup\cdots\cup S_{k-1}$ so that, with $s_i:=|S_i|$ and $A_i:=P_i\cup S_i$, every part receiving at least one sparse vertex has size at most $n/(k-1)+1$. Write $a_i:=|A_i|$ for $i\in[k-1]$.

Let $\cF^\ast:=\bigcup_{\Pi\in\sD_{W^\ast}}\cF^\ast_\Pi$. By comparing binomial coefficients, it is routine to show that for all fixed $S\subseteq V$ with $|S|\leq\delta n/2$, almost all graphs $G\in\cF^\ast$ with $\Pi(G)_\sparse=S$ have a $\sqrt{\log n}$-balanced division $\Pi(G)$. Writing $\Pi(G)=\{P'_1,\dots,P'_{k-1}\}$, this means that for all $1\leq i<j\leq k-1$, we have $||P'_i|-|P'_j||\leq\sqrt{\log n}$. See \cite{perkins2025typical} for calculations in the case $k=3$; for general $k$, the calculation is analogous.
Moreover, since the vertices of $\Pi_\sparse$ are distributed equitably among the sets $P_i$, it follows that $|a_i-a_j|\leq\sqrt{\log n}$ for all $i,j\in[k-1]$. In particular, writing $a_i=n/(k-1)+x_i$, we have $\sum_i x_i=0$ and $x_i=O_k(\sqrt{\log n})$ for all $i\in[k-1]$. Define the quantities
$$N:=\sum_{1\leq i<j\leq k-1}a_ia_j\,,\qquad M:=m-\sum_{i=1}^{k-1}\binom{a_i}{2}\,.$$
For the fixed division $\{A_1,\dots,A_{k-1}\}$, the number of co-$(k-1)$-partite graphs with this division is $\binom{N}{M}$. Since
$$N=\frac{k-2}{2(k-1)}n^2-\frac12\sum_i x_i^2=\frac{k-2}{2(k-1)}n^2+O_k(\log n)$$
and $m=\floor{\gamma_k\binom n2}$, we also have $M=pN+O_k(n)$.

The number of edges crossing between parts in the partition $\Pi$ is $\sum_{i<j}\pi_i\pi_j$. Define
$$K:=N-\sum_{1\leq i<j\leq k-1}\pi_i\pi_j=\sum_{i=1}^{k-1}s_i(n-a_i)-\sum_{1\leq i<j\leq k-1}s_is_j\,,$$
so $K$ is the number of additional cross edges after redistributing the vertices of $\Pi_\sparse$. Note that if $G[\Pi_\sparse]$ has $t$ edges, and if we define
$$L_t:=\sum_{i=1}^{k-1}s_i\pi_i+\sum_{i=1}^{k-1}\binom{s_i}{2}-t\,,$$
then we have $\sum_{i<j}\pi_i\pi_j=N-K=M+L_t$. Since $t\leq\binom{s}{2}$ and $s\leq\delta n/2$, we have $L_t\geq sn/(2(k-1))$ for sufficiently small $\delta$.

Using Stirling's formula and a second-order Taylor series, we have
\begin{align}\label{eqn:critical-binomial-expansion-K1k}
\frac{\binom{N-K}{M+L_t}}{\binom{N}{M}}&\leq \left(\frac{N-M}{N}\right)^K\!\!\left(\frac{N-M}{M}\right)^{L_t}
\!\!\exp\!\left(\!-\frac{(K+L_t)^2}{2(N-M)}+\frac{K^2}{2N}-\frac{L_t^2}{2M}+O_k\!\left(\frac{s^3}{n}+s\right)\!\right)
\end{align}
uniformly over $0\leq t\leq\binom{s}{2}$. Since $M=pN+O_k(n)$ and $p=(1-p)^{k-1}$, the first-order term in \eqref{eqn:critical-binomial-expansion-K1k} is
$$\left(\frac{N-M}{N}\right)^K\left(\frac{N-M}{M}\right)^{L_t}\leq e^{O_k(s)}(1-p)^{(k-2)t}\,.$$
Indeed, we calculate that
\begin{align*}
K-(k-2)L_t&=\sum_{i=1}^{k-1}s_i(n-(k-1)a_i)+\frac{(k-1)\sum_i s_i^2-s^2}{2}+\frac{(k-2)s}{2}+(k-2)t\,,
\end{align*}
where the first term is at least $-O_k(s)$ by the choice of the sets $S_i$, and the quadratic term is nonnegative by Cauchy's inequality.

Next we define the quantity
$$q_k:=\frac{k-1}{(k-2)(1-p)}-\frac{k-2}{k-1}+\frac{1}{p(k-1)(k-2)}\,.$$
Since
$$K=\left(\frac{k-2}{k-1}+O_k\!\left(\frac{s}{n}\right)\right)sn\,,\qquad L_t=\left(\frac{1}{k-1}+O_k\!\left(\frac{s}{n}\right)\right)sn\,,$$
where the $O_k(s/n)$ terms also absorb the error $O_k(\sqrt{\log n}/n)$ because $s\geq L\log n$, the quadratic exponent in \eqref{eqn:critical-binomial-expansion-K1k} is at most $-q_ks^2+O_k\!\left(s^3/n+s\right)$.
We also define $\pi_k:=q_k-\frac12\log\frac{1}{1-p}$, and note that this quantity is positive. Indeed, using $\log(1/(1-p))<p/(1-p)$, we have
$$q_k-\frac{p}{2(1-p)}=\frac{p(k-2)(k-3)+4(k-1)-2}{2(k-1)(k-2)(1-p)}+\frac{1}{p(k-1)(k-2)}>0\,.$$
Choosing $\delta$ sufficiently small and summing over the possible graphs on $\Pi_\sparse$, we get
\begin{align*}
\sum_{t=0}^{\binom{s}{2}}\binom{\binom{s}{2}}{t}\binom{N-K}{M+L_t}&\leq \binom{N}{M}\exp\!\left(-q_ks^2+O_k\!\left(\frac{s^3}{n}+s\right)\right)\sum_{t=0}^{\binom{s}{2}}\binom{\binom{s}{2}}{t}(1-p)^{(k-2)t} \\
&=\binom{N}{M}\exp\!\left(-q_ks^2+O_k\!\left(\frac{s^3}{n}+s\right)\right)(1-p)^{-\binom{s}{2}} \leq \binom{N}{M}e^{-\pi_ks^2/2}
\end{align*}
for all large enough $n$.

Define the set of divisions
$$\sB:=\left\{\Pi'\in\sD_0:||P|-|Q||\leq\sqrt{\log n}\text{ for all }P,Q\in\Pi'\right\}\,.$$
For a fixed enlarged division $\{A_1,\dots,A_{k-1}\}\in\sB$, there are at most $\binom{n}{s}(k-1)^s$ choices of the sparse set and its assignment back to the original division $\Pi$. It follows that
\begin{align*}
\sum_{s\geq L\log n}\sum_{\Pi\in\sD_s}|\cF^\ast_\Pi|&\leq \sum_{s\geq L\log n}\binom{n}{s}(k-1)^se^{-\pi_ks^2/2}\sum_{\Pi'\in\sB}\binom{\binom n2-e((\Pi')^c)}{m-e((\Pi')^c)} =o(C^{k-1}_{n,m})
\end{align*}
provided $L=L(k)$ is sufficiently large. In the last line we used, as in the supercritical proof,
$$\sum_{\Pi'\in\sB}\binom{\binom n2-e((\Pi')^c)}{m-e((\Pi')^c)}=O(C^{k-1}_{n,m})\,,$$
completing the proof.
\end{proof}

\begin{proof}[Proof of \Cref{thm:main-almostall} \ref{item:thm:main-critical}]
The result follows immediately by combining \Cref{lemma:critical-clean-sparse-K1k} with the inclusion given in \eqref{eqn:union-bound-super}, interpreted at $\gamma=\gamma_k$.
\end{proof}

\section{The Subcritical Regime}
\label{sec:subcritical}
In this section we prove \Cref{thm:main-almostall} \ref{item:thm:main-almostall-sub}. We first present definitions and parameter choices that recur in our proofs. Throughout the section, fix $k\geq3$, let $\Delta:=k-2$, fix $\gamma\in(0,\gamma_k)$, and let $m\sim\gamma\binom{n}{2}$. We write $p_k\in(0,1)$ for the unique solution to $p_k=(1-p_k)^{k-1}$, and set $L_k:=\log\frac{1-p_k}{p_k}=-\Delta\log(1-p_k)$. All graphs are on a vertex set $V$ of size $n$.

The parameters below are chosen in the following order. First let $\omega>0$ be arbitrary; all asymptotic notation in this section is with respect to the $\omega\to0$ limit. Next choose $0<\eta=o(\omega)$, and $R_0>0$ sufficiently large with $R_0^{-1}=o(\eta)$. Let $\lambda:=\eta/(4R_0)$ and $\rho:=\min\{p_k,1-p_k\}$. Define
$$\kappa:=\left\lceil\frac{10R_0^2}{\eta^2}\right\rceil$$
and let $c_\matsf=c_\matsf(k,\gamma,\eta,R_0)>0$ be the constant from \Cref{lemma:residual-matching-estimate-K1k}. After $\eta$ and $R_0$ are fixed, choose $\theta>0$ such that
\begin{equation}\label{eqn:sub-theta-parameter-condition-K1k}
\theta=o\!\left(\min\left\{\frac{\eta}{R_0^2},\frac{c_\matsf}{\kappa^2}\right\}\right)\,.
\end{equation}
After $\theta$ is fixed, choose $\alpha>0$ such that
\begin{equation}\label{eqn:sub-matching-parameter-condition-K1k}
\alpha<\frac{\theta}{100k}\,,\qquad H(5\alpha)+\theta<\frac{c_\matsf}{8\kappa^2}\,.
\end{equation}
Next, choose $\delta>0$ and then $\epsilon>0$ such that all inequalities below hold. Let
$$\beta:=\frac{\epsilon^6}{2^{22}\Delta^2}\,,\qquad\tau:=\frac{1}{2}\left(\frac{\beta}{32}\right)^3\,.$$
Writing $C_{k,\theta}>0$ for the constant depending on $k$ and $\theta$ appearing in the proof of \Cref{lemma:WtoWtildeMetricsK1k}, the parameters are chosen so that
\begin{equation}\label{eqn:sub-smaller-components-params}
\epsilon=o\bigg(\bigg(\frac{\alpha\delta\theta}{C_{k,\theta}}\bigg)^3\bigg)\,,\qquad \frac{\beta}{\alpha^2\theta^2}=o(\delta)\,,\qquad \epsilon=o(\alpha^{4k-4}\theta^2)\,.
\end{equation}
Throughout the section, we write $C_k>0$ for a constant depending only on $k$ that may vary from line to line.
Finally, recalling the notation from \Cref{sec:graphon-prob}, we let
$$\mu:=\left(\frac{\gamma(k-1)}{1+(k-2)p_k}\right)^{1/2},\qquad \blambda_\ast:=((\mu,K_{k-1}))\,,\qquad W^\ast:=W_{\blambda_\ast}\,.$$
According to \Cref{prop:graphon-char-fixed-gamma}, among the optimizers of the variational problem $\Phi_k(\gamma)$ stated in \eqref{eqn:var-prob-fixed-gamma}, $W^\ast$ is the unique one satisfying both $\abs{\blambda}=1$ and $G_1=K_{k-1}$.

The following definition introduces the combinatorial object that will model the graphs close to a graphon $W\in\cV_\gamma$.

\begin{definition}[Division]\label{dfn:division}
For any integer $\ell\geq1$, a \emph{division} of $V$ is a set
$$\Pi=\{\Pi_1,\dots,\Pi_\ell\}\,,$$
where for every $i\in[\ell]$ we have $\Pi_i=(Q_i,\cP_i)$, called a \emph{component} of $\Pi$, with the following properties:
\begin{itemize}[leftmargin=\parindent, topsep=4pt]
  \item $Q_i$ is a connected $\Delta$-regular graph on $[q_i]$ for some $q_i\geq\Delta+1$\,;
  \item $\cP_i=\{P_{i,1},\dots,P_{i,q_i}\}$ is a collection of pairwise disjoint nonempty subsets of $V$ called the \emph{parts} of the component $\Pi_i$\,;
  \item the sets $V_i:=\bigcup_{j=1}^{q_i}P_{i,j}$ are pairwise disjoint.
\end{itemize}
We always assume the components $\Pi_i$ with $q_i\leq R_0$ appear first in the ordering $\Pi_1,\dots,\Pi_\ell$, and within each of the sets $\{i:q_i\leq R_0\}$ and $\{i:q_i>R_0\}$, the numbers $v_i:=|V_i|$ are nonincreasing in $i$. Let $\sD$ denote the set of all divisions of $V$. For every $\Pi\in\sD$, define
$$V(\Pi):=\bigcup_{i=1}^\ell V_i\,,\qquad \Pi_\sparse:=V\setminus V(\Pi)\,.$$
We call an edge of $Q_i$ an \emph{active pair}. In this section, we will routinely invoke the various data associated to a division $\Pi$, including $\Pi_i$, $Q_i$, $\cP_i$, $\ell$, and so on, without redefining them every time; the division these data correspond to is always clear from the context.
\end{definition}

Fix $\blambda=((\lambda_1,G_1),(\lambda_2,G_2),\dots)\in\bLambda$ and let $W=W_\blambda\in\cV_\gamma$. For all $1\leq i\leq|\blambda|$, let
$$\mu_i:=\lambda_i-\lambda_{i-1}\,,\qquad r_i:=v(G_i)\,,$$
where $\lambda_0:=0$. Let $L$ be the set $\{1,\dots,|\blambda|\}$ if $|\blambda|<\infty$ and $\N$ otherwise. Fix a division $\Pi=\{\Pi_1,\dots,\Pi_\ell\}$ and define the set of indices
$$I:=\left\{i\in[\ell]:v_i\geq\frac{\eta n}{2}\,,\ q_i\leq R_0\right\}\,.$$
Let us say that $\Pi$ is \emph{$W$-compatible} if there exists an injection $\varphi:I\to L$ such that $Q_i\cong G_{\varphi(i)}$, $|v_i-\mu_{\varphi(i)}n|\leq\delta n$ for all $i\in I$, and for all $j\in L\setminus\varphi(I)$ we have either $\mu_j<\eta$ or $r_j>R_0$. Define the set of $W$-compatible divisions
$$\sD_W:=\left\{\Pi\in\sD:\Pi\text{ is $W$-compatible}\right\}\,.$$

\begin{definition}[$\Delta$-regular blow-up]\label{dfn:D-reg-blowup}
A graph $G$ on $n$ vertices is called a \emph{$\Delta$-regular blow-up} if there exists a division $\Pi=\{\Pi_1,\dots,\Pi_\ell\}\in\sD$ such that the following conditions hold: for all $i\in[\ell]$ and $j\in[q_i]$, the induced subgraph $G[P_{i,j}]$ is a clique; for all $i\in[\ell]$ and distinct $j,j'\in[q_i]$, the induced bipartite graph $G[P_{i,j},P_{i,j'}]$ is nonempty only if $jj'\in E(Q_i)$; for all distinct $i,i'\in[\ell]$, the induced bipartite graph $G[V_i,V_{i'}]$ is empty; and $G$ has no edges incident to $\Pi_\sparse$.
\end{definition}

For all graphs $G$ on $n$ vertices and divisions $\Pi=\{\Pi_1,\dots,\Pi_\ell\}\in\sD$, define the graph $T(G,\Pi)$ on vertex set $V$ by
\begin{align*}
E(T(G,\Pi)) &:= \bigcup_{i=1}^\ell\left(\bigcup_{j=1}^{q_i}E(G^c[P_{i,j}])\cup\bigcup_{jj'\in E(Q_i^c)}E(G[P_{i,j},P_{i,j'}])\right) \\
&\hspace{4cm}\cup\left(\bigcup_{1\leq i<j\leq\ell}E(G[V_i,V_j])\right)\cup E(G[V(\Pi),\Pi_\sparse]) \,.
\end{align*}
We note that $T(G,\Pi)$ has no edges contained in $\Pi_\sparse$.
Let
$$b(G,\Pi):=e(T(G,\Pi))+e(G[\Pi_\sparse])\,,$$
and let $\Pi(G)\in\sD$ be a canonically chosen division minimizing $b(G,\,\cdot\,)$. Let us define
$$T(G):=T(G,\Pi(G))\,,\qquad b(G):=b(G,\Pi(G))\,.$$
We call $T(G)$ the \emph{defect graph} of $G$, and the edges of $T(G)$ are called \emph{defects}. We note that $b(G)$ is the minimum edit distance from $G$ to a $\Delta$-regular blow-up.

For a division $\Pi$ and two parts $P_{i,j},P_{i',j'}$ of $\Pi$, define
$$w_\Pi(P_{i,j},P_{i',j'}):=\begin{cases}1,&i=i'\text{ and }j=j'\,,\\p_k,&i=i'\text{ and }jj'\in E(Q_i)\,,\\0,&\text{otherwise}
\end{cases}\,.$$
Let $H_\Pi$ be the weighted graph on $V$ which has edge weight $w_\Pi(P_{i,j},P_{i',j'})$ on $P_{i,j}\times P_{i',j'}$, and has edge weight $0$ on all pairs incident with $\Pi_\sparse$. Also define
$$I_\theta(\Pi):=\left\{i\in[\ell]:\max_{j\in[q_i]}|P_{i,j}|\geq\theta n\right\}\,,\qquad \cY_\Pi:=\{P_{i,j}:i\in I_\theta(\Pi),\ j\in[q_i]\}\,.$$
We call the parts in $\cY_\Pi$ \emph{$\theta$-visible}.

For all $n,m\in\N$, $W\in\cW$, and $r>0$, define the set of graphs
$$\cB_{n,m}(W,r):=\{G\in\cF^\ast_{n,m}(K_{1,k}):\delta_\square(G,W)<r\}\,,$$
as well as the set
\[\arraycolsep=1.4pt
\begin{array}{ll}
\cF_{W,\Pi} &:= \{G\in\cB_{n,m}(W,\tau):\Pi(G)=\Pi\}\,. \\[7pt]
\end{array}\]
The next lemma is the bridge between $\cB_{n,m}(W,r)$ and the combinatorial objects that will be counted below.

\begin{lemma}\label{lemma:WtoWtildeMetricsK1k}
For all $W\in\cV_\gamma$ and large enough $n$, every graph $G\in\cB_{n,m}(W,\tau)$ satisfies the following conditions:
\begin{enumerate}[
 label=\textit{(\roman*)},
 ref=(\textit{\roman*}),
 topsep=5pt
]
  \item\label{item:bgleqen2} $b(G)\leq \epsilon n^2$.
  \item\label{item:SubCloseK1k-compatible} $\Pi(G)$ is $W$-compatible.
  \item\label{item:SubCloseK1k-balanced} For all $\Pi'=(Q,\{P_1,\dots,P_r\})\in\Pi(G)$ such that $|V(\Pi')|\geq\eta n/(4R_0)$ and $r\leq R_0$, we have
  $$\abs{P_j}\in\left(\frac{1}{r}\pm\min\left\{\alpha,\frac{\omega}{4r}\right\}\right)|V(\Pi')|\,,\qquad |P_j|\geq \frac{|V(\Pi')|}{2R_0}$$
  for all $j\in[r]$.
  \item\label{item:SubCloseK1k-theta} For every component $\Pi_h=(Q_h,\{P_{h,1},\dots,P_{h,q_h}\})\in\Pi(G)$, if $|P_{h,u}|\geq\theta n$ for some $u\in[q_h]$ then $|P_{h,v}|\leq(1+\omega)|P_{h,w}|$ for all $v,w\in[q_h]$.
  \item\label{item:SubCloseK1k-cut} Let $P_{i,a},P_{h,b}\in\cY_{\Pi(G)}$, and let $X\subseteq P_{i,a}$ and $Y\subseteq P_{h,b}$ be disjoint sets with $|X|,|Y|\geq \alpha\theta n/4$. Then
  $$\left|d_G(X,Y)-w_{\Pi(G)}(P_{i,a},P_{h,b})\right|\leq\delta\,.$$
\end{enumerate}
\end{lemma}
\begin{proof}
Fix $\blambda=((\lambda_i,G_i))_{i\in L}\in\bLambda$ such that $W_\blambda=W$, where $L=\{1,\dots,\abs{\blambda}\}$ if $\abs{\blambda}<\infty$ and $L=\N$ otherwise. Let $H_n$ be the weighted graph on $V$ with edge weight $W(i/n,j/n)$ on $ij$, and let $W_n:=W_{H_n}$. Since $W_n\to W$ pointwise almost everywhere, we have $\delta_\square(W_n,W)\to0$. Hence for all sufficiently large $n$, we have $\delta_\square(G,W_n)<2\tau$. By \cite[Lemma~8.9]{lovasz2012large} and \cite[Theorem~2.3]{borgs2008convergent}, together with the remark on p.~1831 there,
$$\widehat\delta_\square(G,H_n)\leq32\delta_\square(G,H_n)^{1/3}\leq\beta\,,$$
by the definition of $\tau$. Since the assertions of the lemma are invariant under relabeling the vertices of $G$, we assume from now on that $d_\square(G,H_n)\leq\beta$.

Write $\mu_i:=\lambda_i-\lambda_{i-1}$ and $r_i:=v(G_i)$, where $\lambda_0:=0$. Let $\ell$ be the greatest index such that $\mu_\ell\geq\epsilon/4$, taking $\ell=0$ if no such index exists, and set $R:=8\Delta\epsilon^{-1}$. For $i\in L$ and $j\in[r_i]$, define the sets of vertices
$$C_{i,j}:=\left\{t\in[n]:\frac{t}{n}\in\left(\lambda_{i-1}+\frac{j-1}{r_i}\mu_i,\lambda_{i-1}+\frac{j}{r_i}\mu_i\right)\right\}\,,\qquad C_i:=\bigcup_{j=1}^{r_i}C_{i,j}\,.$$
Define the set of vertices
$$J:=V\setminus\bigcup_{i\in[\ell],\,r_i\leq R}C_i\,.$$
Only $o(n)$ vertices of $J$ lie on the boundary of one of the intervals defining the sets $C_{i,j}$.

Define the graph $G'$ from $G$ as follows. For all $i\in[\ell]$ with $r_i\leq R$, add all missing edges inside the sets $C_{i,j}$ and delete all edges between $C_{i,j}$ and $C_{i,j'}$ whenever $jj'\in E(G_i^c)$. Delete all edges between distinct $C_i,C_{i'}$ with $i,i'\in[\ell]$ and $r_i,r_{i'}\leq R$, and delete all edges incident to $J$. For large enough $n$, $G'$ is a $\Delta$-regular blow-up, hence $b(G)\leq n^2d_1(G,G')$.

Since $d_\square(G,H_n)\leq\beta$, the edit distance between $G$ and $G'$ inside one set $C_i$ with $r_i\leq R$ is at most $r_i^2\beta n^2$, and the edit distance between two distinct sets $C_i$ and $C_j$ blocks is at most $\beta n^2$. The total edge weight of $H_n$ incident to the nonexceptional vertices of $J$ is at most $\epsilon n^2/2$. Indeed, each block with index larger than $\ell$ contains total edge weight at most $\sum_{i>\ell}\mu_i^2n^2\leq\epsilon n^2/4$, while each block with $i\in[\ell]$ and $r_i>R$ contains total edge weight at most
$$\sum_{i\in[\ell],\,r_i>R}\left(r_i+2e(G_i)\right)\left(\frac{\mu_i n}{r_i}\right)^2\leq\frac{1+\Delta}{R}n^2\leq\frac{\epsilon n^2}{4}\,.$$
Thus the number of edges of $G$ incident to $J$ is at most $3\epsilon n^2/4$ for all sufficiently large $n$. Since $\ell\leq4\epsilon^{-1}$, the definition of $\beta$ gives
$$b(G)\leq n^2d_1(G,G')\leq\ell R^2\beta n^2+\ell^2\beta n^2+\frac{3\epsilon n^2}{4}\leq\epsilon n^2\,,$$
proving \ref{item:bgleqen2}.

Let us write
$$\Pi(G)=\{\Pi_1,\dots,\Pi_t\}\,,\qquad \Pi_a=(Q_a,\{P_{a,1},\dots,P_{a,q_a}\})\,,\qquad V_a:=\bigcup_{b=1}^{q_a}P_{a,b}\,.$$
Define the graph $H$ on $V$ with edge set
$$E(H):=(E(G)\,\triangle\,E(T(G)))\setminus E(G[\Pi_\sparse])\,,$$
so $H$ is obtained from $G$ by toggling the defect edges of the canonical division and deleting the edges in the sparse part $\Pi_\sparse$. Since $d_1(G,H)=b(G)/n^2$, \ref{item:bgleqen2} and $d_\square(G,H_n)\leq\beta$ give
\begin{equation}\label{eqn:sub-HHn-close-K1k}
d_\square(H,H_n)\leq d_\square(H,G)+d_\square(G,H_n)\leq\epsilon+\beta\leq2\epsilon\,.
\end{equation}

We will repeatedly use the following elementary observation. Let $S$ be a finite set partitioned into parts, and let $m>0$. If $|S|\geq3m$ and no part has size greater than $|S|-m$ then there are disjoint unions $X,Y\subseteq S$ of whole parts such that $|X|,|Y|\geq m$. Indeed, if one part has size in $[m,|S|-m]$ then take this part and its complement; otherwise every part has size less than $m$, and a greedy union has size in $[m,2m)$.

Define the parameters
$$m_0:=\epsilon^{1/3}n,\qquad m_1:=\frac{\theta n}{20}\,.$$
By the choice of our parameters above,
\begin{equation}\label{eqn:sub-bridge-parameter-display-K1k}
\rho m_0^2>2\epsilon n^2\,,\qquad m_1\geq6(\Delta+1)m_0\,,\qquad \frac{m_0}{\alpha\theta n}=o(\delta)\,,\qquad R_0\frac{m_0}{n}=o(\delta)\,.
\end{equation}
The first inequality has the following consequence. If $X,Y\subseteq V$ satisfy $|X|,|Y|\geq m_0$ then the edge weights of $H$ and $H_n$ cannot differ by at least $\rho$ on every pair in $X\times Y$, since then \eqref{eqn:sub-HHn-close-K1k} would be contradicted.

We now compare the cells $C_{i,j}$ of $H_n$ with the parts of the canonical division $\Pi(G)$. Suppose first that $C=C_{i,j}$ satisfies $|C|\geq m_1$. Partition $C$ by the sets $P_{a,b}$ and by $\Pi_\sparse$. Define a graph $\Gamma_C$ on the nonempty pieces of this partition by joining two pieces if they lie in adjacent parts of the same component of $\Pi(G)$. Since $\Delta(\Gamma_C)\leq\Delta$, if every piece had size less than $m_0$ then some color class of $\Gamma_C$ would have size at least $|C|/(\Delta+1)\geq6m_0$. The observation above gives disjoint unions $X,Y\subseteq C$ of whole pieces with $|X|,|Y|\geq m_0$ and $H[X,Y]=0$, while every edge weight of $H_n[X,Y]$ equals $1$, a contradiction. Hence some part of $\Pi(G)$ or $\Pi_\sparse$ meets $C$ in at least $m_0$ vertices. The set $\Pi_\sparse$ cannot do so by the same argument, and therefore there is a part of $\Pi(G)$, denoted $P(C)$, such that $|C\cap P(C)|\geq m_0$.
Moreover, we have $|P(C)\setminus C|<m_0$, since otherwise with $X=C\cap P(C)$ and $Y=P(C)\setminus C$, the graph $H[X,Y]$ is complete while every edge weight of $H_n[X,Y]$ is at most $1-\rho$. In particular, distinct cells $C$ of size at least $m_1$ have distinct associated parts $P(C)$.

We next record the converse alignment for large parts of $\Pi(G)$. Suppose that $P=P_{a,b}$ satisfies $|P|\geq\theta n$. Partition $P$ by the cells $C_{i,j}$ and by the boundary vertices of the intervals defining them. Since the boundary contributes $o(n)<m_0$ vertices, we ignore it below. If no cell $C_{i,j}$ satisfies $|P\setminus C_{i,j}|<m_0$ then the above observation gives disjoint unions $X,Y\subseteq P$ of whole intersections $P\cap C_{i,j}$ with $|X|,|Y|\geq m_0$ and such that no pair in $X\times Y$ lies inside a single cell. On the one hand, $H[X,Y]$ is complete, since $X,Y\subseteq P$. On the other hand, every edge weight of $H_n[X,Y]$ is at most $1-\rho$. This contradicts \eqref{eqn:sub-HHn-close-K1k}, hence there is a cell, denoted $C(P)$, such that $|P\setminus C(P)|<m_0$.

Let $\Pi_h=(Q_h,\{P_{h,1},\dots,P_{h,q_h}\})$ be a component of $\Pi(G)$ containing a part $P_{h,u}$ with $|P_{h,u}|\geq\theta n$, and write $C(P_{h,u})=C_{i,j}$. Since $|P_{h,u}\setminus C_{i,j}|<m_0$, the cell $C_{i,j}$ has size at least $\theta n/2\geq m_1$. If $j'\in N_{G_i}(j)$ then $|C_{i,j'}|\geq\theta n/2$, and so $C_{i,j'}$ has an associated part $P(C_{i,j'})$. The cut between $P_{h,u}$ and $P(C_{i,j'})$ has $H_n$-weight $p_k$ on all but $O(m_0n)$ pairs. If these two parts were not adjacent in $Q_h$ then $H$ would have no edges between them, contradicting \eqref{eqn:sub-HHn-close-K1k}. Thus $P(C_{i,j'})$ is adjacent to $P_{h,u}$ in $Q_h$ for every $j'\in N_{G_i}(j)$. Since both $G_i$ and $Q_h$ are $\Delta$-regular and $P_{h,u}$ has exactly $\Delta$ neighbors, these are precisely the neighbors of $P_{h,u}$.

We now claim that no other part of $\Pi_h$ has intersection at least $m_0$ with $C_{i,j}$. Indeed, such a part would have to be adjacent to $P_{h,u}$, since otherwise $H$ would have no edges between two $m_0$-sets on which $H_n$ has weight $1$. But all $\Delta$ neighbors of $P_{h,u}$ have already been identified, and they are associated to the distinct neighboring cells $C_{i,j'}$, $j'\in N_{G_i}(j)$. Hence $C_{i,j}$ is contained in $P_{h,u}$ up to an error of at most $(\Delta+2)m_0$. Applying the same argument after moving to each neighbor and then iterating along paths in the connected graph $G_i$, we see that the whole component $\Pi_h$ is isomorphic to $G_i$, and every part of $\Pi_h$ is within $C_{k,\theta}m_0$ of the corresponding cell of $H_n$, where $C_{k,\theta}>0$ depends only on $k$ and $\theta$. This proves \ref{item:SubCloseK1k-theta}, because $m_0/(\theta n)=o(\omega)$ and $m_0/(\theta n)+\beta/\theta^2=o(\delta)$.

We now consider components with at most $R_0$ parts. If $i\in L$ satisfies $\mu_i\geq\eta$ and $r_i\leq R_0$ then every cell $C_{i,j}$ has size at least $\eta n/(2R_0)\geq m_1$, and the preceding paragraphs associate these cells to a component of $\Pi(G)$ isomorphic to $G_i$ whose total size differs from $\mu_i n$ by at most $\delta n$. Conversely, if $\Pi_h$ has $q_h\leq R_0$ and $|V_h|\geq\eta n/2$ then some part has size at least $\eta n/(2R_0)\geq\theta n$, so the preceding paragraph aligns $\Pi_h$ with a graphon block $G_i$ satisfying $r_i\leq R_0$ and $\mu_i\geq\eta/2-O(\delta)$. This completes the proof of \ref{item:SubCloseK1k-compatible}. The balance assertion \ref{item:SubCloseK1k-balanced} follows in the same way from the cell alignment, since any component $\Pi_h$ with $q_h\leq R_0$ and $|V_h|\geq\eta n/(4R_0)$ has some part of size at least $\eta n/(4R_0^2)\gg\theta n$.

It remains to prove \ref{item:SubCloseK1k-cut}. Let $P_{i,a},P_{h,b}\in\cY_{\Pi(G)}$, and let $X\subseteq P_{i,a}$ and $Y\subseteq P_{h,b}$ be disjoint sets with $|X|,|Y|\geq \alpha\theta n/4$. By \ref{item:SubCloseK1k-theta}, the two parts are each within $C_{k,\theta}m_0$ of their corresponding graphon cells. After deleting at most $C_{k,\theta}m_0$ vertices from each of $X$ and $Y$, we get sets $X'\subseteq X$ and $Y'\subseteq Y$ lying in the corresponding cells, with $|X'|,|Y'|\geq \alpha\theta n/8$. On $X'\times Y'$, the weight of $H_n$ is exactly $w_{\Pi(G)}(P_{i,a},P_{h,b})$. Consequently, $d_\square(G,H_n)\leq\beta$ gives $\left|e_G(X',Y')-w_{\Pi(G)}(P_{i,a},P_{h,b})|X'||Y'|\right|\leq\beta n^2$. Adding back the deleted vertices and using $\beta/(\alpha^2\theta^2)=o(\delta)$ and $m_0/(\alpha\theta n)=o(\delta)$ proves the required result.
\end{proof}

\begin{corollary}\label{cor:cut-closeness}
Fix a graph $G\in\cF_{W,\Pi}$. Fix distinct sets $Y_1,\dots,Y_r\in\cY_{\Pi(G)}$ along with subsets $A_i\subseteq Y_i$, each of size $|A_i|\geq\alpha\theta n/4$. Writing $A:=A_1\cup\cdots\cup A_r$, we have $d_\square(G[A],H_\Pi[A])=o(1)$ as $\omega\to0$.
\end{corollary}
\begin{proof}
This follows immediately from the last paragraph of the proof of \Cref{lemma:WtoWtildeMetricsK1k}. Indeed, after deleting at most $C_{k,\theta}m_0$ vertices from each set $A_i$, the edge weights of $H_\Pi$ agree with those of $H_n$, and we further have $d_\square(G,H_n)\leq\beta$. It follows that
$$d_\square(G[A],H_\Pi[A])\leq O_{k,r}\!\left(\frac{\beta}{\alpha^2\theta^2}+\frac{m_0}{\alpha\theta n}\right)=o(1)\,,$$
where we used $\beta/(\alpha^2\theta^2)=o(\delta)$ and $m_0/(\alpha\theta n)=o(\delta)$.
\end{proof}

Using \Cref{lemma:WtoWtildeMetricsK1k}, we now record three lemmas that establish deterministic consequences of a graph $G$ belonging to the cut-ball $\cB_{n,m}(W,\tau)$. \Cref{lemma:deterministic-nonlow-bound-K1k,lemma:deterministic-medium-companion-K1k,lemma:deterministic-opposite-row-K1k} play an important role in the remainder of this section as they provide useful constraints on the adjacency profiles of vertices.
For $i\in I_\theta(\Pi)$ and $j\in[q_i]$, write
$$N_{i,j}:=\bigcup_{j'\in N_{Q_i}(j)}P_{i,j'}\,.$$
For all $v\in V$, define
$$\cN_{\Pi,G}(v):=\{(i,j):i\in I_\theta(\Pi),\ j\in[q_i],\ d_G(v,P_{i,j})\geq\alpha |P_{i,j}|\}\,.$$

\begin{lemma}\label{lemma:deterministic-nonlow-bound-K1k}
For every $G\in\cF_{W,\Pi}$ and every $v\in V$, we have $|\cN_{\Pi,G}(v)|\leq k-1$.
\end{lemma}
\begin{proof}
Let $\zeta:=\frac12(1/(k-1)-p_k)>0$, and suppose for contradiction that there are distinct pairs $(i_1,j_1),\dots,(i_k,j_k)\in\cN_{\Pi,G}(v)$. For each $a\in[k]$, choose $A_a\subseteq N_G(v)\cap P_{i_a,j_a}$ with $|A_a|=\lceil\alpha |P_{i_a,j_a}|/2\rceil$. By \Cref{lemma:WtoWtildeMetricsK1k} \ref{item:SubCloseK1k-theta}, every part in $\cY_\Pi$ has size at least $\theta n/(1+\omega)$; hence $|A_a|\geq\alpha\theta n/4$. By \Cref{lemma:WtoWtildeMetricsK1k} \ref{item:SubCloseK1k-cut}, for every $a\neq b$ we have $d_G(A_a,A_b)\leq p_k+\zeta$ after decreasing $\delta$. Thus, in the complement of $G$, every bipartite graph between two distinct parts $A_a,A_b$ has density at least $1-p_k-\zeta>(k-2)/(k-1)$. By Tur\'an's theorem, $G^c$ contains a $K_k$ subgraph with one vertex in each $A_a$. These $k$ vertices are pairwise non-adjacent in $G$ and all are adjacent to $v$, so they form an induced copy of $K_{1,k}$ centered at $v$, a contradiction.
\end{proof}

\begin{lemma}\label{lemma:deterministic-medium-companion-K1k}
Let $G\in\cF_{W,\Pi}$, $i\in I_\theta(\Pi)$, $j\in[q_i]$, and $v\in V\setminus N_{i,j}$. If
$$\alpha |P_{i,j}|\leq d_G(v,P_{i,j})\leq(1-\alpha)|P_{i,j}|$$
then there exists $j'\in N_{Q_i}(j)$ such that $d_G(v,P_{i,j'})\geq(1-\alpha)|P_{i,j'}|$.
\end{lemma}
\begin{proof}
Suppose the conclusion does not hold. Choose disjoint sets $X\subseteq N_G(v)\cap P_{i,j}$ and $Y\subseteq P_{i,j}\setminus N_G(v)$, and, for each $t\in N_{Q_i}(j)$, choose $Y_t\subseteq P_{i,t}\setminus N_G(v)$, all of size exactly $\left\lceil\alpha\theta n/4\right\rceil$. This is possible by our hypotheses and by \Cref{lemma:WtoWtildeMetricsK1k} \ref{item:SubCloseK1k-theta}. Consider the family of copies of $K_{1,k}$ centered at $x\in X$ with leaves $v$ and $y\in Y$ and $y_t\in Y_t$ ($t\in N_{Q_i}(j)$). In the weighted graph $H_\Pi$, the induced density of such copies is at least $\rho^{\binom{k+1}{2}}>0$. Indeed, every required edge has weight at least $\rho$ or is forced by a clique part, and every required nonedge has weight at least $\rho$ or is forced by a non-adjacent pair. By \Cref{cor:cut-closeness}, the induced subgraphs of $G$ and $H_\Pi$ on $X\cup Y\cup\bigcup_tY_t$ is $o(1)$. Thus the counting lemma for weighted graphs (e.g.\ \cite[Theorem~2.7]{borgs2008convergent}) implies the induced $K_{1,k}$ density in $G$ is at least $\rho^{\binom{k+1}{2}}-o(1)>0$, a contradiction.
\end{proof}

\begin{lemma}\label{lemma:deterministic-opposite-row-K1k}
If $G\in\cF_{W,\Pi}$ then there do not exist $i\in I_\theta(\Pi)$, $j\in[q_i]$, $s\in N_{Q_i}(j)$, and $v\in P_{i,j}$ such that
$$d_G^c(v,P_{i,j})\geq\alpha|P_{i,j}|\,,\qquad \alpha|P_{i,s}|\leq d_G(v,P_{i,s})\leq(1-\alpha)|P_{i,s}|\,,$$
and $d_G^c(v,P_{i,t})\geq\alpha|P_{i,t}|$ for every $t\in N_{Q_i}(s)\setminus\{j\}$.
\end{lemma}
\begin{proof}
The proof is the same as for \Cref{lemma:deterministic-medium-companion-K1k}, except now we choose disjoint sets $X\subseteq N_G(v)\cap P_{i,s}$ and $Y\subseteq P_{i,s}\setminus N_G(v)$, a set $Y_j\subseteq P_{i,j}\setminus N_G(v)$, and, for every $t\in N_{Q_i}(s)\setminus\{j\}$, a set $Y_t\subseteq P_{i,t}\setminus N_G(v)$, all of size exactly $\left\lceil\alpha\theta n/4\right\rceil$. The rest of the proof proceeds as above.
\end{proof}

\subsection{The counting argument}\label{subsec:counting}
In this subsection, we fix $W\in\cV_\gamma$ and a $W$-compatible division $\Pi=\{\Pi_1,\dots,\Pi_\ell\}$. We use all the same notation relating to $\Pi$ as established in \Cref{dfn:division}. Let $\ell_\ast=\ell_\ast(\Pi)$ be the number of components $\Pi_i$ with $|V_i|\geq \eta n$ and $q_i\leq R_0$, and define
$$\Pi_\ast:=\{\Pi_i:i\in[\ell_\ast]\}\,,\qquad V_\ast:=\bigcup_{i=1}^{\ell_\ast}V_i\,,\qquad S_\Pi:=V\setminus V_\ast\,,\qquad s(\Pi):=|S_\Pi|\,.$$
We call the components $\Pi_i$ with $i\in[\ell_\ast]$, as well as all the vertices and parts these components contain, \emph{retained}. Recall that
$$I_\theta(\Pi):=\left\{i\in[\ell]:\max_{j\in[q_i]}|P_{i,j}|\geq\theta n\right\}\,,\qquad \cY_\Pi:=\{P_{i,j}:i\in I_\theta(\Pi),\ j\in[q_i]\}\,.$$
We also define the retained target set $\cY^\ast_\Pi:=\{P_{i,j}:i\in[\ell_\ast],\ j\in[q_i]\}$,
and split the sparse side into the $\theta$-visible and $\theta$-small portions
$$S_\lgsf(\Pi):=\bigcup_{i\in I_\theta(\Pi)\setminus[\ell_\ast]}V_i\,,\qquad S_\smsf(\Pi):=S_\Pi\setminus S_\lgsf(\Pi)=\Pi_\sparse\cup\bigcup_{i\notin I_\theta(\Pi)}V_i\,.$$
Define $\cI_\Pi:=\{(i,uv):i\in[\ell_\ast],\ uv\in E(Q_i)\}$.
For $e=(i,uv)\in\cI_\Pi$, write
$$N_e:=|P_{i,u}||P_{i,v}|\,,\qquad C_\Pi:=\sum_{i=1}^{\ell_\ast}\sum_{a=1}^{q_i}\binom{|P_{i,a}|}{2}\,.$$
Also define
$$\Omega_\Pi:=\bigcup_{(i,uv)\in\cI_\Pi}E(K_{P_{i,u},P_{i,v}})\,,\qquad \binom{\Pi}{\bsm}:=\prod_{e\in\cI_\Pi}\binom{N_e}{\bsm_e}\,.$$
For an integer $u$, define
$$\cM_{\Pi,u} := \left\{\bsm\in\N^{\cI_\Pi}:
\begin{array}{l}
C_\Pi+\sum_{e\in\cI_\Pi}\bsm_e+u=m\text{ and} \\[8pt]
p_k-2\delta\leq \frac{\bsm_e}{N_e}\leq p_k+2\delta\text{ for all }e\in\cI_\Pi
\end{array}\right\}\,.$$
Let $\cM^\nar_{\Pi,u}$ be defined in the same way, except with the narrower density window $p_k\pm\delta$. We also define
$$\cM^\nar_\Pi:=\bigcup_{-2\epsilon n^2\leq u\leq C_k\eta n^2+2\epsilon n^2}\cM^\nar_{\Pi,u}\,,$$
and
$$Z_\Pi(u):=\sum_{\bsm\in\cM_{\Pi,u}}\binom{\Pi}{\bsm}\,,\qquad Z^\nar_\Pi(u):=\sum_{\bsm\in\cM^\nar_{\Pi,u}}\binom{\Pi}{\bsm}\,.$$
Finally, define the clean retained partition function
\begin{equation}\label{eqn:clean-partition-function-K1k}
\cZ_\Pi:=\sum_{0\leq b\leq C_k\eta n^2}N^\ast_{s(\Pi),b}(K_{1,k})Z_\Pi(b)\,.
\end{equation}

\begin{lemma}\label{lemma:SubCompareSparseSideEdgesK1k}
There exists a constant $C_k>0$ depending only on $k$ such that, for every $G\in\cB_{n,m}(W,\tau)$ with canonical division $\Pi=\Pi(G)$, the following hold:
\begin{enumerate}[label=\textit{(\roman*)}, ref=(\textit{\roman*}), topsep=5pt]
  \item\label{item:SubCompareSparseSideEdgesK1k-total} $e_G(S_\Pi)\leq C_k\eta n^2$.
  \item\label{item:SubCompareSparseSideEdgesK1k-small} For every $A\subseteq S_\smsf(\Pi)$, we have $e_G(A)\leq C_k\theta n|A|+\epsilon n^2$.
  \item\label{item:SubCompareSparseSideEdgesK1k-row} For every $x\in V$, we have $d_G(x,S_\smsf(\Pi))\leq C_k\theta n$.
\end{enumerate}
\end{lemma}
\begin{proof}
Let $J$ be the $\Delta$-regular blow-up obtained from $G$ by toggling the defect graph $T(G,\Pi)$ and deleting all edges of $G[\Pi_\sparse]$. Since $b(G)\leq\epsilon n^2$ by \Cref{lemma:WtoWtildeMetricsK1k}, it suffices to prove the corresponding estimates for $J$, at the cost of increasing $C_k$.

We first prove \ref{item:SubCompareSparseSideEdgesK1k-total}. A component $V_h\subseteq S_\Pi$ with $|V_h|<\eta n$ contributes at most $|V_h|^2$ edges, and hence all such components together contribute at most $\eta n^2$. Now suppose $V_h\subseteq S_\Pi$ and $q_h>R_0$. Since $Q_h$ has maximum degree $\Delta$, we have $e_J(V_h)\leq C_k\sum_{a=1}^{q_h}|P_{h,a}|^2$. If every part $P_{h,a}$ has size less than $\theta n$ then $\sum_{a=1}^{q_h}|P_{h,a}|^2\leq \theta n|V_h|$. Otherwise, \Cref{lemma:WtoWtildeMetricsK1k} \ref{item:SubCloseK1k-theta} gives balance throughout the component, and therefore $\sum_{a=1}^{q_h}|P_{h,a}|^2\leq 2|V_h|^2/q_h\leq 2|V_h|^2/R_0$.
Summing over all components with $q_h>R_0$, and using $\theta=o(\eta)$ and $R_0^{-1}=o(\eta)$, gives $e_J(S_\Pi)\leq C_k\eta n^2$. Adding back the defect graph and the edges of $G[\Pi_\sparse]$ contributes at most $\epsilon n^2$, proving \ref{item:SubCompareSparseSideEdgesK1k-total}.

We next prove \ref{item:SubCompareSparseSideEdgesK1k-small}. Let $A\subseteq S_\smsf(\Pi)$ and decompose $A$ over the components of $\Pi$ contained in $S_\Pi$. If $A_h:=A\cap V_h$ then every part intersecting $A_h$ has size less than $\theta n$. Hence in the $\Delta$-regular blow-up $J$,
\begin{align*}
e_J(A_h) \leq \sum_{a=1}^{q_h}|A_h\cap P_{h,a}|\cdot|P_{h,a}|+\sum_{ab\in E(Q_h)}|A_h\cap P_{h,a}|\cdot|P_{h,b}| \leq C_k\theta n|A_h|\,.
\end{align*}
Summing over $h$ yields $e_J(A)\leq C_k\theta n\,|A|$, and adding back the defect graph contributes at most $\epsilon n^2$, proving \ref{item:SubCompareSparseSideEdgesK1k-small}.

Finally, let $A:=N_G(x,S_\smsf(\Pi))$. Since $G$ is induced-$K_{1,k}$-free, $G[A]$ has independence number less than $k$. Tur\'an's theorem then implies $e_G(A)\geq (k-1)^{-1}\binom{|A|}{2}-O_k(|A|)$. On the other hand, by \ref{item:SubCompareSparseSideEdgesK1k-small} we have $e_G(A)\leq C_k\theta n|A|+\epsilon n^2$. Combining the two estimates gives $|A|\leq C_k\theta n+C_k\sqrt{\epsilon}\,n$.
Since $\epsilon=o(\theta^2)$, this proves \ref{item:SubCompareSparseSideEdgesK1k-row} after increasing $C_k$.
\end{proof}

\subsubsection{Profiles}\label{subsec:subcritical-profiles-K1k}
For $v\in P_{i,j}\subseteq V_\ast$, let us write $P_v:=P_{i,j}$ and
$$\cY_{\Pi,v}:=\{P_{i',j'}\in\cY_\Pi:i'=i\,\Rightarrow j'\not\in N_{Q_i}(j)\cup\{j\}\}\,.$$
We now introduce a bookkeeping device called a profile, which records the local defect data relative to a fixed division that will be needed in the entropy and large-deviation estimates.

\begin{definition}[Profile]\label{def:profile-datum-K1k}
A profile for $\Pi$, denoted $\frp$, consists of the following data:
\begin{enumerate}[label=\textit{(\roman*)}, ref=(\textit{\roman*}), topsep=5pt]
  \item integers $0\leq b(\frp)\leq\binom{s(\Pi)}{2}$ and $0\leq\ell(\frp)\leq n/2$;
  \item disjoint sets $B_\ret(\frp)\subseteq V_\ast$ and $B_\outsf(\frp)\subseteq S_\Pi$;
  \item for every $v\in B_\ret(\frp)$, a collection $\cR_v(\frp)$ of pairs $(Y,r)$, where $Y\in\cY_{\Pi,v}$, $0\leq r\leq |Y\setminus B(\frp)|$, and no $Y$ appears more than once;
  \item for every $v\in B_\ret(\frp)$, an integer $0\leq I_v(\frp)\leq |P_v\setminus B(\frp)|$;
  \item for every $v\in B_\outsf(\frp)$, a collection $\cR_v(\frp)$ of pairs $(Y,r)$, where $Y\in\cY^\ast_\Pi$, $0\leq r\leq |Y\setminus B(\frp)|$, and no $Y$ appears more than once;
  \item a collection $\Tail(\frp)$ of triples $(v,j',t)$, where $v\in B_\ret(\frp)\cap P_{i,j}$, $j'\in N_{Q_i}(j)$, and $t\in\{\Upper,\Lower\}$; for every pair $(v,j')$, at most one of the two triples $(v,j',\Upper)$ and $(v,j',\Lower)$ belongs to $\Tail(\frp)$.
\end{enumerate}
Here and below,
$$B(\frp):=B_\ret(\frp)\cup B_\outsf(\frp)\,.$$
For all $v\in B(\frp)$, an element $(Y,r)\in\cR_v(\frp)$ is called a \emph{row}, $Y$ is called the \emph{target} of the row, and $v$ is called the \emph{root} of the row. Let $\frP_\Pi$ denote the set of all profiles for $\Pi$.
\end{definition}

For a graph $G\in\cF_{W,\Pi}$, write $B_\ret(G)\subseteq V_\ast$ for the set of vertices $v$ such that $d_G(v,Y)\geq4\alpha|Y|$ for some $Y\in\cY_{\Pi,v}$ or $d_G^c(v,P_v)\geq4\alpha|P_v|$. Also write $B_\outsf(G)\subseteq S_\Pi$ for the set of vertices $v\in S_\Pi$ such that $d_G(v,Y)\geq4\alpha|Y|$ for some retained part $Y\in\cY^\ast_\Pi$. Define $B(G):=B_\ret(G)\cup B_\outsf(G)$.

The \emph{residual defect graph} of $G$, denoted $T_\res(G)$, is obtained from $T(G)[V_\ast,V]$ by deleting all edges incident with $B(G)$. The \emph{rooted defect graph} of $G$, denoted $T_B(G)$, is the subgraph of $T:=T(G)[V_\ast,V]$ with the edge set consisting of the recorded rows incident with $B(G)$, namely:
\begin{itemize}[leftmargin=\parindent,topsep=5pt]
  \item for $v\in B_\ret(G)$, the edges from $v$ to $Y\setminus B(G)$ for all $Y\in\cY_{\Pi,v}$ such that $d_G(v,Y)\geq4\alpha|Y|$, together with the edges of $T(G)$ from $v$ to $P_v\setminus B(G)$;
  \item for $v\in B_\outsf(G)$, the edges from $v$ to $Y\setminus B(G)$ for all retained parts $Y\in\cY^\ast_\Pi$ such that $d_G(v,Y)\geq4\alpha|Y|$.
\end{itemize}
Edges with both endpoints in $B(G)$, and all low unrecorded rows incident with $B(G)$, are not included in $T_B(G)$.

\begin{definition}[Profile class]\label{def:profile-class-K1k}
For $\frp\in\frP_\Pi$, define $\cF_{\Pi,\frp}$ to be the set of all graphs $G\in\cF_{W,\Pi}$ satisfying the following conditions:
\begin{enumerate}[label=\textit{(\roman*)}, ref=(\textit{\roman*}), topsep=5pt]
  \item $e(G[S_\Pi])=b(\frp)$;
  \item $B_\ret(\frp)=B_\ret(G)$ and $B_\outsf(\frp)=B_\outsf(G)$;
  \item for every $v\in B_\ret(\frp)$, we have
  $$\cR_v(\frp)=\{(Y,d_G(v,Y\setminus B(\frp))):Y\in\cY_{\Pi,v},\ d_G(v,Y)\geq4\alpha |Y|\}\,;$$
  \item $I_v(\frp)=d_G^c(v,P_v\setminus B(\frp))$ for every $v\in B_\ret(\frp)$;
  \item for every $v\in B_\outsf(\frp)$, we have
  $$\cR_v(\frp)=\{(Y,d_G(v,Y\setminus B(\frp))):Y\in\cY^\ast_\Pi,\ d_G(v,Y)\geq4\alpha |Y|\}\,;$$
  \item $(v,j',\Upper)\in\Tail(\frp)$ if and only if $d_G(v,P_{i,j'})\geq(1-\alpha)|P_{i,j'}|$, and $(v,j',\Lower)\in\Tail(\frp)$ if and only if $d_G(v,P_{i,j'})\leq\alpha |P_{i,j'}|$;
  \item the residual graph $T_\res(G)$ has a maximum matching of size $\ell(\frp)$.
\end{enumerate}
\end{definition}

For $\frp\in\frP_\Pi$, define
$$\cT_{W,\Pi} := \{T(G)[V_\ast,V]:G\in\cF_{W,\Pi}\}\,,\qquad\cT^\root_{\Pi,\frp}:=\{T_B(G):G\in\cF_{\Pi,\frp}\}\,,$$
and, for $T_B\in\cT^\root_{\Pi,\frp}$,
$$\cT^\res_{\Pi,\frp}(T_B):=\{T_\res(G):G\in\cF_{\Pi,\frp}\,,\ T_B(G)=T_B\}\,.$$

\begin{lemma}\label{lemma:profile-covering-K1k}
For all $W$-compatible divisions, we have
\begin{equation}\label{eqn:fwpihigh-covering}
\cF_{W,\Pi}\subseteq\bigcup_{\frp\in\frP_\Pi}\cF_{\Pi,\frp}\,.
\end{equation}
Moreover, if $G\in\cF_{\Pi,\frp}$ then $|B(\frp)|\leq \rho_B n$ where $\rho_B:=2\epsilon/(\alpha\theta)$.
\end{lemma}
\begin{proof}
For every $G\in\cF_{W,\Pi}$, the data specified in \Cref{def:profile-class-K1k} define a profile $\frp\in\frP_\Pi$ with $G\in\cF_{\Pi,\frp}$, proving \eqref{eqn:fwpihigh-covering}. Next, if $v\in B(\frp)$ then $v$ is incident with at least $2\alpha\theta n$ defect edges after decreasing $\alpha$ by a constant factor and using \Cref{lemma:WtoWtildeMetricsK1k} \ref{item:SubCloseK1k-theta}. Since $e(T(G))\leq\epsilon n^2$ by \Cref{lemma:WtoWtildeMetricsK1k}, and every defect edge is counted at most twice in this incidence count, we have $|B(\frp)|\leq \rho_Bn$.
\end{proof}

\begin{lemma}\label{lemma:profile-residual-low-rows-K1k}
Let $\frp\in\frP_\Pi$ and let $G\in\cF_{\Pi,\frp}$. If $v\in V_\ast\setminus B(\frp)$ and $Y\in\cY_{\Pi,v}$ then $d_G(v,Y)<4\alpha |Y|$. If $v\in P_{i,j}\setminus B(\frp)$ then $d_G^c(v,P_{i,j})<4\alpha |P_{i,j}|$. Finally, if $v\in S_\Pi\setminus B(\frp)$ and $Y\in\cY^\ast_\Pi$ then $d_G(v,Y)<4\alpha |Y|$.
\end{lemma}
\begin{proof}
This is immediate from the definitions of $B_\ret(G)$, $B_\outsf(G)$, and $\cR_v(\frp)$.
\end{proof}

\subsubsection{Random graph model}
Here we introduce the random graph that models graphs with a given subgraph on $S_\Pi$, a given defect graph $T$, and edge-count vector $\bsm$. This random graph along with its binomial counterpart play an important role in the remainder, as it allows us to obtain explicit penalties that constrain the count $|\cF_{\Pi,\frp}|$ for profiles $\frp\in\frP$.

\begin{definition}\label{def:active-random-models-K1k}
Fix a graph $H$ on $S_\Pi$, an integer $u$, a vector $\bsm\in\cM^\nar_{\Pi,u}$, and a graph $T\in\cT_{W,\Pi}$. The random graph $G_{\Pi,H,T,\bsm}$ is obtained by choosing, independently for each $e=(i,ab)\in\cI_\Pi$, a uniformly random $\bsm_e$-element subset of edges $A_e\subseteq E(K_{P_{i,a},P_{i,b}})$, and then setting
$$E(G_{\Pi,H,T,\bsm}) := \left(\bigcup_{i\in[\ell_\ast]}\bigcup_{j\in[q_i]}\binom{P_{i,j}}{2}\cup E(H) \cup \bigcup_{e\in\cI_\Pi}A_e\right)\triangle\,E(T)\,.$$
The graph $G^{\bin}_{\Pi,H,T,\bsm}$ is defined in the same way, except that now each $A_e$ is a random subset of $E(K_{P_{i,a},P_{i,b}})$ in which every edge is included independently with probability $\frac{\bsm_e}{|P_{i,a}|\cdot|P_{i,b}|}$.
\end{definition}

In the proofs below, we pass to the binomial version $G^{\bin}_{\Pi,H,T,\bsm}$ for the independence of its edges; this only introduces polynomial overhead, as the following lemma shows.

\begin{lemma}\label{lemma:fixed-count-transfer-K1k}
There exists an absolute constant $C>0$ such that the following holds. For all $\bsm\in\cM^\nar_{\Pi,u}$ and every event $\cE$, we have
$$\bbP\{G_{\Pi,H,T,\bsm}\in\cE\}\leq n^{C|\cI_\Pi|}\,\bbP\{G^\bin_{\Pi,H,T,\bsm}\in\cE\}\,.$$
\end{lemma}
\begin{proof}
In the binomial model, the number of chosen edges in the pair indexed by $e$ has distribution $\Bin(N_e,\bsm_e/N_e)$. Since $\bsm_e/N_e\in p_k\pm2\delta\subseteq[\rho/2,1-\rho/2]$ and $N_e\leq n^2$, Stirling's formula gives $\bbP\{\Bin(N_e,\bsm_e/N_e)=\bsm_e\}\geq n^{-C}$ for an absolute constant $C>0$ and for all $e\in\cI_\Pi$. Conditioning on these events over all $e\in\cI_\Pi$ recovers the fixed-count model $G_{\Pi,H,T,\bsm}$, completing the proof.
\end{proof}

\subsubsection{Local exponents}\label{subsec:local-exponents-K1k}
To bound $|\cF_{W,\Pi}|$, it follows from \eqref{eqn:fwpihigh-covering} that it suffices to bound the sizes of the profile classes $\cF_{\Pi,\frp}$. Since $\delta=o(1)$ and $\rho=\Theta(1)$, uniformly for $p\in[p_k-2\delta,p_k+2\delta]$, we have $\big|\!\log\frac{1-p}{p}-L_k\big|=O_k(\delta)$ and $\big|\!\log\frac{p}{1-p}+L_k\big|=O_k(\delta)$.
For a profile $\frp\in\frP_\Pi$ and a vertex $v\in B(\frp)$, let us write $B=B(\frp)$ and define
\begin{align*}
\Ent_v(\frp) &:= \sum_{(Y,r)\in\cR_v(\frp)}\left(|Y\setminus B|\,H\!\left(\frac{r}{|Y\setminus B|}\right)-L_kr\right)\\
&\hspace{2cm}+\bsone\{v\in B_\ret(\frp)\}\left(|P_v\setminus B|\,H\!\left(\frac{I_v(\frp)}{|P_v\setminus B|}\right)+L_kI_v(\frp)\right)\,.
\end{align*}

For a profile $\frp\in\frP_\Pi$ and a vertex $v\in B_\ret(\frp)\cap P_{i,j}$, define
$$\Tail_v(\frp):=\{(v,j',t)\in\Tail(\frp)\}\,.$$
For $\frt=(v,j',\Upper)\in\Tail_v(\frp)$, define the event
$$E_\frt:=\{G:d_G(v,P_{i,j'}\setminus B(\frp))\geq(1-2\alpha)|P_{i,j'}|\}\,,$$
and for $\frt=(v,j',\Lower)\in\Tail_v(\frp)$, define
$$E_\frt:=\{G:d_G(v,P_{i,j'}\setminus B(\frp))\leq2\alpha|P_{i,j'}|\}\,.$$
If $v\in B_\ret(\frp)$ and $\Tail_v(\frp)\neq\emptyset$ then define the exponent
$$J_v(\frp):=-\max\left\{\log\bbP_{G\sim G^\bin_{\Pi,H,T,\bsm}}\left\{\bigcap_{\frt\in\Tail_v(\frp)}E_\frt\right\}:\bsm\in\cM^\nar_\Pi\right\}\,,$$
while if $v\in B_\outsf(\frp)$ or $\Tail_v(\frp)=\emptyset$ then set $J_v(\frp):=0$. Note that $J_v(\frp)$ is independent of $H$ and $T$. Finally, for all $v\in B(\frp)$ let us define
$$\Psi_v(\frp):=\Ent_v(\frp)-J_v(\frp)\,.$$

We write $A:=L_k/\Delta$ and $U:=L_k+A=\log(1/p_k)$. We will use the elementary estimates
\begin{equation}\label{eqn:local-ent-estimates-K1k}
\max_{0\leq x\leq1}\{H(x)-L_kx\}=A\,,\qquad \max_{0\leq x\leq1}\{H(x)+L_kx\}=U\,.
\end{equation}
Also, if $x\geq1-2\alpha$ then we have
\begin{equation}\label{eqn:high-row-ent-K1k}
H(x)-L_kx\leq -L_k+H(2\alpha)+2L_k\alpha\,.
\end{equation}
For all $Y\in\cY_\Pi$, \Cref{lemma:profile-covering-K1k} and \Cref{lemma:WtoWtildeMetricsK1k} \ref{item:SubCloseK1k-theta} imply that $|Y\setminus B|\geq(1-o(1))|Y|$, so in the estimates below we may replace $|Y|$ by $|Y\setminus B|$ after losing an error $o(|Y|)$, uniformly over all profiles. Our parameter choices ensure $|B|=o(|Y|)$ since $\rho_B=o(\theta)$ and $|Y|\geq\theta n/2$.

\begin{lemma}\label{lemma:sparse-retained-row-budget-K1k}
Let $G\in\cF_{W,\Pi}$ and let $v\in S_\Pi$. Suppose $d_G(v,Y)\geq4\alpha |Y|$ for some retained part $Y\in\cY^\ast_\Pi$. Then there exists a part $P_{h,a}\in\cY_\Pi$ such that $v\in V_h$ and $d_G(v,P_{h,a})\geq\alpha |P_{h,a}|$. Moreover, we have $|\cN_{\Pi,G}(v)\cap\cY^\ast_\Pi|\leq \Delta$.
\end{lemma}
\begin{proof}
Let us write $Y=P_{i,j}$. Since $d_G(v,Y)\geq4\alpha |Y|$, \Cref{lemma:deterministic-medium-companion-K1k} proves there exists a part $P_{i,j'}$ with $j'\in N_{Q_i}(j)\cup\{j\}$ such that $d_G(v,P_{i,j'})\geq(1-\alpha)|P_{i,j'}|$.
We must have $v\not\in\Pi_\sparse$, as otherwise moving $v$ from $\Pi_\sparse$ to $P_{i,j'}$ would produce a division with strictly fewer defect edges than $\Pi=\Pi(G)$. It follows that $v\in P_{h,a}$ for some component $\Pi_h$.

Write $N:=P_{h,a}\cup\bigcup_{b\in N_{Q_h}(a)}P_{h,b}$, and compare $\Pi$ with the division obtained by moving $v$ from $P_{h,a}$ to $P_{i,j'}$. This move removes at least $(1-\alpha)|P_{i,j'}|$ defects from $v$ to $P_{i,j'}$. Apart from the rows inside $N$, all other rows are unchanged or improve in their defect count. Optimality of $\Pi(G)$ thus gives $d_G(v,N)\geq(1-3\alpha)|P_{i,j'}|-o(n)$. If every part $P_{h,a'}\subseteq N$ satisfied $d_G(v,P_{h,a'})<\alpha |P_{h,a'}|$ then $d_G(v,N)<\alpha |N|$, and moving $v$ to $\Pi_\sparse$ would strictly decrease $b(G,\Pi)$, contradicting optimality. It follows that some part $P_{h,a'}\subseteq N$ satisfies $d_G(v,P_{h,a'})\geq\alpha |P_{h,a'}|$. Since this part has size at least $\theta n$ by the preceding lower bound and our choice of parameters, it belongs to $\cY_\Pi$.

The preceding paragraph gives a pair $(h,a')\in\cN_{\Pi,G}(v)$ whose part lies in the component containing $v$, while every retained part counted by $\cN_{\Pi,G}(v)\cap\cY^\ast_\Pi$ also contributes an element of $\cN_{\Pi,G}(v)$. By \Cref{lemma:deterministic-nonlow-bound-K1k}, we have $|\cN_{\Pi,G}(v)|\leq k-1$, hence $|\cN_{\Pi,G}(v)\cap\cY^\ast_\Pi|\leq\Delta$.
\end{proof}

For any graph $T$ on $V$ with $E(T)\subseteq E(K_{V_\ast,V})$, define its signed size by
$$\sigma(T):=e(T)-2\sum_{i=1}^{\ell_\ast}\sum_{j=1}^{q_i}e(T[P_{i,j}])\,,$$
so positive defect edges are counted with weight $1$, while missing defect edges are counted with weight $-1$. Define
$$\Err(\frp):=C\,|B(\frp)|\,\bigl(H(5\alpha)n+\theta n+\delta n+\rho_Bn+\log n\bigr)\,,$$
where $C=C(k,R_0)>0$ is sufficiently large.

\subsubsection{The profile bound}\label{subsec:profile-count-K1k}
Let us say that a triple $(T_B,R,L)$ of graphs is \emph{compatible} (with the profile $\frp$) if there exists a graph $G_0\in\cF_{\Pi,\frp}$ such that, writing $T_0:=T(G_0)[V_\ast,V]$,
$$T_B=T_B(G_0)\,,\qquad R=T_\res(G_0)\,,\qquad L=T_0\setminus(T_B\cup R)\,.$$
By definition, if $(T_B,R,L)$ is compatible then $T_0=T_B\cup R\cup L$ as an edge-disjoint union. Let us say that a pair $(T_B,R)$ is compatible if an $L$ exists such that $(T_B,R,L)$ is compatible. Let us write
$$\cT^\leftsf_{\Pi,\frp}(T_B,R):=\{T(G)[V_\ast,V]\setminus(T_B\cup R):G\in\cF_{\Pi,\frp}\,,\,T_B(G)=T_B\,,\,T_\res(G)=R\}$$
for the set of leftover defect graphs after $T_B(G)$ and $T_\res(G)$ are removed.
For a graph $R$, let $\cE_\matsf(R)$ denote the event, on graphs $F$ on $V$, that no edge of $R$ is contained in an induced copy of $K_{1,k}$ in $F$.
Define
\begin{equation}\label{eqn:Psi-mat-definition-K1k}
\Psi_\matsf(\frp):=\log\Bigg(\max_{T_B\in\cT^\root_{\Pi,\frp}}\Bigg\{\sum_{R\in\cT^\res_{\Pi,\frp}(T_B)}e^{C_ke(R)}\max_{H,L,\bsm}\bbP_{G^\bin_{\Pi,H,T_B\cup R\cup L,\bsm}}\{\cE_\matsf(R)\}\Bigg\}\Bigg)\,,
\end{equation}
where $C_k>0$ is large enough depending only on $k$, and the inner maximum is over induced-$K_{1,k}$-free graphs $H$ on $S_\Pi$, graphs $L$ such that $(T_B,R,L)$ is compatible, and $\bsm\in\cM^\nar_\Pi$. The empty graph is allowed as an element of $\cT^\root_{\Pi,\frp}$. Thus $\cT^\root_{\Pi,\frp}=\emptyset$ occurs only if $\cF_{\Pi,\frp}=\emptyset$; otherwise, the maximum above is over at least one rooted defect graph $T_B$, possibly over the empty graph.

\begin{lemma}\label{lemma:active-level-comparison-K1k}
There exists a constant $C=C(k,\gamma)>0$ such that the following holds. Let $A_\Pi:=\sum_{e\in\cI_\Pi}N_e$. Fix integers $u,u'\in\Z$ and $a_+,a_-\geq0$ such that $|u'-u|\leq\delta A_\Pi/4$ and $u'-u=a_+-a_-$. Then
$$Z^\nar_\Pi(u')\leq n^{C|\cI_\Pi|}\,e^{-(L_k-C\delta)a_+ +(L_k+C\delta)a_-}Z_\Pi(u)\,.$$
\end{lemma}
\begin{proof}
Let $d:=u'-u$. Suppose first that $d\geq0$. For $\bsm=(\bsm_e)_{e\in\cI_\Pi}\in\cM^\nar_{\Pi,u'}$, let $\cS(\bsm)$ be the set of vectors $\bsa=(a_e)_{e\in\cI_\Pi}$ such that $a_e\geq0$, $\sum_ea_e=d$, and $\bsm_e+a_e\leq(p_k+2\delta)N_e$ for all $e$. Since $\bsm_e\leq(p_k+\delta)N_e$ and $d\leq\delta A_\Pi/4$, the set $\cS(\bsm)$ is nonempty. If $\bsa\in\cS(\bsm)$ then $\bsm+\bsa\in\cM_{\Pi,u}$ and
$$\binom{\Pi}{\bsm}\binom{\Pi}{\bsm+\bsa}^{-1}=\prod_{e\in\cI_\Pi}\prod_{q=1}^{a_e}\frac{\bsm_e+q}{N_e-\bsm_e-q+1}\leq e^{-(L_k-C\delta)d}\,.$$
Choose one element $\bsa(\bsm)\in\cS(\bsm)$ for every $\bsm$. For fixed $\bsa$, the map $\bsm\mapsto\bsm+\bsa$ is injective, and the number of possible vectors $\bsa$ is at most $n^{C|\cI_\Pi|}$. It follows that
$$Z^\nar_\Pi(u')\leq n^{C|\cI_\Pi|}e^{-(L_k-C\delta)d}Z_\Pi(u)\,.$$
The case $d<0$ is identical by subtracting a vector $\bsa$ with $\sum_ea_e=-d$.
\end{proof}

\begin{lemma}\label{lemma:profile-root-leftover-bound-K1k}
For all $\frp\in\frP_\Pi$, $T_B\in\cT^\root_{\Pi,\frp}$ and $R\in\cT^\res_{\Pi,\frp}(T_B)$, we have
$$|\cT^\leftsf_{\Pi,\frp}(T_B,R)|\leq e^{\Err(\frp)}\,.$$
Moreover, we have $\abs{\sigma(L)}\leq\Err(\frp)$ for all $L\in\cT^\leftsf_{\Pi,\frp}(T_B,R)$.
\end{lemma}
\begin{proof}
Let $B:=B(\frp)$, and fix $T_B\in\cT^\root_{\Pi,\frp}$ and $R\in\cT^\res_{\Pi,\frp}(T_B)$. Let $L\in\cT^\leftsf_{\Pi,\frp}(T_B,R)$, and choose $G\in\cF_{\Pi,\frp}$ such that
$$T_B(G)=T_B\,,\qquad T_\res(G)=R\,,\qquad L=T(G)[V_\ast,V]\setminus(T_B\cup R)\,.$$
By definition, every edge of $L$ is incident with some vertex of $B$.

We first count the possible edges of $L$ incident with a fixed vertex $v\in B$ and not incident with another vertex of $B$. Consider a $\theta$-visible part $Y\in\cY_\Pi$ for which the pairs $\{vy:y\in Y\setminus B\}$ are contained in the defect graph $T(G)[V_\ast,V]$. If the pair $(Y,d_G(v,Y\setminus B))$ belongs to $\cR_v(\frp)$ then the corresponding edges have already been included in $T_B$ and hence do not belong to $L$. If no pair with target $Y$ belongs to $\cR_v(\frp)$ then by the definition of the profile we have $d_G(v,Y)<4\alpha |Y|$. Since $|B|\leq\rho_Bn$ and $\rho_B=o(\alpha\theta)$, this implies $d_G(v,Y\setminus B)\leq 5\alpha |Y\setminus B|$. Thus after $v$ and $Y$ are fixed, the number of possibilities for the edges of $L$ from $v$ to $Y\setminus B$ is at most $\sum_{0\leq r\leq 5\alpha |Y\setminus B|}\binom{|Y\setminus B|}{r}\leq e^{H(5\alpha)|Y\setminus B|}$.
Multiplying this bound over all such $\theta$-visible target parts $Y$ gives at most $e^{H(5\alpha)n}$ choices for the $\theta$-visible neighbors of $v$ that can appear in $L$.

It remains to consider the vertices of $S_\smsf(\Pi)$. If $v\in B_\ret(\frp)$ then \Cref{lemma:SubCompareSparseSideEdgesK1k} gives $d_G(v,S_\smsf(\Pi))\leq C_k\theta n$.
Since $G[N_G(v,S_\smsf(\Pi))]$ has independence number at most $k-1$, the possible sets $N_G(v,S_\smsf(\Pi))$ are contained in the closed neighborhood, inside $G[S_\smsf(\Pi)]$, of at most $k-1$ vertices. It follows that there are at most $n^{k-1}2^{C_k\theta n}\leq e^{C_k\theta n+C_k\log n}$ choices for the edges of $L$ from $v$ to $S_\smsf(\Pi)$. If $v\in B_\outsf(\frp)$ then no edge from $v$ to $S_\smsf(\Pi)$ lies in $T(G)[V_\ast,V]$, so there is nothing to count.

Finally, the edges of $L$ with both endpoints in $B$ have at most $2^{\binom{|B|}{2}}$ choices. Since $|B|\leq\rho_Bn$ and $\rho_B=o(H(5\alpha)+\theta)$, this factor is absorbed by $e^{C|B|(H(5\alpha)n+\theta n+\log n)}$. Multiplying over all $v\in B$ and increasing the constant in the definition of $\Err(\frp)$ implies
$$|\cT^\leftsf_{\Pi,\frp}(T_B,R)|\leq e^{\Err(\frp)}\,.$$

The same estimates also bound the number of edges of $L$. Indeed, for each $v\in B$, the contribution from $\theta$-visible target parts is at most $5\alpha n$, the contribution from $S_\smsf(\Pi)$ is at most $C_k\theta n$, and the contribution from edges with both endpoints in $B$ is at most $|B|$. It follows that after increasing the constant in $\Err(\frp)$, we have $\abs{\sigma(L)}\leq e(L)\leq \Err(\frp)$.
\end{proof}

\begin{lemma}\label{lemma:fixed-profile-probability-K1k}
Let $T_B\in\cT^\root_{\Pi,\frp}$, let $R\in\cT^\res_{\Pi,\frp}(T_B)$, let $L\in\cT^\leftsf_{\Pi,\frp}(T_B,R)$, and let $T:=T_B\cup R\cup L$. For every induced-$K_{1,k}$-free graph $H$ on $S_\Pi$ and every $\bsm\in\cM^\nar_\Pi$, we have
$$
\bbP\!\left\{G_{\Pi,H,T,\bsm}\in\bigcap_{\frt\in\Tail(\frp)}E_\frt\cap\cE_\matsf(R)\right\}
\leq n^{C_k|\cI_\Pi|}\exp\Bigg(-\sum_{v\in B_\ret(\frp)}J_v(\frp)+\Psi_\matsf(\frp)\Bigg)\,.
$$
\end{lemma}
\begin{proof}
By \Cref{lemma:fixed-count-transfer-K1k}, it is enough to prove the corresponding bound with $G^\bin_{\Pi,H,T,\bsm}$ in place of $G_{\Pi,H,T,\bsm}$, since this incurs only the multiplicative factor $n^{C_k|\cI_\Pi|}$. In $G^\bin_{\Pi,H,T,\bsm}$, the events $\bigcap_{\frt\in\Tail_v(\frp)}E_\frt$ for distinct $v\in B_\ret(\frp)$ are determined by disjoint edge sets since $B(\frp)$ is deleted from every target part in the definition of the events $E_\frt$. The event $\cE_\matsf(R)$ also uses no edge incident with $B(\frp)$. Thus these events are independent in the binomial model, and the lemma follows from the definitions of $J_v(\frp)$ and $\Psi_\matsf(\frp)$.
\end{proof}

\begin{lemma}\label{lemma:profile-bound-K1k}
For every profile $\frp\in\frP_\Pi$,
$$|\cF_{\Pi,\frp}|\leq n^{C_k|\cI_\Pi|}\cZ_\Pi\exp\!\left(\sum_{v\in B(\frp)}\Psi_v(\frp)+\Psi_\matsf(\frp)+\Err(\frp)\right)\,.$$
\end{lemma}
\begin{proof}
We may assume $\cF_{\Pi,\frp}\neq\emptyset$, and write $b:=b(\frp)$ and $B:=B(\frp)$. By \Cref{lemma:SubCompareSparseSideEdgesK1k}, every graph in $\cF_{\Pi,\frp}$ satisfies $b=e(G[S_\Pi])\leq C_k\eta n^2$ after increasing $C_k$. Hence the number of possibilities for $G[S_\Pi]$ is at most $N^\ast_{s(\Pi),b}(K_{1,k})$. We will use the fact that every graph in $\cF_{\Pi,\frp}$ is the outcome of the random graph $G_{\Pi,H,T,\bsm}$ for some induced-$K_{1,k}$-free graph $H$ on $S_\Pi$ with at most $C_k\eta n^2$ edges, some graph $T\in\cT_{W,\Pi}$, and some $\bsm\in\cM^\nar_{\Pi,e(H)+\sigma(T)}$.

For a compatible triple $(T_B,R,L)$, and writing $T:=T_B\cup R\cup L$, a graph $G\in\cF_{\Pi,\frp}$ is an outcome of $G_{\Pi,H,T,\bsm}$ if and only if $\bsm\in\cM^\nar_{\Pi,b'}$ where $b':=b+\sigma(T_B)+\sigma(R)+\sigma(L)$.
We will also use the bound
\begin{equation}\label{eqn:troot-bound}
\abs{\cT^\root_{\Pi,\frp}} \leq \prod_{v\in B(\frp)}\prod_{(Y,r)\in\cR_v(\frp)}\binom{|Y\setminus B|}{r}\cdot\prod_{v\in B_\ret(\frp)}\binom{|P_v\setminus B|}{I_v(\frp)}\,,
\end{equation}
which also gives
\begin{equation}\label{eqn:root-free-energy-identity-K1k}
\sum_{T_B\in\cT^\root_{\Pi,\frp}}e^{-L_k\sigma(T_B)}\leq \abs{\cT^\root_{\Pi,\frp}}\max_{T_B\in\cT^\root_{\Pi,\frp}}e^{-L_k\sigma(T_B)} \leq \exp\Bigg(\sum_{v\in B(\frp)}\Ent_v(\frp)\Bigg)\,,
\end{equation}
by the definitions of $T_B$, $\sigma$, and $\Ent_v$.

Using \Cref{lemma:fixed-profile-probability-K1k}, \Cref{lemma:profile-root-leftover-bound-K1k}, \Cref{lemma:active-level-comparison-K1k}, \eqref{eqn:troot-bound}, and \eqref{eqn:root-free-energy-identity-K1k}, we calculate the following, with details provided below:
\begin{align*}
|\cF_{\Pi,\frp}|&\leq \sum_{H\in\cF_{S_\Pi,b}(K_{1,k})}\sum_{T\in\cT_{W,\Pi}}\sum_{\bsm\in\cM^\nar_{\Pi,b+\sigma(T)}}\binom{\Pi}{\bsm}\,\bbP\{G_{\Pi,H,T,\bsm}\in\cF_{\Pi,\frp}\} \\
&\leq n^{C_k|\cI_\Pi|}N^\ast_{s(\Pi),b}(K_{1,k})\sum_{T\in\cT_{W,\Pi}}\max_{H,\bsm'}\left\{\bbP\{G^\bin_{\Pi,H,T,\bsm'}\in\cF_{\Pi,\frp}\}\right\}\sum_{\bsm\in\cM^\nar_{\Pi,b+\sigma(T)}}\binom{\Pi}{\bsm} \\
&\leq n^{C_k|\cI_\Pi|}N^\ast_{s(\Pi),b}(K_{1,k})\sum_{T_B\in\cT^\root_{\Pi,\frp}}\sum_{R\in\cT^\res_{\Pi,\frp}(T_B)}\sum_{L\in\cT^\leftsf_{\Pi,\frp}(T_B,R)} \bigg(Z^\nar_\Pi(b+\sigma(T_B)+\sigma(R)+\sigma(L)) \\
&\qquad\cdot \max_{\substack{H,\bsm\\T:=T_B\cup R\cup L}}\bigg\{\bbP_{G^\bin_{\Pi,H,T,\bsm}}\bigg\{\bigcap_{\frt\in\Tail(\frp)}E_\frt\bigg\}\cdot\bbP_{G^\bin_{\Pi,H,T,\bsm}}\{\cE_\matsf(R)\}\bigg\}\bigg) \\
&\leq n^{C_k|\cI_\Pi|}N^\ast_{s(\Pi),b}(K_{1,k})\sum_{T_B\in\cT^\root_{\Pi,\frp}}\sum_{R\in\cT^\res_{\Pi,\frp}(T_B)}\sum_{L\in\cT^\leftsf_{\Pi,\frp}(T_B,R)}Z^\nar_\Pi(b+\sigma(T_B)+\sigma(R)+\sigma(L))\\
&\qquad\cdot \exp\Bigg(-\sum_{v\in B_\ret(\frp)}J_v(\frp)+\Psi_\matsf(\frp)\Bigg)\\
&\leq n^{C_k|\cI_\Pi|}N^\ast_{s(\Pi),b}(K_{1,k})Z_\Pi(b)\sum_{T_B\in\cT^\root_{\Pi,\frp}}\sum_{R\in\cT^\res_{\Pi,\frp}(T_B)}\sum_{L\in\cT^\leftsf_{\Pi,\frp}(T_B,R)}e^{-L_k\sigma(T_B)-L_k\sigma(R)-L_k\sigma(L)}\\
&\qquad\cdot \exp\!\left(C\delta(|\sigma(T_B)|+|\sigma(R)|+|\sigma(L)|)-\sum_{v\in B_\ret(\frp)}J_v(\frp)+\Psi_\matsf(\frp)\right)\\
&\leq n^{C_k|\cI_\Pi|}N^\ast_{s(\Pi),b}(K_{1,k})Z_\Pi(b)\,\exp\!\left(\sum_{v\in B(\frp)}\Ent_v(\frp)-\sum_{v\in B_\ret(\frp)}J_v(\frp)+\Psi_\matsf(\frp)+\Err(\frp)\right)\\
&\leq n^{C_k|\cI_\Pi|}N^\ast_{s(\Pi),b}(K_{1,k})Z_\Pi(b)\,\exp\!\left(\sum_{v\in B(\frp)}\Psi_v(\frp)+\Psi_\matsf(\frp)+\Err(\frp)\right)\\
&\leq n^{C_k|\cI_\Pi|}\cZ_\Pi\,
\exp\!\left(\sum_{v\in B(\frp)}\Psi_v(\frp)+\Psi_\matsf(\frp)+\Err(\frp)\right)\,.
\end{align*}
In the third inequality, we decomposed $T$ as $T_B\cup R\cup L$, and we used that the event $\cF_{\Pi,\frp}$ is contained in the intersection of the profile-tail event $\bigcap_{\frt\in\Tail(\frp)}E_\frt$ and the residual matching event $\cE_\matsf(R)$; these events are independent in $G^\bin_{\Pi,H,T,\bsm}$ since the former only uses random edges incident with $B_\ret(\frp)$ while the latter is defined after deleting all edges incident with $B(\frp)$.
The fourth inequality uses the definition of $J_v(\frp)$ and $\Psi_\matsf(\frp)$.
In the fifth inequality, we applied \Cref{lemma:active-level-comparison-K1k}; the adjustment to the subscript $b$ is $\sigma(T_B)+\sigma(R)+\sigma(L)$, which has absolute value $O(\epsilon n^2)$, so the hypothesis of that lemma holds by our choice of parameters.
The sixth inequality follows by combining \Cref{lemma:profile-root-leftover-bound-K1k}, \eqref{eqn:troot-bound}, and \eqref{eqn:root-free-energy-identity-K1k}.
The seventh inequality is simply the definition of $\Psi_v(\frp)$ for $v\in B(\frp)$. Finally, the last inequality uses that $b\leq C_k\eta n^2$, so $N^\ast_{s(\Pi),b}(K_{1,k})Z_\Pi(b)$ is at most $\cZ_\Pi$.
\end{proof}

\subsubsection{Local exponent bounds}\label{subsec:local-exponent-bounds}
To make proper use of \Cref{lemma:profile-bound-K1k}, it remains to prove suitable bounds for the exponents $\Psi_v(\frp)$ and $\Psi_\matsf(\frp)$. In this part, \Cref{lemma:medium-internal-negative-root-local-K1k,lemma:high-internal-negative-root-local-K1k,lemma:high-internal-degree-root-local-K1k,lemma:outside-root-compensation-K1k} together bound $\Psi_v(\frp)$ for all $v\in B(\frp)$ by considering four separate cases. Then \Cref{lemma:residual-matching-estimate-K1k} proves a bound for $\Psi_\matsf(\frp)$.

We begin with some notation and estimates that will be used throughout the next four lemmas. Assume $\cF_{\Pi,\frp}\neq\emptyset$, fix $G\in\cF_{\Pi,\frp}$, and let $v\in B_\ret(\frp)\cap P_{i,j}$. Write $B=B(\frp)$ and $N:=|P_{i,j}|$. Since $P_{i,j}$ is retained, \Cref{lemma:WtoWtildeMetricsK1k} gives $N\geq \eta n/(4R_0)$. Also, by \Cref{lemma:profile-covering-K1k} we have $|P_{i,j}\setminus B|=(1+o(1))N$, and the same estimate holds for every part in $V_i$. Define
$$\arraycolsep=1.4pt
\begin{array}{ll}
U_v&:=\{x\in N_{Q_i}(j):d_G(v,P_{i,x})\geq(1-\alpha)|P_{i,x}|\}\,,\\[7pt]
L_v&:=\{x\in N_{Q_i}(j):d_G(v,P_{i,x})\leq\alpha |P_{i,x}|\}\,,\\[7pt]
M_v&:=N_{Q_i}(j)\setminus(U_v\cup L_v)\,.
\end{array}$$
We also define
$$\cR_v^\insf(\frp):=\{(Y,r)\in\cR_v(\frp):Y\subseteq V_i\}\,,\qquad \cR_v^\outsf(\frp):=\cR_v(\frp)\setminus\cR_v^\insf(\frp)$$
and
$$q_\insf:=|\cR_v^\insf(\frp)|\,,\quad q_\outsf:=|\cR_v^\outsf(\frp)|\,,\quad H_v:=\{(Y,r)\in\cR_v^\insf(\frp):r\geq(1-2\alpha)|Y\setminus B|\}\,.$$

\begin{fact}\label{fact:local-ent-estimates-K1k}
With the notation above, we use \eqref{eqn:local-ent-estimates-K1k} to deduce
\begin{equation}\label{eqn:first-summand-ent}
|P_{i,j}\setminus B|\,H\!\left(\frac{I_v(\frp)}{|P_{i,j}\setminus B|}\right)+L_kI_v(\frp)\leq(U+o(1))N\,.
\end{equation}
If $I_v(\frp)\leq C_0\alpha N$ for fixed $C_0>0$ then the left-hand side of \eqref{eqn:first-summand-ent} is $o(N)$, and if $I_v(\frp)\geq(1-2\alpha)|P_{i,j}\setminus B|$ then it is at most $(L_k+o(1))N$. Moreover, for every $(Y,r)\in\cR_v(\frp)$, we use \eqref{eqn:local-ent-estimates-K1k} to deduce
\begin{equation}\label{eqn:second-summand-ent}
|Y\setminus B|\,H\!\left(\frac{r}{|Y\setminus B|}\right)-L_kr\leq(A+o(1))|Y|\,.
\end{equation}
If $Y\subseteq V_i$ then the left-hand side of \eqref{eqn:second-summand-ent} is at most $(A+o(1))N$, and if additionally $r\geq(1-2\alpha)|Y\setminus B|$ then \eqref{eqn:high-row-ent-K1k} gives the stronger bound $(-L_k+o(1))N$.
\end{fact}

\begin{claim}\label{claim:local-tail-estimate-K1k}
With the notation above,
$$J_v(\frp)\geq |U_v|(U-o(1))N+|L_v|(A-o(1))N\,.$$
\end{claim}
\begin{proof}
Recall the definitions of the upper- and lower-tail events $E_\frt$ for $\frt\in\Tail_v(\frp)$ given in \Cref{subsec:local-exponents-K1k}. The events $\{E_\frt\}_{\frt\in\Tail_v(\frp)}$ are independent in $G^\bin_{\Pi,H,T,\bsm}$ since they correspond to the adjacency of $v$ into distinct sets $P_{i,x}$ with $x\in N_{Q_i}(j)$. If $x\in U_v$ then the upper-tail event $E_{v,x,\Upper}$ has probability at most $e^{-(U-o(1))N}$, by the usual large deviation estimate for binomial random variables. If $x\in L_v$ then the lower-tail event $E_{v,x,\Lower}$ has probability at most $e^{-(A-o(1))N}$. These estimates are uniform over $\bsm\in\cM^\nar_\Pi$ since $\bsm_e/N_e\in p_k\pm\delta$ for every such $\bsm$, hence taking the maximum over $\bsm$ proves the claim.
\end{proof}

\begin{claim}\label{claim:external-component-ent-K1k}
If $a\neq i$ and $N_a$ denotes the size of any part of the component $V_a$ then
$$\sum_{\substack{(Y,r)\in\cR_v(\frp)\\ Y\subseteq V_a}}\left(|Y\setminus B|\,H\!\left(\frac{r}{|Y\setminus B|}\right)-L_kr\right)\leq o(N_a)\,.$$
\end{claim}
\begin{proof}
Let $M'$ be the set of pairs $(Y,r)\in\cR_v(\frp)$ with $Y\subseteq V_a$ and $r<(1-2\alpha)|Y\setminus B|$, and let $H'$ be the set of pairs $(Y,r)\in\cR_v(\frp)$ with $Y\subseteq V_a$ and $r\geq(1-2\alpha)|Y\setminus B|$. By definition of $\cR_v(\frp)$, every pair in $M'$ has density at least $4\alpha+o(1)$ into its target. Since $\rho_B=o(\alpha\theta)$, every such row has density in $[\alpha,1-\alpha]$ into its target for all sufficiently large $n$. Hence \Cref{lemma:deterministic-medium-companion-K1k} implies that for every pair in $M'$, there is a pair in $H'$ whose target is adjacent to its target in $Q_a$. Since each part of $Q_a$ has degree $\Delta$, we have $|M'|\leq\Delta |H'|$. Using \Cref{fact:local-ent-estimates-K1k}, the balance of parts in the component $V_a$ from \Cref{lemma:WtoWtildeMetricsK1k} \ref{item:SubCloseK1k-balanced}, and $L_k=\Delta A$, we obtain
$$(A+o(1))|M'|N_a-(L_k-o(1))|H'|N_a\leq o(N_a)\,,$$
proving the claim.
\end{proof}

\begin{claim}\label{claim:local-nonlow-budget-K1k}
With the notation above, we have
$$|U_v|+|M_v|+q_\insf+q_\outsf\leq\Delta+1\,,\qquad q_\insf+q_\outsf\leq |L_v|+1\,.$$
If additionally $d_G(v,P_{i,j})\geq\alpha |P_{i,j}|$ then
$$|U_v|+|M_v|+q_\insf+q_\outsf\leq\Delta\,,\qquad q_\insf+q_\outsf\leq |L_v|\,.$$
\end{claim}
\begin{proof}
The parts indexed by $U_v\cup M_v$ contribute $|U_v|+|M_v|$ elements of $\cN_{\Pi,G}(v)$, and the rows in $\cR_v(\frp)$ contribute $q_\insf+q_\outsf$ further elements. Thus \Cref{lemma:deterministic-nonlow-bound-K1k} gives $|U_v|+|M_v|+q_\insf+q_\outsf\leq\Delta+1$. Since $|U_v|+|L_v|+|M_v|=\Delta$, this proves $q_\insf+q_\outsf\leq |L_v|+1$. If additionally $d_G(v,P_{i,j})\geq\alpha |P_{i,j}|$ then the own part $P_{i,j}$ contributes one more element of $\cN_{\Pi,G}(v)$. Applying \Cref{lemma:deterministic-nonlow-bound-K1k} with this element included gives $1+|U_v|+|M_v|+q_\insf+q_\outsf\leq\Delta+1$. Since $|U_v|+|L_v|+|M_v|=\Delta$, this proves $q_\insf+q_\outsf\leq |L_v|$.
\end{proof}

\begin{lemma}\label{lemma:medium-internal-negative-root-local-K1k}
There exists $c=c(k,\gamma)>0$ such that the following holds. Assume $\cF_{\Pi,\frp}\neq\emptyset$. If $v\in B_\ret(\frp)\cap P_{i,j}$ and $2\alpha |P_{i,j}\setminus B|\leq I_v(\frp)\leq(1-2\alpha)|P_{i,j}\setminus B|$ then $\Psi_v(\frp)\leq -c\eta n/R_0$.
\end{lemma}
\begin{proof}
Fix $G\in\cF_{\Pi,\frp}$ and use the notation above. Since $I_v(\frp)\leq(1-2\alpha)|P_{i,j}\setminus B|$, we have $d_G(v,P_{i,j})\geq\alpha |P_{i,j}|$ for all sufficiently large $n$. Since also $I_v(\frp)\geq2\alpha |P_{i,j}\setminus B|$, the density of $v$ into $P_{i,j}$ lies in $[\alpha,1-\alpha]$. Applying \Cref{lemma:deterministic-medium-companion-K1k} to the pair $(v,P_{i,j})$ implies $|U_v|\geq1$.

By \Cref{claim:local-nonlow-budget-K1k}, we have $q_\insf+q_\outsf\leq |L_v|$. Also, by \Cref{claim:external-component-ent-K1k}, rows with targets in components other than $V_i$ have nonpositive leading contribution to $\Ent_v(\frp)$. Combining \Cref{fact:local-ent-estimates-K1k,claim:local-tail-estimate-K1k}, we get
\begin{align}
\Psi_v(\frp)&\leq (U+o(1))N+(A+o(1))q_\insf N-|U_v|(U-o(1))N-|L_v|(A-o(1))N \notag\\
&\leq \bigl((1-|U_v|)U+(q_\insf-|L_v|)A+o(1)\bigr)N\,. \label{eqn:medium-main-bound-K1k}
\end{align}
Since $q_\insf+q_\outsf\leq |L_v|$ and $|U_v|\geq1$, the right-hand side is at most $-c\eta n/R_0$ unless $|U_v|=1$, $q_\outsf=0$, and $q_\insf=|L_v|$. Assume we are in this remaining case, and let $r$ be the unique element of $U_v$.

If $I_v(\frp)\leq C_0\alpha N$, where $C_0>4(\Delta+1)$ is fixed then \Cref{fact:local-ent-estimates-K1k} gives
$$|P_{i,j}\setminus B|\,H\!\left(\frac{I_v(\frp)}{|P_{i,j}\setminus B|}\right)+L_kI_v(\frp)=o(N)\,,$$
while the upper-tail event $E_{v,r,\Upper}$ contributes $(U-o(1))N$ to $J_v(\frp)$, hence $\Psi_v(\frp)\leq-c\eta n/R_0$. We may therefore assume that $I_v(\frp)>C_0\alpha N$.

If some pair $(Y,r_Y)\in\cR_v^\insf(\frp)$ satisfies $r_Y\geq(1-2\alpha)|Y\setminus B|$ then \Cref{fact:local-ent-estimates-K1k} gives the stronger bound $(-L_k+o(1))N$ for the corresponding summand in $\Ent_v(\frp)$, instead of the crude $(A+o(1))N$ used in \eqref{eqn:medium-main-bound-K1k}. This again implies $\Psi_v(\frp)\leq-c\eta n/R_0$. Hence every pair $(Y,r_Y)\in\cR_v^\insf(\frp)$ satisfies $r_Y<(1-2\alpha)|Y\setminus B|$. By \Cref{lemma:deterministic-medium-companion-K1k}, each such pair has a high companion in an adjacent part. Since $r\in N_{Q_i}(j)$ is the unique index with $(v,r,\Upper)\in\Tail_v(\frp)$ and since there are no pairs $(Y,r_Y)\in\cR_v^\insf(\frp)$ with $r_Y\geq(1-2\alpha)|Y\setminus B|$, every target $Y$ of a pair in $\cR_v^\insf(\frp)$ is $Q_i$-adjacent to $P_{i,r}$.

We claim that every $s\in M_v$ is adjacent to $r$. Suppose otherwise, and fix $s\in M_v$ with $rs\notin E(Q_i)$. Since $s\in M_v$, we have
$$\alpha|P_{i,s}|\leq d_G(v,P_{i,s})\leq(1-\alpha)|P_{i,s}|\,.$$
Also, since $I_v(\frp)\geq2\alpha|P_{i,j}\setminus B|$ and $\rho_B=o(\alpha\theta)$, we have $d_G^c(v,P_{i,j})\geq\alpha|P_{i,j}|$ for all sufficiently large $n$. Finally, let $t\in N_{Q_i}(s)\setminus\{j\}$. If $t\in N_{Q_i}(j)$ then $t\neq r$ because $rs\notin E(Q_i)$, and hence $t\notin U_v$. Therefore $d_G^c(v,P_{i,t})\geq\alpha|P_{i,t}|$. If $t\notin N_{Q_i}(j)\cup\{j\}$ then either $(P_{i,t},d_G(v,P_{i,t}\setminus B))\in\cR_v^\insf(\frp)$, in which case the preceding paragraph gives
$$d_G(v,P_{i,t}\setminus B)<(1-2\alpha)|P_{i,t}\setminus B|\,,$$
or else $d_G(v,P_{i,t})<4\alpha|P_{i,t}|$. In the first case,
$$d_G(v,P_{i,t})\leq d_G(v,P_{i,t}\setminus B)+|B\cap P_{i,t}|<(1-2\alpha)|P_{i,t}\setminus B|+|B\cap P_{i,t}|\leq(1-\alpha)|P_{i,t}|$$
for all sufficiently large $n$, since $|B|\leq\rho_Bn$ and $\rho_B=o(\theta)$. In the second case, $d_G^c(v,P_{i,t})\geq\alpha|P_{i,t}|$ by our choice of $\alpha$. It follows that in both cases, $d_G^c(v,P_{i,t})\geq\alpha|P_{i,t}|$. It follows that \Cref{lemma:deterministic-opposite-row-K1k} applies, yielding a contradiction.

Let $\Pi'$ be the division obtained from $\Pi$ by moving $v$ from $P_{i,j}$ to $P_{i,r}$. The edges from $v$ to the parts indexed by $M_v$ remain in active pairs after this move, and the rows in $\cR_v^\insf(\frp)$ also become active. The missing edges from $v$ to $P_{i,r}$ contribute at most $2\alpha|P_{i,r}|+o(n)$ to $b(G,\Pi')$, while the lower-tail parts contribute at most $2\alpha\sum_{s\in L_v}|P_{i,s}|+o(n)$. On the other hand, the $I_v(\frp)$ missing edges from $v$ to $P_{i,j}\setminus B$ no longer contribute to the defect count, and all rows outside $V_i$ are unchanged. It follows that
\begin{align*}
b(G,\Pi')-b(G,\Pi)&\leq 2\alpha|P_{i,r}|+2\alpha\sum_{s\in L_v}|P_{i,s}|-I_v(\frp)+o(n)<0\,,
\end{align*}
by the balance of the parts given in \Cref{lemma:WtoWtildeMetricsK1k}, and by the choice of $C_0$. This contradicts the optimality of $\Pi=\Pi(G)$, completing the proof.
\end{proof}

\begin{lemma}\label{lemma:high-internal-negative-root-local-K1k}
There exists $c=c(k,\gamma)>0$ such that the following holds. Assume $\cF_{\Pi,\frp}\neq\emptyset$. If $v\in B_\ret(\frp)\cap P_{i,j}$ and $I_v(\frp)\geq(1-2\alpha)|P_{i,j}\setminus B|$ then $\Psi_v(\frp)\leq -c\eta n/R_0$.
\end{lemma}
\begin{proof}
Fix $G\in\cF_{\Pi,\frp}$ and use the notation introduced above. We first show that $|U_v|+|M_v|\geq1$. If $U_v=M_v=\emptyset$ then for all $P_{i,x}$ with $x\in N_{Q_i}(j)$, we have $d_G(v,P_{i,x})\leq\alpha |P_{i,x}|$. Moving $v$ from $P_{i,j}$ to $\Pi_\sparse$ removes at least $(1-2\alpha)N+o(N)$ missing edges from $v$ to $P_{i,j}$, while the number of new defects inside $V_i$ is at most
$$2\alpha N+2\alpha\sum_{x\in N_{Q_i}(j)}|P_{i,x}|+o(N)\,,$$
contradicting the optimality of $\Pi=\Pi(G)$ after decreasing $\alpha$. It follows that $|U_v|+|M_v|\geq1$.

Assume first that $|U_v|\geq1$. By \Cref{claim:local-nonlow-budget-K1k}, we have $q_\insf+q_\outsf\leq |L_v|+1$. The derivation of \eqref{eqn:medium-main-bound-K1k} applies with the summand involving $I_v(\frp)$ bounded by $(L_k+o(1))N$, giving
\begin{align}
\Psi_v(\frp)&\leq (L_k+o(1))N+(A+o(1))q_\insf N-|U_v|(U-o(1))N-|L_v|(A-o(1))N\notag\\
&\leq\bigl(L_k-|U_v|U+(q_\insf-|L_v|)A+o(1)\bigr)N\,. \label{eqn:high-neg-main-bound-K1k}
\end{align}
Since $U=L_k+A$, the right-hand side is at most $-c\eta n/R_0$ unless $|U_v|=1$, $q_\outsf=0$, and $q_\insf=|L_v|+1$. Assume this remaining case holds, and let $r$ be the unique element of $U_v$.

If $H_v\neq\emptyset$ then \Cref{fact:local-ent-estimates-K1k} gives the stronger bound $(-L_k+o(1))N$ for at least one summand in $\Ent_v(\frp)$, instead of the crude $(A+o(1))N$ used in \eqref{eqn:high-neg-main-bound-K1k}. This gives $\Psi_v(\frp)\leq-c\eta n/R_0$, hence we may assume $H_v=\emptyset$. Similar to above, by \Cref{lemma:deterministic-medium-companion-K1k} every pair in $\cR_v^\insf(\frp)$ has a high companion in an adjacent part. Since $r$ is the unique index with $(v,r,\Upper)\in\Tail_v(\frp)$ and $H_v=\emptyset$, every target $Y$ of a pair in $\cR_v^\insf(\frp)$ is $Q_i$-adjacent to $P_{i,r}$.

The argument in the proof of \Cref{lemma:medium-internal-negative-root-local-K1k} showing that every $s\in M_v$ is adjacent to the unique element of $U_v$ applies here as well. Indeed, the assumption $H_v=\emptyset$ gives the same upper bound on $d_G(v,P_{i,t}\setminus B)$ for rows in $\cR_v^\insf(\frp)$ that was used there. It follows that every $s\in M_v$ is adjacent to $r$.

With the stronger lower bound $I_v(\frp)\geq(1-2\alpha)|P_{i,j}\setminus B|$, the final division comparison in the proof of \Cref{lemma:medium-internal-negative-root-local-K1k} gives a division $\Pi'$ with $b(G,\Pi')<b(G,\Pi)$, contradicting the optimality of $\Pi=\Pi(G)$.

It remains to consider the case $U_v=\emptyset$. Since $|U_v|+|M_v|\geq1$, we have $|M_v|\geq1$. We claim that every $s\in M_v$ has a companion in $H_v$ whose target is adjacent to $P_{i,s}$. Otherwise, choose $s\in M_v$ for which no such companion exists. The same case analysis over $t\in N_{Q_i}(s)\setminus\{j\}$ used in the proof of \Cref{lemma:medium-internal-negative-root-local-K1k} gives $d_G^c(v,P_{i,t})\geq\alpha|P_{i,t}|$ for every such $t$: if $t\in N_{Q_i}(j)$ then $t\notin U_v$, while if $t\notin N_{Q_i}(j)\cup\{j\}$ then either the corresponding row lies outside $H_v$ or $d_G(v,P_{i,t})<4\alpha|P_{i,t}|$. Since $s\in M_v$ and $d_G^c(v,P_{i,j})\geq\alpha|P_{i,j}|$, this contradicts \Cref{lemma:deterministic-opposite-row-K1k}.

If $|H_v|\geq2$ then \Cref{fact:local-ent-estimates-K1k} gives the stronger bound $(-L_k+o(1))N$ for at least two summands in $\Ent_v(\frp)$, which gives $\Psi_v(\frp)\leq-c\eta n/R_0$. If $|H_v|=1$ and $L_v\neq\emptyset$ then the unique summand indexed by $H_v$ contributes at most $(-L_k+o(1))N$, while one lower-tail event contributes $(A-o(1))N$ to $J_v(\frp)$; this again gives $\Psi_v(\frp)\leq-c\eta n/R_0$.

We are left with the case where $H_v$ consists of one pair, say with target $P_{i,r}$, and $L_v=\emptyset$. Then every active-neighbor adjacency belongs to $M_v$, and the claim in the preceding paragraph implies that every $s\in N_{Q_i}(j)$ is adjacent to $r$. Let $\Pi'$ be the division obtained from $\Pi$ by moving $v$ from $P_{i,j}$ to $P_{i,r}$. This makes the row into $P_{i,r}$ internal, makes the old own part an active-neighbor part, and keeps every old active-neighbor adjacency active. In this case it follows that
$$b(G,\Pi')-b(G,\Pi)\leq d_G^c(v,P_{i,r})+d_G(v,P_{i,j})-d_G^c(v,P_{i,j})-d_G(v,P_{i,r})+o(n)<0\,,$$
contradicting the optimality of $\Pi=\Pi(G)$, and proving the lemma.
\end{proof}

\begin{lemma}\label{lemma:high-internal-degree-root-local-K1k}
There exists $c=c(k,\gamma)>0$ such that the following holds. Assume $\cF_{\Pi,\frp}\neq\emptyset$. If $v\in B_\ret(\frp)\cap P_{i,j}$ and $I_v(\frp)<2\alpha|P_{i,j}\setminus B|$ then $\Psi_v(\frp)\leq -c\eta n/R_0$.
\end{lemma}
\begin{proof}
Fix $G\in\cF_{\Pi,\frp}$ and use the notation introduced above. Since $I_v(\frp)<2\alpha|P_{i,j}\setminus B|$, we have $d_G(v,P_{i,j})\geq(1-3\alpha)N$ for all sufficiently large $n$. By \Cref{claim:local-nonlow-budget-K1k}, we have $q_\insf+q_\outsf\leq |L_v|$.

The derivation of \eqref{eqn:medium-main-bound-K1k} applies with the summand involving $I_v(\frp)$ bounded by $o(N)$, so we obtain
\begin{equation}\label{eqn:high-degree-main-bound-K1k}
\Psi_v(\frp)\leq \bigl((q_\insf-|L_v|)A-|U_v|U+o(1)\bigr)N\,,
\end{equation}
with a strict improvement from \Cref{fact:local-ent-estimates-K1k} if $H_v\neq\emptyset$. By $q_\insf+q_\outsf\leq |L_v|$, the right-hand side is at most $-c\eta n/R_0$ unless $U_v=\emptyset$, $q_\outsf=0$, $q_\insf=|L_v|$, and $H_v=\emptyset$. Assume this last case holds.

Since $v\in B_\ret(\frp)$ and $I_v(\frp)<2\alpha|P_{i,j}\setminus B|$, we have $q_\insf+q_\outsf\geq1$, thus $q_\insf\geq1$. Choose a pair $(P_{i,s},r_s)\in\cR_v^\insf(\frp)$. Since $H_v=\emptyset$ and $\rho_B=o(\alpha\theta)$, this pair has density in $[\alpha,1-\alpha]$ into its target. By \Cref{lemma:deterministic-medium-companion-K1k}, it has a high companion in an adjacent part. This companion cannot be a pair in $\cR_v^\insf(\frp)$ because $H_v=\emptyset$, cannot be an upper-tail active-neighbor part because $U_v=\emptyset$, and cannot be the own part $P_{i,j}$ because $s\notin N_{Q_i}(j)\cup\{j\}$. This contradiction proves the lemma.
\end{proof}

\begin{lemma}\label{lemma:outside-root-compensation-K1k}
There exists $c=c(k,\gamma)>0$ such that the following holds. Assume $\cF_{\Pi,\frp}\neq\emptyset$. If $v\in B_\outsf(\frp)$ then $\Psi_v(\frp)\leq -c\eta n/R_0$.
\end{lemma}
\begin{proof}
Fix $G\in\cF_{\Pi,\frp}$. For each retained component $V_i$, let $M_i$ be the set of rows $(P_{i,j},r)\in\cR_v(\frp)$ such that $r<(1-2\alpha)|P_{i,j}\setminus B|$, and let $H_i$ be the set of rows $(P_{i,j},r)\in\cR_v(\frp)$ such that $r\geq(1-2\alpha)|P_{i,j}\setminus B|$. By the same application of \Cref{lemma:deterministic-medium-companion-K1k} used for rows in $\cR_v^\insf(\frp)$ in the proof of \Cref{lemma:medium-internal-negative-root-local-K1k}, every row in $M_i$ has a companion in $H_i$ whose target is adjacent to its target. In particular, $H_i\neq\emptyset$ whenever $M_i\cup H_i\neq\emptyset$.

By \Cref{lemma:sparse-retained-row-budget-K1k}, whenever $M_i\cup H_i\neq\emptyset$, the total number of retained rows rooted at $v$ is at most $\Delta$. In particular, $|M_i|+|H_i|\leq\Delta$. Let $N_i$ denote the size of any part of $V_i$. By \eqref{eqn:local-ent-estimates-K1k}, \eqref{eqn:high-row-ent-K1k}, the balance of the retained component $V_i$, and $L_k=\Delta A$, we have
\begin{align*}
\Psi_v(\frp)&=\Ent_v(\frp)\\
&=\sum_{(Y,r)\in\cR_v(\frp)} |Y\setminus B|\,H\!\left(\frac{r}{|Y\setminus B|}\right)-L_k\sum_{(Y,r)\in\cR_v(\frp)}r\\
&\leq\sum_{i=1}^{\ell_\ast}\left((A+o(1))|M_i|N_i-(L_k-o(1))|H_i|N_i\right)\\
&=\sum_{i=1}^{\ell_\ast}\bigl(A|M_i|-\Delta A|H_i|+o(1)\bigr)N_i\leq\sum_{\substack{i\in[\ell_\ast]\\ M_i\cup H_i\neq\emptyset}}(-A+o(1))N_i\,.
\end{align*}
The last sum is nonempty because $v\in B_\outsf(\frp)$. Since every retained part has size at least $\eta n/(4R_0)$ and the $o(1)$ term is made smaller than $A/2$ by the parameter hierarchy, the lemma follows.
\end{proof}

\begin{lemma}\label{lemma:local-compensation-K1k}
There exists $c=c(k,\gamma)>0$ such that the following holds. For every profile $\frp\in\frP_\Pi$ with $\cF_{\Pi,\frp}\neq\emptyset$ and every $v\in B(\frp)$, we have $\Psi_v(\frp)\leq -c\eta n/R_0$.
\end{lemma}
\begin{proof}
Fix $G\in\cF_{\Pi,\frp}$ and $v\in B_\ret(\frp)$, and write $P_v=P_{i,j}$. By \Cref{lemma:WtoWtildeMetricsK1k}, we have $|P_v|\geq \eta n/(4R_0)$. Moreover, by \Cref{lemma:profile-covering-K1k} we have $|B(\frp)|\leq \rho_B n$. By our choice of parameters, we have $\rho_B=o(\eta/R_0)$, hence $|P_v\setminus B(\frp)|=(1+o(1))|P_v|$. Since \Cref{lemma:medium-internal-negative-root-local-K1k,lemma:high-internal-negative-root-local-K1k,lemma:high-internal-degree-root-local-K1k} cover all possibilities for the value of $I_v(\frp)$, these three lemmas together prove $\Psi_v(\frp)\leq -c\eta n/R_0$. Finally, \Cref{lemma:outside-root-compensation-K1k} proves $\Psi_v(\frp)\leq -c\eta n/R_0$ for all $v\in B_\outsf(\frp)$.
\end{proof}

\begin{lemma}\label{lemma:residual-matching-estimate-K1k}
There exists $c_\matsf=c_\matsf(k,\gamma,\eta,R_0)>0$ such that for every profile $\frp\in\frP_\Pi$ with $\ell(\frp)\geq1$, we have
$$\Psi_\matsf(\frp)\leq -\frac{c_\matsf}{2\kappa}\,\ell(\frp)n\,.$$
\end{lemma}
\begin{proof}
Fix $T_B\in\cT^\root_{\Pi,\frp}$ and write $\ell=\ell(\frp)$. We count the number of residual graphs $R\in\cT^\res_{\Pi,\frp}(T_B)$ as follows. For each such $R$, choose a maximum matching $M_R$ of size $\ell$. The endpoints of $M_R$ form a vertex cover of $R$, and there are at most $n^{2\ell}$ choices for this cover. By definition of the residual graph, no cover vertex belongs to $B(\frp)$. This implies that for all $G\in\cF_{\Pi,\frp}$ such that $T_\res(G)=R$, all cover vertices $v$, and all $Y\in\cY_\Pi$, we have $d_G(v,Y)\leq4\alpha|Y|$. Hence there are at most $e^{C_kH(5\alpha)n}$ possibilities for the neighborhood of $v$ in $\bigcup\cY_\Pi$. By induced-$K_{1,k}$-freeness, the neighborhood of every vertex in $G$ has independence number at most $k-1$; this implies that every neighborhood in $S_\smsf(\Pi)$ is contained in the closed neighborhood of at most $k-1$ vertices. Since \Cref{lemma:SubCompareSparseSideEdgesK1k} gives $d_G(y,S_\smsf(\Pi))\leq C_k\theta n$ for all $y\in V$, it follows that there are at most $n^{k-1}2^{C_k\theta n}\leq e^{C_k\theta n+C_k\log n}$ possibilities for such a neighborhood. The factor $e^{C_ke(R)}$ in the definition of $\Psi_\matsf(\frp)$ is absorbed by the same estimate after increasing $C_k$ since $d_R(v)\leq C_k(\alpha+\theta)n$ for all cover vertices $v$. Combining these estimates, we obtain
\begin{equation}\label{eqn:residual-graph-count-K1k}
\sum_{R\in\cT^\res_{\Pi,\frp}(T_B)}e^{C_ke(R)}\leq e^{C_k\ell(H(5\alpha)+\theta)n+C_k\ell\log n}\,.
\end{equation}

Now fix $R\in\cT^\res_{\Pi,\frp}(T_B)$ and let $M$ be a maximum matching in $R$. Every edge of $M$ is incident with $V_\ast$. We partition $M$ into the following submatchings:
\[\arraycolsep=1cm
\begin{array}{ll}
M[P_{i,j}] & i\in[\ell_\ast]\,,\,j\in[q_i]\,,\\[10pt]
M[P_{i,j},P_{i,j'}] & i\in[\ell_\ast]\,,\,jj'\in E(Q_i^c)\,,\\[10pt]
M[P_{i,j},P_{i',j'}] & 1\leq i<i'\leq\ell_\ast\,,\, j\in[q_i]\,,\, j'\in[q_{i'}]\,,\\[10pt]
M[P_{i,j},S_\Pi] & i\in[\ell_\ast]\,,\,j\in[q_i]\,.
\end{array}\]
The total number of submatchings displayed above is at most
$$2\sum_{i\leq\ell_\ast}q_i+\sum_{i\leq\ell_\ast}\binom{q_i}{2}+\sum_{1\leq i<i'\leq\ell_\ast}q_iq_{i'}\leq \kappa\,,$$
by the definition of $\kappa$, since $\ell_\ast\leq2/\eta$ and $q_i\leq R_0$ for every $i\leq\ell_\ast$. Hence one of the displayed submatchings contains a matching $M_0$ of size at least $\ell/\kappa$. Let $M'\subseteq M_0$ be a submatching of size $\ceil*{\lambda |M_0|/8}$. We now bound, uniformly over every induced-$K_{1,k}$-free graph $H$ on $S_\Pi$, every $L$ with $(T_B,R,L)$ compatible, and every $\bsm\in\cM^\nar_\Pi$, the probability of the event $\cE_\matsf(R)$ in the binomial model $G^\bin_{\Pi,H,T,\bsm}$, where $T:=T_B\cup R\cup L$. We consider the four cases separately.

First suppose $M'\subseteq M[P_{i,a}]$. For every $xy\in M'$, define $\cK_{xy}$ to be the family of copies of $K_{1,k}$ whose center is a vertex $z\in P_{i,a}\setminus (B\cup V(M'))$, whose leaves include $x$ and $y$, and whose remaining leaves consist of one vertex in each part $P_{i,c}\setminus B$ with $c\in N_{Q_i}(a)$. We discard choices using an edge of $T$ other than $xy$ in one of the required adjacencies or non-adjacencies. Since $e(T)\leq\epsilon n^2$, all relevant retained parts have size at least $\lambda n$, and $|V(M')|\leq \lambda n/4+O(1)$, for every $xy\in M'$ we have $|\cK_{xy}|\geq c_1(\lambda n)^{k-1}$. Let $\cK:=\bigcup_{xy\in M'}\cK_{xy}$. For $K\in\cK$, let $A_K$ be the event that $K$ appears as an induced copy of $K_{1,k}$. Every random adjacency or non-adjacency required by $A_K$ is a pair in $P_{h,r}\times P_{h,s}$ for some $h\leq\ell_\ast$ and $rs\in E(Q_h)$, where the binomial edge probability lies in $[\rho/2,1-\rho/2]$, so $\bbP\{A_K\}\geq c_2$ and this implies
$$\mu:=\sum_{K\in\cK}\bbP\{A_K\}\geq c_3|M'|(\lambda n)^{k-1}\,.$$
Moreover, if $A_K$ and $A_{K'}$ are dependent then the two copies share a pair of vertices contained in $P_{h,r}\times P_{h,s}$ for some $h\leq\ell_\ast$ and $rs\in E(Q_h)$. For a fixed $K$, after choosing such a shared pair of vertices, there are at most $C_kn^{k-2}+C_k|M'|n^{k-3}\leq C_kn^{k-2}$ choices for $K'$. It follows that
$$\Delta:=\sum_{K\sim K'}\bbP\{A_K\cap A_{K'}\}\leq C_k|M'|n^{2k-3}\,.$$
It follows that $\mu^2/(4\Delta)\geq c_4|M'|n$ and $\mu/2\geq c_4|M'|n$, after decreasing $c_4>0$. Since the events $A_K$ are positively correlated by the choice of the families $\cK_{xy}$, \Cref{thm:JansonsInequality} gives
\begin{equation}\label{eqn:residual-matching-janson-bound-K1k}
\bbP_{G^\bin_{\Pi,H,T,\bsm}}\{\cE_\matsf(R)\}\leq e^{-c_4|M'|n}\,.
\end{equation}

Next suppose $M'\subseteq M[P_{i,a},P_{i,b}]$ with $ab\in E(Q_i^c)$, and let $X:=V(M')\cap P_{i,a}$. For every $xy\in M'$ with $x\in P_{i,a}$ and $y\in P_{i,b}$, define $\cK_{xy}$ to be the family of copies of $K_{1,k}$ centered at $x$, whose leaves include $y$, one vertex in $P_{i,a}\setminus (B\cup X)$, and one vertex in each part $P_{i,c}\setminus B$ with $c\in N_{Q_i}(a)$. Since $|X|\leq\lambda n/8+O(1)$, discarding choices using a further edge of $T$ leaves at least $c_1(\lambda n)^{k-1}$ choices for each $xy\in M'$. With $\cK:=\bigcup_{xy\in M'}\cK_{xy}$ and $A_K$ defined as above, the same estimates give $\mu\geq c_3|M'|(\lambda n)^{k-1}$ and $\Delta\leq C_k|M'|n^{2k-3}$.
Therefore the same application of Janson's inequality gives \eqref{eqn:residual-matching-janson-bound-K1k}.

Now suppose $M'\subseteq M[P_{i,a},P_{i',b}]$ with $i<i'$, and let $X:=V(M')\cap P_{i,a}$. For every $xy\in M'$ with $x\in P_{i,a}$ and $y\in P_{i',b}$, define $\cK_{xy}$ to be the family of copies of $K_{1,k}$ centered at $x$, whose leaves include $y$, one vertex in $P_{i,a}\setminus (B\cup X)$, and one vertex in each part $P_{i,c}\setminus B$ with $c\in N_{Q_i}(a)$. As above, after discarding choices using an additional edge of $T$, for each $xy\in M'$ there are at least $c_1(\lambda n)^{k-1}$ choices. Hence, with $\cK:=\bigcup_{xy\in M'}\cK_{xy}$, we have $\mu\geq c_3|M'|(\lambda n)^{k-1}$ and $\Delta\leq C_k|M'|n^{2k-3}$.
The same application of Janson's inequality gives \eqref{eqn:residual-matching-janson-bound-K1k}.

Finally suppose $M'\subseteq M[P_{i,a},S_\Pi]$, and let $X:=V(M')\cap P_{i,a}$. For every $xy\in M'$ with $x\in P_{i,a}$ and $y\in S_\Pi$, define $\cK_{xy}$ to be the family of copies of $K_{1,k}$ centered at $x$, whose leaves include $y$, one vertex in $P_{i,a}\setminus (B\cup X)$, and one vertex in each part $P_{i,c}\setminus B$ with $c\in N_{Q_i}(a)$. Since $y\notin B(\frp)$, its adjacency into every $\theta$-visible target is less than $4\alpha$ of that target, and \Cref{lemma:SubCompareSparseSideEdgesK1k} controls its adjacency into $S_\smsf(\Pi)$. Thus deleting choices which create an extra adjacency or non-adjacency involving $y$, and choices using another edge of $T$, removes only an $O_k(\alpha+\theta+o(1))$ proportion of the possibilities. Hence $|\cK_{xy}|\geq c_1(\lambda n)^{k-1}$ for every $xy\in M'$. With $\cK:=\bigcup_{xy\in M'}\cK_{xy}$, we again have $\mu\geq c_3|M'|(\lambda n)^{k-1}$ and $\Delta\leq C_k|M'|n^{2k-3}$.
The same application of Janson's inequality gives \eqref{eqn:residual-matching-janson-bound-K1k}.

The preceding four cases show that \eqref{eqn:residual-matching-janson-bound-K1k} holds uniformly over every induced-$K_{1,k}$-free graph $H$ on $S_\Pi$, every $L$ with $(T_B,R,L)$ compatible, and every $\bsm\in\cM^\nar_\Pi$. Since $|M'|\geq\lambda\ell/(8\kappa)$, after decreasing $c_\matsf=c_\matsf(k,\gamma,\eta,R_0)>0$, its right-hand side is at most $\exp(-c_\matsf\ell n/\kappa)$. Combining this with \eqref{eqn:residual-graph-count-K1k}, and using \eqref{eqn:sub-matching-parameter-condition-K1k}, gives
\begin{align*}
\Psi_\matsf(\frp)&\leq C_k\ell\bigl(H(5\alpha)+\theta\bigr)n+C_k\ell\log n-\frac{c_\matsf}{\kappa}\ell n \leq -\frac{c_\matsf}{2\kappa}\ell n
\end{align*}
for all sufficiently large $n$.
\end{proof}

\subsection{Clean retained comparisons}\label{subsec:clean-retained}
From here until the end of the section, we assume $\omega>0$ is sufficiently small but fixed such that all lemmas in the previous subsection hold. This also fixes all other parameters from the previous subsection as constants, including $\eta$, $R_0$, $\alpha$, $\delta$, $\theta$, $\epsilon$, and $\tau$.

\begin{lemma}\label{lemma:NtaunmWUpperBdK1k}
For all $W\in\cV_\gamma$, there exists $c=c(k,\gamma,W)>0$ such that for all $\Pi\in\sD_W$,
$$|\cF_{W,\Pi}|\leq (1+e^{-cn})\cZ_\Pi\,.$$
Consequently, for a constant $C=C(k,\gamma,W)$, we have
$$|\cB_{n,m}(W,\tau)|\leq e^{Cn}\sum_{\Pi\in\sD_W}\cZ_\Pi\,.$$
\end{lemma}
\begin{proof}
Define
$$\cF^\clean_{W,\Pi}:=\{G\in\cF_{W,\Pi}:T(G)[V_\ast,V]=\emptyset\}$$
and set $\cF'_{W,\Pi}:=\cF_{W,\Pi}\setminus\cF^\clean_{W,\Pi}$. By \Cref{lemma:profile-covering-K1k}, every graph in $\cF_{W,\Pi}$ belongs to $\cF_{\Pi,\frp}$ for some $\frp\in\frP_\Pi$. If $\frp\in\frP_\Pi$ satisfies $B(\frp)=\emptyset$ and $\ell(\frp)=0$ then $\cF_{\Pi,\frp}\subseteq\cF^\clean_{W,\Pi}$, and conversely every graph in $\cF^\clean_{W,\Pi}$ belongs to $\cF_{\Pi,\frp}$ for such a profile. Since \Cref{lemma:SubCompareSparseSideEdgesK1k} gives $e_G(S_\Pi)\leq C_k\eta n^2$, we have $|\cF^\clean_{W,\Pi}|\leq \cZ_\Pi$.

Now fix $\frp\in\frP_\Pi$ such that $B(\frp)\neq\emptyset$ or $\ell(\frp)\geq1$. By \Cref{lemma:profile-bound-K1k,lemma:local-compensation-K1k,lemma:residual-matching-estimate-K1k}, and by the definition of $\Err(\frp)$, we have
\begin{align}
|\cF_{\Pi,\frp}| &\leq n^{C|\cI_\Pi|}\cZ_\Pi\exp\!\left(-c n|B(\frp)|-\frac{c_\matsf}{2\kappa}\ell(\frp)n+\Err(\frp)\right) \nonumber\\
&\leq \cZ_\Pi\exp\!\left(-c n|B(\frp)|-\frac{c_\matsf}{2\kappa}\ell(\frp)n\right)\label{eqn:fpifrp-mcn-bd}
\end{align}
after decreasing $c>0$ on the second line. Here we used $\Err(\frp)\leq c n|B(\frp)|/2$, while $n^{C|\cI_\Pi|}$ is a polynomial factor because $\Pi$ is fixed.

The number of profiles with $|B(\frp)|+\ell(\frp)=r$ is at most $n^{C_\Pi r}$, since one chooses the root set, the row labels, the row sizes, the tail symbols, and the matching size. Summing \eqref{eqn:fpifrp-mcn-bd} over all nontrivial profiles gives at most $e^{-cn}\cZ_\Pi$ after decreasing $c>0$ once more. Combining this with $|\cF^\clean_{W,\Pi}|\leq \cZ_\Pi$ proves $|\cF_{W,\Pi}|\leq (1+e^{-cn})\cZ_\Pi$.

By definition of $\cF_{W,\Pi}$ and \Cref{lemma:WtoWtildeMetricsK1k}, we have $\cB_{n,m}(W,\tau)\subseteq\bigcup_{\Pi\in\sD_W}\cF_{W,\Pi}$. Summing the first bound over $\Pi\in\sD_W$ gives the second assertion after increasing $C$.
\end{proof}

\begin{lemma}\label{lemma:canonical-clean-lower-K1k}
Let $\omega'>0$ be any cut radius to which \Cref{lemma:WtoWtildeMetricsK1k} applies. Then
$$\sum_{\Pi\in\sD_{W^\ast}}\cZ_\Pi\leq e^{o(n)}|\cB_{n,m}(W^\ast,\omega')|\,.$$
\end{lemma}
\begin{proof}
For each $\Pi\in\sD_{W^\ast}$ and $0\leq b\leq C_k\eta n^2$, a graph counted by $Z_\Pi(b)N^\ast_{s(\Pi),b}(K_{1,k})$ is obtained by choosing a graph in $\cF^\ast_{S_\Pi,b}(K_{1,k})$ and then choosing the active edges in the retained pairs according to an edge-count vector in $\cM_{\Pi,b}$. Using standard cut-concentration of random bipartite graphs of constant density, and using \Cref{lemma:WtoWtildeMetricsK1k}, all but a $o(1)$ fraction of these choices are within $\omega'$ in cut metric of $W^\ast$. To account for multiplicity, it suffices to note that almost every graph counted by the sum $\sum_{\Pi\in\sD_{W^\ast}}\cZ_\Pi$ has a unique division $\Pi_\ast\in\sD_{W^\ast}$ for which it is counted by $\cZ_{\Pi_\ast}$, which is routine to show using the random graph $G_{\Pi_\ast,H,\emptyset,\bsm}$.
\end{proof}

\subsubsection{The comparison with $W^\ast$}\label{subsec:subcritical-comparison-K1k}
For $0\leq b\leq C_k\eta n^2$, set
\begin{equation}\label{eqn:sub-gamma-b-K1k}
\gamma_b:=\frac{m-b}{\binom n2}\,,\qquad \mu_b:=\left(\frac{\gamma_b(k-1)}{1+(k-2)p_k}\right)^{1/2}\,.
\end{equation}
Let $\sD_\ast(b)$ be the family of divisions $\Sigma$ for which $\Sigma_\ast$ consists of one component
$$\Sigma_1=(K_{k-1},\{P_1,\dots,P_{k-1}\})\,,$$
where $|P_1|+\cdots+|P_{k-1}|=\lfloor\mu_bn\rfloor+O(1)$ and the parts are as equal as possible, and for which $Z_\Sigma(b)\neq0$.

\begin{lemma}\label{lemma:sub-retained-mass-gap-K1k}
There exists $c_0=c_0(k,\gamma)>0$ such that, if $W\in\cV_\gamma$ and $\delta_\square(W,W^\ast)\geq8\omega$ then every $\Pi\in\sD_W$ satisfies
$$s(\Pi)\leq(1-\mu-c_0\omega)n\,.$$
Moreover, if $0\leq b\leq C_k\eta n^2$ and $\Sigma\in\sD_\ast(b)$ then $s(\Sigma)\geq(1-\mu-C\eta)n$.
\end{lemma}
\begin{proof}
Write $W=W_\blambda$, where $\blambda=((\lambda_i,G_i))_{i\geq1}$, and set $\mu_i:=\lambda_i-\lambda_{i-1}$ and $r_i:=v(G_i)$. Define
$$I_\ret(W):=\{i:r_i\leq R_0\text{ and }\mu_i\geq2\eta\}\,.$$
We claim that
\begin{equation}\label{eqn:sub-graphon-retained-mass-gap-K1k}
\sum_{i\in I_\ret(W)}\mu_i\geq\mu+2c_0\omega\,.
\end{equation}
If \eqref{eqn:sub-graphon-retained-mass-gap-K1k} were not to hold, then using $W\in\cV_\gamma$, we have $\sum_{i\geq1}\frac{\mu_i^2}{r_i}=\frac{\mu^2}{k-1}$. The indices outside $I_\ret(W)$ contribute at most $C\eta+C/R_0$ to the left-hand side. By the choice of $\eta$ and $R_0$, this is smaller than $c_0\omega/100$. Since $r_i\geq k-1$, this forces one index $j\in I_\ret(W)$ to satisfy
$$|\mu_j-\mu|\leq Cc_0\omega+C\eta+C/R_0\,,$$
while the total mass of all other indices is at most $Cc_0\omega+C\eta+C/R_0$. Returning to the preceding identity, we must have $r_j=k-1$, and hence $G_j\cong K_{k-1}$. Aligning this block with the block of $W^\ast$ gives $\delta_\square(W,W^\ast)<8\omega$, provided $c_0$ is sufficiently small and $\eta,R_0^{-1}$ are sufficiently small relative to $\omega$, yielding a contradiction, and proving \eqref{eqn:sub-graphon-retained-mass-gap-K1k}.

Let $\varphi:I\to L$ witness that $\Pi$ is $W$-compatible. If $a\in I_\ret(W)$ then $a\in\varphi(I)$, since otherwise compatibility would give either $\mu_a<\eta$ or $r_a>R_0$. For the unique $i\in I$ with $\varphi(i)=a$, we have $|V_i|\geq(\mu_a-\delta)n\geq\eta n$, hence $i\leq\ell_\ast$. Since $|I_\ret(W)|\leq1/\eta$, \eqref{eqn:sub-graphon-retained-mass-gap-K1k} gives $|V_\ast|\geq(\mu+c_0\omega)n$ after decreasing $c_0$, proving the first assertion. The second assertion follows from $|\mu_b-\mu|\leq C\eta$ and $|V(\Sigma_1)|=\mu_bn+O(1)$.
\end{proof}

\begin{lemma}\label{lemma:clean-retained-comparison-K1k}
For all $W\in\cV_\gamma$ and $0\leq b\leq C_k\eta n^2$, we have
$$\sum_{\Pi\in\sD_W}Z_\Pi(b)\leq n^{o(n)}\sum_{\Sigma\in\sD_\ast(b)}Z_\Sigma(b)\,.$$
\end{lemma}
\begin{proof}
Fix $\Pi\in\sD_W$ and $\bsm\in\cM_{\Pi,b}$. The associated step graphon, equal to $1$ on the clique squares of retained parts, $\bsm_e/N_e$ on retained active rectangles, and $0$ elsewhere, has $K_{1,k}$ density zero and has edge density $\gamma_b+O(1/n)$. Hence, by \Cref{prop:graphon-char-fixed-gamma}, we have $\log\binom{\Pi}{\bsm}\leq \binom{n}{2}\sE_k(\gamma_b)+O(n)$. There are only at most $n^{O_k(R_0/\eta)}$ vectors $\bsm$ for each $\Pi$, and the number of possible choices of the components $\Pi_i$ with $i\in[\ell_\ast]$ and their parts is at most $e^{O(n)}$, hence we have
$$\log\sum_{\Pi\in\sD_W}Z_\Pi(b)\leq \binom{n}{2}\sE_k(\gamma_b)+O(n)\,.$$
On the other hand, for every $\Sigma\in\sD_\ast(b)$, Stirling's formula applied to an edge-count vector with density $p_k+O(1/n)$ on every pair indexed by $\cI_\Sigma$ gives $\log Z_\Sigma(b)\geq \binom{n}{2}\sE_k(\gamma_b)-O(n)$.
The family $\sD_\ast(b)$ is nonempty for every $0\leq b\leq C_k\eta n^2$, after adjusting the $O(1)$ term in the size of the retained component. Combining the two estimates proves the lemma.
\end{proof}

\begin{lemma}\label{lemma:sub-combined}
There exists $C=C(k,\gamma)>1$ such that the following holds. Let $\Pi$ be $W^\ast$-compatible, and let $B_\eta:=C_k\eta n^2+C_kn$. If $0\leq b\leq B_\eta$ then the following hold.
\begin{enumerate}[label=\textit{(\roman*)}, ref=(\textit{\roman*}), topsep=5pt]
  \item\label{item:sub-active-shift-K1k} If $0\leq t\leq B_\eta$ then $Z_\Pi(b)\leq C^{|t-b|+1}Z_\Pi(t)$.
  \item\label{item:sub-reference-ball-K1k} We have $Z_\Pi(b)N^\ast_{s(\Pi),b}(K_{1,k})\leq2|\cB_{n,m}(W^\ast,\omega)|$ for all sufficiently large $n$.
\end{enumerate}
\end{lemma}
\begin{proof}
Since $\Pi$ is $W^\ast$-compatible, the retained division consists of one component $\Pi_1=(Q_1,\cP_1)$, where $Q_1\cong K_{k-1}$, $|V_\ast|\in(\mu\pm C_k\eta)n$, and every part of $\cP_1=\{P_1,\dots,P_{k-1}\}$ has size $\mu n/(k-1)+O_k(\eta n)$. Therefore every $N_e$ with $e\in\cI_\Pi$ is $\Theta_{k,\gamma}(n^2)$, $A_\Pi:=\sum_{e\in\cI_\Pi}N_e=\Theta_{k,\gamma}(n^2)$, and $C_\Pi+p_kA_\Pi=m+O_k(\eta n^2+n)$. Thus, for every $0\leq u\leq B_\eta$, the integer $S_u:=m-u-C_\Pi$ satisfies $S_u=p_kA_\Pi+O_k(\eta n^2+n)$. By the choice of parameters, $S_u/A_\Pi\in p_k\pm\delta/4$ for all sufficiently large $n$.

We first prove \ref{item:sub-active-shift-K1k}. It is enough to compare adjacent values of $u$. Fix $0\leq u<u+1\leq B_\eta$. For each $\bsm\in\cM_{\Pi,u}$, choose the first $e\in\cI_\Pi$ such that $\bsm_e>(p_k-\delta/2)N_e$, and define $\bsm':=\bsm-\mathbf 1_e$. Such an index exists because $\sum_{e\in\cI_\Pi}\bsm_e=S_u$ and $S_u/A_\Pi>p_k-\delta/4$. It also holds that $\bsm'\in\cM_{\Pi,u+1}$. Since all densities in $\cM_{\Pi,u}$ lie in $p_k\pm2\delta$, we have
$$\binom{N_e}{\bsm_e}\binom{N_e}{\bsm_e-1}^{-1}=\frac{N_e-\bsm_e+1}{\bsm_e}\leq C_0$$
for some $C_0=C_0(k,\gamma)>0$. Each vector in $\cM_{\Pi,u+1}$ has at most $|\cI_\Pi|$ preimages, so $Z_\Pi(u)\leq C_0|\cI_\Pi|Z_\Pi(u+1)$. The reverse inequality $Z_\Pi(u+1)\leq C_0|\cI_\Pi|Z_\Pi(u)$ follows in the same way, adding $1$ to the first coordinate $e$ with $\bsm_e<(p_k+\delta/2)N_e$. Iterating the adjacent comparisons proves \ref{item:sub-active-shift-K1k} after increasing $C$.

We now prove \ref{item:sub-reference-ball-K1k}. Fix $H\in\cF^\ast_{S_\Pi,b}(K_{1,k})$ and $\bsm\in\cM_{\Pi,b}$. Choose exactly $\bsm_e$ edges uniformly from each pair $P_i\times P_j$ with $ij\in\cI_\Pi$, independently over all active pairs; make each $P_i$ a clique; put no edges between $S_\Pi$ and $V_\ast$; and put $H$ on $S_\Pi$. Every graph obtained in this way has $m$ edges and is induced-$K_{1,k}$-free, since an independent set in $V_\ast$ meets each clique $P_i$ in at most one vertex and there are no edges between $S_\Pi$ and $V_\ast$.

For every active pair $P_i\times P_j$, the random bipartite graph just defined has edge density in $p_k\pm2\delta$ by definition, and Hoeffding's inequality for hypergeometric random variables, followed by a union bound over all subsets of $P_i$ and $P_j$, gives cut distance $o(1)$ from the constant graphon of density $\bsm_{ij}/N_{ij}$ on this pair with probability $1-o(1)$, uniformly over $H$ and $\bsm$. Since $\Pi$ is $W^\ast$-compatible, the parts $P_1,\dots,P_{k-1}$ have sizes $\mu n/(k-1)+O_k(\eta n)$, and we have $|V(\Pi_1)|=(\mu+O_k(\eta))n$. Moreover, the graph $H$ has at most $B_\eta$ edges. Therefore, after aligning $P_1,\dots,P_{k-1}$ with the positive block of $W^\ast$, the resulting graph lies in $\cB_{n,m}(W^\ast,\omega)$ with probability $1-o(1)$, by the choice of $\eta$ and $\delta$.

Thus at least half of the graphs counted by $Z_\Pi(b)N^\ast_{s(\Pi),b}(K_{1,k})$ belong to $\cB_{n,m}(W^\ast,\omega)$ for all sufficiently large $n$. The choices of $H$ and of the active edges determine the graph uniquely, and hence $Z_\Pi(b)N^\ast_{s(\Pi),b}(K_{1,k})\leq2|\cB_{n,m}(W^\ast,\omega)|$.
\end{proof}

\begin{lemma}\label{lemma:sub-sparse-matching-transfer-K1k}
Let $M_q:=(2q)!/(2^qq!)$. If $s'\geq s+2q$ then for every $b$,
$$M_qN^\ast_{s,b}(K_{1,k})\leq N^\ast_{s',b+q}(K_{1,k})\,.$$
\end{lemma}
\begin{proof}
Given an induced-$K_{1,k}$-free graph on a fixed set of $s$ vertices with $b$ edges, choose a matching of size $q$ on a fixed set of $2q$ new vertices and leave all other new vertices isolated. The resulting graph is induced-$K_{1,k}$-free and has $b+q$ edges, and there are $M_q$ choices for the matching, proving the lemma.
\end{proof}

\begin{lemma}\label{lemma:CompareBallsInCutMetricNlogNK1k}
There exists $\nu=\nu(k,\gamma,\omega)>0$ such that, if $W\in\cV_\gamma$ and $\delta_\square(W,W^\ast)\geq8\omega$ then
$$|\cB_{n,m}(W,\tau)|\leq n^{-\nu n}|\cB_{n,m}(W^\ast,\omega)|\,.$$
\end{lemma}
\begin{proof}
Let $c_0>0$ be the constant in \Cref{lemma:sub-retained-mass-gap-K1k}, and define
$$s_0:=\left\lfloor(1-\mu-c_0\omega/2)n\right\rfloor\,,\qquad q:=\left\lfloor\frac{c_0\omega n}{8}\right\rfloor\,.$$
By \Cref{lemma:sub-retained-mass-gap-K1k}, if $\Pi\in\sD_W$ then $s(\Pi)\leq s_0$, while every $\Sigma\in\sD_\ast(b)$ satisfies $s(\Sigma)\geq s_0+2q$ for all sufficiently large $n$. Since an induced-$K_{1,k}$-free graph on $s(\Pi)$ vertices with $b$ edges can be extended injectively to one on $s_0$ vertices with $b$ edges by adding $s_0-s(\Pi)$ isolated vertices, \Cref{lemma:NtaunmWUpperBdK1k,lemma:clean-retained-comparison-K1k} gives
\begin{align}\label{eqn:sub-comparison-start-K1k}
|\cB_{n,m}(W,\tau)|&\leq n^{o(n)}\sum_{0\leq b\leq C_k\eta n^2}N^\ast_{s_0,b}(K_{1,k})\sum_{\Sigma\in\sD_\ast(b)}Z_\Sigma(b)\,.
\end{align}
For each term in \eqref{eqn:sub-comparison-start-K1k}, \Cref{lemma:sub-sparse-matching-transfer-K1k} and \Cref{lemma:sub-combined} \ref{item:sub-active-shift-K1k} give
$$Z_\Sigma(b)N^\ast_{s_0,b}(K_{1,k})\leq \frac{C^{q+1}}{M_q}Z_\Sigma(b+q)N^\ast_{s(\Sigma),b+q}(K_{1,k})\,.$$
By \Cref{lemma:sub-combined} \ref{item:sub-reference-ball-K1k}, the last product is at most $2|\cB_{n,m}(W^\ast,\omega)|$. The number of possible values of $b$ is $O(n^2)$, and $|\sD_\ast(b)|=e^{O(n)}$. Since $M_q=e^{q\log n+O(n)}$, $C^q=e^{O(n)}$, and $q=\Theta(n)$, we have $C^{q+1}n^{o(n)}/M_q\leq n^{-\nu n}$ after decreasing $\nu>0$, proving the lemma.
\end{proof}

\begin{lemma}\label{lemma:sub-Wstar-sparse-lower-tail-K1k}
There exists $\nu>0$ such that, with $q:=\left\lfloor\omega n/100\right\rfloor$, we have
$$\sum_{\Pi\in\sD_{W^\ast}}\sum_{0\leq b<q}Z_\Pi(b)N^\ast_{s(\Pi),b}(K_{1,k})\leq n^{-\nu n}|\cB_{n,m}(W^\ast,\tau+8\omega)|\,.$$
\end{lemma}
\begin{proof}
Fix $\Pi\in\sD_{W^\ast}$. By \Cref{lemma:WtoWtildeMetricsK1k}, the retained part of $\Pi$ consists of a single balanced $K_{k-1}$-component on $(\mu\pm C\eta)n$ vertices. In particular, $s(\Pi)\geq4q$ for all sufficiently large $n$. If $0\leq b<q$ then every induced-$K_{1,k}$-free graph on $S_\Pi$ with $b$ edges has at least $s(\Pi)-2q$ isolated vertices. For $s\in\N$, define $M_q(s):=s!/(2^qq!(s-2q)!)$. It follows that
$$M_q(s(\Pi)-2q)N^\ast_{s(\Pi),b}(K_{1,k})\leq \binom{b+q}{q}N^\ast_{s(\Pi),b+q}(K_{1,k})\,.$$
Indeed, from a graph with $b$ edges we add a matching of size $q$ on isolated vertices. Conversely, a graph with $b+q$ edges has at most $\binom{b+q}{q}$ possible choices for the added matching edges. Since $b<q$, the binomial factor is $e^{O(n)}$, while $M_q(s(\Pi)-2q)=e^{q\log n+O(n)}$. By \Cref{lemma:sub-combined} \ref{item:sub-active-shift-K1k},
we have $Z_\Pi(b)\leq C^{q+1}Z_\Pi(b+q)$, so we calculate that
$$\sum_{0\leq b<q}Z_\Pi(b)N^\ast_{s(\Pi),b}(K_{1,k})\leq n^{-q+o(n)}\sum_{q\leq u<2q}Z_\Pi(u)N^\ast_{s(\Pi),u}(K_{1,k})\,.$$
Summing over $\Pi\in\sD_{W^\ast}$ contributes only at most $e^{O(n)}$ divisions. By \Cref{lemma:sub-combined} \ref{item:sub-reference-ball-K1k} and monotonicity in the radius, every product in the last sum is at most $2|\cB_{n,m}(W^\ast,\tau+8\omega)|$. Since $q=\Theta(n)$, the lemma follows.
\end{proof}

\begin{proof}[Proof of \Cref{thm:main-almostall} \ref{item:thm:main-almostall-sub}]
Let $\cN=\{W^\ast,W_1,\dots,W_\ell\}\subseteq\cX^\ast_\gamma$ be a $\tau$-net of $\cX^\ast_\gamma$ in cut metric. Define
$$\arraycolsep=1.4pt
\begin{array}{ll}
\cF_\close &:= \displaystyle \cB_{n,m}(W^\ast,\tau)\cup\bigcup_{i=1}^\ell \cB_{n,m}(W_i,\tau)\,, \\[18pt]
\cF_\far &:= \cF^\ast_{n,m}(K_{1,k})\setminus\cF_\close\,.
\end{array}$$
Let $I:=\{i\in[\ell]:\delta_\square(W_i,W^\ast)\geq8\omega\}$ and $\tau':=\tau+8\omega$. By \Cref{lemma:rough-struc-fixed-g}, there exists $C_0=C_0(k,\gamma,\tau)>0$ such that
$$|\cF_\far|\leq e^{-C_0n^2}|\cF^\ast_{n,m}(K_{1,k})|\,.$$
Additionally, we have
$$|\cF_\close|\leq|\cB_{n,m}(W^\ast,\tau')|+\sum_{i\in I}|\cB_{n,m}(W_i,\tau)|\,.$$
By \Cref{lemma:CompareBallsInCutMetricNlogNK1k}, we have
$$\sum_{i\in I}|\cB_{n,m}(W_i,\tau)|\leq \ell n^{-\nu n}|\cB_{n,m}(W^\ast,\omega)|\leq o(|\cB_{n,m}(W^\ast,\tau')|)\,.$$
Since $\cB_{n,m}(W^\ast,\tau')\subseteq \cF^\ast_{n,m}(K_{1,k})$, the preceding displays imply
\begin{equation}\label{eqn:subcritical-almostall-cutclose-K1k}
|\cF^\ast_{n,m}(K_{1,k})|\leq(1+o(1))|\cB_{n,m}(W^\ast,\tau')|\,.
\end{equation}

It remains to identify the typical structure inside $\cB_{n,m}(W^\ast,\tau')$. By shrinking parameters, the estimates preceding this subsubsection apply with $\tau'$ in place of $\tau$. Summing over $\Pi\in\sD_{W^\ast}$ and using \Cref{lemma:NtaunmWUpperBdK1k,lemma:canonical-clean-lower-K1k}, we get
$$\sum_{\Pi\in\sD_{W^\ast}}|\cF'_{W^\ast,\Pi}|=o(|\cB_{n,m}(W^\ast,\tau')|)\,.$$
Therefore, by \eqref{eqn:subcritical-almostall-cutclose-K1k}, almost every graph in $\cF^\ast_{n,m}(K_{1,k})$ lies in $\cF^\clean_{W^\ast,\Pi}$ for some $\Pi\in\sD_{W^\ast}$.

Fix $G\in\cF^\clean_{W^\ast,\Sigma}$ for some $\Sigma\in\sD_{W^\ast}$, and write
$$\Sigma=\{\Sigma_1,\dots,\Sigma_r\},\qquad \Sigma_i=(Q_i,\{P_{i,1},\dots,P_{i,q_i}\})\,.$$
By \Cref{lemma:WtoWtildeMetricsK1k}, the division $\Sigma$ has a unique retained component, say $\Sigma_1$, with $|V(\Sigma_1)|=(\mu\pm C\eta)n$ and $Q_1\cong K_{k-1}$. By definition of $\cF^\clean_{W^\ast,\Sigma}$, the induced subgraph $G[V(\Sigma_1)]$ is co-$(k-1)$-partite and there are no edges between $V(\Sigma_1)$ and $V\setminus V(\Sigma_1)$. It follows that $G=G_1\sqcup G_2$, where $G_1:=G[V(\Sigma_1)]$ is co-$(k-1)$-partite on $(\mu\pm C\eta)n$ vertices and $G_2:=G[V\setminus V(\Sigma_1)]$.

We now prove that $e(G_2)=o(n^2)$. Fix $\xi>0$, and run the argument above with $\eta>0$ chosen such that $C_k\eta<\xi/2$. For every graph in the resulting clean classes, we have $e(G_2)=e_G(S_\Sigma)\leq C_k\eta n^2+o(n^2)\leq \xi n^2$ for all sufficiently large $n$. Since $\xi>0$ was arbitrary, almost every graph satisfies $e(G_2)=o(n^2)$. Finally, \Cref{lemma:sub-Wstar-sparse-lower-tail-K1k} shows that the clean graphs in the $W^\ast$-ball with fewer than $\omega n/100$ edges in $G[S_\Pi]$ form an $o(1)$ fraction of $\cB_{n,m}(W^\ast,\tau')$. Hence, after decreasing the implicit constant,
$e(G_2)\geq \omega n/200$ for almost every $G\in\cF^\ast_{n,m}(K_{1,k})$, proving \Cref{thm:main-almostall} \ref{item:thm:main-almostall-sub}.
\end{proof}

\section{The Conditional \erdosrenyi{} Random Graph}
\label{sec:erdosrenyi}
In this section we prove \Cref{thm:gnp-typ-struc}. The proofs in this section generalize those in \cite{perkins2025typical}. Let $W_0\equiv0$ denote the all-zero graphon, and $W_1\equiv1$ the all-one graphon; also write $\cV_0:=\{W_0\}$ and $\cV_1:=\{W_1\}$. In this section we use notation from \Cref{sec:graphon-prob} relating to the optimal graphons in $\cV_\gamma$. If $\cF$ is a set of graphs on $V$ then define
$$Z_p(\cF):=(1-p)^{\binom n2}\sum_{G\in\cF}\left(\frac{p}{1-p}\right)^{e(G)}\,.$$
For all graphons $W\in\cW$ and all $r>0$, define
$$\cB_n(W,r):=\{G\in\cF^\ast_n(K_{1,k}):\delta_\square(G,W)<r\}\,.$$
Also let $\cC^{k-1}_n$ denote the set of co-$(k-1)$-partite graphs on $V$. The following lemma is the weighted analogue of the final comparison in \Cref{subsec:subcritical-comparison-K1k}.

\begin{lemma}\label{lemma:critical-gnp-comparison-K1k}
For every sufficiently small $\omega>0$, there exist constants $\tau,\nu>0$ such that the following holds. If $W\in\bigcup_{\gamma\in[0,\gamma_k]}\cV_\gamma$ and $\delta_\square(W,W_0)\geq8\omega$ then
$$Z_{p_k}(\cB_n(W,\tau))\leq n^{-\nu n}Z_{p_k}(\cB_n(W_0,\omega))\,.$$
\end{lemma}
\begin{proof}
Let $\Delta:=k-2$ and $p:=p_k$. Since $W$ is nonnegative, we have
$$t(K_2,W)=\delta_\square(W,W_0)\geq8\omega\,.$$
Write $\gamma:=t(K_2,W)$. If $\gamma=\gamma_k$ then $W^\ast_{\gamma_k}$ is the graphon $W_\blambda$ with $\blambda=((1,K_{k-1}))$; otherwise write $W=W_\blambda$ with $\blambda=((\lambda_i,G_i))_{i\geq1}$. Let $\mu_i:=\lambda_i-\lambda_{i-1}$ and $r_i:=v(G_i)$. Choose $\eta>0$ sufficiently small and $R_0>0$ sufficiently large so that $2\eta+R_0^{-1}=o(\omega)$, and define
$$I_\ret(W):=\{i:\mu_i\geq2\eta\text{ and }r_i\leq R_0\}\,.$$
Using the identity defining $\cV_\gamma$, we have
$$\sum_{i\geq1}\frac{\mu_i^2}{r_i}=\frac{\gamma}{1+\Delta p_k}\geq\frac{8\omega}{1+\Delta p_k}\,.$$
The indices outside $I_\ret(W)$ contribute at most $2\eta+R_0^{-1}$ to the left-hand side. Hence, after decreasing a constant $c_0=c_0(k,\omega)>0$, we have
$$\sum_{i\in I_\ret(W)}\mu_i\geq\sum_{i\in I_\ret(W)}\frac{\mu_i^2}{r_i}\geq c_0\,.$$
Choose the remaining parameters $\theta,\alpha,\delta,\epsilon,\tau$ in the same order as in \Cref{sec:subcritical}, with $\eta$ small enough that $C_k\eta<\omega/10$. By $W$-compatibility, every division $\Pi\in\sD_W$ satisfies $|V_\ast|\geq c_0n$ after decreasing $c_0$.

Using the notation from \Cref{subsec:counting}, define the partition function
\begin{equation*}
\widehat{\cZ}_\Pi:=(1-p)^{\binom n2}\sum_{0\leq b\leq C_k\eta n^2}N^\ast_{s(\Pi),b}(K_{1,k})\left(\frac{p}{1-p}\right)^{C_\Pi+b}\prod_{e\in\cI_\Pi}\sum_{\substack{0\leq a\leq N_e\\ a/N_e\in p\pm2\delta}}\binom{N_e}{a}\left(\frac{p}{1-p}\right)^a\,.
\end{equation*}
The proof of \Cref{lemma:NtaunmWUpperBdK1k}, with the summation over $m$ left inside the $p$-weighted partition function, gives
\begin{equation}\label{eqn:critical-weighted-upper-K1k}
Z_p(\cB_n(W,\tau))\leq e^{Cn}\sum_{\Pi\in\sD_W}\widehat{\cZ}_\Pi
\end{equation}
for a constant $C=C(k,\omega)>0$. Indeed, the profile decomposition and all defect penalties are unchanged; only the clean class is counted with the weight in $\widehat{\cZ}_\Pi$ rather than with a fixed edge count.

We next bound the retained contribution in $\widehat{\cZ}_\Pi$. Define the quantity
$$A_\Pi:=\sum_{e\in\cI_\Pi}N_e\,.$$
Since $p_k=(1-p_k)^{k-1}$, we have $p/(1-p)=(1-p)^\Delta$. We also have
$$\prod_{e\in\cI_\Pi}\sum_{\substack{0\leq a\leq N_e\\ a/N_e\in p\pm2\delta}}\binom{N_e}{a}\left(\frac{p}{1-p}\right)^a\leq(1-p)^{-A_\Pi}\,.$$
For a retained component $\Pi_i=(Q_i,\{P_{i,1},\dots,P_{i,q_i}\})$, the adjacency matrix of $Q_i$ has largest eigenvalue $\Delta$, so
$$\sum_{uv\in E(Q_i)}|P_{i,u}||P_{i,v}|\leq\frac{\Delta}{2}\sum_{u=1}^{q_i}|P_{i,u}|^2\,.$$
Summing over retained components gives $A_\Pi\leq\Delta C_\Pi+O(n)$, hence we calculate that
$$\left(\frac{p}{1-p}\right)^{C_\Pi}\prod_{e\in\cI_\Pi}\sum_{\substack{0\leq a\leq N_e\\ a/N_e\in p\pm2\delta}}\binom{N_e}{a}\left(\frac{p}{1-p}\right)^a\leq(1-p)^{\Delta C_\Pi-A_\Pi}\leq e^{Cn}\,.$$
It follows that
\begin{equation}\label{eqn:critical-clean-sparse-side-K1k}
\widehat{\cZ}_\Pi\leq e^{Cn}(1-p)^{\binom n2}\sum_{0\leq b\leq C_k\eta n^2}N^\ast_{s(\Pi),b}(K_{1,k})\left(\frac{p}{1-p}\right)^b\,.
\end{equation}

Let $q:=\lfloor c_0n/8\rfloor$ and $M_q:=(2q)!/(2^qq!)$. Since $|V_\ast|\geq c_0n$, we have $s(\Pi)+2q\leq n$ for large enough $n$. Given an induced-$K_{1,k}$-free graph on $S_\Pi$ with $b\leq C_k\eta n^2$ edges, add a matching of size $q$ on a fixed set of $2q$ vertices in $V_\ast$ and leave all other vertices in $V_\ast$ isolated. The resulting graph is induced-$K_{1,k}$-free and belongs to $\cB_n(W_0,\omega)$ by the choice of $\eta$. This construction is injective after the matching is chosen, so we calculate
\begin{equation}\label{eqn:critical-matching-transfer-K1k}
M_q\left(\frac{p}{1-p}\right)^q(1-p)^{\binom n2}\sum_{0\leq b\leq C_k\eta n^2}N^\ast_{s(\Pi),b}(K_{1,k})\left(\frac{p}{1-p}\right)^b
\leq Z_p(\cB_n(W_0,\omega))\,.
\end{equation}
Combining \eqref{eqn:critical-clean-sparse-side-K1k} and \eqref{eqn:critical-matching-transfer-K1k}, and using $M_q=\exp(q\log n+O(n))$, gives
$$\widehat{\cZ}_\Pi\leq n^{-\nu n}Z_p(\cB_n(W_0,\omega))$$
after decreasing $\nu>0$. Finally, the number of divisions is $e^{O(n)}$, as in \Cref{lemma:clean-retained-comparison-K1k}. Summing over $\Pi\in\sD_W$ in \eqref{eqn:critical-weighted-upper-K1k} and decreasing $\nu$ once more proves the lemma.
\end{proof}

\begin{proof}[Proof of \Cref{thm:gnp-typ-struc}]
Let $G$ be the \erdosrenyi{} random graph $G(n,p)$ conditional on the event $G\in\cF^\ast_n(K_{1,k})$.

First assume $p\in(0,p_k)$. By \eqref{eqn:CpStarFormalDef}, we have $\cX^p_\ast=\{W_0\}$. Hence \Cref{lemma:rough-struc-gnp} implies that $\delta_\square(G,W_0)=o(1)$ with high probability. Since $W_G$ is nonnegative, this implies $e(G)=o(n^2)$.

Now suppose $p\in(p_k,1)$ and define
$$\gamma_p:=\frac{1+(k-2)p}{k-1}=p+\frac{1-p}{k-1},\qquad W_p:=W^\ast_{\gamma_p}\,.$$
By \eqref{eqn:CpStarFormalDef}, the graphon $W_p$ is the unique element of $\cX^p_\ast$ up to equivalence. Therefore \Cref{lemma:rough-struc-gnp} implies that $\delta_\square(G,W_p)=o(1)$ with high probability, and in particular $e(G)=(1+o(1))\gamma_p\binom n2$. It remains to prove that $G$ is co-$(k-1)$-partite with high probability. Let $\zeta>0$ be small enough that $[\gamma_p-\zeta,\gamma_p+\zeta]\subset(\gamma_k,1)$. The proof of \Cref{thm:main-almostall} \ref{item:thm:main-almostall-super} is uniform for all $m$ satisfying $m/\binom n2\in[\gamma_p-\zeta,\gamma_p+\zeta]$, and gives
$$|\cF^\ast_{n,m}(K_{1,k})\setminus\cC^{k-1}_{n,m}|=o(|\cC^{k-1}_{n,m}|)$$
uniformly in this range. Since every graph in $\cB_n(W_p,\zeta/2)$ has $m/\binom n2\in[\gamma_p-\zeta,\gamma_p+\zeta]$ for large enough $n$, we have
\begin{align*}
Z_p(\cB_n(W_p,\zeta/2)\setminus\cC^{k-1}_n)&\leq (1-p)^{\binom n2}\sum_{\substack{m\\ m/\binom n2\in[\gamma_p-\zeta,\gamma_p+\zeta]}}\left(\frac{p}{1-p}\right)^m|\cF^\ast_{n,m}(K_{1,k})\setminus\cC^{k-1}_{n,m}| \\
&= o\vast((1-p)^{\binom n2}\sum_{\substack{m\\ m/\binom n2\in[\gamma_p-\zeta,\gamma_p+\zeta]}}\left(\frac{p}{1-p}\right)^m|\cC^{k-1}_{n,m}|\vast) \\
&= o(Z_p(\cC^{k-1}_n))\,.
\end{align*}
Since $\cC^{k-1}_n\subseteq\cF^\ast_n(K_{1,k})$, and since \Cref{lemma:rough-struc-gnp} gives
$$Z_p(\cF^\ast_n(K_{1,k})\setminus\cB_n(W_p,\zeta/2))\leq e^{-Cn^2}Z_p(\cF^\ast_n(K_{1,k}))\,,$$
it follows that
$$Z_p(\cF^\ast_n(K_{1,k})\setminus\cC^{k-1}_n)=o(Z_p(\cF^\ast_n(K_{1,k})))\,.$$
Thus $G$ is co-$(k-1)$-partite with high probability.

It remains to consider the critical case $p=p_k$. Let $\omega>0$ be sufficiently small, and let $\tau_0<\omega$ and $\nu>0$ be given by \Cref{lemma:critical-gnp-comparison-K1k}. Let
$$\cN=\{W_0,W_1,\dots,W_\ell\}\subseteq\cX^{p_k}_\ast$$
be a $\tau_0/2$-net of $\cX^{p_k}_\ast$ in cut metric. We may assume each $W_i$ can be written $W_\blambda$ for some $\blambda\in\bLambda$. Define
$$\arraycolsep=1.4pt
\begin{array}{ll}
\cF_\close &:= \displaystyle \cB_n(W_0,\tau_0)\cup\bigcup_{i=1}^\ell\cB_n(W_i,\tau_0)\,, \\[18pt]
\cF_\far &:= \cF^\ast_n(K_{1,k})\setminus\cF_\close\,.
\end{array}$$
By \Cref{lemma:rough-struc-gnp}, there is a constant $C=C(k,\omega)>0$ such that
$$Z_{p_k}(\cF_\far)\leq e^{-Cn^2}Z_{p_k}(\cF^\ast_n(K_{1,k}))\,.$$
Let $I:=\{i\in[\ell]:\delta_\square(W_i,W_0)\geq8\omega\}$. Since $\tau_0<\omega$, we have
$$\cF_\close\subseteq\cB_n(W_0,9\omega)\cup\bigcup_{i\in I}\cB_n(W_i,\tau_0)\,.$$
Applying \Cref{lemma:critical-gnp-comparison-K1k} to every $W_i$ with $i\in I$ gives
$$\sum_{i\in I}Z_{p_k}(\cB_n(W_i,\tau_0))\leq \ell n^{-\nu n}Z_{p_k}(\cB_n(W_0,\omega))=o(Z_{p_k}(\cB_n(W_0,9\omega)))\,.$$
Since $\cB_n(W_0,9\omega)\subseteq\cF^\ast_n(K_{1,k})$, the previous displays imply
$$Z_{p_k}(\cF^\ast_n(K_{1,k})\setminus\cB_n(W_0,9\omega))=o(Z_{p_k}(\cF^\ast_n(K_{1,k})))\,.$$
It follows that $\delta_\square(G,W_0)<9\omega$ with high probability. Since $\omega>0$ was arbitrary, we have $\delta_\square(G,W_0)=o(1)$ with high probability, and hence $e(G)=o(n^2)$.
\end{proof}

\section{Extremal Structure and Stability of Edge Colorings}
\label{sec:edge-coloring}
In this section we prove extremal structure and stability results for red--green--blue edge colorings. We use the letters $\red$, $\green$, and $\blue$ to denote the three colors. Let $k\geq3$ be an integer fixed throughout this section.

\begin{definition}\label{dfn:Fk-free-cols}
Let $\cF_k$ be the set of all red--green--blue colorings $\varphi$ of $E(K_k)$ such that for some vertex $v\in V(K_k)$, the following holds: $\varphi(vx)=\red$ for all $x\neq v$, and $\varphi(xy)\in\{\red,\green\}$ for all distinct $x,y\in V(K_k)\setminus\{v\}$. We call the elements of $\cF_k$ \emph{forbidden patterns}. Let $\cC_k(n)$ be the set of red--green--blue colorings of $E(K_n)$ such that for all $S\subseteq V(K_n)$ with $|S|=k$, $\varphi|_S\not\in\cF_k$. We call the elements of $\cC_k(n)$ \emph{$\cF_k$-free colorings}.
\end{definition}

For a red--green--blue coloring $\varphi$ of $E(K_n)$, and all $u\in V(K_n)$ and $U\subseteq V(K_n)$, define 
\begin{align*}
N_\red(u) &:= \{v\in V(K_n):\varphi(uv)=\red\}\,, \\
d_\red(u) &:= |N_\red(u)|\,, \\
e_\red(U) &:= |\{e\in\textstyle\binom{U}{2}:\varphi(e)=\red\}|\,.
\end{align*}
We also define $N_\red(u,U):=N_\red(u)\cap U$ and $d_\red(u,U):=|N_\red(u,U)|$. The green and blue counterparts $N_\green(u)$, $N_\blue(u)$, etc.\ are all defined analogously. Write $E_\red(\varphi)$, $E_\green(\varphi)$, and $E_\blue(\varphi)$ for the sets of red, green, and blue edges of $\varphi$, respectively, and write $e_\red(\varphi)$, $e_\green(\varphi)$, and $e_\blue(\varphi)$ for their cardinalities.

Throughout this section, we will use the variable $\Delta:=k-2$. For any red--green--blue coloring $\varphi$ of $E(K_n)$, define
$$\Phi(\varphi):=e_\red(\varphi)-\Delta\,e_\blue(\varphi)\,.$$
Define the weighted degree $D(v):=d_\red(v)-\Delta\,d_\blue(v)$ of a vertex $v$. The weighted degree restricted to a subset $U$ is $D(v,U):=d_\red(v,U)-\Delta\,d_\blue(v,U)$. Note that $2\Phi(\varphi)=\sum_vD(v)$. We emphasize that $\Delta$, $\Phi$, and $D$ all depend on $k$, but we leave this dependence implicit since we think of $k$ as an integer fixed throughout the whole section. Additionally, the coloring that our notation refers to is typically clear from the context; in cases of ambiguity, we include the coloring as a superscript, as in $N_\red^\varphi(u)$ and $D^\varphi$.

The extremal objects in this section have a multipartite structure amenable to symmetrization. The following cloning operation plays a central role in our analysis.

\begin{definition}[Cloning]\label{dfn:cloning}
Let $\varphi$ be a red--green--blue coloring of $E(K_n)$. If $v\in V(K_n)$ and $U\subseteq V(K_n)\setminus\{v\}$, define the coloring $\varphi^{U\gets v}$ as follows:
\begin{itemize}[leftmargin=\parindent, topsep=4pt]
    \item for all $u\in U$ and $w\in V(K_n)\setminus(U\cup\{u,v\})$, set $\varphi^{U\gets v}(uw):=\varphi(vw)$,
    \item for all distinct $u,w\in U\cup\{v\}$, set $\varphi^{U\gets v}(uw):=\blue$,
    \item all other edges are the same as in $\varphi$.
\end{itemize}
If $U$ contains a single element $u$ then we write $\varphi^{u\gets v}:=\varphi^{U\gets v}$.
\end{definition}

\begin{fact}\label{fact:cloning-keeps-ckn}
If $\varphi\in\cC_k(n)$, $v\in V(K_n)$, and $U\subseteq V(K_n)\setminus\{v\}$ then $\varphi^{U\gets v}\in\cC_k(n)$.
\end{fact}
\begin{proof}
If there were a forbidden pattern in $\varphi^{U\gets v}$ then it would use at most one vertex from $U\cup\{v\}$ (since $U\cup\{v\}$ is a blue clique), and replacing that vertex by $v$ would produce the same forbidden pattern in $\varphi$, which contradicts $\varphi\in\cC_k(n)$.
\end{proof}

\begin{lemma}\label{lemma:kthOrderMantel}
For all $k\geq3$ and $\varphi\in\cC_k(n)$,
\begin{equation}
\Phi(\varphi)\leq\floor*{\frac{\Delta n}{2}}\,.\label{eqn:kthOrderIneq}
\end{equation}
\end{lemma}
\begin{proof}
For a coloring $\chi$, two vertices $x,y$ are \emph{twins} if $\chi(xy)=\blue$ and $\chi(xz)=\chi(yz)$ for every $z\in V(K_n)\setminus\{x,y\}$. The relation of being equal or twin is an equivalence relation; call its equivalence classes \emph{twin classes}. Let $\tau(\chi)$ be the number of pairs of twin vertices in $\chi$.

Choose $\psi\in\cC_k(n)$ so that $\Phi(\psi)$ is maximum, and subject to this, $\tau(\psi)$ is maximum. We have $\Phi(\varphi)\leq\Phi(\psi)$, so it suffices to prove \eqref{eqn:kthOrderIneq} for $\psi$. We first claim that every blue edge of $\psi$ has twin endpoints. Indeed, for every blue edge $xy\in E_\blue(\psi)$, we have
$$\Phi(\psi^{x\gets y})-\Phi(\psi)=D^\psi(y)-D^\psi(x)\,,\qquad \Phi(\psi^{y\gets x})-\Phi(\psi)=D^\psi(x)-D^\psi(y)\,.$$
Since both cloned colorings belong to $\cC_k(n)$ by \Cref{fact:cloning-keeps-ckn}, maximality of $\Phi(\psi)$ implies $D^\psi(x)=D^\psi(y)$. Suppose now that $x$ and $y$ are not twins, and let $C_x$ and $C_y$ be their twin classes. Then $C_x\neq C_y$, and both clonings preserve $\Phi$. Under the cloning $x\gets y$, the vertex $x$ leaves $C_x$ and joins $C_y$, and no twin class not containing $x$ is split. It follows that
\begin{align*}
&\tau(\psi^{x\gets y})-\tau(\psi)\geq |C_y|-(|C_x|-1) \,, \qquad \tau(\psi^{y\gets x})-\tau(\psi)\geq |C_x|-(|C_y|-1)\,.
\end{align*}
Since the sum of the right-hand sides in the last two displays is $2$, one of the two clonings preserves $\Phi$ and strictly increases $\tau$, contradicting the choice of $\psi$ and proving our claim.

It follows that $\psi$ has the following structure: there exists a partition $V(K_n)=V_1\cup\cdots\cup V_l$ such that $\psi|_{V_i}\equiv\blue$ for all $i\in[l]$; and for all $1\leq i<j\leq l$ there exists $c\in\{\red,\green\}$ such that $\psi|_{E(V_i,V_j)}\equiv c$. Let $R$ be the graph with vertex set $[l]$ and edge set $\{ij:\psi|_{E(V_i,V_j)}\equiv\red\}$. Notice that $\Delta(R)\leq\Delta$, as otherwise $\psi$ contains a forbidden coloring. Let $v_i:=|V_i|$ for all $i\in[l]$. We then have
$$e_\red(\psi)=\sum_{ij\in E(R)}v_iv_j\,,\qquad e_\blue(\psi)=\sum_{i=1}^l\binom{v_i}{2}\,.$$
Using that $v_iv_j\leq\frac{1}{2}(v_i^2+v_j^2)$ for all $1\leq i<j\leq l$, we calculate
$$\sum_{ij\in E(R)}v_iv_j\leq\frac{1}{2}\sum_{i=1}^ld_R(i)\,v_i^2\leq\frac{\Delta}{2}\sum_{i=1}^lv_i^2\,,$$
which further implies
$$\Phi(\varphi)\leq\Phi(\psi)=\sum_{ij\in E(R)}v_iv_j-\Delta\sum_{i=1}^l\binom{v_i}{2}\leq\frac{\Delta}{2}\sum_{i=1}^lv_i^2-\frac{\Delta}{2}\sum_{i=1}^lv_i(v_i-1)=\frac{\Delta n}{2}\,,$$
completing the proof.
\end{proof}

\subsection{Stability of the extremal structure}
\begin{definition}[$\Delta$-regular core]\label{dfn:d-regular-cores}
Fix integers $3\leq k\leq n$. Let $\cR_k(n)$ be the set of red--green--blue colorings of $E(K_n)$ obtained as follows: for some integer $\ell\in[k-1,n]$, choose a connected $\Delta$-regular graph $R$ on $[\ell]$; form a partition $V(K_n)=V_1\cup\cdots\cup V_\ell$ such that $||V_i|-|V_j||\leq1$ for all $1\leq i<j\leq \ell$; color all edges inside each $V_i$ blue; for all $ij\in E(R)$, color all edges in $E(V_i,V_j)$ red; color all remaining edges green. The elements of $\cR_k(n)$ are called \emph{$\Delta$-regular cores}.
\end{definition}

\begin{definition}[Extremal family $\cE_k(n)$]\label{dfn:k1k-extremal-fam}
Let $\cE_k(n)$ be the set of red--green--blue colorings of $E(K_n)$ obtained as follows: for some integer $t\geq0$, choose integers $\kappa_1,\dots,\kappa_t$, each being at least $k$, such that $\kappa_1+\cdots+\kappa_t\leq n$; choose colorings $\psi_i\in\cR_k(\kappa_i)$ for all $i\in[t]$; choose injective embeddings $\iota_i:V(K_{\kappa_i})\to V(K_n)$ such that $\image(\iota_i)\cap\image(\iota_j)=\emptyset$ for all $1\leq i<j\leq t$; finally, $K_n$ inherits all the colors from the $\psi_i$ and $\iota_i$, and all remaining edges in $K_n$ are colored green.
\end{definition}

Note that the all-green coloring $\varphi\equiv\green$ is an element of $\cE_k(n)$ since one can take $t=0$.
The \emph{Hamming distance} between two colorings $\varphi$ and $\psi$ of $E(K_n)$ is
$$d(\varphi,\psi):=|\{e\in E(K_n):\varphi(e)\neq\psi(e)\}|\,.$$
The Hamming distance between a coloring $\varphi$ and a set $\cS$ of colorings is
$$d(\varphi,\cS):=\min\{d(\varphi,\psi):\psi\in\cS\}\,.$$
The following stability theorem is the main result of this section.

\begin{theorem}\label{thm:kth-order-stability}
For all $k\geq3$ and $\epsilon>0$ there exists $\delta>0$ such that the following holds for all large enough $n$. If $\varphi\in\cC_k(n)$ and
\begin{equation}\label{eqn:near-extremality}
\Phi(\varphi)\geq-\delta n^2
\end{equation}
then
$$d(\varphi,\cE_k(n))\leq\epsilon n^2\,.$$
\end{theorem}

\subsubsection{Proof overview}
The proof of \Cref{thm:kth-order-stability} works by extracting subcolorings of $\varphi$ that resemble $\Delta$-regular cores. After showing that each core is green-adjacent to most of the rest of the graph, it follows that $\varphi$ is close to an element of $\cE_k(n)$. We first give an overview of the extraction of $\Delta$-regular cores, which is carried out in full in \Cref{lemma:kth-order-recursion}.
\begin{enumerate}[(\emph{\roman*}), topsep=2pt]
    \item First, using \Cref{lemma:D-tails}, we restrict to a subcoloring of $\varphi$ on $(1-o(1))n$ vertices such that $\abs{D(v)}\leq o(n)$ in the restricted coloring. This has the convenient consequence that $d_\red(v)=\Delta\,d_\blue(v)+o(n)$.
    \item We start with a vertex $x^{(0)}$ of maximum red degree $m:=d_\red(x^{(0)})$. The local Tur\'an lemma (\Cref{lemma:locally-turan}) proves that the red--green graph induced by the red neighborhood of $x^{(0)}$, written $H:=\varphi_{\red,\green}[N_\red(x^{(0)})]$, is near-Tur\'an.
    \item The \erdos{}--Simonovits theorem (\Cref{thm:erdos-simonovits}) then implies $N_\red(x^{(0)})$ admits a nearly balanced partition $A_1\cup\cdots\cup A_\Delta$ into blue-dense clusters with blue density $o(1)$ between clusters. Each $A_i$ has $|A_i|=m/\Delta+o(n)$.
    \item The blue-density of $A_i$ and $|D(v)|=o(n)$ imply that a typical vertex $x^{(1)}\in N_\red(x^{(0)})$ satisfies $d_\red(x^{(1)})=m-o(n)$. Now iterate: repeat steps (\emph{ii}) and (\emph{iii}) from $x^{(1)}$.
    \item After the iteration over vertices $x^{(j)}$ terminates, we have found disjoint, nearly balanced, blue-dense subsets $U_1,\dots,U_t$. \Cref{lemma:vertex-level-mixed-mass} proves that for all $1\leq i<j\leq t$, there is a dominant color $c\in\{\red,\green\}$ with $d_c(U_i,U_j)\geq1-o(1)$.
    \item Since a typical vertex satisfies $d_\red(v)=m-o(n)$, the reduced graph on $[t]$ whose edges are the pairs on which red is the dominant color must in fact be $\Delta$-regular. It follows that $U_1\cup\cdots\cup U_t$ induces a coloring close to a $\Delta$-regular core.
\end{enumerate}

In several places we will use the fact that for a coloring $\varphi$, a vertex $v$, and a subset $U\subseteq V(K_n)$ of size $s$,
\begin{equation}
\Phi(\varphi^{U\gets v})-\Phi(\varphi)\geq sD^\varphi(v)-\sum_{u\in U}D^\varphi(u)-(\Delta+1)s^2\,. \label{eqn:phi-clone-ineq}
\end{equation}
This inequality holds since for each $u\in U$, we recolor every edge $uw$ with $w\in V(K_n)\setminus(U\cup\{v\})$ to match $vw$, so this increases $\Phi$ by $sD(v)-\sum_{u\in U}D(u)$, and the at most $s^2$ edges in $\binom{U\cup\{v\}}{2}$ are recolored blue, each decreasing $\Phi$ by at most $\Delta+1$.

\subsubsection{Restriction to a subset with uniform weighted degrees}
The following lemma will allow us to restrict to a subset $X$ of $(1-o(1))n$ vertices, all of which have weighted degree $|D(v,X)|=o(n)$. In particular, $d_\red(v,X)=\Delta\,d_\blue(v,X)+o(n)$, which is suggestive of the extremal structure $\cE_k(n)$ and will turn out to be a very convenient identity.

\begin{lemma}\label{lemma:D-tails}
For all fixed $k\geq3$ and $\delta\in(0,1)$, there exists $n_0$ such that the following holds for all integers $n\geq n_0$. If $\varphi\in\cC_k(n)$ and $\Phi(\varphi)\geq-\delta n^2$ then the following hold:
\begin{enumerate}[
  label=\textit{(\roman*)},
  ref=(\textit{\roman*}),
  topsep=5pt
]
	\item\label{item:lemma:D-tails-i} $\max_{v\in V(K_n)}D(v)\leq8\Delta\sqrt{\delta}\,n$.
	\item\label{item:lemma:D-tails-ii} All but at most $10\Delta\delta^{3/4}n$ vertices satisfy $D(v)\geq-\delta^{1/4}n$.
	\item\label{item:lemma:D-tails-iii} Let $X:=\left\{v\in V(K_n):D(v)\geq-\delta^{1/4} n\right\}$. For all $v\in X$,
    $$\abs{D(v,X)} \leq 12\Delta^2\delta^{1/4}n\,.$$
\end{enumerate}
\end{lemma}
\begin{proof}
{\it(i)} Suppose to the contrary that there exists $z\in V(K_n)$ with $D(z)\geq8\Delta\sqrt{\delta}n$.
Let $s:=\floor{\sqrt{\delta}\,n}$ and choose $U\subseteq V(K_n)\setminus\{z\}$ with $|U|=s$ so as to minimize $\sum_{u\in U}D(u)$ among all $s$-subsets of $V(K_n)\setminus\{z\}$. Write $S:=U\cup\{z\}$ and $T:=V(K_n)\setminus S$. Let $\psi:=\varphi^{U\gets z}$. Since the average of $D(\cdot)$ over $U$ is at most the overall average,
$$\frac{1}{s}\sum_{u\in U}D(u)\leq\frac{1}{n}\sum_{v\in V(K_n)}D(v)\leq \Delta\,,$$
so combining with \eqref{eqn:phi-clone-ineq} implies, for all large enough $n$,
\begin{align*}
\Phi(\psi)-\Phi(\varphi) &\geq sD(z)-\sum_{u\in U}D(u)-(\Delta+1)s^2 \\
&\geq s\cdot 8\Delta\sqrt{\delta}n - s\Delta - (\Delta+1)s^2 \\
&\geq 7\Delta\delta n^2 - \Delta\sqrt{\delta}n - 2\Delta\delta n^2 \geq 4\Delta\delta n^2\,,
\end{align*}
where we used that $\frac{7}{8}\sqrt{\delta}n\leq s\leq\sqrt{\delta}n$. On the other hand, since $\psi\in\cC_k(n)$ by \Cref{fact:cloning-keeps-ckn}, \Cref{lemma:kthOrderMantel} implies $\Phi(\psi)\leq\Delta n/2$. It follows that $\Phi(\psi)-\Phi(\varphi)\leq\Delta n/2+\delta n^2\leq2\delta n^2$. This contradicts the above display and completes the proof of {\it(i)}.

{\it(ii)} Let $x:=|X^c|$ (the set $X$ is defined in {\it(iii)}). Using part {\it(i)}, the definition of $X$, and near-extremality of $\Phi$,
$$-\delta n^2\leq\Phi(\varphi)=\frac{1}{2}\left(\sum_{v\in X}D(v)+\sum_{v\in X^c}D(v)\right)\leq 4\Delta\sqrt{\delta}n^2-\frac{x\delta^{1/4} n}{2}\,,$$
which immediately implies $x\leq\frac{\delta+4\Delta\sqrt{\delta}}{\delta^{1/4}/2}n\leq10\Delta\delta^{3/4}n$.

{\it(iii)} For all $v\in X$ we have $D(v)\geq-\delta^{1/4} n$, so using parts {\it(i)} and {\it(ii)},
$$D(v,X)=D(v)-d_\red(v,X^c)+\Delta\,d_\blue(v,X^c)\geq-\delta^{1/4}n-10\Delta\delta^{3/4}n\geq-12\Delta\delta^{1/4}n\,,$$
since $|X^c|\leq10\Delta\delta^{3/4}n$. Similarly,
$$D(v,X)\leq D(v)+\Delta\,d_\blue(v,X^c)\leq8\Delta\sqrt{\delta}n+10\Delta^2\delta^{3/4}n\leq12\Delta^2\delta^{1/4}n\,,$$
completing the proof.
\end{proof}

\subsubsection{Near-Tur\'an structure in red neighborhoods}
Next, we present a key structural lemma that states, roughly speaking, that if a coloring $\varphi\in\cC_k(n)$ is close to extremal, then inside the red neighborhood of a vertex of large enough red degree, the union of the red and green edges is close to Tur\'an. We write $T_{n,r}$ for the Tur\'an graph and denote $t_{n,r}:=e(T_{n,r})$.

\begin{lemma}\label{lemma:locally-turan}
For all fixed $\delta\in(0,1)$, $\alpha\in[0,1]$ and $k\geq4$, there exists $n_0$ such that the following holds for all $n\geq n_0$. Let $\varphi\in\cC_k(n)$ and assume
$$\Phi(\varphi)\geq-\delta n^2\,.$$
Fix $x\in V(K_n)$ such that
$$d_\red(x)\geq \max_{v\in V(K_n)} d_\red(v)-\alpha n\,,$$
and write $m:=d_\red(x)$. Assume $m\geq \vartheta n$ for some $\vartheta\in(0,1]$. Let $H$ be the graph on $N_\red(x)$ with edge set $E(H):=\{e\in\textstyle\binom{N_\red(x)}{2}:\varphi(e)\in\{\red,\green\}\}$. Then
\begin{equation}
e(H)\geq t_{m,\Delta}-12\bigg(\frac{\sqrt{\delta}}{\vartheta}+\frac{\alpha}{\vartheta}\bigg)m^2\,. \label{eqn:local-stability}
\end{equation}
\end{lemma}
\begin{proof}
Let $\rho:=\frac{1}{m^2}(t_{m,\Delta}-e(H))$ and $C:=12$. We will prove $\rho\leq C\big(\frac{\sqrt{\delta}}{\vartheta}+\frac{\alpha}{\vartheta}\big)$, which is precisely \eqref{eqn:local-stability}. Assume to the contrary that $\rho\geq C\big(\frac{\sqrt{\delta}}{\vartheta}+\frac{\alpha}{\vartheta}\big)$.

Let $A:=N_\red(x)$. Since $\varphi\in\cC_k(n)$, $H$ is a $K_{\Delta+1}$-free graph, which implies $e(H)\leq t_{m,\Delta}$. If we define $b^\ast:=\binom{m}{2}-t_{m,\Delta}$ (the minimum number of nonedges in a $K_{\Delta+1}$-free graph), then we have $e_\blue(A)=b^\ast+t_{m,\Delta}-e(H)$. Using $b^\ast\geq\frac{m^2}{2\Delta}-\frac{m}{2}$, we calculate
\begin{equation}
\sum_{v\in A}d_\blue(v,A)=2e_\blue(A)=2b^\ast+2(t_{m,\Delta}-e(H)) \geq \frac{m^2}{\Delta}-m+2\rho m^2\,,\label{eqn:sumvAdb1}
\end{equation}
Define the high blue degree set
$$B:=\left\{v\in A:d_\blue(v,A)\geq\left(\frac{1}{\Delta}+\rho\right)m\right\}\,.$$
We claim $|B|\geq \rho m$. Indeed, we have $d_\blue(v,A)\leq m$ for all $v\in A$, and $d_\blue(v,A)\leq(\frac{1}{\Delta}+\rho)m$ for all $v\in A\setminus B$\,; hence, letting $b:=|B|$,
$$\sum_{v\in A}d_\blue(v,A)\leq bm+(m-b)\left(\frac{1}{\Delta}+\rho\right)m=\frac{m^2}{\Delta}+\rho m^2+bm\left(1-\frac{1}{\Delta}-\rho\right)\,,$$
and combining this with \eqref{eqn:sumvAdb1} gives $b\geq\frac{\rho m-1}{1-1/\Delta-\rho}\geq\rho m$ (since $C\big(\frac{\sqrt{\delta}}{\vartheta}+\frac{\alpha}{\vartheta}\big)\leq \rho<\frac{1}{2}$ and $0<1-\frac{1}{\Delta}-\rho<1$) for large enough $n$.

All vertices $v$ in $B$ have small weighted degree $D(v)$: by the choice of $x$, they satisfy $d_\red(v)\leq m+\alpha n$, and by definition of $B$ they satisfy $d_\blue(v)\geq d_\blue(v,A)\geq(\frac{1}{\Delta}+\rho)m$, hence
\begin{equation}\label{eqn:small-Ddeg-B}
D(v)=d_\red(v)-\Delta\,d_\blue(v)\leq (m+\alpha n)-\Delta\left(\frac{1}{\Delta}+\rho\right)m=\alpha n-\Delta\rho m\,.
\end{equation}
There must exist a vertex $z\in V(K_n)$ with $D(z)\geq-2\sqrt{\delta}n$; otherwise we would have $2\Phi(\varphi)=\sum_{v\in V(K_n)}D(v)\leq-2\sqrt{\delta}n^2$, which contradicts our assumption that $\Phi(\varphi)\geq-\delta n^2$. Let $s:=\floor{\rho m/16}$ and let $W\subseteq B$ be a subset of size $s$ not containing $z$. Let $\psi:=\varphi^{W\gets z}$. Using \eqref{eqn:phi-clone-ineq}, we calculate
\begin{align*}
\Phi(\psi)-\Phi(\varphi) &\geq s\,D^\varphi(z)-\sum_{w\in W}D^\varphi(w)-(\Delta+1)\,s^2 \\
&\geq - 2s\sqrt{\delta}n + s\cdot\Delta\rho m - s\alpha n - (\Delta+1)s^2 \\
&\geq - \frac{\rho\sqrt{\delta}mn}{8} + \frac{\Delta\rho^2m^2}{16} - \frac{\rho\alpha mn}{16} - \frac{(\Delta+1)\rho^2m^2}{2^8} - o(n^2) \\
&\geq \rho^2m^2\left(-\frac{\sqrt{\delta}n}{8\rho m}-\frac{\alpha n}{16\rho m}+\frac{\Delta}{16}-\frac{\Delta}{2^7}\right) - o(n^2) \\
&\geq \rho^2m^2\left(\frac{\Delta}{32}-\frac{3}{16C}\right)-o(n^2) \\
&\geq C^2\delta n^2\left(\frac{\Delta}{32}-\frac{3}{16C}\right)-o(n^2) \\[3pt]
&\geq 2\delta n^2-o(n^2)\,,
\end{align*}
where in the second inequality we used $D(z)\geq-2\sqrt{\delta}n$ and \eqref{eqn:small-Ddeg-B}; in the third inequality we used $s=(1+o(1))\rho m/16$; in the fifth inequality we used $m\geq\vartheta n$, $\rho\geq C\frac{\sqrt{\delta}}{\vartheta}$, and $\rho\geq C\frac{\alpha}{\vartheta}$, which imply
$$\frac{\sqrt{\delta}n}{8\rho m}\leq \frac{1}{8C}\qquad\text{and}\qquad\frac{\alpha n}{16\rho m}\leq \frac{1}{16C}\,;$$
in the sixth inequality we used $\rho^2m^2\geq C^2(\delta+\alpha^2+2\sqrt{\delta}\alpha)n^2\geq C^2\delta n^2$; and in the last inequality we used $\frac{\Delta C^2}{32}-\frac{3C}{16}\geq2$, which holds since $C=12$ and $\Delta\geq2$.

But since $\Phi(\varphi)\geq-\delta n^2$, this means $\Phi(\psi)>\Delta n/2$, contradicting \Cref{lemma:kthOrderMantel}, and completing the proof that $\rho\leq C\big(\frac{\sqrt{\delta}}{\vartheta}+\frac{\alpha}{\vartheta}\big)$.
\end{proof}

\begin{theorem}{\normalfont(\citeauthor{erdos1967some}~\cite{erdos1967some},~\citeauthor{simonovits1968method}~\cite{simonovits1968method})}\label{thm:erdos-simonovits}
For every $\epsilon>0$ there exist $\delta>0$ and $n_0\geq0$ such that the following holds for all $n\geq n_0$. Fix $r\geq2$ and a $K_{r+1}$-free graph $G$ on $n$ vertices. If
$$e(G)\geq t_{n,r}-\delta n^2$$
then there exists a partition $V(G)=V_1\cup\cdots\cup V_r$ such that
$$\sum_{i=1}^re(V_i)\leq\epsilon n^2$$
and $\left||V_i|-\frac{n}{r}\right|\leq \epsilon n$ for all $i\in[r]$.
\end{theorem}

\subsubsection{Between-cluster dominant colors}
\begin{lemma}\label{lemma:vertex-level-mixed-mass}
For all fixed $k\geq3$ and $\eta\in(0,1)$ there exist $\beta_0>0$ and $C\geq1$ such that the following holds for all $\beta\in(0,\beta_0)$ and all sufficiently large $n$. Let $\varphi\in\cC_k(n)$, and let $U_1,\dots,U_t\subseteq V(K_n)$ be $t\geq k-1$ disjoint sets, and let $U:=\bigcup_{i=1}^tU_i$. Assume the following conditions hold:
\begin{enumerate}[
  label=\textit{(\roman*)},
  ref=(\textit{\roman*}),
  topsep=5pt
]
    \item\label{item:vertex-level-mixed-mass-1} $|U_i|\geq \eta n$ for all $i\in[t]$, and $||U_i|-|U_j||\leq \beta n$ for all $i,j\in[t]$;
    \item\label{item:vertex-level-mixed-mass-2} $d_\blue(U_i,U_j)\leq\beta$ for all distinct $i,j\in[t]$;
    \item\label{item:vertex-level-mixed-mass-3} for all $i\in[t]$, every $v\in U_i$ satisfies
    $$d_\blue(v,U_i)\geq (1-\beta)|U_i|\,,\qquad d_\blue(v,U\setminus U_i)\leq \beta n\,,\qquad\abs{D(v,U)}\leq \beta n\,.$$
\end{enumerate}
Then for all distinct $i,j\in[t]$,
\begin{equation}\label{eqn:dominant-color}
\max\{d_\red(U_i,U_j),\,d_\green(U_i,U_j)\} \geq 1-C\,\frac{\sqrt{\beta}}{\eta^2} \,.
\end{equation}
\end{lemma}
\begin{proof}
We will first show that for all $i\in[t]$ and $v\in U_i$,
\begin{equation}\label{eq:vertex-mixed-mass-final}
\sum_{j=1}^t\min\left\{\frac{d_\red(v,U_j)}{|U_j|},\,1-\frac{d_\red(v,U_j)}{|U_j|}\right\}\leq C\,\frac{\sqrt{\beta}}{\eta^2}
\end{equation}
for a constant $C=C(k)>0$. Define
$$\rho_{vj}:=\frac{d_\red(v,U_j)}{|U_j|},\qquad M(v):=\sum_{j\neq i}\min\{\rho_{vj},\,1-\rho_{vj}\}\,.$$
Let $\theta:=8\Delta\sqrt{\beta}$ and $A(v):=\{j\in[t]\setminus\{i\}: \rho_{vj}\geq\theta\}$.

We will first show that $|A(v)|\leq\Delta$. Assume to the contrary that $|A(v)|\geq \Delta+1$. Choose distinct indices $j_1,\dots,j_{\Delta+1}\in A(v)$, and for each $r\in[\Delta+1]$ let $S_r:=N_\red(v,U_{j_r})$. Then $|S_r|=d_\red(v,U_{j_r})\geq \theta |U_{j_r}|$. Choose $y_r\in S_r$ uniformly and independently at random for each $r\in[\Delta+1]$ and let $Z$ be the number of pairs $1\leq r<q\leq\Delta+1$ such that $\varphi(y_ry_q)=\blue$. We calculate that
\begin{align*}
\bbP\{Z=0\} &\geq 1-\sum_{1\leq r<q\leq\Delta+1}\bbP\{\varphi(y_ry_q)=\blue\} \\
&= 1-\sum_{1\leq r<q\leq\Delta+1}\frac{e_\blue(S_r,S_q)}{|S_r||S_q|} \\
&\geq 1-\sum_{1\leq r<q\leq\Delta+1}\frac{e_\blue(U_{j_r},U_{j_q})}{\theta^2|U_{j_r}||U_{j_q}|} \\
&\geq 1-\binom{\Delta+1}{2}\frac{\beta}{\theta^2} \geq 1-\frac{\Delta(\Delta+1)}{2}\cdot\frac{1}{64\Delta^2} > 0 \,,
\end{align*}
so there exists a choice of $y_1,\dots,y_{\Delta+1}$ with no blue edge among them. By construction, all edges $vy_r$ are red. Therefore the $k=\Delta+2$ vertices $\{v,y_1,\dots,y_{\Delta+1}\}$ span a forbidden pattern with center $v$, all center-incident edges red, and all leaf-leaf edges non-blue, which contradicts $\varphi\in\cC_k(n)$, proving that indeed $|A(v)|\leq\Delta$.

From item \emph{(iii)} of our hypotheses and the definition of $D(\cdot)$, it is easy to check that $d_\blue(v,U)=|U_i|+O_k(\beta n)$ and $d_\red(v,U)=\Delta|U_i|+O_k(\beta n)$. We use these identities to estimate $\sum_j\rho_{vj}$ as follows. Using that $||U_i|-|U_j||\leq\beta n$ and $\rho_{vj}\in[0,1]$,
\begin{align*}
\sum_{j=1}^t\rho_{vj} &= \frac{1}{|U_i|}\left(\sum_{j=1}^t\rho_{vj}|U_j|+\sum_{j=1}^t\rho_{vj}(|U_i|-|U_j|)\right) \\
&= \frac{d_\red(v,U)}{|U_i|} + \frac{t}{|U_i|}\cdot O(\beta n) = \Delta + O_k\!\left(\frac{\beta}{\eta^2}\right) \,,
\end{align*}
where we used that $|U_i|\geq\eta n$ and $t\leq1/\eta+1\leq2/\eta$.

For $j\in A(v)$ we have $\min\{\rho_{vj},1-\rho_{vj}\}\leq 1-\rho_{vj}$, while for $j\notin A(v)$ we have $\rho_{vj}\leq\theta$. Hence splitting indices into $A(v)$ and its complement,
\begin{align*}
\sum_{j=1}^t\min\{\rho_{vj},1-\rho_{vj}\} &\leq \sum_{j\in A(v)}(1-\rho_{vj}) + \sum_{j\in[t]\setminus A(v)}\rho_{vj} \\
&= |A(v)| - \left(\sum_{j=1}^t\rho_{vj}-\sum_{j\in[t]\setminus A(v)}\rho_{vj}\right) + \sum_{j\in[t]\setminus A(v)}\rho_{vj} \\
&\leq \Delta - \sum_{j=1}^t\rho_{vj} + \sum_{j\in[t]\setminus A(v)}2\rho_{vj} \\
&\leq \Delta - \sum_{j=1}^t\rho_{vj} + 2t\theta \leq O_k\!\left(\frac{\beta}{\eta^2}\right) + O\!\left(\frac{\sqrt{\beta}}{\eta}\right) \leq C\,\frac{\sqrt{\beta}}{\eta^2}
\end{align*}
for a constant $C=C(k)$ and all $\beta\in(0,\beta_0(k,\eta))$, completing the proof of \eqref{eq:vertex-mixed-mass-final}.

It remains to deduce \eqref{eqn:dominant-color} from \eqref{eq:vertex-mixed-mass-final}.
Set $\mu:=C\sqrt{\beta}/\eta^2$ and (by decreasing $\beta_0$ and increasing $C$) assume $\beta\leq \mu$ and $\mu\leq1/100$.
From \eqref{eq:vertex-mixed-mass-final} we know that for every $v\in U$ and every $j\in[t]$, $\min\{\rho_{vj},\,1-\rho_{vj}\}\leq \mu$, hence $\rho_{vj}\leq\mu$ or $\rho_{vj}\geq 1-\mu$.
Fix distinct $i,j\in[t]$ and suppose for a contradiction that
$$\max\{d_\red(U_i,U_j),\,d_\green(U_i,U_j)\}<1-10\mu\,.$$
Since $d_\blue(U_i,U_j)\leq\beta\leq\mu$, we have
\begin{align*}
1-10\mu &\geq d_\red(U_i,U_j) = 1 - d_\green(U_i,U_j) - d_\blue(U_i,U_j) \geq 10\mu - d_\blue(U_i,U_j) \geq 9\mu \,.
\end{align*}
If we define the vertex sets
$$H:=\{v\in U_i:\rho_{vj}\geq 1-\mu\}\,,\qquad L:=U_i\setminus H\,,$$
then for $v\in H$ we have $d_\green(v,U_j)+d_\blue(v,U_j)\leq \mu|U_j|$, and for $v\in L$ we have $d_\red(v,U_j)\leq \mu|U_j|$.
Double counting gives
$$(1-\mu)|H||U_j|\leq e_\red(U_i,U_j)\leq |H||U_j|+\mu|L||U_j|\,,$$
thus $\abs{d_\red(U_i,U_j)-|H|/|U_i|}\leq \mu$,
and in particular $|H|/|U_i|\in[8\mu,\,1-9\mu]$. Moreover,
$$e_\green(H,U_j)+e_\blue(H,U_j)\leq \mu|H||U_j|\,,\qquad e_\red(L,U_j)\leq \mu|L||U_j|\,,$$
so there exists $u\in U_j$ with
$$d_\green(u,H)+d_\blue(u,H)\leq 2\mu|H|\,,\qquad d_\red(u,L)\leq 2\mu|L|\,.$$
For this $u$,
$$\rho_{ui}=\frac{d_\red(u,U_i)}{|U_i|}=\frac{d_\red(u,H)+d_\red(u,L)}{|U_i|}\in\left[\frac{|H|}{|U_i|}-2\mu,\ \frac{|H|}{|U_i|}+2\mu\right]\subseteq[6\mu,\,1-7\mu]\,,$$
and hence $\min\{\rho_{ui},\,1-\rho_{ui}\}\geq 6\mu>\mu$, a contradiction. This proves \eqref{eqn:dominant-color}.
\end{proof}

\subsubsection{Extraction of a $\Delta$-regular core}

\begin{lemma}\label{lemma:kth-order-recursion}
For every fixed $k\geq3$ and $\eta\in(0,1)$ and every $\xi\in(0,\eta/20)$, there exist $\delta_0=\delta_0(k,\eta,\xi)\in(0,1)$ and $n_0=n_0(k,\eta,\xi)$ such that the following holds for all $n\geq n_0$ and all $\delta\in(0,\delta_0)$. Let $\varphi\in\cC_k(n)$ and assume
\begin{equation}\label{eq:recursion-assumption-phi-new}
\Phi(\varphi)\geq-\delta n^2
\end{equation}
and
\begin{equation}\label{eq:maxreddeg-assump}
\max_{v\in V(K_n)}d_\red(v)\geq \eta n\,.
\end{equation}
Then there exist a constant $c=c(k,\eta)>0$, an integer $\ell\geq k-1$, disjoint sets $U_0,U_1,\dots,U_\ell\subseteq V(K_n)$, and a $\Delta$-regular graph $R$ on $[\ell]$ such that, writing
$$U_\ast:=\bigcup_{i=0}^{\ell}U_i,\qquad T:=V(K_n)\setminus U_\ast\,,$$
the following hold:
\begin{enumerate}[
  label=\textit{(\roman*)},
  ref=(\textit{\roman*}),
  topsep=5pt
]
\item\label{item:kth-rec-blue-dense} $e_\blue(U_i)\geq(1-\xi)\binom{|U_i|}{2}$ for all $i\in[\ell]$.
\item\label{item:kth-rec-sizes} $|U_i|\geq cn$ and $\big||U_i|-|U_j|\big|\leq \xi n$ for all $i,j\in[\ell]$, $|U_0|\leq \xi n$, and $\abs{\bigcup_{i=1}^{\ell}U_i}\geq (\eta-\xi)n$.
\item\label{item:kth-rec-red-dens} $d_\red(U_i,U_j)\geq1-\xi$ for all $ij\in E(R)$.
\item\label{item:kth-rec-green-dens} $d_\green(U_i,U_j)\geq1-\xi$ for all $ij\notin E(R)$.
\item\label{item:kth-rec-small-bdry} $e_\red(U_\ast,U_\ast^c)+e_\blue(U_\ast,U_\ast^c)\leq \xi n^2$.
\item\label{item:kth-rec-extr-pres} $e_\red(T)-\Delta\,e_\blue(T)\geq-(\delta+\xi)n^2$.
\end{enumerate}
\end{lemma}
\begin{proof}
Throughout the proof, $C_k>0$ is a large enough constant depending only on $k$ that may change from line to line. If $k\geq4$, let $\sE:(0,1)\to(0,1)$ be the mapping $\epsilon_{\ref{thm:erdos-simonovits}}\mapsto\delta_{\ref{thm:erdos-simonovits}}$ from the statement of \Cref{thm:erdos-simonovits} with $r=\Delta$; if $k=3$, let $\sE(\epsilon)=\epsilon$. Let us define the following parameters:
\begin{align}
&\tau := \frac{1}{C_k}\min\left\{\eta^4\xi^2\,,\ \beta_0^{\ref{lemma:vertex-level-mixed-mass}}\Big(\frac{\eta}{C_k}\Big)\,,\ \eta^6\right\}\,,\label{eq:betaL}\quad L:=\ceil*{\frac{C_k}{\tau}} \,.\\[5pt]
&\beta_L := \frac{\tau^6}{C_k} \,,\quad \beta_q := \frac{1}{C_k}\min\left\{\tau^2\,,\ (\eta\,\beta_{q+1})^2\,,\ \left(\eta\,\sE(\beta_{q+1})\right)^2\right\}\,,\quad q=L-1,\dots,0\,, \label{eq:betaq}
\end{align}
where $\beta_0^{\ref{lemma:vertex-level-mixed-mass}}(\cdot)$ is the threshold from \Cref{lemma:vertex-level-mixed-mass}. We have $0<\beta_0<\cdots<\beta_L<\tau<1$.

Choose $\delta_0=\delta_0(k,\eta,\xi)$ small enough such that
\begin{equation}\label{eq:delta-small}
\delta_0 \leq \min\left\{\frac{\beta_0^{6}}{C_k^2},\,\frac{\eta}{C_k}\right\}\,.
\end{equation}
Henceforth fix $\delta\in(0,\delta_0)$ and assume $\varphi$ is a coloring satisfying the hypotheses of the lemma.

\Cref{lemma:D-tails} (applied with $\delta_{\ref{lemma:D-tails}}\gets\delta$) and \eqref{eq:delta-small} imply there exists a set $X\subseteq V(K_n)$ of at least $(1-C_k\delta^{3/4})n\geq(1-\beta_0^{4})n$ vertices such that for all $v\in X$,
$$\abs{D^\varphi(v,X)} \leq \beta_0n\,.$$
Fix a vertex
$$x^{(0)}\in\argmax_{v\in V(K_n)}d_\red^\varphi(v)$$
and set $\widetilde{X}:=X\cup\{x^{(0)}\}$. In the remainder, we work with the coloring $\psi\in\cC_k(n)$ defined by
$$\psi(e):=\begin{cases}\varphi(e)&e\subseteq \widetilde{X}\\\green&\text{otherwise}\end{cases}\,.$$
After increasing $n_0$ if necessary, for all $v\in X$ we have
\begin{equation}\label{eqn:vpr-deg-control}
\abs{D^\psi(v)} \leq 2\beta_0n\,.
\end{equation}
Since $\psi$ differs from $\varphi$ only on edges incident to $\widetilde{X}^c$, we have
\begin{equation}\label{eq:psi-near-extremal}
\Phi(\psi)\geq\Phi(\varphi)-\bigg(\binom{|\widetilde{X}^c|}{2}+|\widetilde{X}^c|\cdot n\bigg)\geq\Phi(\varphi)-\bigg(\binom{\ceil{\beta_0^{4}n}}{2}+\beta_0^{4}n^2\bigg)\geq-\deltatilde n^2\,,
\end{equation}
where $\deltatilde:=C_k\,\beta_0^{4}$. Henceforth, all uses of $d(\cdot)$ and $N(\cdot)$ are with respect to $\psi$.

\noindent\textbf{First step of the iteration.} Let $A^{(0)}:=N_\red^\psi(x^{(0)})$ and $m:=d^\psi_\red(x^{(0)})$. From \eqref{eq:maxreddeg-assump}, the choice of $x^{(0)}$, and $|\widetilde{X}^c|\leq\beta_0^4n\leq\eta n/2$, we have $m\geq\eta n/2$.

If $k=3$ then $\Delta=1$ and $A^{(0)}$ is a blue clique. Indeed, if $y,z\in A^{(0)}$ and $\psi(yz)\neq\blue$ then $\{x^{(0)},y,z\}$ spans a forbidden pattern centered at $x^{(0)}$. In this case we set $A^{(0)}_1:=A^{(0)}$. If $k\geq4$, by \eqref{eq:psi-near-extremal} the hypotheses of \Cref{lemma:locally-turan} are met with $x_{\ref{lemma:locally-turan}}=x^{(0)}$, $\delta_{\ref{lemma:locally-turan}}=\deltatilde$, $\alpha_{\ref{lemma:locally-turan}}=\beta_0^4$, and $\vartheta_{\ref{lemma:locally-turan}}=\eta/2$. Hence the edge-colored graph $H:=\psi_{\red,\green}[N_\red^\psi(x^{(0)})]$ satisfies
\begin{equation}\label{eqn:near-turan-iter0}
e(H)\geq t_{m,\Delta}-12\left(\frac{\sqrt{\deltatilde}}{\eta/2}+\frac{\beta_0^4}{\eta/2}\right)m^2\geq t_{m,\Delta}-12\beta_0 m^2\geq t_{m,\Delta}-\sE(\beta_1)\,m^2\,,
\end{equation}
where in the second and third inequalities we used \eqref{eq:betaq} and the definition $\deltatilde=C_k\beta_0^4$, specifically $\sqrt{\deltatilde}/\eta\leq \beta_0/4$, $\beta_0^4/\eta\leq\beta_0/4$, and $12\beta_0\leq\sE(\beta_1)$.
In this case, \Cref{thm:erdos-simonovits} applies to $H$, yielding a partition $A^{(0)}=A^{(0)}_1\cup\cdots\cup A^{(0)}_\Delta$. In both cases,
\begin{equation}\label{eq:A0-ES}
\sum_{i=1}^{\Delta}e_{\red,\green}(A^{(0)}_i)\leq \beta_1 m^2\,,\qquad ||A^{(0)}_i|-m/\Delta|\leq \beta_1m \,, \qquad i=1,\dots,\Delta \,.
\end{equation}

Next, let $\lambda:=(1-C_k\sqrt{\beta_1})m/\Delta$. By \eqref{eq:A0-ES}, after increasing $C_k$ and taking $n$ sufficiently large, we have $|A^{(0)}_i|-1-\lambda\geq \sqrt{\beta_1}m/C_k$. Hence Markov's inequality implies
\begin{align}
\begin{aligned}\label{eqn:stab-markov-1}
\left|\left\{v\in A^{(0)}_i:d_\blue(v,A^{(0)}_i)\leq\lambda\right\}\right| &\leq \frac{2\cdot e_{\red,\green}(A^{(0)}_i)}{|A^{(0)}_i|-1-\lambda} \leq C_k\sqrt{\beta_1}m \,.
\end{aligned}
\end{align}
For a vertex $v\in A^{(0)}_i$ satisfying $d_\blue(v,A^{(0)}_i)\geq\big(1-C_k\sqrt{\beta_1}\big)m/\Delta$, we use \eqref{eqn:vpr-deg-control} and the definition of $D(\cdot)$ to calculate that
\begin{align*}
d_\blue^\psi(v,V(K_n)\setminus A^{(0)}_i) &= d_\blue^\psi(v)-d_\blue^\psi(v,A^{(0)}_i) \\
&= \frac{d_\red^\psi(v)-D^\psi(v)}{\Delta}-d_\blue^\psi(v,A^{(0)}_i) \\
&\leq \frac{m+\beta_0^4n}{\Delta} + \frac{2\beta_0n}{\Delta} - \big(1-C_k\sqrt{\beta_1}\big)\frac{m}{\Delta} \leq C_k\sqrt{\beta_1}n \,.
\end{align*}
It follows that the set
$$W^{(0)}_i:=\left\{v\in A^{(0)}_i:d_\blue(v,A^{(0)}_i)\geq\big(1-C_k\sqrt{\beta_1}\big)\,\frac{m}{\Delta}\,,\,d_\blue(v,V(K_n)\setminus A^{(0)}_i)\leq C_k\sqrt{\beta_1}n\right\}$$
satisfies $\abs{W^{(0)}_i}\geq(1-C_k\sqrt{\beta_1})m/\Delta$. Using \eqref{eqn:vpr-deg-control}, we have that for all $v\in W^{(0)}_i$,
\begin{align*}
d_\red^\psi(v) &= \Delta\,d_\blue^\psi(v)+D^\psi(v) \\
&\geq \big(1-C_k\sqrt{\beta_1}\big)m - 2\beta_0n \geq m - C_k\sqrt{\beta_1}n
\end{align*}
and combining this with $\beta_1\leq\tau^2/C_k\leq\eta^2/C_k$ from \eqref{eq:betaq}, we have
\begin{equation}\label{eqn:vartheta-verif}
d_\red^\psi(v) \geq m - C_k\sqrt{\beta_1}n \geq (\eta /2 - C_k\sqrt{\beta_1})n \geq \eta n/4 \,.
\end{equation}
Let $W^{(0)}:=\bigcup_{i=1}^\Delta W^{(0)}_i$.

\smallskip
\noindent\textbf{Inductive step.}
Suppose that for some $1\leq q\leq L-2$ we have chosen vertices $x^{(0)},\dots,x^{(q-1)}$, with $x^{(0)}$ as above and $x^{(j)}\in X$ for $1\leq j<q$, sets $A^{(j)}:=N_\red^\psi(x^{(j)})$, partitions $A^{(j)}=\bigcup_{i=1}^\Delta A_i^{(j)}$, and typical subsets
$W_i^{(j)}\subseteq A_i^{(j)}$ for each $j=0,\dots,q-1$, with the following properties:
\begin{equation}\label{eq:Aj-sizes-uniform}
\sum_{i=1}^{\Delta}e_{\red,\green}(A^{(j)}_i)\leq \beta_{j+1}m^2 \,,\qquad \big||A_i^{(j)}|-m/\Delta\big|\leq C_k\beta_{j+1}m
\end{equation}
for all $i\in[\Delta]$;
\begin{equation}\label{eq:Wj-size-uniform}
|W_i^{(j)}|\geq \big(1-C_k\sqrt{\beta_{j+1}}\big)\,m/\Delta
\end{equation}
for all $i\in[\Delta]$; and every $v\in W_i^{(j)}$ satisfies
\begin{align}
& d^\psi_\blue(v,A_i^{(j)}) \geq \big(1-C_k\sqrt{\beta_{j+1}}\big)\,m/\Delta\,,\label{eq:Wj-good-1}\\
& d^\psi_\blue(v,V(K_n)\setminus A_i^{(j)}) \leq C_k\sqrt{\beta_{j+1}}\,n\,,\label{eq:Wj-good-2}\\
& d^\psi_\red(v) \geq m-C_k\sqrt{\beta_{j+1}}\,n\,.\label{eq:Wj-good-3}
\end{align}
For $j=0$ these are exactly the properties proved above (after enlarging the implicit constants $C_k$ once). Let $W^{(j)}:=W^{(j)}_1\cup\cdots\cup W^{(j)}_\Delta$ for all $j\in[q-1]$.

If there exists a vertex $x^{(q)}\in W^{(j)}_i$ for some $j\in[q-1]$ and $i\in[\Delta]$ such that
\begin{equation}\label{eq:choose-next-seed}
d_\red^\psi\Bigg(x^{(q)},V(K_n)\setminus\bigcup_{j=0}^{q-1}A^{(j)}\Bigg)\geq \tau n \,,
\end{equation}
then we continue the iteration as follows. Since $A^{(q)}=N_\red^\psi(x^{(q)})$, the set $A^{(q)}\setminus\bigcup_{j<q}A^{(j)}$ has size at least $\tau n$. Hence each successful step contributes at least $\tau n$ new vertices to $\bigcup_{j<q}A^{(j)}$, so the number of iterations is at most $1/\tau+1\leq L-1$, after increasing $C_k$.

By \eqref{eq:Wj-good-3} and \eqref{eq:betaq}, $d_\red^\psi(x^{(q)})\geq \eta n/4$ for all large $n$. Let $A^{(q)}:=N_\red^\psi(x^{(q)})$. If $k=3$ then $\Delta=1$ and $A^{(q)}$ is a blue clique, so we set $A^{(q)}_1:=A^{(q)}$. If $k\geq4$, apply \Cref{lemma:locally-turan} to $x^{(q)}$ with
$\delta_{\ref{lemma:locally-turan}}=\deltatilde$,
$\vartheta_{\ref{lemma:locally-turan}}=\eta/4$, and $\alpha_{\ref{lemma:locally-turan}}=C_k\sqrt{\beta_{q+1}}$. By \eqref{eq:betaq} and \eqref{eq:delta-small}, the error term in
\Cref{lemma:locally-turan} is at most $\sE(\beta_{q+2})$. Hence, writing $H:=\psi_{\red,\green}[A^{(q)}]$, \Cref{thm:erdos-simonovits} applies to $H$ and yields
a partition
$A^{(q)}=A_1^{(q)}\cup\cdots\cup A_\Delta^{(q)}$.

In both cases, the partition satisfies \eqref{eq:Aj-sizes-uniform} with $j=q$. Indeed, in the case $k=3$ this follows from $e_{\red,\green}(A^{(q)}_1)=0$, \eqref{eq:Wj-good-3}, and the bound $d_\red^\psi(x^{(q)})\leq m+\beta_0^4n$; in the case $k\geq4$ it follows from the preceding paragraph and the same estimates. Here \eqref{eq:Wj-good-3} gives $m-d_\red^\psi(x^{(q)})\leq C_k\sqrt{\beta_q}n$, while $d_\red^\psi(x^{(q)})-m\leq \beta_0^4n$, and both errors are at most $C_k\beta_{q+1}m$ by \eqref{eq:betaq}. Repeating the Markov-type argument from the first step, and discarding the possible vertex in $\widetilde{X}\setminus X$, then yields sets $W_i^{(q)}\subseteq A_i^{(q)}\cap X$ satisfying \eqref{eq:Wj-size-uniform}--\eqref{eq:Wj-good-3} with $j=q$. This completes the inductive step.

\smallskip
\noindent\textbf{Overlap dichotomy.}
We claim that for all $p,q\geq0$ and $r,s\in[\Delta]$,
\begin{equation}\label{eqn:OGP}
\abs{W^{(p)}_r\,\triangle\, W^{(q)}_s}\leq C_k\sqrt{\beta_L}n\qquad\text{or}\qquad\abs{W^{(p)}_r\cap W^{(q)}_s}=0
\end{equation}
where $C_k>0$ is large enough compared with the constants in \eqref{eq:Wj-size-uniform}--\eqref{eq:Wj-good-2}. To see this, assume $\abs{W^{(p)}_r\,\triangle\, W^{(q)}_s}\geq C_k\sqrt{\beta_L}n$, assume there exists $v\in W^{(p)}_r\cap W^{(q)}_s$, and let $C'_k$ denote the maximum of the constants in \eqref{eq:Wj-size-uniform}--\eqref{eq:Wj-good-2}. Let us assume $C_k\geq16 C'_k$ in \eqref{eqn:OGP}. We have
\begin{align*}
\begin{array}{l @{\hskip 8mm} r @{\hskip 1mm} l}
&W^{(q)}_s\setminus W^{(p)}_r &\subseteq (W^{(q)}_s\setminus A^{(p)}_r)\cup(A^{(p)}_r\setminus W^{(p)}_r) \\[5pt]
\Longrightarrow & \abs{W^{(q)}_s\setminus A^{(p)}_r} & \geq \abs{W^{(q)}_s\setminus W^{(p)}_r} - \abs{A^{(p)}_r\setminus W^{(p)}_r} \geq \dfrac{C_k\sqrt{\beta_L}n}{2} - 2C'_k\sqrt{\beta_L}n \\[7pt]
\Longrightarrow & d_\blue(v,W^{(q)}_s\setminus A^{(p)}_r) &\geq \abs{W^{(q)}_s\setminus A^{(p)}_r} - d_{\red,\green}(v,A^{(p)}_r) \geq \dfrac{C_k\sqrt{\beta_L}n}{2} - 4C'_k\sqrt{\beta_L}n \,,
\end{array}
\end{align*}
where in the last inequality we used that $d_{\red,\green}(v,A^{(p)}_r)\leq2C'_k\sqrt{\beta_L}n$, which is immediate from \eqref{eq:Aj-sizes-uniform} and \eqref{eq:Wj-good-1}.
It follows that
\begin{align*}
d_\blue^\psi(v) &\geq d_\blue(v,W^{(p)}_r) + d_\blue(v,W^{(q)}_s\setminus A^{(p)}_r) \\
&\geq (1-2C'_k\sqrt{\beta_L})\,\frac{m}{\Delta} + \left(\dfrac{C_k\sqrt{\beta_L}n}{2} - 4C'_k\sqrt{\beta_L}n\right) \\
&\geq \frac{m}{\Delta} + \left(\frac{C_k}{2}-6C'_k\right)\sqrt{\beta_L}n \geq \frac{m}{\Delta} + 2C'_k\sqrt{\beta_L}n \,,
\end{align*}
but this contradicts $d_\blue(v)\leq m/\Delta+C'_k\sqrt{\beta_L}n$ (which is immediate from \eqref{eq:Aj-sizes-uniform} and \eqref{eq:Wj-good-2}). Thus we have proven the dichotomy \eqref{eqn:OGP}. Henceforth we drop the distinction between $C_k$ and $C'_k$.

\smallskip
\noindent\textbf{Applying \Cref{lemma:vertex-level-mixed-mass}.}
Let $J\leq L-1$ be the number of steps performed in the above induction. In light of \eqref{eqn:OGP}, the relation $W^{(p)}_r\sim W^{(q)}_s$ defined by the condition $\abs{W^{(p)}_r\,\triangle\, W^{(q)}_s}\leq C_k\sqrt{\beta_L}n$ is an equivalence relation. Let
$$\{W_1,\dots,W_t\}:=\{W_i^{(q)}:0\leq q\leq J-1,\ i\in[\Delta]\}/\sim$$
be the set of equivalence classes. For each $r\in[t]$, let
$$U_r:=\bigcap_{W\in W_r}W\,,\qquad A:=\bigcup_{j=0}^{J-1}A^{(j)}\,,\qquad U:=\bigcup_{r=1}^tU_r\,.$$

Every successful step $q\in\{1,\dots,J-1\}$ creates at least one new equivalence class. Indeed, since $A^{(q)}\setminus\bigcup_{j<q}A^{(j)}$ has size at least $\tau n$, some index $i\in[\Delta]$ satisfies
$$\abs{A_i^{(q)}\setminus\bigcup_{j<q}A^{(j)}}\geq \frac{\tau n}{\Delta}\,.$$
Using \eqref{eq:Wj-size-uniform} and \eqref{eq:Aj-sizes-uniform}, this implies for all sufficiently large $n$
$$\abs{W_i^{(q)}\setminus\bigcup_{j<q}A^{(j)}}\geq \frac{\tau n}{2\Delta}>C_k\sqrt{\beta_L}n\,,$$
so $W_i^{(q)}$ is not equivalent to any set discovered earlier. Also, two distinct sets arising in the same step are disjoint and have size $\Omega_k(n)$ by \eqref{eq:Wj-size-uniform}, so they cannot be equivalent. Thus each class contains at most one representative from each step, and hence at most $L$ representatives in total.

We claim \Cref{lemma:vertex-level-mixed-mass} applies to $U=U_1\cup\cdots\cup U_t$ with $\eta_{\ref{lemma:vertex-level-mixed-mass}}=\eta/C_k$ and $\beta_{\ref{lemma:vertex-level-mixed-mass}}=C_k\tau$. We verify the hypotheses of \Cref{lemma:vertex-level-mixed-mass} one at a time.

First, disjointness of the sets $U_1,\dots,U_t$ is automatic from \eqref{eqn:OGP}. Since each class contains at most $L$ representatives and each pairwise symmetric difference inside a class has size at most $C_k\sqrt{\beta_L}n$, a union bound gives
\begin{equation}\label{eq:Ur-size}
\abs{U_i}\geq(1-C_k\sqrt{\beta_L})\frac{m}{\Delta}-\frac{L}{2}\cdot C_k\sqrt{\beta_L}n=\frac{m}{\Delta}-C_k\,L\sqrt{\beta_L}\,n\,.
\end{equation}
Moreover, if $W_i^{(p)}\in W_r$ then
\begin{equation}\label{eqn:Ajiur-ub}
\abs{A_i^{(p)}\setminus U_r}\leq \abs{A_i^{(p)}\setminus W_i^{(p)}}+\abs{W_i^{(p)}\setminus U_r}\leq C_kL\sqrt{\beta_L}n \,.
\end{equation}
Summing over the at most $\Delta L$ sets $A_i^{(p)}$ gives
\begin{equation}\label{eqn:AsmU-ub}
|A\setminus U| \leq \Delta L\cdot C_kL\sqrt{\beta_L}n\leq C_kL^2\sqrt{\beta_L}n \leq \frac{\tau}{C_k}\,n \,.
\end{equation}

We next show that $t\geq \Delta+1=k-1$. Fix $r\in[t]$ and $v\in U_r$, and choose $(i,p)$ such that $W_i^{(p)}\in W_r$. Since the iteration has terminated,
\[
d_\red^\psi(v,V(K_n)\setminus A)<\tau n \,.
\]
Using \eqref{eq:Wj-good-3} and \eqref{eqn:AsmU-ub}, we obtain
\begin{equation}\label{eqn:dredU-lb}
d_\red(v,U)\geq d_\red(v)-\tau n-\abs{A\setminus U}\geq m-C_k\sqrt{\beta_L}n-\tau n-\frac{\tau}{C_k}n\geq m-C_k\tau n \,.
\end{equation}
On the other hand, \eqref{eq:Wj-good-1} and \eqref{eqn:Ajiur-ub} imply
\[
d_\blue(v,U_r)\geq d_\blue(v,A_i^{(p)})-\abs{A_i^{(p)}\setminus U_r}\geq |U_r|-C_kL\sqrt{\beta_L}n \,,
\]
and hence
\[
d_\red(v,U_r)\leq |U_r|-d_\blue(v,U_r)\leq C_kL\sqrt{\beta_L}n \,.
\]
Therefore, using \eqref{eq:Ur-size},
\begin{equation}\label{eqn:dredU-ub}
d_\red(v,U)\leq d_\red(v,U_r)+\sum_{s\neq r}|U_s|\leq C_kL\sqrt{\beta_L}n+(t-1)\left(\frac{m}{\Delta}+C_k\beta_Ln\right) \,.
\end{equation}
Comparing \eqref{eqn:dredU-lb} and \eqref{eqn:dredU-ub}, and using $m\geq\eta n/2$, we conclude that $t\geq\Delta+1$ for all sufficiently large $n$.

In particular, for all large $n$,
\begin{equation}\label{eq:Ur-size-lower}
|U_r|\geq \eta n/C_k\qquad\text{and}\qquad \big||U_r|-|U_s|\big|\leq \tau n/C_k\,,
\end{equation}
$r,s\in[t]$, where the second estimate uses $\abs{U_r}\leq m/\Delta+C_k\beta_Ln$, \eqref{eq:Ur-size}, and $\beta_L\leq \tau^4/C_k$, completing the verification of \Cref{item:vertex-level-mixed-mass-1}. \Cref{item:vertex-level-mixed-mass-2} follows from \eqref{eq:Wj-size-uniform}--\eqref{eq:Wj-good-2}: for all distinct $r,s\in[t]$ we have
$$d_\blue(U_r,U_s)=\frac{e_\blue(U_r,U_s)}{|U_r|\cdot|U_s|}\leq\frac{m\cdot C_k\sqrt{\beta_L}n}{(m/\Delta-C_kL\sqrt{\beta_L}n)^2}\leq\frac{C_k\sqrt{\beta_L}n^2}{m^2/4\Delta^2}\leq 4\Delta^2C_k\,\frac{\sqrt{\beta_L}}{\eta^2}\leq\frac{\tau}{C_k}\,,$$
where we used \eqref{eq:betaL}. It remains to verify \Cref{item:vertex-level-mixed-mass-3}. Fix $v\in U_r$ and let $(i,p)$ be indices such that $U_r\subseteq W^{(p)}_i$. We calculate that
\begin{align}
d_\blue(v,U_r) &= d_\blue(v,A^{(p)}_i) - d_\blue(v,A^{(p)}_i\setminus U_r) \nonumber\\[3pt]
&\geq d_\blue(v,A^{(p)}_i) - \abs{A^{(p)}_i\setminus W^{(p)}_i} - \abs{W^{(p)}_i\setminus U_r} \nonumber\\[3pt]
&\geq (1-C_k\sqrt{\beta_L})\,\frac{m}{\Delta} - C_k\sqrt{\beta_L}m - (|W^{(p)}_i| - |U_r|) \label{eqn:dbvur-lb}\\
&\geq (1-C_k\sqrt{\beta_L})(|A^{(p)}_i|-C_k\beta_Lm) - C_k\sqrt{\beta_L}m \nonumber\\[5pt]
&\hspace{3cm}- \left(\left(\frac{m}{\Delta}+C_k\beta_Lm\right) - \left(\frac{m}{\Delta} - C_kL\sqrt{\beta_L}n\right)\right) \nonumber\\
&\geq (1-C_k\sqrt{\beta_L})|U_r| - C_kL\sqrt{\beta_L}n \nonumber\\
&\geq \left(1-\frac{C_kL\sqrt{\beta_L}}{\eta}\right)|U_r| \geq \left(1-\frac{\tau}{C_k}\right)|U_r| \,, \nonumber
\end{align}
where the third line used \eqref{eq:Aj-sizes-uniform}--\eqref{eq:Wj-good-1}, the fourth used \eqref{eq:Aj-sizes-uniform} and \eqref{eq:Ur-size}, and the last line used \eqref{eq:Ur-size-lower}. Next, using the inclusion
$$U\setminus U_r\subseteq (U\setminus A^{(p)}_i)\cup(A^{(p)}_i\setminus U_r)\subseteq(V(K_n)\setminus A^{(p)}_i)\cup(A^{(p)}_i\setminus U_r)\,,$$
we calculate that
\begin{equation}
\begin{aligned}\label{eqn:dbUUr-lb}
d_\blue(v,U\setminus U_r) &\leq d_\blue(v,V(K_n)\setminus A^{(p)}_i)+\abs{A^{(p)}_i\setminus U_r} \\
&\leq C_k\sqrt{\beta_L}n + C_kL\sqrt{\beta_L}n \leq \frac{C_k\sqrt{\beta_L}}{\tau}\,n \leq \frac{\tau^2}{C_k}\,n\,,
\end{aligned}
\end{equation}
where we used \eqref{eq:Wj-good-2} and the same bound on $\abs{A^{(p)}_i\setminus U_r}$ as in the previous display.

To verify that $\abs{D(v,U)}\leq C_k\tau n$, we need two-sided bounds on $d_\red(v,U)$ and $d_\blue(v,U)$.
From the condition \eqref{eq:choose-next-seed} defining the iteration, we have $d_\red(v,A)\geq d_\red(v)-\tau n$, hence
\begin{align*}
d_\red(v,U) &\leq d_\red(v)\leq m+\beta_0^4n\leq m+C_k\tau n\,,\\[5pt]
d_\red(v,U) &\geq d_\red(v,A) - \abs{A\setminus U} \\
&\geq d_\red(v) - \tau n - \Delta L\cdot\max_{r\in[t]}\max_{(i,j):W^{(j)}_i\in W_r}\abs{A^{(j)}_i\setminus U_r} \\
&\geq m - C_k\sqrt{\beta_L}n - \tau n - C_kL^2\sqrt{\beta_L}n \\
&\geq m - C_k\tau n \,,
\end{align*}
where in the first line we used the choice of $x^{(0)}$, $|\widetilde{X}^c|\leq\beta_0^4n$, and $\beta_0^4\leq C_k\tau$; in the third line we used that there are at most $\Delta L$ distinct sets $A^{(j)}_i$; and in the fourth line we used \eqref{eq:Wj-good-3} and the estimate $\abs{A^{(j)}_i\setminus U_r}\leq C_kL\sqrt{\beta_L}n$. We also have
\begin{align*}
\frac{m}{\Delta} - C_kL\sqrt{\beta_L}n &\leq d_\blue(v,U_r) \\
&\leq d_\blue(v,U) \\
&\leq |U_r| + d_\blue(v,U\setminus U_r) \\
&\leq \frac{m}{\Delta}+C_k\tau^2 n\,,
\end{align*}
where the first inequality follows from \eqref{eqn:dbvur-lb}, and the fourth inequality used \eqref{eqn:dbUUr-lb}. Combining the two-sided bounds on $d_\red(v,U)$ and $d_\blue(v,U)$, we have
\begin{align*}
D(v,U)&\geq \left(m-C_k\tau n\right)-\Delta\left(\frac{m}{\Delta}+C_k\tau^2n\right)\geq -C_k\tau n\,, \\[10pt]
D(v,U)&\leq m+C_k\tau n-\Delta\left(\frac{m}{\Delta}-C_kL\sqrt{\beta_L}n\right)\leq C_k\tau n\,,
\end{align*}
completing the verification of \Cref{item:vertex-level-mixed-mass-3}.

Define the parameter
$$\mu:=C_k\,\frac{\sqrt{\tau}}{\eta^2}\,.$$
By \eqref{eq:betaL}, we have $\tau\leq \eta^4\xi^2/C_k^3$, so after enlarging $C_k$ if needed, we may assume $\mu\leq \xi/20$ and $\mu\leq 1/4$. It follows from \Cref{lemma:vertex-level-mixed-mass}, with $\beta=C_k\tau$, that for all $1\leq i<j\leq t$,
\begin{equation}\label{eq:dominant-color-UiUj}
\max\{d_\red(U_i,U_j),\,d_\green(U_i,U_j)\} \geq 1-\mu \,.
\end{equation}

\smallskip
\noindent\textbf{Reduced graph and $\Delta$-regularity.}
Let $R$ be the graph on $[t]$ such that $rs\in E(R)$ if and only if $d_\red(U_r,U_s)\geq 1/2$. Since $d_\blue(U_r,U_s)\leq\tau/C_k\leq \mu$ for all distinct $r,s$, \eqref{eq:dominant-color-UiUj} implies that
\begin{equation}\label{eq:red-green-dichotomy-UiUj}
rs\in E(R)\Longrightarrow d_\red(U_r,U_s)\geq 1-\mu\,,\qquad rs\notin E(R)\Longrightarrow d_\red(U_r,U_s)\leq \mu\ \text{ and }\ d_\green(U_r,U_s)\geq 1-\mu \,.
\end{equation}
To prove $R$ is $\Delta$-regular, fix $r\in[t]$. Summing the identity $D(v,U)=d_\red(v,U)-\Delta d_\blue(v,U)$ over $v\in U_r$, and using $|D(v,U)|\leq C_k\tau n$, we obtain
\begin{align}
\sum_{s\in[t]\setminus\{r\}} e_\red(U_r,U_s) &= \sum_{v\in U_r}D(v,U) - \sum_{v\in U_r}d_\red(v,U_r) + \sum_{v\in U_r}\Delta\,d_\blue(v,U) \nonumber\\
&= O_k(\tau n^2) + |U_r|\cdot O_k(\sqrt{\beta_L}n) + |U_r|\cdot\Delta\left(\frac{m}{\Delta}+O_k(\tau^2n)\right) \nonumber\\
&= \frac{m^2}{\Delta} + O_k(\tau n^2) \,. \label{eqn:erurus-rhs}
\end{align}
On the other hand, by definition of $R$,
\begin{align*}
\sum_{s\neq r} e_\red(U_r,U_s) &= \sum_{s\in N_R(r)} (1+O_k(\mu))|U_r||U_s| + \sum_{s\in V(R)\setminus(N_R(r)\cup\{r\})} O_k(\mu)|U_r||U_s| \\
	&= d_R(r)\cdot\frac{m^2}{\Delta^2}+O_{k}(\mu n^2)
\end{align*}
where the second equality used \eqref{eq:Ur-size-lower}. Comparing with \eqref{eqn:erurus-rhs} and using $|U_r|\geq \eta n/C_k$, we obtain $d_R(r)=\Delta$ for all large enough $n$, so $R$ is $\Delta$-regular.

\smallskip
\noindent\textbf{Defining the output sets and verifying (i)--(vi).}
Let $\ell:=t$ and define
$$U_0:=(\widetilde{X}^c)\cup(A\setminus U)\,,\qquad U_\ast:=\bigcup_{i=0}^{\ell}U_i\,,\qquad T:=V(K_n)\setminus U_\ast\,.$$
Note that $U\subseteq A\subseteq \widetilde{X}$, hence $T\subseteq \widetilde{X}$.

All density statements below concerning pairs of sets contained in $\widetilde{X}$ are with respect to $\psi$ (as in the preceding proof) and thus hold for $\varphi$ as well, since $\psi$ agrees with $\varphi$ on $\widetilde{X}$.

\smallskip
\noindent\emph{Proof of \ref{item:kth-rec-blue-dense}.}
Fix $i\in[\ell]$ and $v\in U_i$. The verification of \Cref{item:vertex-level-mixed-mass-3} above (specifically the lower bound on $d_\blue(v,U_r)$) yields
$$d_\blue(v,U_i)\geq \left(1-\frac{\tau}{C_k}\right)|U_i|\geq (1-\xi)|U_i|\,,$$
so double counting gives $e_\blue(U_i)\geq (1-\xi)\binom{|U_i|}{2}$, proving \ref{item:kth-rec-blue-dense}.

\smallskip
\noindent\emph{Proof of \ref{item:kth-rec-sizes}.}
Since $|\widetilde{X}^c|\leq \beta_0^{4}n$ and $\abs{A\setminus U}\leq \tau n/C_k$ by \eqref{eqn:AsmU-ub}, we have $|U_0|\leq \xi n$.

Next, \eqref{eq:Ur-size-lower} gives $|U_i|\geq cn$, with $c=\eta/C_k$, and $\big||U_i|-|U_j|\big|\leq \tau n/C_k\leq \xi n$ for all $i,j\in[\ell]$.
Finally, since $W^{(0)}\subseteq A^{(0)}$ and $|W^{(0)}|\geq (1-C_k\sqrt{\beta_1})m$ by \eqref{eq:Wj-size-uniform}, while $W^{(0)}\subseteq A\subseteq U\cup U_0$, we have
$$|U|\geq |W^{(0)}|-|U_0|\geq (1-C_k\sqrt{\beta_1})m-\xi n\,.$$
Moreover $m=d_\red^\psi(x^{(0)})\geq \max_{v\in V(K_n)}d_\red^\varphi(v)-|\widetilde{X}^c|\geq \eta n-\xi n$ for all large $n$, so $|U|\geq (\eta-\xi)n$. This completes the proof of \ref{item:kth-rec-sizes}.

\smallskip
\noindent\emph{Proof of \ref{item:kth-rec-red-dens} and \ref{item:kth-rec-green-dens}.}
By \eqref{eq:red-green-dichotomy-UiUj} and by definition of $\mu$, if $ij\in E(R)$ then
\[
d_\red(U_i,U_j)\geq 1-\mu\geq 1-\xi \,,
\]
proving \ref{item:kth-rec-red-dens}. Likewise, if $ij\notin E(R)$ then
\[
d_\green(U_i,U_j)\geq 1-\mu\geq 1-\xi \,,
\]
which proves \ref{item:kth-rec-green-dens}.

\smallskip
\noindent\emph{Proof of \ref{item:kth-rec-small-bdry}.}
Since $U\subseteq A$ and $T=V(K_n)\setminus U_\ast=V(K_n)\setminus((\widetilde{X}^c)\cup A)$, every vertex $v\in U$ satisfies
\[
d_\red^\psi(v,T)\leq d_\red^\psi(v,V(K_n)\setminus A)<\tau n \,.
\]
Also, if $v\in U_r\subseteq W_i^{(p)}$ then $T\subseteq V(K_n)\setminus A_i^{(p)}$, so \eqref{eq:Wj-good-2} gives
\[
d_\blue^\psi(v,T)\leq d_\blue^\psi(v,V(K_n)\setminus A_i^{(p)})\leq C_k\sqrt{\beta_L}n \,.
\]
Summing these bounds over $v\in U$ yields
\[
e_\red^\psi(U,T)+e_\blue^\psi(U,T)\leq C_k(\tau+\sqrt{\beta_L})n^2 \leq \frac{\xi n^2}{2}
\]
by the choice of $\tau$. Since $\abs{U_0}\leq \beta_0^4n+\tau n/C_k\leq \xi n/2$ after decreasing $\delta_0$ if needed, we also have
\[
e_\red^\varphi(U_0,T)+e_\blue^\varphi(U_0,T)\leq \abs{U_0}n\leq \frac{\xi n^2}{2} \,.
\]
Combining the two bounds and using that $U_\ast=U_0\cup U$ yields
\[
e_\red(U_\ast,U_\ast^c)+e_\blue(U_\ast,U_\ast^c)=e_\red(U_\ast,T)+e_\blue(U_\ast,T)\leq \xi n^2 \,,
\]
which proves \ref{item:kth-rec-small-bdry}.

\smallskip
\noindent\emph{Proof of \ref{item:kth-rec-extr-pres}.}
Since $T\subseteq \widetilde{X}$, the induced colorings $\psi|_T$ and $\varphi|_T$ coincide, and hence
$$e_\red(T)-\Delta e_\blue(T)=\Phi(\varphi|_T)=\Phi(\psi|_T)\,.$$
Decompose $\Phi(\psi)$ over $U_\ast$, $T$, and the cut:
$$\Phi(\psi)=\Phi(\psi|_{U_\ast})+\Phi(\psi|_T)+\big(e_\red^\psi(U_\ast,T)-\Delta e_\blue^\psi(U_\ast,T)\big)\,.$$
Using \eqref{eq:recursion-assumption-phi-new} and that $\Phi(\psi)\geq \Phi(\varphi)-O(|\widetilde{X}^c|n)\geq -\delta n^2-\xi n^2$ (by $|\widetilde{X}^c|\leq \xi n$), together with the bound $e_\red(U_\ast,T)\leq \xi n^2$ from \ref{item:kth-rec-small-bdry} and the trivial inequality $-\Delta e_\blue(U_\ast,T)\leq 0$, we obtain
$$\Phi(\psi|_T)\geq -(\delta+2\xi)n^2-\Phi(\psi|_{U_\ast})\,.$$
Finally, since $\psi|_{U_\ast}\in\cC_k(|U_\ast|)$, \Cref{lemma:kthOrderMantel} gives $\Phi(\psi|_{U_\ast})\leq \Delta |U_\ast|/2\leq \Delta n/2\leq \xi n^2$ for all sufficiently large $n$.
Therefore $\Phi(\psi|_T)\geq -(\delta+3\xi)n^2$, and running the argument with $\xi/3$ in place of $\xi$ in the parameter hierarchy gives
$$e_\red(T)-\Delta e_\blue(T)\geq -(\delta+\xi)n^2\,,$$
which is \ref{item:kth-rec-extr-pres}.
\end{proof}

\begin{proof}[Proof of \Cref{thm:kth-order-stability}]
First, if $e_\red(\varphi)\leq\epsilon n^2/4$ and we take $\delta\leq\epsilon/4$ then by \eqref{eqn:near-extremality},
$$e_\blue(\varphi)=\frac{e_\red(\varphi)-\Phi(\varphi)}{\Delta}\leq\frac{\epsilon n^2/4+\delta n^2}{\Delta}\leq\frac{\epsilon n^2}{2}\,,$$
which implies $d(\varphi,\cE_k(n))\leq\epsilon n^2$, since it requires at most $\epsilon n^2$ edits to make all edges of $\varphi$ green. Thus in the remainder we may assume $e_\red(\varphi)\geq\epsilon n^2/4$.

Fix $\eta:=\epsilon/100$, let $r=r(\epsilon)\in\N$ be minimal such that $(1-\eta/2)^r\leq \epsilon/20$, and set $A:=(20/\epsilon)^2$. Let $C=C(k,\epsilon)>0$ be a sufficiently large constant for the editing estimates below. We choose positive constants $\Gamma_r,\Gamma_{r-1},\dots,\Gamma_0$ and $\xi_{r-1},\dots,\xi_0$ backwards as follows. First define $\Gamma_r:=\eta$. For $s=r-1,\dots,0$, having chosen $\Gamma_{s+1}$, choose $\xi_s\in(0,\eta/20)$ sufficiently small such that $A\xi_s\leq\Gamma_{s+1}/2$ and $\xi_s\leq \epsilon\eta/(8C)$. Let $\delta_s^\ast:=\delta_0^{\ref{lemma:kth-order-recursion}}(k,\eta,\xi_s)$ and $n_s^\ast:=n_0^{\ref{lemma:kth-order-recursion}}(k,\eta,\xi_s)$ be the thresholds from \Cref{lemma:kth-order-recursion}, and define
$$\Gamma_s:=\min\left\{\delta_s^\ast/2,\,\Gamma_{s+1}/(2A),\,\eta\right\}\,.$$
Finally choose $\delta>0$ small enough that $\delta\leq\min\{\Gamma_0,\epsilon/4\}$, and take $n$ large enough that $n\geq (20/\epsilon)\max_{0\leq s<r}n_s^\ast$.

We now define an iterative decomposition of $V(K_n)$.
Set $W_0:=V(K_n)$, $n_0:=n$, and $\delta_0:=\delta$, and for each $s\geq0$ let $\varphi^{(s)}:=\varphi|_{W_s}$.
Note that $\varphi^{(s)}\in\cC_k(n_s)$ for every $s$.
We maintain the inductive hypothesis
\begin{equation}\label{eq:iter-phi-lb-fixed}
\Phi(\varphi^{(s)})\geq -\delta_s n_s^2\,.
\end{equation}
We also maintain that the coloring already constructed on $V(K_n)\setminus W_s$ is a member of $\cE_k(n-n_s)$, every edge between $W_s$ and $V(K_n)\setminus W_s$ is green, and $\delta_s\leq\Gamma_s$.

\smallskip
\noindent\textbf{Stopping conditions.}
Fix $s\geq0$ and assume \eqref{eq:iter-phi-lb-fixed}.
If $n_s\leq \epsilon n/20$ then we recolor every non-green edge inside $W_s$ green. Since all edges between $W_s$ and its complement are already green, this uses at most $n_s^2\leq \epsilon n^2/20$ edits, and the process terminates.

Assume now $n_s\geq\epsilon n/20$. If $\max_{v\in W_s}d_\red^{\varphi^{(s)}}(v)<\eta n_s$ then $e_\red(\varphi^{(s)})\leq \eta n_s^2/2$ and hence by \eqref{eq:iter-phi-lb-fixed},
$$e_\blue(\varphi^{(s)})=\frac{e_\red(\varphi^{(s)})-\Phi(\varphi^{(s)})}{\Delta}\leq \frac{\eta n_s^2/2+\delta_s n_s^2}{\Delta}\,.$$
Thus the number of non-green edges inside $W_s$ is at most
$$e_\red(\varphi^{(s)})+e_\blue(\varphi^{(s)})\leq(\eta+\delta_s)n_s^2\leq2\eta n_s^2\,,$$
where we used $\delta_s\leq\Gamma_s\leq\eta$. Therefore recoloring every non-green edge inside $W_s$ green uses at most $2\eta n_s^2\leq 2\eta n^2$ edits, and the process terminates.

\smallskip
\noindent\textbf{Applying \Cref{lemma:kth-order-recursion}.}
Finally assume $n_s\geq\epsilon n/20$ and
$$\max_{v\in W_s}d_\red^{\varphi^{(s)}}(v)\geq \eta n_s\,.$$
Since $n_s\leq(1-\eta/2)^s n$ at every step at which the process continues, the present assumption implies $s<r$. Also $n_s\geq\epsilon n/20\geq n_s^\ast$ and $\delta_s\leq\Gamma_s\leq\delta_s^\ast/2$. Apply \Cref{lemma:kth-order-recursion} to $\varphi^{(s)}$ on $W_s$ with parameters $(\delta_{\ref{lemma:kth-order-recursion}},\eta_{\ref{lemma:kth-order-recursion}},\xi_{\ref{lemma:kth-order-recursion}})=(\delta_s,\eta,\xi_s)$.
We obtain an integer $\ell_s\geq k-1$, disjoint sets $V^{(s)}_0,V^{(s)}_1,\dots,V^{(s)}_{\ell_s}\subseteq W_s$, and a $\Delta$-regular graph $R_s$ on $[\ell_s]$ satisfying \ref{item:kth-rec-blue-dense}--\ref{item:kth-rec-extr-pres} of \Cref{lemma:kth-order-recursion}.
Write
$$S_s:=\bigcup_{i=0}^{\ell_s}V^{(s)}_i\,,\qquad C_s:=\bigcup_{i=1}^{\ell_s}V^{(s)}_i\,,\qquad W_{s+1}:=W_s\setminus S_s\,,\qquad n_{s+1}:=|W_{s+1}|\,.$$
By \ref{item:kth-rec-sizes} and $\xi_s\leq \eta/20$ we have $|C_s|\geq (\eta-\xi_s)n_s\geq \eta n_s/2$, hence $n_{s+1}\leq (1-\eta/2)n_s$. In particular $n_s\leq (1-\eta/2)^s n$, so whenever the process continues we must have $s<r$.

\smallskip
\noindent\textbf{Updating the near-extremality parameter.}
By \ref{item:kth-rec-extr-pres} we have $\Phi(\varphi^{(s+1)})\geq-(\delta_s+\xi_s)n_s^2$.
If $n_{s+1}>\epsilon n/20$ then $n_s/n_{s+1}\leq 20/\epsilon$, so \eqref{eq:iter-phi-lb-fixed} holds with $s\gets s+1$ after defining
$$\delta_{s+1}:=(\delta_s+\xi_s)\left(\frac{20}{\epsilon}\right)^2\,.$$
Moreover $\delta_{s+1}\leq A\Gamma_s+A\xi_s\leq\Gamma_{s+1}$, so the parameter induction is preserved. If $n_{s+1}\leq\epsilon n/20$, no further near-extremality estimate is needed.

\smallskip
\noindent\textbf{Editing $C_s$ to cores and making the cut green.}
Let $R_s$ have connected components $Q_{s,1},\dots,Q_{s,t_s}$.
For each $a\in[t_s]$, let $C_{s,a}:=\bigcup_{i\in V(Q_{s,a})}V^{(s)}_i$.
Using \ref{item:kth-rec-blue-dense}, \ref{item:kth-rec-sizes}, \ref{item:kth-rec-red-dens}, and \ref{item:kth-rec-green-dens}, for each $a$ we can modify the coloring on $C_{s,a}$ by changing at most $C\xi_s |C_{s,a}|^2$ edges so that it becomes a blow-up of $Q_{s,a}$ with all within-part edges blue and all between-part edges constant (red on edges of $Q_{s,a}$ and green otherwise).
By moving at most $C\xi_s |C_{s,a}|$ vertices between the parts inside $C_{s,a}$, we may additionally arrange that the part sizes differ by at most $1$, at the cost of at most $C\xi_s |C_{s,a}|^2$ further edge edits.
After these edits, the induced coloring on $C_{s,a}$ is an element of $\cR_k(|C_{s,a}|)$.

If $a\neq b$ then every pair of parts from $Q_{s,a}$ and $Q_{s,b}$ corresponds to a nonedge of $R_s$, so \ref{item:kth-rec-green-dens} implies that there are at most $\xi_s |C_{s,a}||C_{s,b}|$ non-green edges between $C_{s,a}$ and $C_{s,b}$. Recoloring all these edges green therefore requires at most $\xi_s |C_s|^2\leq \xi_s n_s^2$ edits.

Next, we recolor all non-green edges incident to $V^{(s)}_0$ inside $W_s$ green. Since $\abs{V^{(s)}_0}\leq \xi_s n_s$, this requires at most $\xi_s n_s^2$ edits.

Finally, \ref{item:kth-rec-small-bdry} gives
$$e_\red(S_s,W_{s+1})+e_\blue(S_s,W_{s+1})\leq \xi_s n_s^2 \,,$$
so recoloring all remaining non-green edges in $E(C_s,W_{s+1})$ green requires at most $\xi_s n_s^2$ edits. At this point, the partially-edited coloring, restricted to $V(K_n)\setminus W_{s+1}$, is a member of $\cE_k(n-n_{s+1})$, and every edge between $W_{s+1}$ and its complement is green, so the structural inductive hypothesis is preserved.

We finally tally the total edge count. Let $\xi_{\max}:=\max_{0\leq s<r}\xi_s$. Each application of \Cref{lemma:kth-order-recursion} at step $s$ produces cores on $C_s$ and makes all edges incident to $S_s$ consistent with the target structure at a total cost of at most $C\xi_{\max} n_s^2$ edits.
Since $n_{s+1}\leq (1-\eta/2)n_s$ whenever we apply \Cref{lemma:kth-order-recursion}, we have $\sum_{s\geq0}n_s^2\leq (2/\eta)n^2$.
Thus the total number of edits over all core-extraction steps is at most $(2C\xi_{\max}/\eta)n^2\leq \epsilon n^2/4$.
When the process terminates, the final stopping step contributes at most $\max\{2\eta,\epsilon/20\}n^2\leq\epsilon n^2/20$ further edits.
Altogether we obtain a coloring $\psi\in\cE_k(n)$ with $d(\varphi,\psi)\leq \epsilon n^2$, completing the proof.
\end{proof}

\section*{Acknowledgments}
The author is grateful to Will Perkins for many helpful conversations. The author is supported by an Algorithms and Randomness Center Fellowship and an NSF Graduate Research Fellowship.

\printbibliography

\appendix
\markboth{APPENDIX}{APPENDIX}

\section{Deferred Results}
\label{sec:deferred}
\begin{lemma}\label{lemma:type-lemma}
For all $\delta>0$ and integers $\ell,L,t\geq1$ there exists $\epsilon^\ast=\epsilon^\ast(\delta,\ell)>0$ such that for all $\eta\in(0,\epsilon^\ast)$ there exist integers $U=U(\delta,\ell,L,t,\eta)$ and $n_0=n_0(\delta,\ell,L,t,\eta)$ with the following property. Let $G$ be a graph on $n\geq n_0$ vertices and let $P'=\{V_1',\dots,V_s'\}$ be a partition of $V(G)$ such that $1\leq s\leq t$ and all parts differ in size by at most 1.
Then there exist an integer $r\geq1$ and an $(\eta,\delta,\ell)$-type $(P,R)$ for $G$ with
$P=\{V_0,V_1,\dots,V_k\}$ such that:
\begin{enumerate}[
  label=\textit{(\roman*)},
  ref=(\textit{\roman*}),
  topsep=5pt
]
    \item $|V_0|\leq \eta n$.
    \item $L\leq k\leq U$.
    \item There is a partition $[k]=J_1\cup\cdots\cup J_s$ with $|J_i|=r$ for all $i\in[s]$ such that $V_i'\setminus V_0=\bigcup_{j\in J_i} V_j$ and $V_j\subseteq V_i'$ for all $j\in J_i$.
\end{enumerate}
\end{lemma}
\begin{proof}
The lemma follows by slightly modifying the proofs in \cite{bottcher2012perfect}. Note that we use the labels $(\red,\green,\blue)$ in types instead of $(\frac{1}{2},0,1)$.

First, apply the induced embedding lemma \cite[Lemma~2.9]{bottcher2012perfect} with $k_{2.9}=k'_{2.9}=\ell$ and $d_{2.9}=\delta$, yielding $\epsilon_{2.9}$ and $\epsilon'_{2.9}$. Set $\epsilon^\ast:=\epsilon_{2.9}$ and assume $\eta\in(0,\epsilon^\ast)$ is given. Next, apply the type lemma \cite[Lemma~2.7]{bottcher2012perfect} with the following enhancement: in the proof of \cite[Lemma~2.7]{bottcher2012perfect}, instead of applying the standard regularity lemma \cite[Lemma~2.2]{bottcher2012perfect}, apply the standard energy-increment proof of the regularity lemma starting from an initial partition of size at most $t$. The additional uniformity in item~(iii) is obtained by refining every current class into the same number of children at each refinement step. This only changes the bound $U$. We apply this enhanced version of \cite[Lemma~2.7]{bottcher2012perfect} with $\epsilon_{2.7}=\eta$, $\epsilon'_{2.7}=\epsilon'_{2.9}$, $k'_{2.7}=\ell$, $(k_0)_{2.7}=L$, and $t_{2.7}=t$, yielding $U:=u_{2.7}$ and $n_0:=(n_0)_{2.7}$.
\end{proof}

\GraphonSeq*
\begin{proof}
Let $G_n:=G(n,W)$ for all $n\geq1$. Since $t_{\ind}(F,W)=0$, each $G_n$ is induced-$F$-free with probability $1$. We now follow the argument of \citeauthor{lovasz2006limits} (LS) \cite[Lemmas~5.1~\&~5.2]{lovasz2006limits}. We explain the only substantive change that is needed from their argument: the use of types in the sense of \Cref{dfn:type} instead of weak regularity partitions.

In the proof of \cite[Lemma~5.1]{lovasz2006limits}, the weak regularity lemma \cite[Lemma~4.2]{lovasz2006limits} is used to produce, for all $m\geq1$, a subsequence $\{G^m_n\}_{n\geq1}$ wherein each graph has a weakly $\frac{1}{m}$-regular partition $P_{n,m}$ into the same number $t_m$ of parts. In the induction step from $m$ to $m+1$, LS refine $P_{n,m}$ to obtain a weakly $\frac{1}{m+1}$-regular partition $P_{n,m+1}$ of $G^m_n$. The refinement divides each of the parts of $P_{n,m}$ into the same number $r_m$ of parts, so $P_{n,m+1}$ has $t_{m+1}=r_mt_m$ parts. The numbers $r_m$ arise from applying the weak regularity lemma with $\epsilon=\frac{1}{m+1}$.

We now explain how we tweak the induction step. We distinguish two kinds of partitions of the vertex sets $V(G_n)$: \emph{type partitions} (the partitions $P^\type_{n,m}$ associated with a type $(P^\type_{n,m},R)$ of $G_n$) and \emph{clean partitions} (partitions constructed \emph{from} type partitions by equitably distributing the exceptional set vertices to the non-exceptional classes). The purpose of clean partitions is that we use them as the initial partition $P'$ in \Cref{lemma:type-lemma}.

Fix the constant $\epsilon^\ast:=\epsilon^\ast(\delta,v(F))$ returned by \Cref{lemma:type-lemma}, and set $G_n^1:=G_n$ for all $n\geq1$. Apply \Cref{lemma:type-lemma} with $\delta_{\ref{lemma:type-lemma}}=\delta$, $\ell_{\ref{lemma:type-lemma}}=v(F)$, $L_{\ref{lemma:type-lemma}}=1$, $t_{\ref{lemma:type-lemma}}=1$, and $\eta_{\ref{lemma:type-lemma}}=\min\{\epsilon^\ast/2,1\}$, using the trivial partition $\{V(G_n)\}$ as the initial partition. Passing to a subsequence if necessary, we may assume that all resulting types have the same number $t_1$ of non-exceptional classes. Denote the corresponding type partitions by $P^\type_{n,1}$ and the corresponding colored graphs by $R_{n,1}$. Now define $P^\clean_{n,1}$ from $P^\type_{n,1}$ by distributing the vertices of the exceptional set $V^\type_{0,n,1}$ among the non-exceptional classes so that the resulting class sizes differ by at most $1$. For $m\geq1$, assume we have already formed clean partitions $P^\clean_{n,m}$.

Set $\eta_{m+1}:=\min\{\epsilon^\ast/2,1/m\}$. We apply \Cref{lemma:type-lemma} with $\delta_{\ref{lemma:type-lemma}}=\delta$, $\ell_{\ref{lemma:type-lemma}}=v(F)$, $L_{\ref{lemma:type-lemma}}=2t_m$, $t_{\ref{lemma:type-lemma}}=t_m$, and $\eta_{\ref{lemma:type-lemma}}=\eta_{m+1}$, yielding integers $u_{m+1}:=U_{\ref{lemma:type-lemma}}$ and $n_{m+1}:=(n_0)_{\ref{lemma:type-lemma}}$.

For $n\geq n_{m+1}$, \Cref{lemma:type-lemma} has the implication that with the initial partition $P'_{\ref{lemma:type-lemma}}=P^{\clean}_{n,m}$ and $s_{\ref{lemma:type-lemma}}=t_m$, the graph $G^m_n$ has an $(\eta_{m+1},\delta,v(F))$-type $(P^{\type}_{n,m+1},R_{n,m+1})$ with
$$P^{\type}_{n,m+1}=\{V^{\type}_{0,n,m+1},V^{\type}_{1,n,m+1},\dots,V^{\type}_{t_{m+1},n,m+1}\}$$
and a partition $[t_{m+1}]=J_1\cup\cdots\cup J_{t_m}$ with each $|J_i|=r_m$ such that
$$V^{\clean}_{i,n,m}\setminus V^{\type}_{0,n,m+1}=\bigcup_{j\in J_i}V^{\type}_{j,n,m+1}$$
and $V^{\type}_{j,n,m+1}\subseteq V^{\clean}_{i,n,m}$ for $j\in J_i$. Passing to a further subsequence if necessary, we may assume that the integers
$r$ and $k$ returned by \Cref{lemma:type-lemma} are independent of $n$.
We denote them by $r_m$ and $t_{m+1}$, respectively. In particular, $t_{m+1}=r_m t_m$.
Finally, define $P^{\clean}_{n,m+1}$ from $P^{\type}_{n,m+1}$ by distributing the vertices of $V^{\type}_{0,n,m+1}$ among the classes $V^{\type}_{1,n,m+1},\dots,V^{\type}_{t_{m+1},n,m+1}$ in such a way that each vertex of $V^{\type}_{0,n,m+1}$ is assigned to one of the $r_m$ children lying inside its parent class $V^{\clean}_{i,n,m}$. Then $P^{\clean}_{n,m+1}$ is a uniform refinement of $P^{\clean}_{n,m}$ in the sense that each class of $P^{\clean}_{n,m}$ is subdivided into exactly $r_m$ classes of $P^{\clean}_{n,m+1}$, and the classes of $P^{\clean}_{n,m+1}$ differ in size by at most $1$.

The rest of the argument in \cite[Lemma~5.1]{lovasz2006limits} and the proof of \cite[Lemma~5.2]{lovasz2006limits} now apply with the clean partitions $P^\clean_{n,m}$ in place of the weakly regular partitions of LS. Let $Q^\clean_{n,m}$ denote the $t_m\times t_m$ density matrix of the clean partition $P^\clean_{n,m}$. Since each class of $P^\clean_{n,m}$ is subdivided into exactly $r_m$ classes of $P^\clean_{n,m+1}$, and all classes of $P^\clean_{n,m}$ and $P^\clean_{n,m+1}$ differ in size by at most $1$, the matrices $Q^\clean_{n,m}$ satisfy the block-averaging relation of \cite[Lemma~5.1(ii)]{lovasz2006limits} up to an error $o_n(1)$ for each fixed $m$. Passing to the limit $n\to\infty$, the limiting matrices $Q_m$ therefore satisfy the exact relation in \cite[Lemma~5.1(ii)]{lovasz2006limits}. Hence we obtain the following:
\begin{enumerate}[(\emph{\alph*})]
    \item A subsequence $\{G'_n\}_{n\geq1}$ of $\{G_n\}_{n\geq1}$ such that for all $n\geq1$ and $1\leq m\leq n$, $G'_n$ has an $(\eta_m,\delta,v(F))$-type $(P^\type_{n,m},R_{n,m})$ on $t_m$ vertices and a clean partition $P^\clean_{n,m}$ into $t_m$ classes.
    \item We have $\eta_m\to0$ as $m\to\infty$.
    \item For fixed $m$, we have $Q^\clean_{n,m}\to Q_m$ entrywise as $n\to\infty$ for a $t_m\times t_m$ matrix $Q_m$.
    \item From \cite[Lemma~5.2(a)]{lovasz2006limits}, we have $W_{Q_m}\to W'$ a.e.\ for a graphon $W'$.
    \item For all $m\geq1$, choose an index $i(m)\geq m$ large enough such that $Q^\clean_{i(m),m}$ is entrywise within $1/m$ of $Q_m$, and let $U_m:=W_{Q^\clean_{i(m),m}}$. Then $\norm{U_m-W'}_1\to0$.
\end{enumerate}

Now let $R_m:=R_{i(m),m}$ and let $W_m$ be the graphon associated with the type $R_m$. Then $W_m$ satisfies~(\emph{iii}) by construction. Also, since $P^\clean_{i(m),m}$ is obtained from $P^\type_{i(m),m}$ by equitably redistributing the vertices of the exceptional set $V^\type_{0,i(m),m}$, each clean class differs from the corresponding non-exceptional type class by at most
$$\frac{|V^\type_{0,i(m),m}|}{t_m}+1\leq \frac{\eta_m |V(G'_{i(m)})|}{t_m}+1$$
vertices. In particular, after increasing $i(m)$ if necessary, every entry of $Q^\clean_{i(m),m}$ differs from the corresponding density between the non-exceptional classes of $P^\type_{i(m),m}$ by at most $C\eta_m$, where $C>0$ is an absolute constant, hence
$$\norm{W_m-U_m}_1\leq C\eta_m\,.$$
Since $\eta_m\to0$ and $\norm{U_m-W'}_1\to0$, it follows that
$$\norm{W_m-W'}_1\to0\,,$$
proving~(\emph{ii}). Moreover, $v(R_m)=t_m\to\infty$ by construction, so~(\emph{iv}) holds.

It only remains to verify that $\delta_\square(W_m,W)\to0$. Letting $H_m:=G'_{i(m)}$, we use the triangle inequality
$$\delta_\square(W_m,W)\leq\delta_\square(W_m,W_{H_m})+\delta_\square(W_{H_m},W)\,.$$
First, $\delta_\square(W_m,W_{H_m})\leq C\eta_m$ since $P^\type_{i(m),m}$ is an $\eta_m$-regular partition for $H_m$ (see \cite[\S9.1.2]{lovasz2012large} for this inequality). Second, $\delta_\square(W_{H_m},W)\to0$ since, by \cite[Corollary~2.6]{lovasz2006limits}, the sequence $G(n,W)$ is almost surely convergent with limit $W$, and by the standard equivalence between convergence in homomorphism density and cut metric (see \cite[Theorem~3.3]{borgs2008convergent}), this implies
$\delta_\square(W_{H_m},W)\to0$.
\end{proof}

\ActivateWarningFilters[pdftoc]
\section{Induced-$C_4$-Free Graphs of Fixed Density}
\DeactivateWarningFilters[pdftoc]
\label{sec:c4-free}
This appendix is devoted to proving \Cref{thm:c4-main}. Throughout this section, fix a constant $\gamma\in(0,1)$ and assume $m\sim\gamma\binom{n}{2}$. We use the parameter $\epsilon>0$ throughout this section. All statements are for sufficiently small $\epsilon>0$ and all large enough $n$. The proof of \Cref{thm:c4-main} is organized as follows:
\begin{enumerate}[
  label=\textit{(\roman*)},
  ref=(\textit{\roman*}),
  topsep=5pt
]
	\item First, \Cref{lemma:c4-rough-struc} shows almost all induced-$C_4$-free graphs of constant density $\gamma$ are within $\epsilon n^2$ edit distance of a split graph. This lemma follows from the solution to an entropy maximization problem and stability for colored graphs (\Cref{lemma:c4-stability}).
	\item Next, \Cref{lemma:almost-all-nondeg} shows that among the graphs $G\in\cF^\ast_{n,m}(C_4)$ that are $\epsilon n^2$-close to split, at most an exponentially small fraction have an optimal partition $\Pi=(A,B)$ whose sizes deviate by $\Theta(n)$ from an explicit optimal size that depends on $\gamma$.
	\item \Cref{lemma:FPi1,lemma:FPi2} show that among the remaining graphs, at most an exponentially small fraction have ``defect graphs'' containing a vertex of degree $\Theta(n)$ inside either of the parts $A$ or $B$ of an optimal partition $\Pi=(A,B)$.
	\item Finally, \Cref{lemma:c4-matching} is used to show that among all remaining graphs, at most a $2^{-\Omega(n)}$ fraction contain any defect edges at all, completing the proof of \Cref{thm:c4-main}.
\end{enumerate}

\subsection{Rough structure lemma}
Let us say that a graph $G$ is \emph{$t$-far from split} if the minimum edit distance from $G$ to a split graph is at least $t$. Define the sets of graphs
\[\arraycolsep=0.4mm
\begin{array}{ll}
\cF^\far &:= \{G\in\cF^\ast_{n,m}(C_4):\text{$G$ is $\epsilon n^2$-far from split}\} \,, \\[6pt]
\cF^\close &:= \cF^\ast_{n,m}(C_4)\setminus\cF^\far \,.
\end{array}\]
In this subsection we prove the following rough structure lemma.

\begin{lemma}\label{lemma:c4-rough-struc}
There exists a constant $c=c(\gamma,\epsilon)>0$ such that
$$|\cF^\far|\leq 2^{-cn^2}N^\ast_{n,m}(C_4)\,.$$
\end{lemma}

To prepare for the proof of \Cref{lemma:c4-rough-struc}, we introduce several definitions and two lemmas relating to colored graphs and entropy. Recall also the definitions set forth in \Cref{subsec:colored-graphs}.
For $\gamma\in(0,1)$ and a colored graph $J$ on $n$ vertices, let
$$f_\gamma(J):=e_{\red}(J) \cdot H\!\left(\frac{\gamma\binom{n}{2}-e_{\blue}(J)}{e_{\red}(J)}\right)\,,$$
where we use the convention that $H(x)=-\infty$ whenever $x\not\in[0,1]$. For $n\geq1$ and a graph $H$, define the optimization problem
\begin{equation}\label{eqn:coloring-entropy-0}
\Phi(n,\gamma,H) := \max\left\{f_\gamma(J):J\in\cC(n,H)\right\}\,,
\end{equation}
where $\cC(n,H)$ is defined in \Cref{subsec:colored-graphs}.
As we will see below, to prove \Cref{lemma:c4-rough-struc}, it suffices to show that near-maximizers of $\Phi(n,\gamma,H)$ are close in edit distance to a coloring that represents a split graph.

\begin{definition}[Stability]\label{dfn:Jgamma-stable}
Let $H$ be a fixed graph, $\gamma\in(0,1)$ a constant, and $\cJ(n)$ a set of colored graphs on $n$ vertices. For a colored graph $J$, let $d(J,\cJ(n))$ denote the minimum Hamming distance between $J$ and an element of $\cJ(n)$, where vertex edits incur zero cost. Assume that for all $\epsilon>0$ there exist $\delta>0$ and $\eta>0$ such that the following holds for all large enough $n$: if $J\in\cC(n,H)$ such that $e(J^c)\leq\eta n^2$ and
$$f_\gamma(J)\geq\Phi(n,\gamma,H)-\delta n^2$$
then $d(J,\cJ(n))\leq\epsilon n^2$. In this case we say $H$ is \emph{$(\cJ,\gamma)$-stable}.
\end{definition}

\Cref{lemma:c4-stability} below establishes stability for $C_4$ and is the key technical lemma needed to prove \Cref{lemma:c4-rough-struc}. We also need the following lemma.

\begin{lemma}\label{lemma:entropy-BOOST}
For all fixed $\gamma\in(0,1)$ and $\alpha\in(0,\frac{1}{2})$ there exists $\delta>0$ such that the following holds for all large enough $n$. Let $F_\gamma(R,B):=R\cdot H\big(\frac{\gamma\binom{n}{2}-B}{R}\big)$. If
$$\alpha\leq\frac{\gamma\binom{n}{2}-B}{R}\leq1-\alpha\,,\qquad R\geq\alpha n^2\,,\qquad D_r\geq\alpha n^2\,,\qquad\text{and}\qquad|D_b|\leq n\,,$$
then
$$F_\gamma(R+D_r,B+D_b)\geq F_\gamma(R,B)+\delta n^2\,.$$
\end{lemma}
\begin{proof}
Since $\frac{\partial}{\partial B}F_\gamma(R,B)=-H'\!\big(\frac{\gamma\binom{n}{2}-B}{R}\big)$, there is a constant $L=L(\alpha)>0$ such that
$$|F_\gamma(R,B+D_b)-F_\gamma(R,B)|\leq L|D_b|\leq L n\,.$$
Now let $h(x):=F_\gamma(x,B+D_b)$. Since $|D_b|\leq n$ and $R\geq\alpha n^2$, for all large enough $n$ and all $x\in[R,R+\alpha n^2]$, we have
$$\frac{\gamma\binom{n}{2}-(B+D_b)}{x}\geq\frac{\gamma\binom{n}{2}-(B+D_b)}{R+\alpha n^2}\geq\frac{(\alpha/2)R}{R+\alpha n^2}\geq\frac{\alpha}{4}\,.$$
Since
$$h'(x)=-\log\!\left(1-\frac{\gamma\binom{n}{2}-(B+D_b)}{x}\right)\,,$$
it follows that $h'(x)\geq c_\alpha:= -\log(1-\alpha/4)>0$ on $[R,R+\alpha n^2]$. Using $D_r\geq\alpha n^2$,
$$F_\gamma(R+D_r,B+D_b)-F_\gamma(R,B+D_b)\geq\int_R^{R+\alpha n^2}h'(x)\,dx\geq c_\alpha\alpha n^2\,.$$
It follows that
$$F_\gamma(R+D_r,B+D_b)-F_\gamma(R,B)\geq c_\alpha\alpha n^2-L n\,,$$
so taking $\delta:=c_\alpha\alpha/2$ completes the proof.
\end{proof}

A colored graph $J\in\cC(n)$ is called a \emph{split coloring} if $J$ is complete and there exists a partition $V(J)=A\cup B$ such that all vertices and edges in $A$ are green, all vertices and edges in $B$ are blue, and all edges between $A$ and $B$ are red. It is easy to see that $C_4\nHookrightarrow J$ for such a coloring, hence $J\in\cC(n,C_4)$. For all $n$, let $\cS(n)$ be the set of split colorings on $n$ vertices, so that $\cS(n)\subseteq\cC(n,C_4)$.

\begin{lemma}\label{lemma:c4-stability}
For all fixed $\gamma\in(0,1)$, $C_4$ is $(\cS,\gamma)$-stable.
\end{lemma}
\begin{proof}
Fix $\epsilon>0$.
Let $\alpha'\in(0,1/2)$ such that $H(\alpha')=\gamma(1-\gamma)/4$. Let $\delta_{\ref{lemma:entropy-BOOST}}$ be the constant from \Cref{lemma:entropy-BOOST} applied with $\gamma$ and $\alpha:=\min\{\frac{\epsilon\gamma(1-\gamma)}{320},\alpha'\}$. Let
$$\eta:=\left(\frac{\epsilon\gamma(1-\gamma)}{3000}\right)^2,\quad\delta:=\min\left\{\frac{\gamma(1-\gamma)}{8},\,\frac{\delta_{\ref{lemma:entropy-BOOST}}}{2}\right\}\,.$$
Let $J\in\cC(n,C_4)$ satisfy $e(J^c)\leq\eta n^2$ and
$$f_\gamma(J)\geq \Phi(n,\gamma,C_4)-\delta n^2\,.$$

Let $\sigma$ denote the coloring of $J$ and define the two vertex sets $V_\green:=\{v:\sigma(v)=\green\}$ and $V_\blue:=\{v:\sigma(v)=\blue\}$. For all $v\in V_\green$, every edge in $N_\red(v)$ is colored blue since $C_4\hookrightarrow\rrgtriangle$\,. (We use white vertices to convey that the colored homomorphism holds for any $\{\blue,\green\}$-labeling of those vertices.) Similarly, for all $v\in V_\blue$, every edge in $N_\red(v)$ is colored green since $C_4\hookrightarrow\rrbtriangle$\,. Moreover, all edges in $J[V_\green]$ are colored green since $C_4\hookrightarrow\grg$ and $C_4\hookrightarrow\gbg$; and $J[V_\blue]$ contains no red edges since $C_4\hookrightarrow\brb$.

To prove the lemma, it suffices to prove the inequalities
\begin{align}
&e_\green(V_\blue)\leq {\textstyle\frac{1}{3}}\epsilon n^2 \,, \label{eqn:many-green-in-vb}\\[3pt]
&e_\green(V_\green,V_\blue) + e_\blue(V_\green,V_\blue) \leq {\textstyle\frac{1}{3}}\epsilon n^2\,, \label{eqn:few-nonred-on-cut}
\end{align}
since by making all green edges in $V_\blue$ blue, all non-red edges between $V_\green$ and $V_\blue$ red, and forming the at most $\eta n^2$ nonedges, we can make $J$ a split coloring via at most $\epsilon n^2$ edits.

Assuming \eqref{eqn:many-green-in-vb} or \eqref{eqn:few-nonred-on-cut} fails, we will construct a coloring with entropy greater than $\Phi(n,\gamma,C_4)$, a contradiction. We introduce the construction now. Let $b^\ast$ be the greatest integer such that $\binom{b^\ast}{2}\leq e_\blue(J)$, and let $g^\ast:=n-b^\ast$. Define the complete colored graph $J^\ast$ as follows: partition the vertex set $V(K_n)$ into subsets $U_\green$ and $U_\blue$ of sizes $g^\ast$ and $b^\ast$, respectively; $J^\ast[U_\green]$ is a green clique and all its vertices are green; $J^\ast[U_\blue]$ is a blue clique and all its vertices are blue; all edges between $U_\green$ and $U_\blue$ are red. Henceforth the notation $d_c$ and $N_c$ for a color $c\in\{\red,\green,\blue\}$ refers to $J$ and not $J^\ast$. Write $g:=|V_\green|$ and $b:=|V_\blue|$.

To obtain the contradiction, we will apply \Cref{lemma:entropy-BOOST} with $R:=e_\red(J)$, $D_r:=e_\red(J^\ast)-e_\red(J)$, $B:=e_\blue(J)$, and $D_b:=e_\blue(J^\ast)-e_\blue(J)$. We will verify that the hypotheses of \Cref{lemma:entropy-BOOST} are satisfied with $\gamma$ and $\alpha$. First, from \eqref{eqn:Phi0-LB} we have
\begin{equation}\label{eqn:c4-red-lb}
e_\red(J)\geq f_\gamma(J)\geq\Phi(n,\gamma,C_4)-\delta n^2\geq\left(\frac{\gamma(1-\gamma)}{4}-\delta\right)n^2\geq\frac{\gamma(1-\gamma)}{8}\,n^2\,.
\end{equation}
The inequality $|D_b|\leq n$ holds trivially by definition of $J^\ast$. Using \eqref{eqn:c4-red-lb}, we also verify that
$$H\bigg(\frac{\gamma\binom{n}{2}-B}{R}\bigg)=\frac{f_\gamma(J)}{R}\geq\frac{\gamma(1-\gamma)}{4}\,,$$
so the symmetry of $H$ and the fact that $\alpha\leq\alpha'$ implies $\alpha\leq\frac{\gamma\binom{n}{2}-B}{R}\leq1-\alpha$. To complete the proof of the lemma, it suffices to prove $D_r\geq\alpha n^2$, since then all the hypotheses of \Cref{lemma:entropy-BOOST} are met, implying
$$f_\gamma(J^\ast)-f_\gamma(J)=F_\gamma(R+D_r,B+D_b)-F_\gamma(R,B)\geq\delta_{\ref{lemma:entropy-BOOST}} n^2\,,$$
which contradicts $f_\gamma(J)\geq\Phi(n,\gamma,C_4)-\delta_{\ref{lemma:entropy-BOOST}}n^2/2$. In the remainder, we show that the failure of \eqref{eqn:many-green-in-vb} or \eqref{eqn:few-nonred-on-cut} leads to $D_r\geq\alpha n^2$.

We will use the following observations. For all $v\in V_\green$,
$$\binom{b^\ast}{2}+n\geq e_\blue(J[N_\red(v)])\geq\binom{d_\red(v)}{2}-\eta n^2\,,$$
which implies $d_\red(v)\leq b^\ast+2\sqrt{\eta}n$ (for all large enough $n$). We calculate
\begin{equation}\label{eqn:c4-vgbeta-bd}
e_\red(J)=\sum_{v\in V_\green}d_\red(v)\leq g\cdot(b^\ast+2\sqrt{\eta}n)\,,
\end{equation}
which we combine with \eqref{eqn:c4-red-lb} to deduce $b^\ast\geq\frac{\gamma(1-\gamma)}{10}n$ and $g\geq\frac{\gamma(1-\gamma)}{10}n$.

We consider three cases depending on the discrepancy $g^\ast-g$.

\noindent\textbf{Case 1:} First, if $g^\ast\geq g+{\textstyle\frac{1}{16}}\epsilon n$ then
from \eqref{eqn:c4-vgbeta-bd} and $e_\red(J^\ast)=g^\ast b^\ast$, we deduce
\begin{align*}
e_\red(J^\ast)-e_\red(J)&\geq b^\ast(g^\ast-g)-2\sqrt{\eta}\,g\,n \\
&\geq \frac{\epsilon b^\ast n}{16}-2\sqrt{\eta}n^2\geq\frac{\epsilon\gamma(1-\gamma)}{160}n^2-\frac{\epsilon\gamma(1-\gamma)}{1500}n^2\geq\alpha n^2\,.
\end{align*}

\noindent\textbf{Case 2:} If $g\geq g^\ast+\tfrac{1}{16}\epsilon n$ then using $e_\green(J)\geq\binom{g}{2}-e(J^c)$ and $e_\green(J^\ast)=\binom{g^\ast}{2}$,
\begin{align*}
e_\green(J)-e_\green(J^\ast) &\geq \binom{g}{2}-\binom{g^\ast}{2}-e(J^c) \\
&= \frac{(g-g^\ast)(g+g^\ast-1)}{2}-e(J^c) \geq \frac{\epsilon\gamma(1-\gamma)}{320}n^2-e(J^c) \,,
\end{align*}
where we used $g-g^\ast\geq\tfrac{1}{16}\epsilon n$ and $g\geq\frac{\gamma(1-\gamma)}{10}n$. Using the identity
\begin{equation}\label{eqn:c4-hhstar-id}
\begin{aligned}
e_\red(J^\ast) &= \binom{n}{2} - e_\green(J^\ast) - e_\blue(J^\ast) \\
&= e_\red(J) + \big(e_\green(J)- e_\green(J^\ast)+ e(J^c)\big) + \big(e_\blue(J) - e_\blue(J^\ast)\big) \,,
\end{aligned}
\end{equation}
we obtain $e_\red(J^\ast)\geq e_\red(J)+\alpha n^2$.

\noindent\textbf{Case 3:} Finally assume $|g^\ast-g|\leq\tfrac{1}{16}\epsilon n$. If \eqref{eqn:many-green-in-vb} fails then comparing green edges,
$$e_\green(J)\geq\binom{g}{2}+e_\green(J[V_\blue])-\eta n^2\geq\binom{g}{2}+\frac{\epsilon n^2}{3}-\eta n^2\,.$$
We also have $e_\green(J^\ast)=\binom{g^\ast}{2}$ and
$$\binom{g^\ast}{2}-\binom{g}{2}=\frac{(g^\ast-g)(g^\ast+g-1)}{2}\leq(g^\ast-g)n\leq\frac{\epsilon n^2}{16}\,,$$
so combining the above bounds gives $e_\green(J)\geq e_\green(J^\ast)+{\textstyle\frac{1}{4}}\epsilon n^2$. Plugging this into \eqref{eqn:c4-hhstar-id} yields $e_\red(J^\ast)\geq e_\red(J)+{\textstyle\frac14}\epsilon n^2$. On the other hand, if \eqref{eqn:few-nonred-on-cut} fails then since the mapping $x\mapsto x(n-x)$ is $n$-Lipschitz on $[0,n]$,
$$|g^\ast b^\ast-gb|=|g^\ast(n-g^\ast)-g(n-g)| \leq n|g^\ast-g| \leq \tfrac{1}{16}\epsilon n^2\,,$$
which we combine with $e_\red(J)\leq gb-\frac13\epsilon n^2$ to deduce
$$D_r=g^\ast b^\ast-e_{\red}(J) \geq g^\ast b^\ast-\Big(gb-\tfrac13\epsilon n^2\Big) \geq -\tfrac{1}{16}\epsilon n^2+\tfrac13\,\epsilon n^2 = \tfrac14\epsilon n^2\,,$$
completing the proof.
\end{proof}

For a type $T$ and integers $n$ and $m$, define
$$N_{n,m}(T):=\abs{\{G\in\cG(n,m):G\text{ has }T\}}\,,$$
where \emph{having a type} is the notion defined in \Cref{dfn:type}.

\begin{definition}[Distance to a colored graph]
Let $G$ be a graph on $[n]$ and let $J=(F,\sigma)\in\cC(n)$ be a colored graph on $[n]$. Define
$$d(G,J):=\abs{\{e\in E(J):\sigma(e)=\blue,\, e\notin E(G)\}}+\abs{\{e\in E(J):\sigma(e)=\green,\, e\in E(G)\}}\,,$$
and for a subset $\cJ(n)\subseteq\cC(n)$, set $d(G,\cJ(n)):=\min\{d(G,J):J\in\cJ(n)\}$.
\end{definition}

We will use the following explicit lower bound on $\Phi(n,\gamma,C_4)$. Let $J_\spl\in\cS(n)$ be a split coloring whose blue clique is on $k:=\floor{\gamma n}$ vertices and whose green clique is on $l:=n-k$ vertices. It is easy to check that
$$\frac{\gamma\binom{n}{2}-e_\blue(J_\spl)}{e_\red(J_\spl)}=\frac{1}{2}+O_\gamma\!\left(\frac1n\right)\,,$$
which implies
\begin{equation}
\Phi(n,\gamma,C_4)\geq f_\gamma(J_\spl)\geq\frac{e_\red(J_\spl)}{2}\geq\frac{\gamma(1-\gamma)}{4}\,n^2 \label{eqn:Phi0-LB}
\end{equation}
for all large enough $n$.

\begin{proof}[Proof of \Cref{lemma:c4-rough-struc}]
We first prove a lower bound on $N^\ast_{n,m}(C_4)$ using an entropy maximizer. Let $J^\ast\in\cC(n,C_4)$ be a colored graph attaining the maximum $f_\gamma(J^\ast)=\Phi(n,\gamma,C_4)$ and write $q:=\frac{\gamma\binom{n}{2}-e_\blue(J^\ast)}{e_\red(J^\ast)}$.
By \eqref{eqn:Phi0-LB}, we have
$$e_\red(J^\ast)\,H(q)=\Phi(n,\gamma,C_4)\geq \frac{\gamma(1-\gamma)}{4}\,n^2\,.$$
In particular, $\frac{\gamma(1-\gamma)}{4}\,n^2\leq e_\red(J^\ast)\leq\binom{n}{2}$, hence $H(q)\geq \gamma(1-\gamma)/2$ and there exists $\alpha=\alpha(\gamma)\in(0,\frac12)$ such that $q\in[\alpha,1-\alpha]$ for all large enough $n$.

Now let $p:=\frac{m-e_\blue(J^\ast)}{e_\red(J^\ast)}$. Since $m=\gamma\binom{n}{2}+o(n^2)$, we have $|p-q|=o(1)$,
and therefore $p\in[0,1]$ for all large enough $n$. Consider the family $\cP(J^\ast,m)$ of (simple, uncolored) graphs $G$ on $n$ vertices and $m$ edges obtained as follows: include every blue edge of $J^\ast$, exclude every green edge of $J^\ast$, and choose the remaining $m-e_\blue(J^\ast)$ edges among the red edges of $J^\ast$. Note $\cP(J^\ast,m)$ is nonempty precisely because $p\in[0,1]$. Every $G\in\cP(J^\ast,m)$ is induced-$C_4$-free, since an induced copy of $C_4$ in $G$ would give a colored homomorphism $C_4\hookrightarrow J^\ast$, contradicting $J^\ast\in\cC(n,C_4)$. We now compute that
\begin{align}
\log N^\ast_{n,m}(C_4)&\geq \log|\cP(J^\ast,m)| \nonumber\\
&= \log\binom{e_\red(J^\ast)}{p\,e_\red(J^\ast)} \nonumber\\
&= e_\red(J^\ast)\,H(p)+o(n^2) \nonumber\\
&= e_\red(J^\ast)\,H(q)+o(n^2) \nonumber\\
&= f_\gamma(J^\ast)+o(n^2) \nonumber\\
&= \Phi(n,\gamma,C_4)+o(n^2) \,. \label{eqn:lower-bound-Nstar}
\end{align}

We now prove an upper bound for the number of graphs of a fixed type.
Using \Cref{lemma:c4-stability}, let $\delta_0>0$ and $\eta_0>0$ be the constants from $(\cS,\gamma)$-stability of $C_4$ for $\epsilon/2$.
Let $C=C(\gamma)>0$ be a sufficiently large constant. Choose $\tau\in(0,\frac12)$ so that $4\tau<\epsilon/2$ and $C\cdot H(\tau)+3\tau\leq \delta_0/4$.
Apply \Cref{lemma:type-lemma} with parameters $\tau$, $\ell:=4$, $L:=\lceil 1/\tau\rceil$, $t:=1$, and the trivial initial partition $\{V(G)\}$, and let $\eta>0$, $u\in\mathbb{N}$, and $n_1$ be the resulting constants. By reducing $\eta$ if necessary, we may assume $\eta\leq \min\{\tau,\eta_0/10\}$.
Hence every graph on at least $n_1$ vertices has an $(\eta,\tau,4)$-type $(P,R)$ with $\lceil 1/\tau\rceil\leq v(R)\leq u$. Fix $n\geq n_1$.

Let $T=(P,R)$ be such a type, and write $P=\{V_0,V_1,\dots,V_k\}$. Let $\sigma$ be the coloring of $R$. Define a colored graph $R'$ on the vertex set $[n]$ as follows:
for all $i\in[k]$, include all edges $xy\in\binom{V_i}{2}$ in $R'$, and color all vertices and edges in $V_i$ by $\sigma(i)$; for all $ij\in E(R)$ and all $(x,y)\in V_i\times V_j$, include $xy\in E(R')$ and color $xy$ by $\sigma(ij)$; all remaining pairs (i.e.\ pairs incident to $V_0$ or corresponding to $ij\notin E(R)$) are nonedges of $R'$, and vertices in $V_0$ are colored green.
Thus $R'$ is a colored graph on $[n]$ with $e((R')^c)\leq 2\eta n^2$, and a routine counting argument implies
\begin{equation}\label{eqn:type-counting}
N_{n,m}(T) \leq 2^{f_\gamma(R')+(2\eta+C\cdot H(\tau)+\tau)n^2+o(n^2)}\,.
\end{equation}
Indeed, the number of choices for edges incident to $V_0$ or between irregular pairs is at most $2^{2\eta n^2}$; for edges inside the parts $V_i$, it is at most $2^{\tau n^2}$; and for the remaining pairs, Stirling's formula gives a contribution of at most $2^{f_\gamma(R')+C\cdot H(\tau)n^2+o(n^2)}$, since green and blue pairs have densities in $[0,\tau)\cup(1-\tau,1]$, while the binomial coefficient for the red pairs has lower entry differing from $\gamma\binom{n}{2}-e_\blue(R')$ by $O(\tau n^2)$.

Since $d(G,\cS(n))$ is exactly the edit distance from $G$ to a split graph, we can write $\cF^\far=\{G\in\cF^\ast_{n,m}(C_4):d(G,\cS(n))\geq\epsilon n^2\}$.
For each $G\in\cF^\far$, fix a type $T_G=(P_G,R_G)$ for $G$ given by \Cref{lemma:type-lemma}, and let $R'_G$ be the associated colored graph on $[n]$ defined above. To pass from $G$ to a graph consistent with $R'_G$, one only needs to edit pairs incident to $V_0$, irregular pairs, regular green/blue pairs, and edges inside the parts $V_i$, hence
$$d(G,R'_G)\leq |V_0|n+\eta\binom{k}{2}\Big(\frac{n}{k}\Big)^2+\tau\binom{n}{2}+\sum_{i=1}^k \binom{|V_i|}{2}\leq 3\tau n^2\,,$$
where we used $\eta\leq\tau$ and $k\geq \lceil 1/\tau\rceil$. Since $G\in\cF^\far$, it follows that
$$d(R'_G,\cS(n))\geq d(G,\cS(n))-d(G,R'_G)\geq (\epsilon-3\tau)n^2\geq \frac{\epsilon}{2}\,n^2\,.$$
Moreover, we have $e((R'_G)^c)\leq 2\eta n^2\leq \eta_0 n^2$, and $R'_G\in\cC(n,C_4)$ since any colored homomorphism $\phi:C_4\hookrightarrow R'_G$ projects to a colored homomorphism $C_4\hookrightarrow R_G$, contradicting property \ref{item:type-colored-homo} of the type $T_G$, since $G$ is induced-$C_4$-free. Therefore, the $(\cS,\gamma)$-stability of $C_4$ from \Cref{lemma:c4-stability} implies
\begin{equation}\label{eqn:entropy-gap}
f_\gamma(R'_G) \leq \Phi(n,\gamma,C_4)-\delta_0 n^2\,.
\end{equation}

Now group graphs in $\cF^\far$ by their type $T=(P,R)$. The number of possible partitions $P$ and
colored graphs $R$ with $k\leq u$ is at most $\exp(O(n\log n))=2^{o(n^2)}$. Hence, combining
\eqref{eqn:type-counting} and \eqref{eqn:entropy-gap} and using the choice of $\tau$, we obtain
$$|\cF^\far| \leq 2^{\Phi(n,\gamma,C_4)-\frac{1}{2}\delta_0 n^2+o(n^2)}\,.$$
Together with \eqref{eqn:lower-bound-Nstar}, this completes the proof.
\end{proof}

\subsection{Final counting argument}
For all $G\in\cF^\ast_{n,m}(C_4)$ and all (ordered) partitions $\Pi=(A,B)$ of $V$, define
$$b(G,\Pi):=e(G[A])+e(G^c[B])$$
and let $\Pi(G)$ be a canonically chosen bipartition minimizing $b(G,\cdot)$. If $\Pi(G)=(A,B)$ then let $T(G)$ be the graph on $V$ with edge set
$$E(T(G)):=E(G[A])\cup E(G^c[B])\,.$$
Define the parameters
$$\zeta:=\epsilon^{1/20}\,,\qquad \alpha:=\zeta^4\,,$$
and define the function
$$\psi_\gamma(x):=x(1-x)\,H\!\left(\frac{\gamma-x^2}{2x(1-x)}\right)$$
for all $x\in[1-\sqrt{1-\gamma},\,\sqrt{\gamma}]$. A short calculus argument shows that $\psi_\gamma$ has a unique maximizer $\lambda_\gamma\in(0,1)$. For $\zeta>0$, let us say that a partition $(A,B)$ is \emph{$\zeta$-degenerate} if $||B|-\lambda_\gamma n|>\zeta n$, and \emph{$\zeta$-nondegenerate} otherwise. Let us also say that a graph $G\in\cF^\close$ is $\zeta$-degenerate if its optimal partition $\Pi(G)$ is $\zeta$-degenerate. Define the set of graphs
$$\cF^\degenerate:=\{G\in\cF^\close:G\text{ is $\zeta$-degenerate}\}\,.$$
Let $S_{n,m}$ be the number of split graphs on $n$ vertices and $m$ edges. Let us write $\sP$ for the set of all $\zeta$-nondegenerate partitions $(A,B)$ of the vertex set $V$. For all partitions $\Pi=(A,B)$, define the sets of graphs
\[\arraycolsep=0.4mm
\begin{array}{ll}
\cF_\Pi &:= \{G\in\cF^\close:\Pi(G)=\Pi\} \,, \\[4pt]
\cF^\ast_\Pi &:= \{G\in\cF^\close:\Pi(G)=\Pi,\,b(G,\Pi(G))=0\} \,, \\[4pt]
\cT_\Pi &:= \{T(G):G\in\cF_\Pi\} \,, \\[4pt]
\cF_{\Pi,T} &:= \{G\in\cF_\Pi:T(G)=T\} \,,
\end{array}\]
and write
$$s_\Pi:=\binom{|A|\cdot|B|}{m-e(K_B)}\,.$$

\begin{lemma}\label{lemma:almost-all-nondeg}
There exists a constant $c=c(\gamma,\epsilon)>0$ such that
$$|\cF^\degenerate|\leq 2^{-cn^2}\,N^\ast_{n,m}(C_4)\,.$$
\end{lemma}
\begin{proof}
Fix a partition $(A_0,B_0)$ with $|B_0|=\floor{\lambda_\gamma n}$ and consider all graphs $G$ with $G[A_0]$ empty, $G[B_0]$ complete, and exactly $m-\binom{|B_0|}{2}$ edges between $A_0$ and $B_0$. Clearly, every such graph is split, and hence induced-$C_4$-free. Writing $p_\gamma:=\frac{\gamma-\lambda_\gamma^2}{2\lambda_\gamma(1-\lambda_\gamma)}$,
we have
$$\frac{m-\binom{|B_0|}{2}}{|A_0|\cdot|B_0|}=p_\gamma+o(1)\,,$$
and therefore
\begin{equation}\label{eqn:c4-split-lb}
N^\ast_{n,m}(C_4)\geq \binom{|A_0|\cdot|B_0|}{m-\binom{|B_0|}{2}}
=2^{\psi_\gamma(\lambda_\gamma)n^2+o(n^2)}\,.
\end{equation}

Now fix a $\zeta$-degenerate partition $\Pi=(A,B)$ and write $b:=|B|/n$. For every graph $G\in\cF_\Pi$ we have $e(T(G))\leq\epsilon n^2$, hence the number of possibilities for the defect graph $T(G)$ is at most $\sum_{t=0}^{\epsilon n^2}\binom{\binom{n}{2}}{t}\leq 2^{H(\epsilon)n^2}$. For each fixed $T$, writing $t:=e(T[A])-e(T[B])$, a graph $G$ with defect graph $T(G)=T$ is determined by its edges across the cut $(A,B)$, which implies
$$|\cF_{\Pi,T}|\leq \binom{|A|\cdot|B|}{m-\binom{|B|}{2}-t}\,.$$
Since $|t|\leq\epsilon n^2$, it follows from \cite[Lemma~A.1(ii)]{perkins2025typical} that $|\cF_{\Pi,T}|\leq 2^{(\psi_\gamma(b)+c\epsilon)n^2+o(n^2)}$
for some constant $c=c(\gamma)>0$. Hence
$$|\cF_\Pi|\leq 2^{(\psi_\gamma(b)+c\epsilon+H(\epsilon))n^2+o(n^2)}\,.$$
Since $\lambda_\gamma$ is the unique maximizer of $\psi_\gamma$ and $|b-\lambda_\gamma|\geq\zeta$, there exists a constant $c'=c'(\gamma,\zeta)>0$ such that $\psi_\gamma(b)\leq\psi_\gamma(\lambda_\gamma)-2c'$. Taking $\epsilon>0$ sufficiently small gives
$$|\cF_\Pi|\leq 2^{(\psi_\gamma(\lambda_\gamma)-c')n^2+o(n^2)}\,.$$
There are at most $2^n$ $\zeta$-degenerate partitions, so summing over partitions gives
$$|\cF^\degenerate|\leq 2^{(\psi_\gamma(\lambda_\gamma)-c')n^2+o(n^2)}\,.$$
Comparing with \eqref{eqn:c4-split-lb} implies the statement of the lemma.
\end{proof}

\begin{remark}[Consequences of nondegeneracy]\label{rem:c4-nondeg-bounds}
Since $\lambda_\gamma\in(0,1)$ is fixed in terms of $\gamma$, taking $\epsilon>0$ sufficiently small implies that every $\Pi=(A,B)\in\sP$ satisfies $|A|,|B|=\Theta_\gamma(n)$. Additionally, there exists a constant $\beta=\beta(\gamma)\in(0,\frac12)$ such that for all large enough $n$, if $\Pi=(A,B)\in\sP$, $|t|\leq \epsilon n^2+n$, and $0\leq z\leq n$ then
\begin{equation}\label{eqn:c4-p-bounded}
\beta\leq \frac{m-\binom{|B|}{2}-t}{|A|\cdot|B|-z}\leq1-\beta\,,
\end{equation}
and the same bound also holds if one subtracts $z$ from the numerator (by decreasing $\beta$).
\end{remark}

The following lemma quantifies how the number of graphs in $\cF_{\Pi,T}$ is constrained if the defect graph $T$ contains a matching of a given size.

\begin{lemma}\label{lemma:c4-matching}
There exists a constant $c=c(\gamma)>0$ such that the following holds. Let $\Pi=(A,B)\in\sP$ and $T\in\cT_\Pi$. If $T[A]$ or $T[B]$ has a matching of size $k\geq1$ then
$$|\cF_{\Pi,T}|\leq 2^{-ckn}\binom{|A|\cdot|B|}{m-e(K_B)-e(T[A])+e(T[B])}\,.$$
\end{lemma}
\begin{proof}
Define $t:=e(T[A])-e(T[B])$ and $m':=m-e(K_B)-t$. Let $H'$ be the uniformly random bipartite subgraph of $K_{A,B}$ with $m'$ edges, and let $H''$ be the random bipartite subgraph of $K_{A,B}$ where each edge is included independently with probability $p:=\frac{m'}{|A|\cdot|B|}$. Let $G'$ and $G''$ be the random graphs on $V$ with edge sets
$$E(G')=(E(H')\cup K_B)\,\triangle\,E(T)\,,\qquad E(G'')=(E(H'')\cup K_B)\,\triangle\,E(T)\,,$$
where $\triangle$ is the symmetric difference. Using \cite[Lemma~A.4]{perkins2025typical} to relate $G'$ and $G''$, we have
\begin{equation}\label{eqn:c4-mat-init}
|\cF_{\Pi,T}|\leq\bbP\{G'\in\cF^\ast_n(C_4)\}\cdot\binom{|A|\cdot|B|}{m'}\leq\sqrt{2p}\,n\,\bbP\{G''\in\cF^\ast_n(C_4)\}\cdot\binom{|A|\cdot|B|}{m'}\,.
\end{equation}
In the remainder, we bound $\bbP\{G''\in\cF^\ast_n(C_4)\}$ by the probability $G''$ avoids a specific family $\cK$ of copies of $C_4$.

Let $M\subseteq T$ be a matching of size $k$ in $T[A]$ or $T[B]$. First assume $M\subseteq T[A]$. By \Cref{rem:c4-nondeg-bounds}, we have $|A|,|B|=\Theta_\gamma(n)$, and therefore for all $G\in\cF_{\Pi,T}$,
$$\min\{e(G^c[A]),e(G[B])\}\geq c n^2$$
for some constant $c=c(\gamma)>0$ since $e(T)\leq\epsilon n^2$ and $\epsilon>0$ is small enough. Using \cite[Lemma~A.3]{perkins2025typical}, this implies each of the graphs $G^c[A]$ and $G[B]=K_B\setminus T[B]$ has a matching on at least $cn$ edges. Fix a matching $M'\subseteq K_B\setminus T[B]$ with at least $cn$ edges. For each $e=\{x_1,x_2\}\in E(M)$ and $f=\{y_1,y_2\}\in E(M')$, write
$x_1<x_2$ and $y_1<y_2$, and let $K_{e,f}$ be the copy of $C_4$ on
$\{x_1,x_2,y_1,y_2\}$ with edge set $E(K_{e,f})=\{e,\,f,\,x_1y_1,\,x_2y_2\}$. Define
$$\cK:=\{K_{e,f}:e\in E(M),\,f\in E(M')\}$$
and note that $|\cK|=|M|\cdot|M'|\geq ckn$. For all $K\in\cK$, if we define the event $E_K:=\{G''[V(K)]\cong K\}$ then since the events $\{E_K:K\in\cK\}$ are independent and \eqref{eqn:c4-p-bounded} gives $p\in[\beta,1-\beta]$, we have
\begin{equation}\label{eqn:G''-janson-mat}
\begin{aligned}
\bbP\{G''\in\cF^\ast_n(C_4)\}\leq\bbP\!\left\{\bigcap_{K\in\cK}E_K^c\right\} &\leq \prod_{K\in\cK}\bbP\{E_K^c\} \\
&\leq (1-\beta^2(1-\beta)^2)^{|\cK|}\leq2^{-c|\cK|}\leq2^{-ckn}\,.
\end{aligned}
\end{equation}
By combining \eqref{eqn:c4-mat-init} and \eqref{eqn:G''-janson-mat}, this gives the bound in the case $V(M)\subseteq A$.

If instead $M\subseteq T[B]$ then the same argument applies with $G^c[A]$ in place of $G[B]$: fix a matching $M'\subseteq K_A\setminus T[A]$ of linear size and, for each $e\in E(M)$ and $f\in E(M')$, use the induced copy of $C_4$ on the four vertices whose non-edges are $e$ and $f$. Using \eqref{eqn:c4-p-bounded}, this again yields the $2^{-ckn}$ bound, possibly with a different constant $c=c(\gamma)>0$.
\end{proof}

For all $\Pi=(A,B)\in\sP$ define the sets of graphs
\[\arraycolsep=0.4mm
\begin{array}{ll}
\cF^\high_{\Pi,1} &:= \{G\in\cF_\Pi:\exists\,v\in A\,,\,d_{T(G)}(v)\geq\alpha|A|\}\,, \\[7pt]
\cF^\high_{\Pi,2} &:= \{G\in\cF_\Pi:\exists\,v\in B\,,\,d_{T(G)}(v)\geq\alpha|B|\}\,.
\end{array}\]

\begin{lemma}\label{lemma:FPi1}
There exists a constant $c=c(\gamma)>0$ such that for all $\Pi\in\sP$,
$$\abs{\cF^\high_{\Pi,1}}\leq2^{-c\alpha^2n^2}s_\Pi\,.$$
\end{lemma}
\begin{proof}
In this proof, $c=c(\gamma)>0$ is a constant that may change from line to line. Let $T\in\cT_\Pi$ and assume there is a vertex $v$ (canonically chosen for each $T$) such that $d_T(v)\geq\alpha|A|$. Let us denote $N:=N_T(v)$. We split into two cases.

\smallskip
\noindent\textbf{Case 1.}
First assume $e_T(N)\geq|N|^2/4$. Then using \cite[Lemma~A.3]{perkins2025typical}, $T[N]$ has a matching of size at least $|N|/32\geq c\alpha n$. Hence by \Cref{lemma:c4-matching},
\begin{align}
|\cF_{\Pi,T}| &\leq 2^{-c\alpha n^2}\binom{|A|\cdot|B|}{m-e(K_B)-e(T[A])+e(T[B])} \nonumber\\
&\leq 2^{-c\alpha n^2+c\epsilon n^2}\binom{|A|\cdot|B|}{m-e(K_B)}
=2^{-c\alpha n^2+c\epsilon n^2}s_\Pi \leq2^{-c\alpha^2n^2}s_\Pi\,, \label{eqn:fhipi1-case1-bd}
\end{align}
where in the second inequality we used \cite[Lemma~A.1(ii)]{perkins2025typical}.
\smallskip

\noindent\textbf{Case 2.} Now assume $e_T(N)\leq|N|^2/4$. For all $G\in\cF_{\Pi,T}$ and $v\in A$, the optimality of $\Pi$ implies $d_G^c(v,B)\geq d_G(v,A)$, so let $Z=Z(G)\subseteq B$ be a canonically chosen subset of size at least $d_G(v,A)$ such that $Z\subseteq N^c_G(v,B)$. Let us write $\cF_{\Pi,T,v,Z}$ for the set of all $G\in\cF_{\Pi,T}$ for which $v$ and $Z$ are exactly as chosen above.

Let $t:=e(T[A])-e(T[B])$ and $m':=m-e(K_B)-t$. Let $H'$ be the random graph on $V$ whose edge set is a uniformly random choice of $m'$ edges from $E(K_{A,B}\setminus K_{v,Z})$. Let $H''$ be the random graph on $V$ in which each edge $e\in E(K_{A,B}\setminus K_{v,Z})$ is included in $H''$ independently with probability $p:=\frac{m'}{|A|\cdot|B|-|Z|}$. Let $G'$ be the random graph $(H'\cup K_B)\,\triangle\,T$ and let $G''$ be the random graph $(H''\cup K_B)\,\triangle\,T$. By definition, we have $e(G')=m$. Using \cite[Lemma~A.4]{perkins2025typical} to relate $G'$ and $G''$, we have
\begin{align}
|\cF_{\Pi,T,v,Z}| &\leq \bbP\{G'\in\cF^\ast_n(C_4)\}\cdot\binom{|A|\cdot|B|-|Z|}{m'} \nonumber\\
&\leq \sqrt{2p}\,n\,\bbP\{G''\in\cF^\ast_n(C_4)\}\cdot\binom{|A|\cdot|B|-|Z|}{m'} \nonumber\\
&\leq \bbP\{G''\in\cF^\ast_n(C_4)\} \cdot 2^{c\epsilon n^2}\,s_\Pi \,,\label{eqn:G''-PiTvZ-bd}
\end{align}
where in the third inequality we used \cite[Lemma~A.1(i),(ii),(iii)]{perkins2025typical}.
We now bound $\bbP\{G''\in\cF^\ast_n(C_4)\}$ by the probability $G''$ avoids a specific set of copies of $C_4$. For each $xy\in E(T^c[N])$ and $z\in Z$, let $K_{xy,z}$ be the copy of $C_4$ on $\{v,x,y,z\}$ with edge set $\{vx,vy,xz,yz\}$. Define the set of copies
$$\cK:=\{K_{xy,z}:xy\in E(T^c[N])\,,\, z\in Z\}\,,$$
so that, using the choice of $Z$ and $\Pi\in\sP$,
$$|\cK|=e(T^c[N])\cdot|Z|\geq\left(\binom{|N|}{2}-\frac{|N|^2}{4}\right)|Z|\geq \frac{|N|^2|Z|}{8}\,.$$
If we define the event $E_K:=\{G''[V(K)]\cong K\}$ for all $K\in\cK$ then $G''\in\cF^\ast_n(C_4)$ implies $\bigcap_{K\in\cK}E_K^c$. Since distinct events $E_{K_1}$ and $E_{K_2}$ are dependent only if the corresponding copies share the same vertex in $Z$ and share one endpoint in $N$, we have, using the notation of \Cref{thm:JansonsInequality},
\begin{align*}
\mu &:= \sum_{K\in\cK}\bbP\{E_K\} \geq p^2\,|\cK| \geq p^2\cdot\frac{|N|^2|Z|}{8} \,, \\
\Delta &:= \sum_{K_1\sim K_2}\bbP\{E_{K_1}\cap E_{K_2}\}\leq p^3\cdot|N|^3\cdot|Z|\,,
\end{align*}
which implies
$$\frac{\mu^2}{\Delta}\geq \frac{p^4|\cK|^2}{p^3|N|^3|Z|}\geq c\frac{|N|^4|Z|^2}{|N|^3|Z|}\geq c\alpha^2n^2\,.$$
Applying \Cref{thm:JansonsInequality}, it follows that
\begin{equation}\label{eqn:G''-pitvz-janson}
\bbP\{G''\in\cF^\ast_n(C_4)\} \leq \bbP\!\left\{\bigcap_{K\in\cK}E_K^c\right\} \leq \exp\!\left(-\min\left\{\frac{\mu}{2}\,,\,\frac{\mu^2}{4\Delta}\right\}\right) \leq e^{-c\alpha^2n^2}\,.
\end{equation}

As in the proof of \Cref{lemma:almost-all-nondeg}, the number of possibilities for $T$ is at most $2^{H(\epsilon)n^2}$. Also, the number of possibilities for $v$ and $Z$ is at most $n2^n$. Hence, the result follows by combining \eqref{eqn:fhipi1-case1-bd}, \eqref{eqn:G''-PiTvZ-bd}, and \eqref{eqn:G''-pitvz-janson}, summing over $T\in\cT_\Pi$, $v\in V$, and $Z\subseteq V$, and taking $\epsilon>0$ sufficiently small.
\end{proof}

\begin{lemma}\label{lemma:FPi2}
There exists a constant $c=c(\gamma)>0$ such that for all $\Pi\in\sP$,
$$\abs{\cF^\high_{\Pi,2}}\leq2^{-c\alpha^3n^2}s_\Pi\,.$$
\end{lemma}
\begin{proof}
Let $T\in\cT_\Pi$ and assume there is a vertex $v$ (canonically chosen for each $T$) such that $d_T(v)\geq\alpha|B|$. Let us denote $N:=N_T(v)$. For all $G\in\cF_{\Pi,T}$ and $v\in B$, the optimality of $\Pi$ implies $d_G(v,A)\geq d_G^c(v,B)$, so let $Z=Z(G)\subseteq A$ be a canonically chosen subset of size at least $d_G^c(v,B)$ such that $Z\subseteq N_G(v,A)$. Let us write $\cF_{\Pi,T,v,Z}$ for the set of all $G\in\cF_{\Pi,T}$ for which $v$ and $Z$ are exactly as chosen above.

Let $t:=e(T[A])-e(T[B])$ and $m':=m-e(K_B)-t-|Z|$. Let $H'$ be the random graph on $V$ whose edge set is a uniformly random choice of $m'$ edges from $E(K_{A,B}\setminus K_{v,Z})$. Let $H''$ be the random graph on $V$ in which each edge $e\in E(K_{A,B}\setminus K_{v,Z})$ is included in $H''$ independently with probability $p:=\frac{m'}{|A|\cdot|B|-|Z|}$. Let $G'$ be the random graph $(H'\cup K_B\cup K_{v,Z})\,\triangle\,T$ and let $G''$ be the random graph $(H''\cup K_B\cup K_{v,Z})\,\triangle\,T$. By definition, we have $e(G')=m$. Using \cite[Lemma~A.4]{perkins2025typical} to relate $G'$ and $G''$, we have
\begin{align}
|\cF_{\Pi,T,v,Z}| &\leq \bbP\{G'\in\cF^\ast_n(C_4)\}\cdot\binom{|A|\cdot|B|-|Z|}{m'} \nonumber\\
&\leq \sqrt{2p}\,n\,\bbP\{G''\in\cF^\ast_n(C_4)\}\cdot\binom{|A|\cdot|B|-|Z|}{m'} \nonumber\\
&\leq \bbP\{G''\in\cF^\ast_n(C_4)\} \cdot 2^{c\epsilon n^2}\,s_\Pi \,, \label{eqn:G''-PiTvZ-bd-2}
\end{align}
where in the third inequality we used \cite[Lemma~A.1(i),(ii),(iii)]{perkins2025typical}.
It remains to bound $\bbP\{G''\in\cF^\ast_n(C_4)\}$. Since $|Z|\geq|N|\geq\alpha|B|\geq c\alpha n$ and $e(T)\leq\epsilon n^2$, we have
$$e_{T^c}(Z)\geq\binom{|Z|}{2}-\epsilon n^2\geq c\alpha^2n^2$$
for every graph $G\in\cF_{\Pi,T}$. Hence by \cite[Lemma~A.3]{perkins2025typical}, there exists a matching $M=M(T)\subseteq T^c[Z]$ of size at least $c\alpha^2n$. For all $x\in N$ and $yz\in E(M)$, let $K_{yz,x}$ be the copy of $C_4$ on $\{v,x,y,z\}$ with edge set $\{vy,vz,xy,xz\}$. Define the set of copies
$$\cK:=\{K_{yz,x}:yz\in E(M),\,x\in N\}\,,$$
so that
$$|\cK|=e(M)\cdot |N|\geq c\alpha^3n^2\,.$$
If we write $E_K:=\{G''[V(K)]\cong K\}$ for all $K\in\cK$ then again $G''\in\cF^\ast_n(C_4)$ implies $\bigcap_{K\in\cK}E_K^c$. Moreover, the events $\{E_K:K\in\cK\}$ are independent since each of these events depends on a distinct pair of random cross edges. Using \Cref{rem:c4-nondeg-bounds},
\begin{align*}
\bbP\{G''\in\cF^\ast_n(C_4)\} &\leq \bbP\!\left\{\bigcap_{K\in\cK}E_K^c\right\}=\prod_{K\in\cK}\bbP\{E_K^c\} \leq (1-\beta^2)^{|\cK|} \leq 2^{-c|\cK|}\leq 2^{-c\alpha^3n^2}\,.
\end{align*}
As in \Cref{lemma:FPi1}, the result now follows from \eqref{eqn:G''-PiTvZ-bd-2} by summing over the possibilities for $T$, $v$, and $Z$.
\end{proof}

\begin{proof}[Proof of \Cref{thm:c4-main}]
In this proof, the constant $c>0$ depends on both $\gamma$ and $\epsilon$ and may change from line to line. For all $G\in\cF^\close$, if $\Pi(G)=(A,B)$ then let $M(G)\subseteq T(G)$ be a canonically-chosen maximum matching subject to either $V(M(G))\subseteq A$ or $V(M(G))\subseteq B$. For all $\Pi\in\sP$, $k\geq1$, and graphs $T$, let us define
\[\arraycolsep=0.4mm
\begin{array}{ll}
\cF^\mat_{\Pi,k} &:= \{G\in\cF_\Pi\setminus(\cF^\high_{\Pi,1}\cup\cF^\high_{\Pi,2}):e(M(G))=k\} \,, \\[8pt]
\cT^\mat_{\Pi,k} &:= \{T(G):G\in\cF^\mat_{\Pi,k}\} \,.
\end{array}\]
Using the definitions of all the sets in this section, we have
\begin{equation}\label{eqn:c4-union-bound}
\cF^\ast_{n,m}(C_4)\subseteq\bigcup_{\Pi\in\sP}\Bigg(\cF^\ast_\Pi\cup\bigcup_{k\geq1}\cF^\mat_{\Pi,k}\cup\cF^\high_{\Pi,1}\cup\cF^\high_{\Pi,2}\Bigg)\cup\cF^\degenerate\cup\cF^\far\,.
\end{equation}
For all $\Pi\in\sP$, $k\geq1$, and $T\in\cT^\mat_{\Pi,k}$, let $M=M(T)$ be a canonically chosen maximum matching of $T$ subject to $V(M)\subseteq A$ or $V(M)\subseteq B$. Since $M$ is maximum among matchings contained in $A$ or in $B$, both graphs $T[A]$ and $T[B]$ have matching number at most $k$, and hence each has a vertex cover of size at most $2k$. Using additionally that $\Delta(T)\leq\alpha n$, it follows that
$$e(T)\leq e(T[A])+e(T[B])\leq 4k\alpha n\,.$$
Using \Cref{lemma:c4-matching}, we calculate that
\begin{align*}
|\cF_{\Pi,T}| &\leq 2^{-ckn}\binom{|A|\cdot|B|}{m-e(K_B)-e(T[A])+e(T[B])} \\
&\leq 2^{-ckn+c\alpha kn}\binom{|A|\cdot|B|}{m-e(K_B)}
=2^{-ckn+c\alpha kn}s_\Pi \leq 2^{-ckn}s_\Pi \,,
\end{align*}
where in the second inequality we used \cite[Lemma~A.1(ii)]{perkins2025typical}.
Again using the fact that both $T[A]$ and $T[B]$ have vertex covers of size at most $2k$, we have
$$|\cT^\mat_{\Pi,k}|\leq\sum_{X\in\binom{V}{\leq4k}}\prod_{v\in X}\binom{n}{\leq\alpha n}\leq\binom{n}{\leq4k}2^{4k\cdot H(\alpha)n}\leq 2^{5H(\alpha)kn}\,,$$
where we used $\binom{n}{\leq\alpha n}\leq2^{H(\alpha)n}$.
Summing over $k\geq1$ and $T\in\cT^\mat_{\Pi,k}$, we obtain
\begin{equation}\label{eqn:c4-fmat-sum}
\begin{aligned}
\sum_{k\geq1}|\cF^\mat_{\Pi,k}|\leq\sum_{k\geq1}\sum_{T\in\cT^\mat_{\Pi,k}}|\cF_{\Pi,T}| &\leq \sum_{k\geq1}|\cT^\mat_{\Pi,k}|\cdot2^{-ckn}s_\Pi \\
&\leq s_\Pi\sum_{k\geq1}2^{-ckn+5H(\alpha)kn} \leq 2^{-cn}\,s_\Pi \,.
\end{aligned}
\end{equation}
By the same estimate as used in the proof of \Cref{lemma:almost-all-nondeg} with $T=0$, we have
\begin{equation*}
\sum_{\Pi\notin\sP}\binom{|A|\cdot|B|}{m-e(K_B)}\leq 2^{-cn^2}S_{n,m}\,,
\end{equation*}
so using the fact that all but at most a $2^{-cn}$ fraction of split graphs have a unique cover by a clique and independent set, it is standard to show
\begin{equation}\label{eqn:c4-sPi-sum}
\sum_{\Pi\in\sP}s_\Pi\leq(1+2^{-cn})\,S_{n,m}\,.
\end{equation}

Combining \eqref{eqn:c4-union-bound}, \eqref{eqn:c4-fmat-sum}, \eqref{eqn:c4-sPi-sum}, \Cref{lemma:c4-rough-struc,lemma:almost-all-nondeg,lemma:FPi1,lemma:FPi2}, and using $\sum_{\Pi\in\sP}|\cF^\ast_\Pi|\leq S_{n,m}$, we calculate
\begin{align*}
N^\ast_{n,m}(C_4) &\leq \sum_{\Pi\in\sP}\Bigg(|\cF^\ast_\Pi|+\sum_{k\geq1}|\cF^\mat_{\Pi,k}|+|\cF^\high_{\Pi,1}|+|\cF^\high_{\Pi,2}|\Bigg)+\abs{\cF^\degenerate}+\abs{\cF^\far} \\
&\leq S_{n,m} + \sum_{\Pi\in\sP}\Big(2^{-cn}+2^{-c\alpha^2n^2}+2^{-c\alpha^3n^2}\Big)s_\Pi +2\cdot2^{-cn^2}\,N^\ast_{n,m}(C_4) \\
&\leq S_{n,m} + \Big(2^{-cn}+2^{-c\alpha^2n^2}+2^{-c\alpha^3n^2}\Big)(1+2^{-cn})\,S_{n,m} + 2^{-cn^2}\,N^\ast_{n,m}(C_4) \\
&\leq (1+2^{-cn})\,S_{n,m}+2^{-cn^2}\,N^\ast_{n,m}(C_4) \,.
\end{align*}
It follows that $N^\ast_{n,m}(C_4)=(1+O(2^{-cn}))\cdot S_{n,m}$, proving the theorem.

To deduce \eqref{eqn:c4-entropy-main}, for each $b\in\{0,\dots,n\}$ let $S_{n,m}(b):=\binom{n}{b}\binom{b(n-b)}{m-\binom{b}{2}}$, where the second binomial coefficient is defined as $0$ if its lower entry does not lie in $[0,b(n-b)]$. We have
$$\max_{0\leq b\leq n}\binom{b(n-b)}{m-\binom{b}{2}}\leq S_{n,m}\leq\sum_{b=0}^n S_{n,m}(b)\,,$$
and \eqref{eqn:c4-entropy-main} follows by using Stirling's formula, $m\sim\gamma\binom{n}{2}$, and $N^\ast_{n,m}(C_4)=(1+o(1))\cdot S_{n,m}$.
\end{proof}

\end{document}